\documentclass[11pt]{gsm-l}

\usepackage{extsizes}

\usepackage{tikz}
\usepackage{rotating}

\usepackage{enumitem}
\setlist[itemize]{topsep=0pt,itemsep=0pt}
\setlist[enumerate]{topsep=0pt,itemsep=0pt}

\usepackage{scalerel}
\usepackage{tkz-euclide}
\usetikzlibrary{arrows, arrows.meta, bending, decorations.pathmorphing, decorations.markings} % hobby} }

\usepackage{auto-pst-pdf}
\usepackage{hyperref}

\usepackage{pict2e, amssymb,latexsym,amstext,amsmath,amsthm,epsfig,ifthen}
\usepackage{color}
\usepackage{blkarray}
\usepackage[capitalize, nameinlink, noabbrev, nosort]{cleveref}
\usepackage{youngtab}
\usepackage[boxsize=.35 em]{ytableau}

\usepackage[dvips,centering,includehead,width=12.61cm,height=22.1cm,paperheight=23.8cm,paperwidth=17cm]{geometry}

\input epsf

\makeatletter
\newcommand{\manuallabel}[2]{\phantomsection\def\@currentlabel{#2}\label{#1}}
\makeatother

\def\LW#1{(\textcolor{blue}{LW:#1})}

\newtheorem{theorem}{Theorem}[section]
\newtheorem{proposition}[theorem]{Proposition}

\newtheorem{corollary}[theorem]{Corollary}
\newtheorem{lemma}[theorem]{Lemma}

\theoremstyle{definition}

\newtheorem{remark}[theorem]{Remark}
\newtheorem{example}[theorem]{Example}
\newtheorem{definition}[theorem]{Definition}

\newtheorem{exercise}[theorem]{Exercise}

\numberwithin{section}{chapter}
\numberwithin{equation}{section}
\numberwithin{figure}{chapter}
\numberwithin{table}{chapter}

\DeclareMathOperator{\aff}{aff}
\DeclareMathOperator{\algn}{align}

\def\ZZ{\mathbb{Z}}

\def\source{\operatorname{source}}
\def\target{\operatorname{target}}

\def\Gr{\operatorname{Gr}}

\newcommand{\light}[1]{{#1}}
\newcommand{\dark}[1]{{#1}}

\newcommand{\white}[1]{{\color{white} #1}}
\newcommand{\red}[1]{{\color{red} #1}}
\newcommand{\yellow}[1]{{\color{yellow} #1}}

\definecolor{darkred}{rgb}{1,0,0}        
\definecolor{lightred}{rgb}{1,0.4,0}     
\definecolor{darkblue}{cmyk}{1,0.4,0,0.4}  %blue
\definecolor{lightblue}{cmyk}{1,0.4,0,0}  %blue
\definecolor{darkgreen}{cmyk}{1,0.5,1,0}  
\definecolor{lightgreen}{cmyk}{1,0,1,0}  
\newcommand{\dpi}{{\widetilde{\pi}}}
\def\Bip{\operatorname{Bip}}

\newcommand{\notch}{\scriptstyle\bowtie}

\newsavebox{\digon}
\savebox{\digon}(10,10)[bl]{
\qbezier(5,0)(0,5)(5,10)
\qbezier(5,0)(10,5)(5,10)
\put(5,0){\circle*{1}} 
\put(5,5){\circle*{1}} 
\put(5,10){\circle*{1}} 
}

\newsavebox{\lowbar}
\savebox{\lowbar}(10,10)[bl]{
\put(5,0){\line(0,1){5}}
}

\newsavebox{\lowtag}
\savebox{\lowtag}(10,10)[bl]{
\put(5,3.5){\makebox(0,0){$\notch$}}
}

\newsavebox{\highbar}
\savebox{\highbar}(10,10)[bl]{
\put(5,5){\line(0,1){5}}
}

\newsavebox{\hightag}
\savebox{\hightag}(10,10)[bl]{
\put(5,6.5){\makebox(0,0){$\notch$}}
}

\makeindex

\begin{document}

%%%%%%%%%%%%% mock manual labels temporarily created to verify the actual ones
%%%%%%%%%%%%% comment out all chapters in \includeonly above, 
%%%%%%%%%%%%% then run the commands below for all labels to see their values 

%\makeatletter
%\newcommand{\mockmanuallabel}[2]{\ref{#1}\ \ #2\ \ \def\@currentlabel{#2}\label{#1}\par}
%\makeatother

%RESTORE FRONTMATTER FOR FINAL VERSION
%\iffalse

\frontmatter

\title{
\textsf{\Huge \hbox{Introduction to Cluster Algebras}}\\[.1in]
\textsf{\Huge Chapter 7} \\[.2in]
{\rm\textsf{\LARGE (preliminary version)}}
}

\author{\Large \textsc{Sergey Fomin}}
%\address{Department of Mathematics, 
%University of Michigan,
%Ann Arbor, MI 48109} 
%\email{fomin@umich.edu}

\author{\Large \textsc{Lauren Williams}}
%\address{Department of Mathematics, University of California, Berkeley, CA 94720} 
%\email{williams@math.berkeley.edu}

\author{\Large \textsc{Andrei Zelevinsky}}
%\address{Department of Mathematics, Northeastern University, Boston, MA 02115}
%\email{andrei@neu.edu}

\maketitle

\noindent
\textbf{\Huge Preface}

\vspace{1in}

\manuallabel{ch:tp-examples}{1}
\manuallabel{ch:combinatorics-of-mutations}{2}
\manuallabel{def:graphquiver}{2.5.1}
\manuallabel{urban}{2.5}
\manuallabel{def:urban-renewal}{2.5.3}
\manuallabel{def:tildeB_I}{4.1.1}
\manuallabel{sec:triangulations}{2.2}
\manuallabel{sec:mut-wiring}{2.3}
\manuallabel{fig:quiver-triangulation}{2.2}
\manuallabel{sec:matrices}{1.4}
\manuallabel{rem:opposite}{3.1.10}
\manuallabel{fig:chamber-quiver2}{2.6}
\manuallabel{fig:Q_3}{2.3}
\manuallabel{exercise:FG-sl3-d4}{2.2.3}
\manuallabel{rem:reduceddecomp}{2.3.4}
%\manuallabel{}{}

\manuallabel{ex:SL_2}{1.1.2}
\manuallabel{sec:Ptolemy}{1.2}
\manuallabel{eq:grassmann-plucker-3term}{1.2.1}
\manuallabel{sec:baseaffine}{1.3}
\manuallabel{prop:Muir}{1.3.5}
\manuallabel{fig:moves}{1.10}

\manuallabel{def:Q(T)-polygon}{2.2.1}
\manuallabel{def:quiverwd}{2.3.1}
\manuallabel{sec:mut-double-wiring}{2.4}
\manuallabel{fig:markov-quiver}{2.10}
% OLD LABEL: \manuallabel{fig:markov-quiver}{2.9}

\manuallabel{eq:exch-rel-geom}{3.1.1}
\manuallabel{exercise:quiverexchange}{3.1.3}
\manuallabel{def:Tn}{3.1.4}
\manuallabel{def:cluster-algebra}{3.1.6}
%THE EXAMPLE BELOW IS ABSENT IN THE OLD arXiv VERSION
\manuallabel{ex:base-affine-SL3}{3.2.1}
\manuallabel{example:A(1,2)}{3.2.7}
%OLD LABEL: \manuallabel{example:A(1,2)}{3.2.5}
\manuallabel{thm:Laurent}{3.3.1}
\manuallabel{th:Laurent-sharper}{3.3.6}

\manuallabel{ex:mutation-acyclic}{4.1.5}
\manuallabel{def:clustersubalgebra}{4.2.6}

\manuallabel{th:finite-type-classification}{5.2.8}
\manuallabel{sec:type-A}{5.3}
\manuallabel{example:SL4}{5.3.8}
\manuallabel{sec:type-D}{5.4}
\manuallabel{def:arcs-Dn}{5.4.3}
\manuallabel{def:P-gamma}{5.4.9}
\manuallabel{eq:Omega}{5.3.2}
\manuallabel{eq:Adefn}{5.4.1}
\manuallabel{tab:cluster-numbers}{5.17}

%\manuallabel{ch:plabic}{7}
%\manuallabel{ch:Grassmannians}{8}

\manuallabel{ch:surfaces}{10}
\manuallabel{ch:Grassmannians}{8}
\manuallabel{ch:dynamical}{11}
\manuallabel{ch:general-cluster-algebras}{12}
\manuallabel{thm:Dynkin-hereditary}{10.4.1}
\manuallabel{sec:finite-mutation-type-skew-symmetrizable}{10.2}
\manuallabel{sec:generalized-associahedra}{9.3}
\manuallabel{sec:exchangegraph}{9.1}

\noindent
This is a preliminary draft of Chapter 7 of our forthcoming textbook
\textsl{Introduction to cluster algebras}, 
joint with Andrei Zelevinsky (1953--2013). 

Other chapters have been posted as 
\begin{itemize}[leftmargin=.2in]
	\item  \href{https://arxiv.org/abs/1608.05735}{(Chapters~1--3)}, 
	\item \href{https://arxiv.org/abs/1707.07190}{(Chapters~4--5)}, and 
	\item \href{https://arxiv.org/abs/2008.09189}{(Chapter~6)}.
%\item \texttt{arXiv:2106.02160} (Chapter~7). 
\end{itemize}
We expect to post additional chapters in the not so distant future. 

\medskip

Anne Larsen and Raluca Vlad made a number of valuable suggestions 
that helped improve the quality of the first version of this manuscript.
We are also grateful to 
Zenan Fu, Amal Mattoo, Hanna Mularczyk, and Ashley Wang for their comments on 
a subsequent version  of this chapter, 
and for assistance with creating figures, and to 
Gregory Li, Stella Li, Annabel Ma, Jacob Paltrowitz, and Katherine Tung for their comments on 
the current version of this chapter and assistance with figures.
We would also like to thank Melissa Sherman-Bennett for useful comments.  
Last but not least, we are grateful to Pasha Galashin, 
%and his students Ariana Chin, Thomas Martinez, Robert Miranda, Olha Shevchenko, and Matty
%Tyler, 
%Not sure where to thank his students?  My impression is that they 
% couldn't get through Chapter 7 but they did give feedback on earlier chapters.
whose suggestions greatly clarified
our exposition.

Our work was partially supported by the
NSF grants DMS-1664722, DMS-1854512, DMS-1854316, DMS-2054231, DMS-2152991,
DMS-2348501, and by the Radcliffe Institute for Advanced Study.

\medskip

Comments and suggestions are welcome. 

\bigskip

\rightline{Sergey Fomin}
\rightline{Lauren Williams}

\vfill


\bigskip

\noindent
\copyright \ 2021 by 
Sergey Fomin, Lauren Williams, and Andrei Zelevinsky

\tableofcontents

\mainmatter

\setcounter{chapter}{6}

%RESTORE FRONTMATTER FOR FINAL VERSION
%\fi

%\include{introduction}
%\include{tp-examples}
%\include{combinatorics-of-mutations}
%\include{geometric-type}
%\include{newfromold}
%\include{finite-type}
%\setcounter{page}{0}

% !TEX root = /Users/sf/Dropbox/Book/Sources/chapter7-hyper.tex
\chapter{Plabic graphs}
\label{ch:plabic}
\def\X{\mathfrak{X}}
\def\affpi{\widetilde{\pi}_{\aff}}

In this chapter, 
we present the combinatorial machinery of \emph{plabic graphs},   
which were introduced by A.~Postnikov~\cite{postnikov}. 
These are planar (unoriented) graphs with bicolored vertices
satisfying some mild technical conditions. 
Plabic graphs can be transformed using certain %naturally defined 
\emph{local moves}. 
A key observation is that each plabic graph gives rise to a quiver,
so that local moves on plabic graphs translate into (a subclass of) quiver mutations.  

Crucially, the combinatorics underlying several important classes of cluster structures 
that arise in applications fits into the plabic graphs framework. 
This in particular applies to the basic examples introduced in Chapter~\ref{ch:tp-examples}. 
More concretely, we show that the combinatorics of flips in triangulations of a convex polygon 
(resp., braid moves in wiring diagrams, either ordinary or double)
can be entirely recast in the language of plabic graphs. 
In these and other examples,  an important role is played by 
the subclass of \emph{reduced} plabic graphs 
that are analogous to---and indeed generalize---reduced decompositions in symmetric groups.

D.~Thurston's \emph{triple diagrams}~\cite{thurston} are closely
related to plabic graphs.  
After making this connection precise and developing the machinery of triple diagrams, 
we use this machinery to establish the fundamental properties of reduced plabic graphs. 

Plabic graphs and related combinatorics have 
arisen in the study of shallow water waves \cite{kodwil,kodwilpnas} (via the KP equation)
and in connection with scattering amplitudes in $\mathcal{N}=4$ super Yang-Mills theory~\cite{amplitudes}.
Constructions closely related to plabic graphs were studied by T.~Kawamura~\cite{kawamura} 
in the context of the topological theory of graph divides.

A reader interested exclusively in the combinatorics of plabic graphs can read this chapter 
independently of the previous ones. 
While we occasionally refer to combinatorial constructions introduced in 
Chapters \ref{ch:tp-examples}--\ref{ch:combinatorics-of-mutations}, 
they are not relied upon 
in the development of the general theory of plabic graphs. % does not rely upon them. 

Cluster algebras as such  do not appear in this chapter. 
On the other hand, as we will see in a subsequent chapter,
reduced plabic graphs introduced herein will 
prominently feature 
in the study of cluster structures in Grassmannians and related varieties. 
%see Chapter~\ref{ch:Grassmannians}.
A reader who wishes to skim the chapter for the main ideas may choose
to read only Sections \ref{sec:overview}-
\ref{sec:plabictriangulations}, and 
Sections \ref{sec:edge-labels}-\ref{weaksep}.

%\smallskip

The structure of this chapter is as follows.

In \cref{sec:overview} we introduce plabic graphs and give an overview of 
the main results about them.
In particular, we introduce the important notions of \emph{reduced plabic graph} and 
 \emph{trip permutation}, and 
 we state 
 (but do not prove) 
the ``fundamental theorem of reduced plabic graphs,'' 
which characterizes the move-equivalence classes
of reduced plabic graphs in terms of associated \emph{decorated permutations}. 

\cref{sec:plabic} describes how to associate a quiver to any plabic graph.

In \cref{sec:plabictriangulations}, 
we recast the combinatorics of triangulations of 
a %convex 
polygon and (ordinary or double) wiring diagrams in the language of plabic graphs.  

\cref{sec:tri} discusses the version of the theory in which all internal vertices 
of plabic graphs are \emph{trivalent}. 
%(This version naturally arises in some applications, cf., e.g.,~\cite{fpst}.)
As \cref{sec:tri} is not strictly necessary for the sections that follow, 
it can be skipped if desired.

\cref{sec:triple-diagrams} introduces the basic notions of  
\emph{triple diagrams}. 
We then show that triple diagrams are in bijection with 
\emph{normal plabic graphs}. 

In \cref{sec:mintriple}, we study \emph{minimal} triple diagrams, largely following~\cite{thurston}.
These diagrams can be viewed as counterparts of reduced plabic graphs. 

In \cref{minred}, we explain how to go between minimal triple diagrams and reduced
plabic graphs.
We then use this correspondence to prove the fundamental theorem of 
reduced plabic graphs. 

In \cref{sec:bad}, we state and prove the \emph{bad features criterion}
that detects whether a plabic graph is reduced or not.

In \cref{sec:affine}, we describe a bijection between decorated permutations
and a certain subclass of \emph{affine permutations}.

In \cref{sec:bridge}, we use a factorization algorithm for affine permutations 
 to construct a family of reduced plabic graphs called \emph{bridge decompositions}.

\cref{sec:edge-labels} discusses \emph{edge labelings} of reduced plabic graphs
and gives the \emph{resonance criterion} for 
recognizing whether a plabic graph is reduced.  

\cref{sec:face-labels} introduces \emph{face labelings} of reduced plabic graphs.

In \cref{weaksep}, we provide an intrinsic combinatorial characterization of collections of
face labels that arise via this construction.  
Face labels  will reappear in 
a subsequent chapter %Chapter~\ref{ch:Grassmannians}
on cluster structures in Grassmannians.

\newpage
\section{Plabic graphs and the main results}
\label{sec:overview}

\begin{definition}
\label{def:plabic}
	A \emph{plabic} (planar bicolored) \emph{graph} is a (planar)
	graph~$G$ embedded
into a closed disk~$\mathbf{D}$, so that:
\begin{itemize}[leftmargin=.2in]
\item
the embedding of $G$ into~$\mathbf{D}$ is planar, i.e., the edges do not cross; 
\item
each internal vertex is colored black or white;
\item
each internal vertex is connected by a path to some boundary vertex;
\item
the (uncolored) vertices lying on the boundary of~$\mathbf{D}$ 
are labeled $1,2,\dots,b$ in clockwise order, for some positive integer~$b$; 
\item
each of these $b$ \emph{boundary vertices} is incident to a single edge. 
%\item
%a leaf (i.e., a degree~1 vertex) cannot be colored white, 
%unless it is a \emph{lollipop}, i.e., is connected by an edge to a boundary vertex. 
\end{itemize}
Loops and multiple edges are allowed.

We consider plabic graphs up to isotopy of the ambient disk~$\mathbf{D}$,
fixing the disk's boundary. 
The \emph{faces} of~$G$ are 
the connected components of the complement of 
$G$ inside~$\mathbf{D}$. 
A degree~1 internal vertex that 
is connected by an edge to a boundary vertex is called
	a \emph{lollipop}.  Any other degree~1 internal vertex is called a \emph{leaf}.  Boundary vertices are not leaves.
\end{definition}

Two examples of plabic graphs are shown in Figure~\ref{fig:plabic}. 
Many more examples will appear throughout this chapter. 
In what follows, we will often omit the boundary 
of the ambient disk when drawing plabic graphs.

\begin{figure}[ht]
\vspace{-5pt}
\begin{center}
{\ }\qquad 
\begin{tabular}{cc}
\setlength{\unitlength}{1pt}
\begin{picture}(120,125)(-57,-70)
\thicklines
\put(15,-5){\circle{120}}
\multiput(1,0)(1,30){2}{\line(1,0){27}}
\multiput(0,1)(30,1){2}{\line(0,1){27}}
\put(0,30){\circle{4}}
\put(30,30){\circle*{4}}
\put(30,0){\circle{4}}
\put(0,0){\circle*{4}}
\put(-10,-25){\circle{4}}
\put(15,-40){\circle*{4}}
\put(40,-25){\circle{4}}
\put(0,0){\line(-10,-25){9.2}}
\put(30.5,-2){\line(10,-25){8.5}}
\put(15,-40){\line(-25,15){23}}
\put(15,-40){\line(25,15){23}}
\put(15,-65){\circle*{3}}
\put(15,-40){\line(0,-1){24}}
\put(-12,48){\circle*{3}}
\put(42,48){\circle*{3}}
\put(-12,48){\line(12,-18){11}}
\put(42,48){\line(-12,-18){11}}
\put(-29,-46){\circle*{3}}
\put(59,-46){\circle*{3}}
\put(-29,-46){\line(19,21){17.5}}
\put(59,-46){\line(-19,21){17.5}}
\put(-17,53){\makebox(0,0){$\mathbf{1}$}}
\put(47,53){\makebox(0,0){$\mathbf{2}$}}
\put(64,-50){\makebox(0,0){$\mathbf{3}$}}
\put(15,-72){\makebox(0,0){$\mathbf{4}$}}
\put(-34,-50){\makebox(0,0){$\mathbf{5}$}}
\end{picture}
\qquad{\ }
& 
{\ }\qquad
\setlength{\unitlength}{1pt}
\begin{picture}(120,120)(-42,-70)
\thicklines
\put(15,-5){\circle{120}}
\multiput(1,0)(1,30){2}{\line(1,0){27}}
\multiput(0,1)(30,1){2}{\line(0,1){27}}
\put(-30,-5){\circle{4}}
\put(-45,-5){\circle*{3}}
\put(-45,-5){\line(1,0){13}}
\put(0,30){\circle{4}}
\put(30,30){\circle*{4}}
\put(30,0){\circle{4}}
\put(0,0){\circle*{4}}
\put(-10,-25){\circle{4}}
\put(15,-40){\circle*{4}}
\put(40,-25){\circle{4}}
\put(0,0){\line(-10,-25){9.2}}
\put(30.5,-2){\line(10,-25){8.5}}
\put(15,-40){\line(-25,15){23}}
\put(15,-40){\line(25,15){23}}
\put(15,-65){\circle*{3}}
\put(15,-40){\line(0,-1){24}}
\put(-12,48){\circle*{3}}
\put(42,48){\circle*{3}}
\put(-12,48){\line(12,-18){11}}
\put(42,48){\line(-12,-18){11}}
\put(-29,-46){\circle*{3}}
\put(59,-46){\circle*{3}}
\put(-29,-46){\line(19,21){17.5}}
\put(59,-46){\line(-19,21){17.5}}
\put(-17,53){\makebox(0,0){$\mathbf{1}$}}
\put(47,53){\makebox(0,0){$\mathbf{2}$}}
\put(64,-50){\makebox(0,0){$\mathbf{3}$}}
\put(15,-72){\makebox(0,0){$\mathbf{4}$}}
\put(-34,-50){\makebox(0,0){$\mathbf{5}$}}
\put(-51,-5){\makebox(0,0){$\mathbf{6}$}}
\end{picture}
\qquad {\ }
\\[.1in]
(a) & (b)
\end{tabular}
\end{center}
\vspace{-.2in} 
\caption{(a) A plabic graph $G$. (b) A plabic graph $G'$ with a white lollipop.}
\label{fig:plabic}
\end{figure}
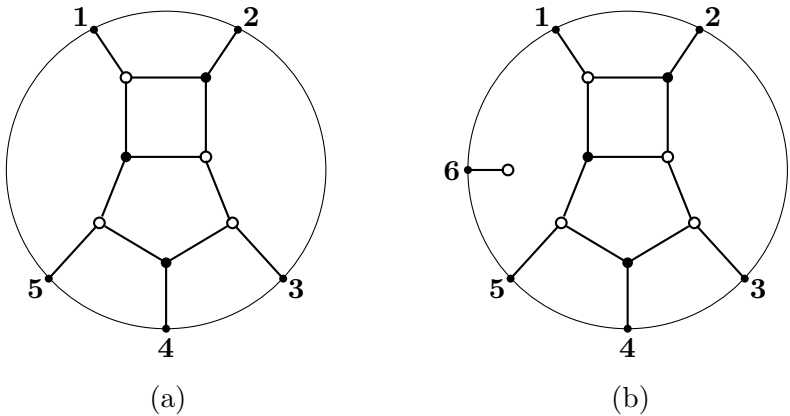

\begin{remark}
\label{rem:plabic-history}
Plabic graphs were originally introduced by A.~Postnikov
\cite[Section~12]{postnikov}, who used them to describe 
parametrizations of cells in totally nonnegative Grassmannians. 
A~closely related class of graphs was defined by 
T.~Kawamura~\cite{kawamura} in the context of the topological theory of \emph{graph divides}.
Our definition is very close to Postnikov's. 
\end{remark}

\begin{remark}
In some contexts, it is useful to drop the third condition in 
\cref{def:plabic}
to allow components of $G$ disconnected from the boundary.
\end{remark}

\begin{definition}\label{def:leafless}
We say that a plabic graph is \emph{leafless}
	if it has no leaves. (Note that a leafless plabic graph may have lollipops.)
\end{definition}

For simplicity, we will mainly 
 restrict our attention to leafless plabic graphs.  The results are simpler to state
 and the proofs are less technical in this setting.  
Moreover, for most applications
of plabic graphs, it suffices to work with leafless plabic graphs.
  A notable exception is the class of 
\emph{normal plabic graphs}, which we will study in \cref{sec:triple-diagrams},
and which are in bijection with \emph{triple diagrams}.  Normal plabic graphs
and triple diagrams  will be a key tool
in the proof of 
\cref{thm:moves}.

A key role in the theory of plabic graphs is played by the  
equivalence relation generated by a family of transformations
called  \emph{(local) moves}. 

\begin{definition}
\label{def:moves}
We say that two leafless plabic graphs $G$ and $G'$ are
\emph{move-equi\-valent}, and write $G \sim G'$, 
if $G$ and~$G'$ can be related to each other via a sequence of 
the following \emph{local moves}, denoted (M1), (M2), and~(M3): 
\begin{itemize}[leftmargin=.4in]
\item[(M1)]
(The \emph{square move}) 
Change the colors of all vertices on the boundary of a quadrilateral face, 
provided these colors alternate and these vertices are trivalent.  
See Figure~\ref{fig:M1}.

\begin{figure}[ht]%[htbp]
\begin{center}
\vspace{-.1in} 
\setlength{\unitlength}{.8pt}
\begin{picture}(30,36)(0,0)
\thicklines
\multiput(1,0)(1,30){2}{\line(1,0){27.3}}
\multiput(0,1)(30,1){2}{\line(0,1){27.3}}
\put(0,0){\circle*{4}}
\put(0,30){\circle{4}}
\put(30,30){\circle*{4}}
\put(30,0){\circle{4}}
\put(-10,40){\line(1,-1){8.8}}
\put(-10,-10){\line(1,1){9}}
\put(40,-10){\line(-1,1){8.8}}
\put(40,40){\line(-1,-1){9}}
\end{picture}
\quad
\begin{picture}(30,36)(0,0)
\put(15,15){\makebox(0,0){\Large{$\longleftrightarrow$}}}
\end{picture}
\quad
\begin{picture}(30,36)(0,0)
\thicklines
\multiput(1.5,0)(0,30){2}{\line(1,0){27}}
\multiput(0,1.5)(30,-0.5){2}{\line(0,1){27.5}}
\put(0,0){\circle{4}}
\put(0,30){\circle*{4}}
\put(30,30){\circle{4}}
\put(30,0){\circle*{4}}
\put(-10,40){\line(1,-1){9}}
\put(-10,-10){\line(1,1){8.5}}
\put(40,-10){\line(-1,1){9}}
\put(40,40){\line(-1,-1){8.5}}
\end{picture}
\end{center}
\vspace{-.1in} 
\caption{Move (M1) on plabic graphs.}
\label{fig:M1}
\end{figure}
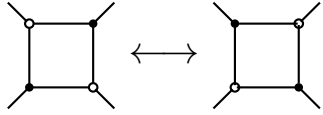

\item[(M2)]
Remove a bivalent vertex (of any color) and merge the edges adjacent to it;  
or, conversely, insert a bivalent vertex in the middle of an~edge. 
See Figure~\ref{fig:M2}.

\begin{figure}[ht]%[htbp]
\begin{center}
\vspace{-.15in} 
\setlength{\unitlength}{1pt}
\begin{picture}(30,20)(0,10)
\thicklines
\put(15,15){\circle*{4}}
\put(0,15){\line(1,0){30}}
\end{picture}
\quad
\begin{picture}(30,20)(0,10)
\put(15,15){\makebox(0,0){\Large{$\longleftrightarrow$}}}
\end{picture}
\quad
\begin{picture}(30,20)(0,10)
\thicklines
\put(0,15){\line(1,0){30}}
\end{picture}
\quad \text{ or }
\quad
\begin{picture}(30,20)(0,10)
\thicklines
\put(15,15){\circle{4}}
\put(0,15){\line(1,0){13}}
\put(17,15){\line(1,0){13}}
\end{picture}
\quad
\begin{picture}(30,20)(0,10)
\put(15,15){\makebox(0,0){\Large{$\longleftrightarrow$}}}
\end{picture}
\quad
\begin{picture}(30,20)(0,10)
\thicklines
\put(0,15){\line(1,0){30}}
\end{picture}
\end{center}
\vspace{-.2in} 
\caption{Move (M2) on plabic graphs.}
\label{fig:M2}
\end{figure}
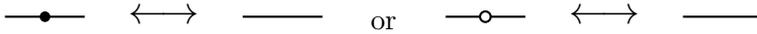

\item[(M3)]
Contract an edge connecting two internal vertices of the same color,
as in \cref{fig:M3}; or, conversely, ``uncontract'' an 
internal vertex of degree $d$ into two vertices of the same color,
joined by an edge,
with degrees
$d_1+1$ and $d_2+1$, where $d_1, d_2 >0$ and $d_1+d_2=d$.
%In the other direction,
% split an internal vertex into two vertices of the same color joined by an edge.
%See \cref{fig:M3}. 

\begin{figure}[ht]
\begin{center}
\vspace{-.1in} 
\setlength{\unitlength}{1pt}
\begin{picture}(45,20)(0,10)
\thicklines
\put(15,15){\circle*{4}}
\put(30,15){\circle*{4}}
\put(15,15){\line(-2,1){15}}
\put(15,15){\line(-2,-1){15}}
\put(15,15){\line(1,0){15}}

\put(30,15){\line(2,1){15}}
\put(30,15){\line(6,1){15}}
\put(30,15){\line(6,-1){15}}
\put(30,15){\line(2,-1){15}}

\end{picture}
\quad
\begin{picture}(30,20)(0,10)
\put(15,15){\makebox(0,0){\Large{$\longleftrightarrow$}}}
\end{picture}
\quad
\begin{picture}(30,20)(15,10)
\thicklines
\put(30,15){\circle*{4}}
\put(30,15){\line(-2,1){15}}
\put(30,15){\line(-2,-1){15}}
\put(30,15){\line(2,1){15}}
\put(30,15){\line(6,1){15}}
\put(30,15){\line(6,-1){15}}
\put(30,15){\line(2,-1){15}}
\end{picture}
\quad \text{ or }
\quad
\begin{picture}(45,20)(0,10)
\thicklines
\put(15,15){\circle{4}}
\put(15,15){\circle{3.9}}
\put(30,15){\circle{4}}
\put(13.3,15){\line(-2,1){15}}
\put(13.3,15){\line(-2,-1){15}}
\put(17,15){\line(1,0){11}}
\put(31.75,15){\line(2,1){15}}
\put(31.75,15){\line(1,0){15}}
\put(31.75,15){\line(2,-1){15}}
\end{picture}
\quad
\begin{picture}(30,20)(0,10)
\put(15,15){\makebox(0,0){\Large{$\longleftrightarrow$}}}
\end{picture}
\quad
\begin{picture}(30,20)(15,10)
\thicklines
\put(30,15){\circle{4}}
\put(28,15){\line(-2,1){15}}
\put(28,15){\line(-2,-1){15}}
\put(32,15){\line(2,1){15}}
\put(32,15){\line(1,0){15}}
\put(32,15){\line(2,-1){15}}
\end{picture}
\vspace{-5pt}\end{center}
\caption{Move (M3) on plabic graphs.
The number of ``hanging'' edges on each side must be positive.}
\label{fig:M3}
\end{figure}

\end{itemize}
\end{definition}

Note that 
if $G$ is bipartite, then the square move~(M1)
is closely related to the \emph{spider move}, see Definition~\ref{def:urban-renewal} 
as well as \cref{def:urban-normal} below.

\begin{definition}
\label{def:reduced-plabic}
%\label{rem:overG}
A leafless plabic graph $G$ is \emph{reduced} if 
there is no plabic graph $G'\sim G$ such that $G'$
contains
%in its move-equivalence class 
%one of the following ``forbidden configurations:''
%\begin{enumerate}[leftmargin=.2in]
%	\item \label{forbid:1}
		a \emph{hollow monogon} or 
		a \emph{hollow digon}, 
		that is, an internal face bounded by one or two edges, see
 \cref{fig:fail}.
%$u$ and $v$ are connected in $G'$ by more than one edge, 
%with no other edges in the region between them; or 
%\item \label{forbid:2}
%an internal leaf that is not a lollipop and does not belong to a collapsible tree. 
%\end{enumerate}
%Applying (M2)--(M3) moves if needed, we can restate this definition as follows:
%$G$~is reduced if no plabic graph $G'\sim G$ contains 
%one of the ``forbidden configurations'' shown 
%See \cref{fig:fail}(a,b,c) for some hollow digons.

	A plabic graph $G$ with leaves is \emph{reduced} if it can be converted into 
	a reduced leafless plabic graph $\overline{G}$ via the 
	following operations:
\begin{enumerate}
\item contract a leaf edge if both of its endpoints have the same color
	(note that this is not an (M3) move);
\item delete a $2$-valent vertex via (M2).
\end{enumerate}
%If $\overline{G}$
%is a leafless reduced plabic graph as above, then we say that $G$ is \emph{reduced}. 
Otherwise
we say
that $G$ is \emph{non-reduced}.
\iffalse
Now let $G$ be a plabic graph with leaves.
Suppose that we can convert $G$ into a leafless plabic graph
$\overline{G}$ 
via the following operations:
\begin{enumerate}
\item contract a leaf edge if both of its endpoints have the same color
	(note that this is not an (M3) move);
\item delete a $2$-valent vertex via (M2).
\end{enumerate}
If $\overline{G}$
is a leafless reduced plabic graph as above, then we say that $G$ is \emph{reduced}. 
Otherwise
we say
that $G$ is \emph{non-reduced}.
\fi
\end{definition}
\begin{figure}[ht]
\begin{center}
\vspace{-.2in}
%\includegraphics[height=1in]{FailureReduced.ps}
%\\
\setlength{\unitlength}{1pt}
\begin{picture}(35,20)(0,10)
\thicklines
\put(15,15){\circle*{4.5}}
\put(13,15){\line(-1,0){10}}
\cbezier(15,17)(30,40)(30,-10)(15,13)
\end{picture}
\quad
\begin{picture}(45,20)(0,10)
\thicklines
\put(15,15){\circle{4}}
\cbezier(15,17)(30,40)(30,-10)(15,13)
\put(15,15){\color{white}\circle*{3}}
\put(13,15){\line(-1,0){10}}
\end{picture}
\quad
\begin{picture}(45,20)(0,10)
\thicklines
\put(15,15){\circle{4}}
\put(30,15){\circle{4}}
\put(13,15){\line(-1,0){10}}
\put(32,15){\line(1,0){10}}
\qbezier(15,17)(22,30)(30,17)
\qbezier(15,13)(22,0)(30,13)
\end{picture}
\quad
\begin{picture}(45,20)(0,10)
\thicklines
\put(15,15){\circle*{4}}
\put(30,15){\circle{4}}
\put(13,15){\line(-1,0){10}}
\put(32,15){\line(1,0){10}}
\qbezier(15,15)(22,30)(30,17)
\qbezier(15,15)(22,0)(30,13)
\end{picture}
\quad
\begin{picture}(45,20)(0,10)
\thicklines
\put(15,15){\circle*{4}}
\put(30,15){\circle*{4}}
\put(13,15){\line(-1,0){10}}
\put(32,15){\line(1,0){10}}
\qbezier(15,15)(22,30)(30,16)
\qbezier(15,15)(22,0)(30,14)
\end{picture}
\vspace{-5pt}
\end{center}
\caption{A leafless plabic graph
 is not reduced if and only if it is move-equivalent
	to a graph containing a hollow monogon or digon.}
\label{fig:fail}
\end{figure}

%\begin{remark}
In this chapter, we will mainly restrict our attention to leafless graphs.
%\end{remark}

\begin{remark}
\label{rem:comments-on-reduced-plabic}
For a leafless plabic graph, 
the property of being reduced is invariant under local moves.
Note however that this property is not readily testable because it requires 
understanding all graphs in  the move-equiv\-a\-lence class
of a plabic graph. % is usually infinite. 
We will later obtain several other criteria
for testing this property, see 
\cref{cor:reduced=min-faces},
%	says that a leafless 
%	plabic graph is reduced if it has the
%smallest number of faces among plabic graphs in its equivalence class.
\cref{thm:reduced}, % gives a criterion for reducedness that involves checking for 
%forbidden subgraphs; 
	and \cref{thm:resonance}.
%gives a criterion that 
%comes from labeling edges of $G$ using trips and checking 
%a kind of tropical balancing condition at each vertex.
\end{remark}

The most fundamental result concerning reduced plabic graphs
is their classification up to move-equivalence (cf.\ \cref{rem:comments-on-reduced-plabic}),
to be given in \cref{thm:moves} below. 
To state this result, we will need some preparation. 

\begin{definition}
\label{def:trip}
A \emph{trip} $\tau$ in a plabic graph $G$ is a walk %possibly a closed one, 
along the edges of $G$ that either begins and ends at boundary vertices (with all intermediate vertices internal),
or is a closed walk entirely contained in the interior of the disk.
This walk must obey the following
 \emph{rules of the road}: 
\begin{itemize}[leftmargin=.2in]
\item
	at a black (respectively, white) vertex of degree at least $2$, 
		$\tau$ always makes the sharpest possible right (respectively, left)
		turn onto a different edge;
%\item
%at a white vertex of degree at least $2$, $\tau$ always makes the sharpest possible left turn onto a different edge;
\item at a vertex of degree $1$ (e.g., a lollipop), $\tau$ makes a U-turn.
\end{itemize}
\end{definition}

\begin{remark}
Any walk which starts at a boundary vertex and obeys the rules of 
the road must necessarily end at a boundary vertex; we call such a 
trip 
 a \emph{one-way trip}.  We refer to a trip which is 
entirely contained in the interior of the disk as 
 a \emph{roundtrip}.  
\end{remark}
%		if $\tau$ is closed, then it is entirely contained in the interior of the ambient disk; 
%such trips are called \emph{roundtrips}; 
%\item
%if $\tau$ is not closed, then it begins and ends at boundary vertices, while 
%all its intermediate vertices lie in the interior of the disk;
%such trips are called \emph{one-way trips}.

A one-way trip may begin and end at the same vertex.
For example, a trip that starts at a boundary vertex~$i$ incident to a lollipop will 
end at~$i$.

\begin{remark}
Just as different countries have different rules regarding 
which side of the road one should drive on,
different authors make conflicting choices for the rules of the 
road for plabic graphs.
In this book, we consistently use the convention chosen in \cref{def:trip}.
\end{remark}

\begin{remark}
The notion of a trip and the condition of being reduced have appeared in the 
study of dimer models in statistical mechanics, 
	 %on quivers.  In this literature,
	wherein trips have been called \emph{zigzag paths} \cite{Kenyon}. 
Reduced plabic graphs were called ``marginally geometrically consistent"  
in \cite[Section 3.4]{Broomhead}, 
and were said to ``obey condition Z" in \cite[Section 8]{Bocklandt}.
\end{remark}

\begin{exercise}\label{ex:diff}
Show that %every one-way trip terminates (necessarily at a boundary vertex).  
one-way trips starting at different vertices terminate at different vertices.
\end{exercise}

\begin{remark}
\label{rem:atmost2}
For any edge $e$ in~$G$ and a choice of direction along $e$, % with distinct endpoints $u$ and~$v$, 
there is a unique trip %$\tau_1$
traversing $e$ in the chosen direction.
It may happen that the same trip traverses $e$ twice (once in each direction). 
\end{remark}

\begin{definition}\label{def:trippermutation}
Let $G$ be a plabic graph with $b$ boundary vertices. 
The \emph{trip permutation} $\pi_G: \{1,\dots,b\} \to \{1,\dots,b\}$
is defined by setting \hbox{$\pi_G(i)=j$} whenever the trip originating at  $i$ terminates at~$j$.
%  If boundary vertex $i$ is attached to a black (respectively, white)
%leaf, then we say that $\pi_G(i)=i$ is a black (respectively, white) fixed point.  
We will mostly use the one-line notation $\pi_G=(\pi_G(1), \dots, \pi_G(b))$
to represent these permutations.
\end{definition}

To illustrate, in \cref{fig:plabic}(a), we have $\pi_G = (3,4,5,1,2)$.
Note that \cref{def:trippermutation} is well-defined because of \cref{ex:diff}.

%\pagebreak[3]

\begin{exercise}
\label{exercise:trip-invariant}
Show that move-equivalent plabic graphs have the same 
 trip permutation. % cf.\ \cref{exercise:forward} below. 
\end{exercise}

\begin{remark}
The notion of a trip permutation can be further enhanced to construct
finer invariants of local moves. 
For example, we can record, in addition to the trip permutation,
the suitably defined \emph{winding number} of each trip. 
These winding numbers do not change under local moves.
A~more powerful invariant associates to any plabic graph 
a particular (transverse) link, see~\cite{fpst}.
\end{remark}

The following statement
will be proved in \cref{minred}.
\begin{proposition}
\label{prop:fixedlollipop}
Let $G$ be a reduced leafless 
plabic graph. 
If $\pi_G(i) = i$, then the connected component of~$G$ 
containing the boundary vertex~$i$ 
is a lollipop. 
\end{proposition}

\begin{definition}\label{def:dectrip}
A \emph{decorated permutation} $\dpi$ on $b$ letters is 
	a permutation of the set $\{1,\dots,b\}$ together with a 
	\emph{decoration} of each fixed point by either an overline or an underline.
	In other words, for every~$i$, we have 
\[
\dpi(i)\in\{\overline{i}, \underline{i}\}\cup \{1,\dots,b\}\setminus\{i\}.
\] 
\end{definition}

An example of a decorated permutation on $6$ letters is  $(3,4,5,1,2,\overline{6})$.

\pagebreak[3]

\begin{exercise}
\label{ex:enumeration}
Show that the number of decorated permutations on $b$ letters is 
equal to $b!\sum_{k=0}^b \frac{1}{k!}$. 
%A more refined enumeration formula for decorated permutations will be given in \cref{pr:A_{a,b}(1)}.
\end{exercise}

\begin{definition}
\label{def:dtp}
Let $G$ be a reduced leafless plabic graph. % with $b$ boundary vertices. 
The \emph{decorated trip permutation} $\pi_G$ associated with~$G$ is defined~by 
\[
\dpi_G(i) = 
\begin{cases}
\pi_G(i) & \text{if  $\pi_G(i) \neq i$;} \\
\ \ \overline{i} & \text{if $G$ has a %tree collapsing to a 
	white lollipop at~$i$;}\\
\ \ \underline{i} & \text{if $G$ has
	%a tree collapsing to 
	a black lollipop at~$i$.}
\end{cases}
\]
\end{definition}
\begin{remark}
	If $G$ is a reduced plabic graph with leaves, we can define
	$\dpi_G$ to be $\dpi_{\overline{G}}$,
	where $\overline{G}$ is as in 
\cref{def:reduced-plabic}.
\end{remark}

\cref{fig:plabic-move-equiv} shows two reduced plabic graphs with the same 
 decorated trip permutation $%\dpi_{G'} = 
 (3,4,5,1,2,\overline{6})$. %cf.\ also \cref{fig:plabic}.
 
\begin{figure}[ht]
\begin{center}
\vspace{-.1in}
\setlength{\unitlength}{1pt}
\begin{picture}(100,125)(-30,-70)
\put(15,-5){\circle{120}}
\thicklines
\multiput(1,0)(1,30){2}{\line(1,0){27.5}}
\multiput(0,1)(30,1){2}{\line(0,1){27.5}}
\put(0,30){\circle{4}}
\put(30,30){\circle*{4}}
\put(30,0){\circle{4}}
\put(0,0){\circle*{4}}
\put(-10,-25){\circle{4}}
\put(15,-40){\circle*{4}}
\put(40,-25){\circle{4}}
\put(0,0){\line(-10,-25){9}}
\put(30,-1.5){\line(10,-25){9}}
\put(15,-40){\line(-25,15){23}}
\put(15,-40){\line(25,15){23}}
\put(15,-65){\circle*{3}}
\put(15,-40){\line(0,-1){24}}
\put(-12,48){\circle*{3}}
\put(42,48){\circle*{3}}
\put(-12,48){\line(12,-18){11}}
\put(42,48){\line(-12,-18){11}}
\put(-29,-46){\circle*{3}}
\put(59,-46){\circle*{3}}
\put(-29,-46){\line(19,21){18}}
\put(59,-46){\line(-19,21){18}}
\put(-17,53){\makebox(0,0){$\mathbf{1}$}}
\put(47,53){\makebox(0,0){$\mathbf{2}$}}
\put(64,-50){\makebox(0,0){$\mathbf{3}$}}
\put(15,-72){\makebox(0,0){$\mathbf{4}$}}
\put(-34,-50){\makebox(0,0){$\mathbf{5}$}}

\put(-51,-5){\makebox(0,0){$\mathbf{6}$}}
\put(-45,-5){\circle*{3}}
\put(-45,-5){\line(1,0){13.5}}
\put(-30,-5){\circle{4}}
\end{picture}\qquad \qquad \qquad
\begin{picture}(100,125)(-30,-70)
\put(15,-5){\circle{120}}
\thicklines
\put(2,30){\line(1,0){26}}
\put(30,2){\line(0,1){26}}
\put(0,-13){\line(0,1){41}}
\put(2,-15){\line(1,0){11}}
\put(0,30){\circle*{4}}
\put(30,30){\circle{4}}
\put(30,0){\circle{4}}
\put(0,-15){\circle{4}}
\put(15,-15){\circle*{4}}
\put(15,-17){\line(0,-1){21}}
\put(15,-15){\line(1,1){13.5}}

%\put(-10,-25){\circle{4}}
\put(40,-25){\circle*{4}}
\put(15,-40){\circle{4}}

%\put(0,0){\line(-10,-25){9}}
\put(30,-1.5){\line(10,-25){9}}
%\put(15,-40){\line(-5,3){23}}
\put(16.5,-39.1){\line(5,3){23}}
\put(15,-65){\circle*{3}}
\put(15,-42){\line(0,-1){22}}

\put(-12,48){\circle*{3}}
\put(42,48){\circle*{3}}
\put(-12,48){\line(12,-18){11}}
\put(42,48){\line(-12,-18){11}}
\put(-29,-46){\circle*{3}}
\put(59,-46){\circle*{3}}
\put(-29,-46){\line(19,21){27.5}}
\put(59,-46){\line(-19,21){18}}
\put(-17,53){\makebox(0,0){$\mathbf{1}$}}
\put(47,53){\makebox(0,0){$\mathbf{2}$}}
\put(64,-50){\makebox(0,0){$\mathbf{3}$}}
\put(15,-72){\makebox(0,0){$\mathbf{4}$}}
\put(-34,-50){\makebox(0,0){$\mathbf{5}$}}

\put(-51,-5){\makebox(0,0){$\mathbf{6}$}}
\put(-45,-5){\circle*{3}}
\put(-45,-5){\line(1,0){13.5}}
\put(-30,-5){\circle{4}}
\end{picture}
\end{center}
\vspace{-.1in} 
\caption{Two reduced plabic graphs sharing the same decorated trip permutation $(3,4,5,1,2,\overline 6)$. 
Cf.\  \cref{fig:plabic}(b).
}
\label{fig:plabic-move-equiv}
\end{figure}

Exercise~\ref{exercise:trip-invariant} can be strengthened as follows. 

\begin{exercise}
\label{exercise:forward}
The decorated trip permutation of a reduced plabic graph is invariant under local moves.
\end{exercise}

We will later show (see \cref{permtoG}) that for each  
decorated permutation $\dpi$ on $b$ letters, 
there exists a reduced leafless
plabic graph whose decorated trip permutation is $\dpi$.

%\cref{thm:moves} gives a partial converse of \cref{exercise:forward}
%in the setting of \emph{reduced} plabic graphs.
%Moreover it will be crucially important for the 
%combinatorial applications that follow as well as for 
%total positivity on the Grassmannian.

Crucially, the move-equivalence class of a reduced plabic graph
is completely determined by its decorated trip permutation: 

%\pagebreak[3]

\begin{theorem}[Fundamental theorem of reduced plabic graphs]
\label{thm:moves}
	Let $G$ and $G'$ be reduced leafless
	plabic graphs.  The following statements are equivalent:
\begin{enumerate}[leftmargin=.3in]
\item[{\rm (1)}] $G$ and $G'$ are move-equivalent;
\item[{\rm (2)}] $G$ and $G'$ have the same decorated trip permutation.
\end{enumerate}
\end{theorem}

\begin{remark}
It is possible
to extend \cref{thm:moves} to the setting of plabic graphs with 
leaves, using a generalized  version of move (M3) that allows
the number of hanging edges to be $0$.  However, the proof 
is  more technical as it needs to deal with 
``collapsible trees'', see \cite{chapter7v1}.
	\end{remark}

To illustrate, the two reduced plabic graphs shown in
\cref{fig:plabic-move-equiv} have the same decorated trip permutation 
and consequently are move-equivalent.  

We note that the number of faces of a plabic graph is an invariant 
of its move-equivalence class: that is, the number of faces is preserved
under moves (M1), (M2), and (M3).  We can use the number of 
faces and 
\cref{thm:moves} to characterize the leafless plabic graphs which 
are reduced:
\begin{corollary}
\label{cor:reduced=min-faces}
Let $\pi$ be a permutation on $b$ letters. 
Consider all leafless plabic graphs $G$ 
%without internal leaves (other than lollipops)
whose trip permutation is~$\pi$; in particular, $G$~has $b$ boundary vertices. 
Among all such~plabic graphs~$G$, the reduced ones are precisely those 
that have the smallest number of faces. 
\end{corollary}
In 
\cref{cor:number-faces},
we will give a formula for the number of faces in a reduced 
plabic graph in terms of the associated decorated trip permutation.

\begin{remark}
%\label{rem:}
In \cref{cor:reduced=min-faces}, the requirement that $G$ is leafless cannot be dropped;
cf.\ \cref{fig:fork-not-lollipop}. 
%For example, the graph in \cref{fig:fork-not-lollipop} has a single face but is not reduced. 
\end{remark}

\begin{figure}[ht]
\begin{center}
%\vspace{-.05in}
\setlength{\unitlength}{1.5pt}
\begin{picture}(30,15)(0,5)
\put(20,15){\circle{30}}
\thinlines
%\qbezier(-5,10)(20,-10)(45,10)
\thicklines
\put(20,0){\circle*{1.2}}
\put(20,10){\circle{2}}
\put(30,13){\circle*{2}}
\put(10,13){\circle*{2}}
%\put(20,-4){\makebox(0,0){$1$}}
%
\put(20,0.6){\line(0,1){8.4}}
\put(21,10){\line(3,1){8}}
\put(19,10){\line(-3,1){8}}
\end{picture}
\vspace{-.13in}
\end{center}
\caption{This plabic graph $G$ with one boundary vertex
	%and $\pi_G(1)\!=\!1$
has a single face but is not reduced. Note that it is not leafless.}
%since it contains a forbidden configuration, 
%cf.\ \cref{fig:fail}(d). The component containing~$i$ is not collapsible.}
\label{fig:fork-not-lollipop}
\end{figure}

\begin{remark}
%\label{rem:}
Some authors call a plabic graph reduced if it has the smallest number of faces among all graphs with a given decorated trip permutation, cf.\ \cref{cor:reduced=min-faces}. 
If one adopts this definition, then the graph~$G$ in \cref{fig:fork-not-lollipop}  becomes reduced.
%This leads to a failure of  \cref{prop:fixedlollipop}, since the component of~$G$
%attached to the boundary vertex~$i$ is not a lollipop. 
However, there are several reasons that we want to consider this graph
to be non-reduced, including the correspondence with triple diagrams,
cf.\ \cref{minred}, where we will 
%Another reason to treat this kind of plabic graph~$G$ as non-reduced 
%will arise in the context of triple diagrams, 
want reduced plabic graphs to be in bijection with minimal triple diagrams.
%(see \cref{red-minimal}), 
%the triple diagram corresponding to~$G$ is not minimal. 
\end{remark}

\cref{thm:moves} and \cref{cor:reduced=min-faces} 
will be proved in \cref{minred}.
The statement (1)$\Rightarrow$(2) in \cref{thm:moves} is easy, cf.\  \cref{exercise:forward}.  
The converse  implication (2)$\Rightarrow$(1) is much harder.
In \cref{minred}, we give a proof of this implication that 
%follows a different ---and, we hope, fairly intuitive---approach 
utilizes D.~Thurston's machinery of triple diagrams, which is presented in 
Sections~\ref{sec:triple-diagrams}--\ref{sec:mintriple}.

A very intricate argument justifying the implication (2)$\Rightarrow$(1)  
was described in Postnikov's original preprint \cite[Section 13]{postnikov}. 
Another proof of \cref{thm:moves}, involving some difficult results about \emph{plabic tilings}
(and relying on \cref{thm:ops} below), %\cite{oh-postnikov-speyer}[Theorem~1.15]), 
was given by S.~Oh and D.~Speyer \cite{oh-speyer}. % see Theorem~1.2.  

\newpage
\section{Plabic graphs and their quivers}
\label{sec:plabic}

We next associate a quiver to any plabic graph, extending the construction in 
 Definition~\ref{def:graphquiver} to the setting of plabic graphs that need not be
 bipartite.

\begin{definition} 
\label{def:Q(G)}
The \emph{quiver~$Q(G)$ associated to a plabic 
graph~$G$} is defined as follows.
The vertices of $Q(G)$ are in one-to-one correspondence
	with the faces of~$G$. A vertex is mutable (respectively, frozen)
	if the corresponding face is internal (respectively, incident to the boundary of the disk).
%	A~vertex of $Q(G)$ is frozen if the corresponding face is incident to the boundary of the 
%	disk, and is mutable otherwise.
%declared mutable or frozen 
%depending on whether the corresponding face is \emph{internal} 
%(i.e., disjoint from the boundary of~$\mathbf{D}$) or not.
%
For each edge $e$ in $G$ connecting a white vertex to a black vertex and 
separating two distinct faces, we introduce an arrow
connecting the faces separated by $e$; this arrow is oriented
so that it ``sees'' the white endpoint of $e$ to the left and the black 
	endpoint to the right as it crosses over~$e$.
%The arrows of $Q(G)$ are constructed in the following way. 
%Let $e$ be an edge in~$G$ that connects a white vertex to a black vertex 
%and separates two distinct faces, at least one of which is internal.
%For each such edge~$e$, we introduce an arrow in $Q(G)$ 
%connecting the  faces separated by~$e$;
%this arrow is oriented so that when we move along the arrow in the direction of its orientation, 
%we see the white endpoint of~$e$ on our left and the black endpoint on our right.
We then remove oriented $2$-cycles from the resulting quiver,
one by one, to get~$Q(G)$.
See~\cref{fig:plabic2}. 
\end{definition}

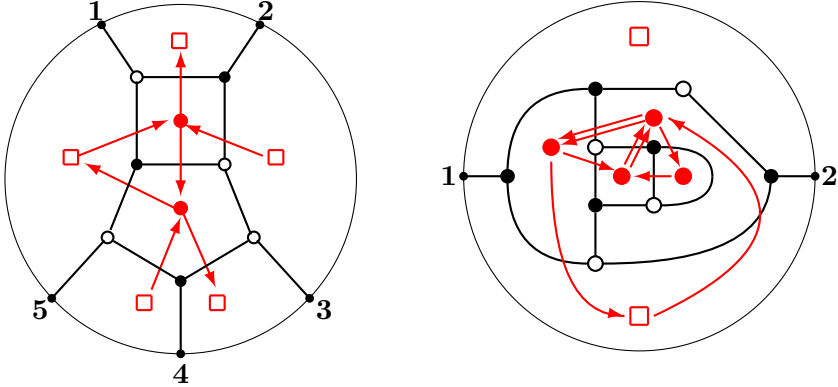
\begin{figure}[ht]%[htbp]
\begin{center}
\vspace{-5pt}
\setlength{\unitlength}{1pt}

\begin{picture}(100,120)(-30,-70)
\thicklines
\put(15,-5){\circle{120}}
\thicklines
\multiput(1,0)(1,30){2}{\line(1,0){27.3}}
\multiput(0,1)(30,1){2}{\line(0,1){27}}
\put(0,30){\circle{4}}
\put(30,30){\circle*{4}}
\put(30,0){\circle{4}}
\put(0,0){\circle*{4}}
\put(-10,-25){\circle{4}}
\put(15,-40){\circle*{4}}
\put(40,-25){\circle{4}}
\put(0,0){\line(-10,-25){9.2}}
\put(30,-2){\line(10,-25){8.6}}
\put(15,-40){\line(-25,15){23}}
\put(15,-40){\line(25,15){23}}
\put(15,-40){\line(0,-1){24}}
\put(-12,48){\line(12,-18){11}}
\put(42,48){\line(-12,-18){11}}
\put(15,-65){\circle*{3}}
\put(-12,48){\circle*{3}}
\put(42,48){\circle*{3}}
\put(-29,-46){\circle*{3}}
\put(59,-46){\circle*{3}}
\put(-29,-46){\line(19,21){17.5}}
\put(59,-46){\line(-19,21){17.5}}
\put(-17,53){\makebox(0,0){$\mathbf{1}$}}
\put(47,53){\makebox(0,0){$\mathbf{2}$}}
\put(64,-50){\makebox(0,0){$\mathbf{3}$}}
\put(15,-72){\makebox(0,0){$\mathbf{4}$}}
\put(-34,-50){\makebox(0,0){$\mathbf{5}$}}

% boxes and circles
\put(15,15){\red{\circle*{5}}}
	\put(15,-15){\red{\circle*{5}}}

% centered the box
	\put(12.5,40){\red{\line(0,1){5}}}
	\put(12.5,40){\red{\line(1,0){5}}}
	\put(12.5,45){\red{\line(1,0){5}}}
	\put(17.5,40){\red{\line(0,1){5}}}

	\put(50,0){\red{\line(0,1){5}}}
	\put(50,0){\red{\line(1,0){5}}}
	\put(50,5){\red{\line(1,0){5}}}
	\put(55,0){\red{\line(0,1){5}}}

	\put(-25,0){\red{\line(0,1){5}}}
	\put(-25,0){\red{\line(1,0){5}}}
	\put(-25,5){\red{\line(1,0){5}}}
	\put(-20,0){\red{\line(0,1){5}}}

	\put(25,-50){\red{\line(0,1){5}}}
	\put(25,-50){\red{\line(1,0){5}}}
	\put(25,-45){\red{\line(1,0){5}}}
	\put(30,-50){\red{\line(0,1){5}}}

	\put(0,-50){\red{\line(0,1){5}}}
	\put(0,-50){\red{\line(1,0){5}}}
	\put(0,-45){\red{\line(1,0){5}}}
	\put(5,-50){\red{\line(0,1){5}}}

% arrows
% big circle centered at (15,-5)
% boxes are centered at (15, 42.5), (-22.5, 2.5), (47.5, 2.5), (27.5, -47.5), (2.5, -47.5)

% should be centered at (52.5, 2.5) - fixed
\put(15,19){\red{\vector(0,1){20}}}
\put(15,11){\red{\vector(0,-1){22}}}
\put(-17,3){\red{\vector(2,.8){27}}}
\put(47,3){\red{\vector(-2,.8){27}}}
\put(11,-13){\red{\vector(-4,2){27}}}
\put(16,-19){\red{\vector(1,-2){12}}}
\put(2.5,-43){\red{\vector(1,2){12}}}
\end{picture}
\qquad
\setlength{\unitlength}{1pt}
\begin{picture}(120,58)(-50,-30)
%\begin{picture}(120,58)(-50,-50)
\thicklines

% \thinlines
\put(25,30){\circle{120}}
%\put(25,30){\circle{130}}

\put(-35,30){\circle*{3}}
\put(85,30){\circle*{3}}
%\put(-45,30){\makebox(0,0){$\mathbf{1}$}}
%\put(95,30){\makebox(0,0){$\mathbf{2}$}}
\put(-40,30){\makebox(0,0){$\mathbf{1}$}}
\put(90,30){\makebox(0,0){$\mathbf{2}$}}

\thicklines
\put(10,22.5){\line(0,1){15}}
\put(10,2.5){\line(0,1){15}}
\put(10,42.5){\line(0,1){15}}
\put(30,22.5){\line(0,1){15}}
\qbezier(32.5,20)(50,20)(50,30)
\qbezier(32.5,40)(50,40)(50,30)
\qbezier(12.5,0)(70,0)(70,30)
\qbezier(7.5,0)(-20,0)(-20,30)
\qbezier(7.5,60)(-20,60)(-20,30)
%\qbezier(12.5,60)(70,60)(70,30)
\put(12.5,60){\line(1,0){25}}
\put(40,60){\circle{5}}
\put(70,30){\line(-1,1){28}}

\put(-20,30){\line(-1,0){15}}
\put(70,30){\line(1,0){15}}

\put(12.5,20){\line(1,0){15}}
\put(12.5,40){\line(1,0){15}}
\put(10,0){\circle{5}}
\put(10,60){\circle*{5}}
\put(10,20){\circle*{5}}
\put(10,40){\circle{5}}
\put(30,20){\circle{5}}
\put(30,40){\circle*{5}}
\put(-20,30){\circle*{5}}
\put(70,30){\circle*{5}}

%quiver
\put(-5,40){\red{\circle*{6}}}
\put(19,30){\red{\circle*{6}}}
\put(40,30){\red{\circle*{6}}}
\put(30,50){\red{\circle*{6}}}

% circle centered at (25,30)
% prevous boxes are length 5
\put(22.5,75){\red{\line(0,1){5}}}
\put(22.5,75){\red{\line(1,0){5}}}
\put(22.5,80){\red{\line(1,0){5}}}
\put(27.5,75){\red{\line(0,1){5}}}

\put(22.5,-20){\red{\line(0,1){5}}}
\put(22.5,-20){\red{\line(1,0){5}}}
\put(22.5,-15){\red{\line(1,0){5}}}
\put(27.5,-20){\red{\line(0,1){5}}}

% boxes are centered at (25, 75.5) and (25, -17.5), symmetric
\put(36,30){\red{\vector(-1,0){13}}}
\put(32,47){\red{\vector(1,-2){6.8}}}
\put(20,33.5){\red{\vector(1,1.6){8.5}}}
\put(18,33.5){\red{\vector(1,1.6){9}}}

\put(-1,38){\red{\vector(3,-1){18}}}
\put(26.5,49){\red{\vector(-7,-2.2){28.1}}}
\put(26.5,51){\red{\vector(-7,-2.2){28.8}}}

\red{\qbezier(30,-18)(100,15)(36,48)}
\put(38,47){\red{\vector(-2,1.2){4}}}
\red{\qbezier(19,-17.5)(-5,-18)(-5,35)}
\put(19,-17.5){\red{\vector(9,-1){2}}}

\end{picture}
\end{center}
\vspace{-.1in} 
\caption{Two plabic graphs and their associated quivers. 
Shown on the left is the graph~$G$ from \cref{fig:plabic}(a).   
The quiver on the right has double arrows, corresponding to the instances where a pair of faces 
 share two boundary segments disconnected from each other.
The frozen vertex~$v$ at the top of the picture is isolated: 
the two arrows between~$v$ and an internal vertex
located underneath~$v$ cancel each other.} 
\label{fig:plabic2}
\end{figure}

%\begin{proposition}
%As in the setting of bipartite graphs
%(\emph{cf.} \cref{ex:urban}), if two plabic graphs $G$ and 
%$G'$ are related
%	via the square move (M1) (subject to the above restriction), $Q(G)$ and $Q(G')$ are related via mutation
%at the corresponding vertex. If two plabic graphs $G$ and $G'$ are 
%	related via (M2) or (M3) then $Q(G)$ and $Q(G')$ are the same.
%\end{proposition}

%The following simple but crucially important observation is implicit in Postnikov's original work~\cite{postnikov}.
%(Cf.\ also \cref{ex:urban}.)

\begin{proposition}
\label{pr:plabic-vs-quivers}
Let $G$ and $G'$ be two plabic graphs %move-equivalent to each other, 
related to each other by one of the local moves {\rm(M1)}, {\rm(M2)}, or~{\rm(M3)},
	subject to the restriction that  we only allow a square move~{\rm(M1)} at a face $F$ if
\vspace{-2pt}
\begin{align}
\label{eq:square-move-tricky}
%&\text{among the four faces surrounding the square, the consecutive ones}\\
%&\text{must be distinct, see \cref{fig:square-move-tricky}.} 
	&\text{$F$ does not share two consecutive sides with another face;} \\[-20pt] \notag
%	&\text{no face sharing an edge with $F$ is incident to two}\\
%	\notag
%	&\text{consecutive sides of the square, see \cref{fig:square-move-tricky}.}\\
%\notag
\end{align}
cf.\ \cref{fig:square-move-tricky}.
Then the quivers $Q(G)$ and $Q(G')$ are mutation equivalent.
\end{proposition}

%There is one (rarely relevant) constraint on permissibility of square moves, see 
%A casual reader may ignore this subtle point. 

\begin{figure}[ht]
\begin{center}
\vspace{-5pt}
\setlength{\unitlength}{1pt}
\begin{picture}(50,25)(-30,15)
\thinlines
%\put(30,30){\circle{72}}
\thicklines
\put(10,22.5){\line(0,1){15}}
\put(30,22.5){\line(0,1){15}}
\qbezier(32.5,20)(50,20)(50,30)
\qbezier(32.5,40)(50,40)(50,30)
\put(12.5,20){\line(1,0){15}}

\put(8,39){\line(-2,-1){16}}
\put(12.5,40){\line(1,0){15}}
\put(-27.5,30){\line(1,0){15}}
\put(-87.5,30){\line(1,0){15}}
\put(8,21){\line(-2,1){16}}
\put(-52,21){\line(-2,1){16}}
\put(-32,31){\line(-2,1){16}}
\put(-32,29){\line(-2,-1){16}}
\put(-52,39){\line(-2,-1){16}}
\put(-50,42.5){\line(0,1){10}}
\put(-50,17.5){\line(0,-1){10}}

\put(10,20){\circle*{5}}
\put(10,40){\circle{5}}
\put(30,20){\circle{5}}
\put(30,40){\circle*{5}}
\put(20,30){\makebox(0,0){$A$}}
\put(-50,30){\makebox(0,0){$B$}}
\put(65,30){\makebox(0,0){$C$}}
\put(-10,30){\circle*{5}}
\put(-30,30){\circle{5}}
\put(-70,30){\circle{5}}
\put(-50,40){\circle*{5}}
\put(-50,20){\circle*{5}}
\end{picture}
\vspace{-10pt}
\end{center}
\caption{
Restriction \eqref{eq:square-move-tricky} allows 
the square move at~$A$---but not at~$B$, 
since face~$C$ is adjacent to two consecutive sides~of~$B$.
 }
\label{fig:square-move-tricky}
\end{figure}
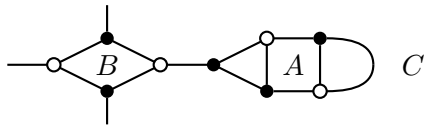

\begin{proof}
It is straightforward to check that a square move in a plabic graph
translates into a quiver mutation at the vertex associated to that 
	square face, provided that condition~\eqref{eq:square-move-tricky}
is satisfied. 
%consecutive faces surrounding the square face are distinct, 
%cf. Figure~\ref{fig:square-move-tricky}. 
It is also straightforward to check that the quiver associated with a plabic graph 
does not change under moves (M2) or (M3), % or a switch (as in Definition~\ref{def:switch}).
see Figure~\ref{fig:plabic-moves-are-mutations}.
\end{proof}

\pagebreak[3]

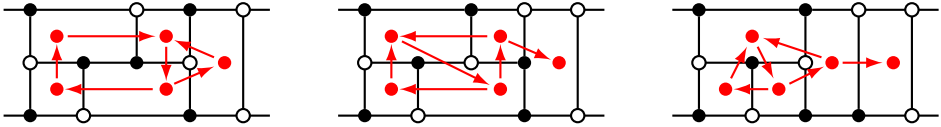
\begin{figure}[ht]
\begin{center}
\vspace{-.1in}
\setlength{\unitlength}{1pt}
\begin{picture}(100,40)(0,0)
\thicklines
%vertices
\put(10,0){\circle*{5}}
\put(30,0){\circle{5}}
\put(70,0){\circle*{5}}
\put(90,0){\circle{5}}
\put(10,20){\circle{5}}
\put(30,20){\circle*{5}}
\put(50,20){\circle*{5}}
\put(70,20){\circle{5}}
\put(10,40){\circle*{5}}
\put(50,40){\circle{5}}
\put(70,40){\circle*{5}}
\put(90,40){\circle{5}}

%edges
\put(0,0){\line(1,0){27.5}}
\put(32.5,0){\line(1,0){55}}
\put(92.5,0){\line(1,0){7.5}}
\put(12.5,20){\line(1,0){55}}
\put(0,40){\line(1,0){47.5}}
\put(52.5,40){\line(1,0){35}}
\put(92.5,40){\line(1,0){7.5}}

\put(10,2.5){\line(0,1){15}}
\put(10,22.5){\line(0,1){15}}
\put(30,17.5){\line(0,-1){15}}
\put(50,22.5){\line(0,1){15}}
\put(70,2.5){\line(0,1){15}}
\put(70,22.5){\line(0,1){15}}
\put(90,2.5){\line(0,1){35}}

%nodes
\put(20,10){\red{\circle*{5}}}
\put(61,10){\red{\circle*{5}}}

\put(20,30){\red{\circle*{5}}}
\put(61,30){\red{\circle*{5}}}
\put(83,20){\red{\circle*{5}}}

\put(24,30){\red{\vector(1,0){33}}}
\put(56,10){\red{\vector(-1,0){33}}}
\put(79,22){\red{\vector(-2.3,1){15}}}
\put(20,14){\red{\vector(0,1){13}}}
\put(61,26){\red{\vector(0,-1){13}}}
\put(64,12){\red{\vector(2.3,1){15}}}

\end{picture}
\qquad
\begin{picture}(100,40)(0,0)
\thicklines
%vertices
\put(10,0){\circle*{5}}
\put(30,0){\circle{5}}
\put(70,0){\circle*{5}}
\put(90,0){\circle{5}}
\put(10,20){\circle{5}}
\put(30,20){\circle*{5}}
\put(50,20){\circle{5}}
\put(70,20){\circle*{5}}
\put(10,40){\circle*{5}}
\put(50,40){\circle*{5}}
\put(70,40){\circle{5}}
\put(90,40){\circle{5}}

%edges
\put(0,0){\line(1,0){27.5}}
\put(32.5,0){\line(1,0){55}}
\put(92.5,0){\line(1,0){7.5}}
\put(12.5,20){\line(1,0){35}}
\put(52.5,20){\line(1,0){15}}
\put(0,40){\line(1,0){47.5}}
\put(52.5,40){\line(1,0){15}}
\put(72.5,40){\line(1,0){15}}
\put(92.5,40){\line(1,0){7.5}}

\put(10,2.5){\line(0,1){15}}
\put(10,22.5){\line(0,1){15}}
\put(30,17.5){\line(0,-1){15}}
\put(50,22.5){\line(0,1){15}}
\put(70,2.5){\line(0,1){15}}
\put(70,22.5){\line(0,1){15}}
\put(90,2.5){\line(0,1){35}}

%nodes
\put(20,10){\red{\circle*{5}}}
\put(61,10){\red{\circle*{5}}}

\put(20,30){\red{\circle*{5}}}
\put(61,30){\red{\circle*{5}}}
\put(83,20){\red{\circle*{5}}}

\put(56,30){\red{\vector(-1,0){33}}}
\put(56,10){\red{\vector(-1,0){33}}}
%\put(79,22){\red{\vector(-2.3,1){15}}}
\put(20,14){\red{\vector(0,1){13}}}
\put(61,14){\red{\vector(0,1){13}}}
\put(64,28){\red{\vector(2.3,-1){16}}}
\put(24,28){\red{\vector(2,-1){33}}}

\end{picture}
\qquad
\begin{picture}(100,40)(0,0)
\thicklines
%vertices
\put(10,0){\circle*{5}}
\put(30,0){\circle{5}}
\put(70,0){\circle*{5}}
\put(90,0){\circle{5}}
\put(10,20){\circle{5}}
\put(30,20){\circle*{5}}
\put(50,20){\circle{5}}
\put(50,0){\circle*{5}}
\put(10,40){\circle*{5}}
\put(50,40){\circle*{5}}
\put(70,40){\circle{5}}
\put(90,40){\circle{5}}

%edges
\put(0,0){\line(1,0){27.5}}
\put(32.5,0){\line(1,0){55}}
\put(92.5,0){\line(1,0){7.5}}
\put(12.5,20){\line(1,0){35}}
\
\put(0,40){\line(1,0){47.5}}
\put(52.5,40){\line(1,0){15}}
\put(72.5,40){\line(1,0){15}}
\put(92.5,40){\line(1,0){7.5}}

\put(10,2.5){\line(0,1){15}}
\put(10,22.5){\line(0,1){15}}
\put(30,17.5){\line(0,-1){15}}
\put(50,22.5){\line(0,1){15}}
\put(50,2.5){\line(0,1){15}}
\put(70,2.5){\line(0,1){35}}
\put(90,2.5){\line(0,1){35}}

%nodes
\put(20,10){\red{\circle*{5}}}
\put(40,10){\red{\circle*{5}}}

\put(30,30){\red{\circle*{5}}}
\put(60,20){\red{\circle*{5}}}
\put(83,20){\red{\circle*{5}}}

\put(56,22){\red{\vector(-3,1){22}}}
\put(36,10){\red{\vector(-1,0){13}}}
%\put(79,22){\red{\vector(-2.3,1){15}}}
\put(22,14){\red{\vector(1,2){6}}}
\put(44,12){\red{\vector(2,1){13}}}

\put(64,20){\red{\vector(1,0){15}}}
\put(32,26){\red{\vector(1,-2){6}}}

\end{picture}
\vspace{-.05in}
\end{center}
\caption{Fragments of plabic graphs and their associated quivers.
The first two plabic graphs are related by a square move (M1); their quivers are related by a single mutation.
The second and the third graphs are related by moves of type (M3), 
and have isomorphic quivers.
}
%\vspace{-.1in}
\label{fig:plabic-moves-are-mutations}
\end{figure}
\begin{remark}
\label{rem:square-move-tricky}
Suppose that condition~\eqref{eq:square-move-tricky} fails at a square face~$B$,
with $B$~incident to another face~$C$ along
two consecutive edges, as in \cref{fig:square-move-tricky}. 
Then the arrows transversal to these edges cancel each other,
so they do not appear in the associated quiver. 
This leads to a discrepancy 
between the square move and the quiver mutation. 
\end{remark}

\begin{remark} 
\label{rem:urban}
%The key difference between the setting of this section \emph{vs.}\ 
%Section~\ref{urban} is that here we do not require plabic graphs to be bipartite.  
%This distinction is not particularly important, 
%since we 
One can always apply a sequence of moves (M2)  to a plabic graph to make it bipartite.
\end{remark}

\begin{remark}
\label{rem:restricted-M1}
In light of \cref{pr:plabic-vs-quivers}, one may choose to adjust the definition
of the square move~(M1)---hence the notion of move-equivalence of plabic graphs---by
forbidding square moves violating condition~\eqref{eq:square-move-tricky}. 
(This convention was adopted in~\cite{fpst}.) 
We note that for the important subclass of \emph{reduced} plabic graphs
(see \cref{def:reduced-plabic}),
condition~\eqref{eq:square-move-tricky} is automatically satisfied, so there is no need to worry about it.  
\end{remark}

\begin{remark}
\label{rem:cluster-algebra-from-plabic}
Using \cref{def:Q(G)}, we can associate 
a seed pattern---hence a cluster algebra---to any plabic graph~$G$.
By \cref{pr:plabic-vs-quivers}, 
this cluster algebra only depends on the move-equivalence class of~$G$,
assuming that we adopt a restricted notion of move-equivalence, cf.\ \cref{rem:restricted-M1}. 
We will soon see that this family of cluster algebras includes all the main examples
of cluster algebras (defined by quivers) introduced in the earlier chapters. 
This justifies the importance of the combinatorial study of plabic graphs,
and in particular their classification up to move-equivalence. 
\end{remark}

The quivers arising from plabic graphs are quite special; in particular, they are planar. 
%In particular, each quiver $Q(G)$ is planar.  
%moreover $Q(G)$ possesses a planar embedding in which
%the arrows surrounding every bounded face form an oriented cycle. 
Nevertheless, \cref{pr:embed-into-plabic} below, stated here without proof, 
shows that
%asserts that the mutation class of any quiver without frozen vertices 
%can be embedded into the mutation class of a quiver of a plabic~graph.
%Thus 
quivers from plabic graphs are, in some sense, ``universal.''

\begin{proposition}[{\cite{FIL}}]
\label{pr:embed-into-plabic}
Let $Q$ be a quiver whose vertices are all mutable. 
Then there exists a plabic graph~$G$ such that $Q$ 
is a full subquiver (see Definition~\ref{def:tildeB_I})
of a quiver mutation-equivalent to~$Q(G)$. 
\end{proposition}

\clearpage
\newpage
\section{Triangulations and wiring diagrams via plabic graphs}
\label{sec:plabictriangulations}

In this section, we explain how the machinery of local moves on plabic graphs 
unifies the combinatorial constructions of Chapter~\ref{ch:combinatorics-of-mutations},
%Sections~\ref{sec:triangulations}--\ref{sec:mut-double-wiring}, 
including: 
\begin{itemize}[leftmargin=.2in]
\item
flips in triangulations of a polygon 
%by its diagonals 
		(Section~\ref{sec:triangulations});
\item
braid moves for wiring diagrams (Section~\ref{sec:mut-wiring});
\item
their analogues for double wiring diagrams (Section~\ref{sec:mut-double-wiring}). 
\end{itemize}
As mentioned 
in Remark~\ref{rem:urban},
the spider move for bipartite graphs  (Section~\ref{urban}) 
is equivalent to the square move for plabic graphs. 

\begin{example}[{\textbf{Triangulations of a polygon}, see \cite[Algorithm~12.1]{kodwil}}]  
\label{TriangulationA}
Let $T$ be a triangulation of a convex $m$-gon $\mathbf{P}_m\,$. 
The plabic graph $G(T)$ associated to $T$ is constructed as follows: 
\begin{enumerate}[leftmargin=.3in]
\item Place a white vertex of $G(T)$ at each vertex of~$\mathbf{P}_m$.
\item Place a black vertex of $G(T)$ in the interior of each triangle~of~$T$.
Connect it by edges to the three white vertices of the triangle.
\item Embed $\mathbf{P}_m$ into the interior of a disk~$\mathbf{D}$.
\item Place $m$ uncolored vertices of $G(T)$ on the boundary of~$\mathbf{D}$.
\item Connect each white vertex of $G(T)$ to a boundary vertex. These edges must not cross.  
\end{enumerate}
We emphasize that the set of edges of $G(T)$ includes
neither the sides of $\mathbf{P}_m$ nor the diagonals of~$T$.
See \cref{fig:triang-plabic0}. 
\end{example}

\begin{figure}[ht]
\begin{center}
%\vspace{.1in}
\setlength{\unitlength}{1.7pt}
\begin{picture}(50,55)(0,-3)
\thicklines
  \multiput(0,20)(60,0){2}{\red{\line(0,1){20}}}
  \multiput(20,0)(0,60){2}{\red{\line(1,0){20}}}
  \multiput(0,40)(40,-40){2}{\red{\line(1,1){20}}}
  \multiput(20,0)(40,40){2}{\red{\line(-1,1){20}}}

  \multiput(20,0)(20,0){2}{\red{\circle*{1.5}}}
  \multiput(20,60)(20,0){2}{\red{\circle*{1.5}}}
  \multiput(0,20)(0,20){2}{\red{\circle*{1.5}}}
  \multiput(60,20)(0,20){2}{\red{\circle*{1.5}}}

% vertices are: (20,0), (40,0), (20,60), (40,60), (0,20), (0,40), (60,20), (60,40)
\thicklines

% \thinlines
\put(40,0){\red{\line(1,2){20}}}
\put(0,40){\red{\line(1,0){60}}}
\put(0,20){\red{\line(2,-1){40}}}
\put(0,40){\red{\line(1,-1){40}}}
\put(20,60){\red{\line(2,-1){40}}}

\end{picture}
	\hspace{1in}
\setlength{\unitlength}{1.35pt}
\begin{picture}(50,50)(0,-12)

% layer the green lines above the black lines
% adjust line lengths
% add dots on circle
\thicklines
  \multiput(20,0)(20,0){2}{\circle{3}}
  \multiput(20,60)(20,0){2}{\circle{3}}
  \multiput(0,20)(0,20){2}{\circle{3}}
  \multiput(60,20)(0,20){2}{\circle{3}}

\put(32,28){\circle*{3}}
\put(12,21){\circle*{3}}
\put(22,5){\circle*{3}}
\put(55,21){\circle*{3}}
\put(20,48){\circle*{3}}
\put(38,55){\circle*{3}}

\thicklines
\put(30,30){\circle{80}}

% \put(32,28){\circle*{3}}
% \put(12,21){\circle*{3}}
% \put(22,5){\circle*{3}}
% \put(55,21){\circle*{3}}
% \put(20,48){\circle*{3}}
% \put(38,55){\circle*{3}}

% \put(19,-1){\line(-0.6,-1){3.8}}
% \put(41,-1){\line(0.6,-1){3.8}}
% \put(19,61){\line(-0.6,1){3.8}}
% \put(41,61){\line(0.6,1){3.8}}

\put(32,28){\line(1,0.42){27}}
\put(32,28){\line(-1,0.38){30.5}}
\put(32,28){\line(0.3,-1){8}}

\put(12,21){\line(-0.65,1){11.5}}
\put(12,21){\line(1,-0.75){27}}
\put(12,21){\line(-1,-0.1){10.5}}

\put(22,5){\line(1,-0.27){16.5}}
\put(22,5){\line(-0.35,-1){1.3}}
\put(22,5){\line(-1,0.67){21}}

\put(55,21){\line(1, -0.15){3.5}}
\put(55,21){\line(0.27, 1){4.7}}
\put(55,21){\line(-0.7, -1){14}}

\put(20,48){\line(0, 1){10.5}}
\put(20,48){\line(-1, -0.4){18.8}}
\put(20,48){\line(5, -1){38.5}}

\put(38,55){\line(1, 2.5){1.5}}
\put(38,55){\line(-3.6, 1){16.5}}
\put(38,55){\line(4.4, -3){21}}

\multiput(20,60)(20,0){2}{\circle{3}}
\multiput(0,20)(0,20){2}{\circle{3}}
\multiput(60,20)(0,20){2}{\circle{3}}
  
\put(19,-1){\line(-0.6,-1){3.8}}
\put(41,-1){\line(0.6,-1){3.8}}
  
\put(19,61){\line(-0.6,1){3.8}}
\put(41,61){\line(0.6,1){3.8}}
  
\put(-1, 19){\line(-1, -0.6){6}}
\put(-1, 41){\line(-1, 0.6){6}}

\put(61, 19){\line(1, -0.6){6}}
\put(61, 41){\line(1, 0.6){6}}
\thicklines

% \thinlines
	\multiput(0,20)(60,0){2}{\red{\line(0,1){20.05}}}
	\multiput(20,0)(0,60){2}{\red{\line(1,0){20.05}}}
	\multiput(0,40)(40,-40){2}{\red{\line(1,1){20.05}}}
	\multiput(20,0)(40,40){2}{\red{\line(-1,1){20.05}}}
\thicklines

% \thinlines
	\put(40,0){\red{\line(1,2){20.05}}}
	\put(0,40){\red{\line(1,0){60.05}}}
	\put(0,20){\red{\line(2,-1){40.05}}}
	\put(0,40){\red{\line(1,-1){40.05}}}
	\put(20,60){\red{\line(2,-1){40.05}}}

% \put(19,-1){\line(-0.6,-1){3.8}}
% \put(41,-1){\line(0.6,-1){3.8}}
% \put(19,61){\line(-0.6,1){3.8}}
% \put(41,61){\line(0.6,1){3.8}}
% \put(-1, 19){\line(-1, -0.6){6}}
% \put(-1, 41){\line(-1, 0.6){6}}
% \put(61, 19){\line(1, -0.6){6}}
% \put(61, 41){\line(1, 0.6){6}}
\multiput(15.3,-7.2)(29.2,0){2}{\circle*{2}}
\multiput(15.3,67.2)(29.2,0){2}{\circle*{2}}
\multiput(-7.3,15.4)(0,29.2){2}{\circle*{2}}
\multiput(67.3,15.4)(0,29.2){2}{\circle*{2}}
\end{picture}
\vspace{-15pt}
\end{center}
\caption{A triangulation $T$ of an octagon, and the corresponding
	plabic graph $G(T)$, cf.\ Figure~\ref{fig:quiver-triangulation}.}
\label{fig:triang-plabic0}
\end{figure}

\begin{exercise}
	Show that $Q(G(T)) = Q(T)$, 
i.e., the quiver associated to the plabic graph of a triangulation $T$ 
	coincides with the quiver $Q(T)$ associated to $T$, as in Definition~\ref{def:Q(T)-polygon}.
\end{exercise}

\begin{exercise}
Show that if triangulations $T$ and $T'$ are related by a flip, 
then the plabic graphs $G(T)$ and $G(T')$ are move-equivalent to each~other.
More concretely, flipping a diagonal in~$T$ translates into a square move 
at the corresponding quadrilateral face of~$G(T)$,
plus some (M3) moves to make each vertex of that face trivalent.
\end{exercise}

\begin{example}[{\textbf{Wiring diagrams}}]
\label{def:wiringplabic}

Let $D$ be a wiring diagram, as in Section~\ref{sec:baseaffine}.
We associate a plabic graph $G(D)$ to~$D$ by replacing 
each crossing in $D$ by a pair of trivalent
vertices connected vertically, with a black vertex on top
and a white vertex on the bottom.
We then enclose the resulting graph~in~a~disk.
%and label
%the boundary vertices from $1$ to $2n$ in clockwise order,
%starting at the lower left.

\nopagebreak

% !!! Testing -- commented this one for now but put it back later
%\newsavebox{\ssone}
%\setlength{\unitlength}{2.4pt} 
%\savebox{\ssone}(10,20)[bl]{
%\thicklines 
%\qbezier(5,5)(7,10)(10,10)
%\qbezier(5,5)(3,0)(0,0)
%\qbezier(5,5)(3,10)(0,10)
%\qbezier(5,5)(7,0)(10,0)
%\put(0,20){\line(1,0){10}} 
%}

%\begin{figure}
%\begin{picture}(60,40)(0,-18) 
%\put(0,0){\makebox(0,0){\usebox{\ssone}}} 
%\put(10,0){\makebox(0,0){\usebox{\sstwo}}} 
%\put(20,0){\makebox(0,0){\usebox{\ssone}}} 
%\end{picture}
%\setlength{\unitlength}{1.2pt}
%\begin{picture}(60,40)(0,-18)
%\thicklines
%\put(0,40){\line(1,0){28.5}}
%\put(60,40){\line(-1,0){28.5}}
%\put(0,20){\line(1,0){8.5}}
%\put(11.5,20){\line(1,0){37}}
%\put(60,20){\line(-1,0){8.5}}
%\put(0,0){\line(1,0){60}}
%\put(10,1.5){\line(0,1){17}}
%\put(30,21.5){\line(0,1){17}}
%\put(50,1.5){\line(0,1){17}}
%\put(10,0){\circle{3}}
%\put(30,20){\circle{3}}
%\put(50,0){\circle{3}}
%\put(10,20){\circle*{3}}
%\put(30,40){\circle*{3}}
%\put(50,20){\circle*{3}}
%\end{picture}
%	\caption{ A wiring diagram $D$ and the 
%	corresponding plabic graph $G(D)$. }
%\end{figure}

%%%%%%%%%%%%%%%%%%%%%%%%%%%%%%%%%%%%%%%%%%%%%%%%%%%%%%%%

\newsavebox{\sssone}
\setlength{\unitlength}{2pt} 
\savebox{\sssone}(10,30)[bl]{
\thicklines 
\qbezier(5,5)(7,10)(10,10)
\qbezier(5,5)(3,0)(0,0)
\qbezier(5,5)(3,10)(0,10)
\qbezier(5,5)(7,0)(10,0)
\put(0,20){\line(1,0){10}} 
\put(0,30){\line(1,0){10}} 
}

\newsavebox{\ssstwo}
\setlength{\unitlength}{2pt} 
\savebox{\ssstwo}(10,30)[bl]{
\thicklines 
\qbezier(5,15)(7,20)(10,20)
\qbezier(5,15)(3,10)(0,10)
\qbezier(5,15)(3,20)(0,20)
\qbezier(5,15)(7,10)(10,10)
\put(0,0){\line(1,0){10}} 
\put(0,30){\line(1,0){10}} 
}

\newsavebox{\sssthree}
\setlength{\unitlength}{2pt} 
\savebox{\sssthree}(10,30)[bl]{
\thicklines 
\qbezier(5,25)(7,30)(10,30)
\qbezier(5,25)(3,20)(0,20)
\qbezier(5,25)(3,30)(0,30)
\qbezier(5,25)(7,20)(10,20)
\put(0,0){\line(1,0){10}} 
\put(0,10){\line(1,0){10}} 
}

\newsavebox{\sssonethree}
\setlength{\unitlength}{2pt} 
\savebox{\sssonethree}(10,30)[bl]{
\thicklines 
\qbezier(5,25)(7,30)(10,30)
\qbezier(5,25)(3,20)(0,20)
\qbezier(5,25)(3,30)(0,30)
\qbezier(5,25)(7,20)(10,20)
\qbezier(5,5)(7,10)(10,10)
\qbezier(5,5)(3,0)(0,0)
\qbezier(5,5)(3,10)(0,10)
\qbezier(5,5)(7,0)(10,0)
}

\newsavebox{\ssslines}
\setlength{\unitlength}{2pt} 
\savebox{\ssslines}(5,30)[bl]{
\thicklines 
\put(0,0){\line(1,0){5}} 
\put(0,10){\line(1,0){5}} 
\put(0,20){\line(1,0){5}} 
\put(0,30){\line(1,0){5}} 
}

%%%%%%%%%%%%%%%%%%%%%%%%%%%%%%%%%%%%%%%%%%%%%%%%%%%%%%%%

\begin{figure}[ht]
\begin{center}
\vspace{-1.2cm}
\setlength{\unitlength}{2pt} 
\begin{picture}(70,30)(-10,0)
\thicklines
% \thinlines
%\put(10,-10){\makebox(0,0){\lightblue{$\Delta_{1}$}}}
%\put(-8,0){\makebox(0,0){\lightblue{$\Delta_{12}$}}}
%\put(0,10){\makebox(0,0){\lightblue{$\Delta_{123}$}}}

%\put(10,0){\makebox(0,0){\lightred{$\Delta_{13}$}}}
%\put(30,0){\makebox(0,0){\lightred{$\Delta_{14}$}}}
%\put(40,-10){\makebox(0,0){\lightblue{$\Delta_{4}$}}}

%\put(30,10){\makebox(0,0){\lightred{$\Delta_{134}$}}}
%\put(50,0){\makebox(0,0){\lightblue{$\Delta_{34}$}}}

%\put(58,10){\makebox(0,0){\lightblue{$\Delta_{234}$}}}

\put(-7.5,0){\makebox(0,0){\usebox{\ssslines}}} 
\put(0,0){\makebox(0,0){\usebox{\ssstwo}}} 
\put(10,0){\makebox(0,0){\usebox{\sssthree}}} 
\put(20,0){\makebox(0,0){\usebox{\ssstwo}}} 
\put(30,0){\makebox(0,0){\usebox{\sssone}}} 
\put(40,0){\makebox(0,0){\usebox{\ssstwo}}} 
\put(50,0){\makebox(0,0){\usebox{\sssthree}}} 
\put(57.5,0){\makebox(0,0){\usebox{\ssslines}}} 
\end{picture}
\quad
	\hspace{.02cm}
\begin{picture}(70,30)(-10,0)
\thicklines

% \thinlines
%\put(10,-10){\makebox(0,0){\lightblue{$\Delta_{1}$}}}
%\put(0,0){\makebox(0,0){\lightblue{$\Delta_{12}$}}}
%\put(-8,10){\makebox(0,0){\lightblue{$\Delta_{123}$}}}

%\put(10,10){\makebox(0,0){\lightred{$\Delta_{124}$}}}
%\put(25,0){\makebox(0,0){\lightred{$\Delta_{14}$}}}
%\put(40,-10){\makebox(0,0){\lightblue{$\Delta_{4}$}}}

%\put(35,10){\makebox(0,0){\lightred{$\Delta_{134}$}}}
%\put(50,0){\makebox(0,0){\lightblue{$\Delta_{34}$}}}

%\put(58,10){\makebox(0,0){\lightblue{$\Delta_{234}$}}}
\put(-7.5,0){\makebox(0,0){\usebox{\ssslines}}} 
\put(0,0){\makebox(0,0){\usebox{\sssthree}}} 
\put(10,0){\makebox(0,0){\usebox{\ssstwo}}} 
\put(20,0){\makebox(0,0){\usebox{\sssthree}}} 
\put(30,0){\makebox(0,0){\usebox{\sssone}}} 
\put(40,0){\makebox(0,0){\usebox{\ssstwo}}} 
\put(50,0){\makebox(0,0){\usebox{\sssthree}}} 
\put(57.5,0){\makebox(0,0){\usebox{\ssslines}}} 
\end{picture}
\end{center}
	\vspace{1cm}
\begin{center}
\setlength{\unitlength}{1pt}
\begin{picture}(140,65)(0,0)
\thicklines

% \thinlines
\put(0,40){\line(1,0){18.5}}
\put(21.5,40){\line(1,0){37}}
\put(61.5,40){\line(1,0){37}}
\put(101.5,40){\line(1,0){38.5}}
\put(0,20){\line(1,0){78.5}}
\put(140,20){\line(-1,0){58.5}}
\put(70,20){\line(-1,0){28.5}}
\put(0,0){\line(1,0){140}}
\put(0,60){\line(1,0){38.5}}
\put(41.5,60){\line(1,0){77}}
\put(121.5,60){\line(1,0){18.5}}
\put(20,20){\line(0,1){18.5}}
\put(80,0){\line(0,1){18.5}}
\put(60,20){\line(0,1){18.5}}
\put(100,20){\line(0,1){18.5}}
\put(40,40){\line(0,1){18.5}}
\put(120,40){\line(0,1){18.5}}
\put(20,20){\circle{3}}
\put(80,0){\circle{3}}
\put(60,20){\circle{3}}
\put(100,20){\circle{3}}
\put(40,40){\circle{3}}
\put(120,40){\circle{3}}
\put(20,40){\circle*{3}}
\put(80,20){\circle*{3}}
\put(60,40){\circle*{3}}
\put(100,40){\circle*{3}}
\put(40,60){\circle*{3}}
\put(120,60){\circle*{3}}
\put(20,20){\white{\circle*{2.7}}}
\put(80,0){\white{\circle*{2.7}}}
\put(60,20){\white{\circle*{2.7}}}
\put(100,20){\white{\circle*{2.7}}}
\put(40,40){\white{\circle*{2.7}}}
\put(120,40){\white{\circle*{2.7}}}

\thicklines
%\put(40,10){\lightblue{\circle{6}}}\put(40,10){\lightblue{\circle{3}}}
%\put(100,10){\lightblue{\circle{6}}}\put(100,10){\lightblue{\circle{3}}}
%\put(10,30){\lightblue{\circle{6}}}\put(10,30){\lightblue{\circle{3}}}
%\put(120,30){\lightblue{\circle{6}}}\put(120,30){\lightblue{\circle{3}}}
%\put(20,50){\lightblue{\circle{6}}}\put(20,50){\lightblue{\circle{3}}}
%\put(130,50){\lightblue{\circle{6}}}\put(130,50){\lightblue{\circle{3}}}
%\put(40,30){\lightred{\circle{6}}}
%\put(80,30){\lightred{\circle{6}}}
%\put(80,50){\lightred{\circle{6}}}
\linethickness{1.5pt}
%\put(15,30){\vector(1,0){20}}
%\put(45,30){\vector(1,0){30}}
%\put(85,30){\vector(1,0){30}}
%\qbezier(42.5,35)(50,50)(75,50)
%\put(42.5,35){\vector(-1,-2){0.7}}
%\qbezier(117.5,35)(110,50)(86,50)
%\put(86,50){\vector(-1,0){2.3}}
%\qbezier(84,52)(105,60)(125,52)
%\put(125,52){\vector(4,-1){1.5}}
%\qbezier(24,52)(50,60)(74,52)
%\put(74,52){\vector(4,-1){1.5}}
%\put(36.5,33.5){\vector(-1,1){13.5}}
%\qbezier(77.5,25)(70,10)(46,10)
%\put(46,10){\vector(-1,0){2.3}}
%\put(96.5,13.5){\vector(-1,1){13.5}}
\end{picture}
\quad
\setlength{\unitlength}{1pt}
\begin{picture}(140,65)(0,0)
\thicklines

% \thinlines
\put(0,40){\line(1,0){38.5}}
\put(41.5,40){\line(1,0){57}}
\put(101.5,40){\line(1,0){38.5}}
\put(0,20){\line(1,0){78.5}}
\put(140,20){\line(-1,0){58.5}}
\put(70,20){\line(-1,0){28.5}}
\put(0,0){\line(1,0){140}}
\put(0,60){\line(1,0){18.5}}
\put(21.5,60){\line(1,0){37}}
\put(61.5,60){\line(1,0){57}}
\put(121.5,60){\line(1,0){18.5}}

\put(20,40){\line(0,1){18.5}}
\put(80,0){\line(0,1){18.5}}
\put(60,40){\line(0,1){18.5}}
\put(100,20){\line(0,1){18.5}}
\put(40,20){\line(0,1){18.5}}
\put(120,40){\line(0,1){18.5}}
\put(20,40){\circle{3}}
\put(80,0){\circle{3}}
\put(60,40){\circle{3}}
\put(100,20){\circle{3}}
\put(40,20){\circle{3}}
\put(120,40){\circle{3}}
\put(20,60){\circle*{3}}
\put(80,20){\circle*{3}}
\put(60,60){\circle*{3}}
\put(100,40){\circle*{3}}
\put(40,40){\circle*{3}}
\put(120,60){\circle*{3}}

\put(20,40){\white{\circle*{2.7}}}
\put(80,0){\white{\circle*{2.7}}}
\put(60,40){\white{\circle*{2.7}}}
\put(100,20){\white{\circle*{2.7}}}
\put(40,20){\white{\circle*{2.7}}}
\put(120,40){\white{\circle*{2.7}}}

\thicklines
%\put(40,10){\lightblue{\circle{6}}}\put(40,10){\lightblue{\circle{3}}}
%\put(100,10){\lightblue{\circle{6}}}\put(100,10){\lightblue{\circle{3}}}
%\put(20,30){\lightblue{\circle{6}}}\put(20,30){\lightblue{\circle{3}}}
%\put(120,30){\lightblue{\circle{6}}}\put(120,30){\lightblue{\circle{3}}}
%\put(10,50){\lightblue{\circle{6}}}\put(10,50){\lightblue{\circle{3}}}
%\put(130,50){\lightblue{\circle{6}}}\put(130,50){\lightblue{\circle{3}}}
%\put(40,50){\lightred{\circle{6}}}
%\put(80,30){\lightred{\circle{6}}}
%\put(80,50){\lightred{\circle{6}}}
\linethickness{1.5pt}
%\put(25,30){\vector(1,0){50}}
%\put(85,30){\vector(1,0){30}}
%\put(15,50){\vector(1,0){20}}
%\put(45,50){\vector(1,0){30}}
%\qbezier(43.5,45)(50,36)(75,32)
%\put(43.5,45){\vector(-1,1){1.2}}
%\qbezier(117.5,35)(110,50)(86,50)
%\put(86,50){\vector(-1,0){2.3}}
%\qbezier(84,52)(105,60)(125,52)
%\put(125,52){\vector(4,-1){1.5}}
%\qbezier(24,52)(50,60)(74,52)
%\put(74,52){\vector(4,-1){1.5}}
%\put(36.5,46.5){\vector(-1,-1){13.5}}
%\qbezier(77.5,25)(70,10)(46,10)
%\put(46,10){\vector(-1,0){2.3}}
%\put(96.5,13.5){\vector(-1,1){13.5}}
%\put(80,45){\vector(0,-1){10.8}}
\end{picture}
\end{center}
	\vspace{.1cm}
\begin{center}
\setlength{\unitlength}{1pt}
\begin{picture}(140,65)(0,0)
\thicklines
% \thinlines

\put(0,40){\line(1,0){18.5}}
\put(21.5,40){\line(1,0){37}}
\put(61.5,40){\line(1,0){37}}
\put(101.5,40){\line(1,0){38.5}}
\put(0,20){\line(1,0){78.5}}
\put(140,20){\line(-1,0){58.5}}
\put(70,20){\line(-1,0){28.5}}
\put(0,0){\line(1,0){140}}
\put(0,60){\line(1,0){38.5}}
\put(41.5,60){\line(1,0){77}}
\put(121.5,60){\line(1,0){18.5}}
\put(20,20){\line(0,1){18.5}}
\put(80,0){\line(0,1){18.5}}
\put(60,20){\line(0,1){18.5}}
\put(100,20){\line(0,1){18.5}}
\put(40,40){\line(0,1){18.5}}
\put(120,40){\line(0,1){18.5}}
\put(20,20){\circle{3}}
\put(80,0){\circle{3}}
\put(60,20){\circle{3}}
\put(100,20){\circle{3}}
\put(40,40){\circle{3}}
\put(120,40){\circle{3}}
\put(20,40){\circle*{3}}
\put(80,20){\circle*{3}}
\put(60,40){\circle*{3}}
\put(100,40){\circle*{3}}
\put(40,60){\circle*{3}}
\put(120,60){\circle*{3}}

\put(20,20){\white{\circle*{2.7}}}
\put(80,0){\white{\circle*{2.7}}}
\put(60,20){\white{\circle*{2.7}}}
\put(100,20){\white{\circle*{2.7}}}
\put(40,40){\white{\circle*{2.7}}}
\put(120,40){\white{\circle*{2.7}}}
\thicklines
	\put(37,7){\red{\line(0,1){6}}}
	\put(37,7){\red{\line(1,0){6}}}
	\put(43,13){\red{\line(0,-1){6}}}
	\put(43,13){\red{\line(-1,0){6}}}
	\put(97,7){\red{\line(0,1){6}}}
	\put(97,7){\red{\line(1,0){6}}}
	\put(103,13){\red{\line(0,-1){6}}}
	\put(103,13){\red{\line(-1,0){6}}}
	\put(7,27){\red{\line(0,1){6}}}
	\put(7,27){\red{\line(1,0){6}}}
	\put(13,33){\red{\line(0,-1){6}}}
	\put(13,33){\red{\line(-1,0){6}}}
	\put(117,27){\red{\line(0,1){6}}}
	\put(117,27){\red{\line(1,0){6}}}
	\put(123,33){\red{\line(0,-1){6}}}
	\put(123,33){\red{\line(-1,0){6}}}
	\put(17,47){\red{\line(0,1){6}}}
	\put(17,47){\red{\line(1,0){6}}}
	\put(23,53){\red{\line(0,-1){6}}}
	\put(23,53){\red{\line(-1,0){6}}}
	\put(127,47){\red{\line(0,1){6}}}
	\put(127,47){\red{\line(1,0){6}}}
	\put(133,53){\red{\line(0,-1){6}}}
	\put(133,53){\red{\line(-1,0){6}}}
	
%	\put(67,67){\red{\line(0,1){6}}}
%	\put(67,67){\red{\line(1,0){6}}}
%	\put(73,73){\red{\line(0,-1){6}}}
%	\put(73,73){\red{\line(-1,0){6}}}
%	\put(67,-13){\red{\line(0,1){6}}}
%	\put(67,-13){\red{\line(1,0){6}}}
%	\put(73,-7){\red{\line(0,-1){6}}}
%	\put(73,-7){\red{\line(-1,0){6}}}
\put(40,30){\red{\circle*{5}}}
\put(80,30){\red{\circle*{5}}}
\put(80,50){\red{\circle*{5}}}
\linethickness{1pt}
\put(35,30){\red{\vector(-1,0){20}}}
\put(75,30){\red{\vector(-1,0){30}}}
\put(115,30){\red{\vector(-1,0){30}}}
\put(42.5,35){\red{\vector(2.2,1){31}}}
% {\red{\qbezier(42.5,35)(50,50)(75,50)}}
% \put(76,50){\red{\vector(1,0){0.7}}}
\put(86,48){\red{\vector(2.5,-1){32}}}
% {\red{\qbezier(117.5,35)(110,50)(86,50)}}
%\put(117,36){\red{\vector(2,-3){1.5}}}
\put(125,50){\red{\vector(-1,0){40}}}
% {\red{\qbezier(84,52)(105,60)(125,52)}}
% \put(84,52){\red{\vector(-4,-1){1.5}}}
\put(74,50){\red{\vector(-1,0){47}}}
% {\red{\qbezier(25,52)(50,60)(74,52)}}
% \put(25,52){\red{\vector(-4,-1){1.5}}}
\put(46,10){\red{\vector(2,1){31.5}}}
% \put(24,46){\red{\vector(1,-1){12.5}}}
% {\red{\qbezier(77.5,25)(70,10)(46,10)}}
% \put(77.5,25){\red{\vector(2,3){1}}}
\put(83.5,26.5){\red{\vector(1,-1){12.5}}}
\end{picture}
\quad
\setlength{\unitlength}{1pt}
\begin{picture}(140,65)(0,0)
\thicklines
% \thinlines
\put(0,40){\line(1,0){38.5}}
\put(41.5,40){\line(1,0){57}}
\put(101.5,40){\line(1,0){38.5}}
\put(0,20){\line(1,0){78.5}}
\put(140,20){\line(-1,0){58.5}}
\put(70,20){\line(-1,0){28.5}}
\put(0,0){\line(1,0){140}}
\put(0,60){\line(1,0){18.5}}
\put(21.5,60){\line(1,0){37}}
\put(61.5,60){\line(1,0){57}}
\put(121.5,60){\line(1,0){18.5}}

\put(20,40){\line(0,1){18.5}}
\put(80,0){\line(0,1){18.5}}
\put(60,40){\line(0,1){18.5}}
\put(100,20){\line(0,1){18.5}}
\put(40,20){\line(0,1){18.5}}
\put(120,40){\line(0,1){18.5}}
\put(20,40){\circle{3}}
\put(80,0){\circle{3}}
\put(60,40){\circle{3}}
\put(100,20){\circle{3}}
\put(40,20){\circle{3}}
\put(120,40){\circle{3}}
\put(20,60){\circle*{3}}
\put(80,20){\circle*{3}}
\put(60,60){\circle*{3}}
\put(100,40){\circle*{3}}
\put(40,40){\circle*{3}}
\put(120,60){\circle*{3}}
        \thicklines
	\put(37,7){\red{\line(0,1){6}}}
	\put(37,7){\red{\line(1,0){6}}}
	\put(43,13){\red{\line(0,-1){6}}}
	\put(43,13){\red{\line(-1,0){6}}}
	\put(97,7){\red{\line(0,1){6}}}
	\put(97,7){\red{\line(1,0){6}}}
	\put(103,13){\red{\line(0,-1){6}}}
	\put(103,13){\red{\line(-1,0){6}}}
	\put(17,27){\red{\line(0,1){6}}}
	\put(17,27){\red{\line(1,0){6}}}
	\put(23,33){\red{\line(0,-1){6}}}
	\put(23,33){\red{\line(-1,0){6}}}
	\put(117,27){\red{\line(0,1){6}}}
	\put(117,27){\red{\line(1,0){6}}}
	\put(123,33){\red{\line(0,-1){6}}}
	\put(123,33){\red{\line(-1,0){6}}}
	\put(7,47){\red{\line(0,1){6}}}
	\put(7,47){\red{\line(1,0){6}}}
	\put(13,53){\red{\line(0,-1){6}}}
	\put(13,53){\red{\line(-1,0){6}}}
	\put(127,47){\red{\line(0,1){6}}}
	\put(127,47){\red{\line(1,0){6}}}
	\put(133,53){\red{\line(0,-1){6}}}
	\put(133,53){\red{\line(-1,0){6}}}
	
%	\put(67,67){\red{\line(0,1){6}}}
%	\put(67,67){\red{\line(1,0){6}}}
%	\put(73,73){\red{\line(0,-1){6}}}
%	\put(73,73){\red{\line(-1,0){6}}}
%	\put(67,-13){\red{\line(0,1){6}}}
%	\put(67,-13){\red{\line(1,0){6}}}
%	\put(73,-7){\red{\line(0,-1){6}}}
%	\put(73,-7){\red{\line(-1,0){6}}}
\put(40,50){\red{\circle*{5}}}
\put(70,30){\red{\circle*{5}}}
\put(90,50){\red{\circle*{5}}}

% \linethickness{1pt}
%\qbezier(70,25)(70,10)(46,10)
\thicklines
\put(115,30){\red{\vector(-1,0){40}}}
\put(65,30){\red{\vector(-1,0){40}}}
\put(85,50){\red{\vector(-1,0){40}}}
\put(125,50){\red{\vector(-1,0){30}}}
\put(35,50){\red{\vector(-1,0){20}}}
\put(24.5,34.5){\red{\vector(1,1){12.5}}}
\put(74,34){\red{\vector(1,1){13}}}
\put(44.5,47){\red{\vector(3,-2){22.5}}}
\put(94.5,47){\red{\vector(3,-2){21.5}}}
\put(74.5,27){\red{\vector(3,-2){21.5}}}
\put(44.5,13){\red{\vector(3,2){22.5}}}
\thicklines
\put(20,40){\white{\circle*{2.7}}}
\put(80,0){\white{\circle*{2.7}}}
\put(60,40){\white{\circle*{2.7}}}
\put(100,20){\white{\circle*{2.7}}}
\put(40,20){\white{\circle*{2.7}}}
\put(120,40){\white{\circle*{2.7}}}

\end{picture}
\vspace{-5pt}
\end{center}
\caption{Top: wiring diagrams $D_1$ and $D_2$ associated to 
reduced expressions $s_2 s_3 s_2 s_1 s_2 s_3$ and $s_3 s_2 s_3 s_1 s_2 s_3$
for $w_0=(4,3,2,1)\in S_4$.
These wiring diagrams (resp., reduced expressions) are related via a braid move. 
Middle: the plabic graphs $G(D_1)$ and $G(D_2)$. 
Bottom: the quivers $Q(G(D_1))$ and $Q(G(D_2))$, with isolated frozen vertices removed.
}
\label{fig:wiring-plabic0}
\end{figure}
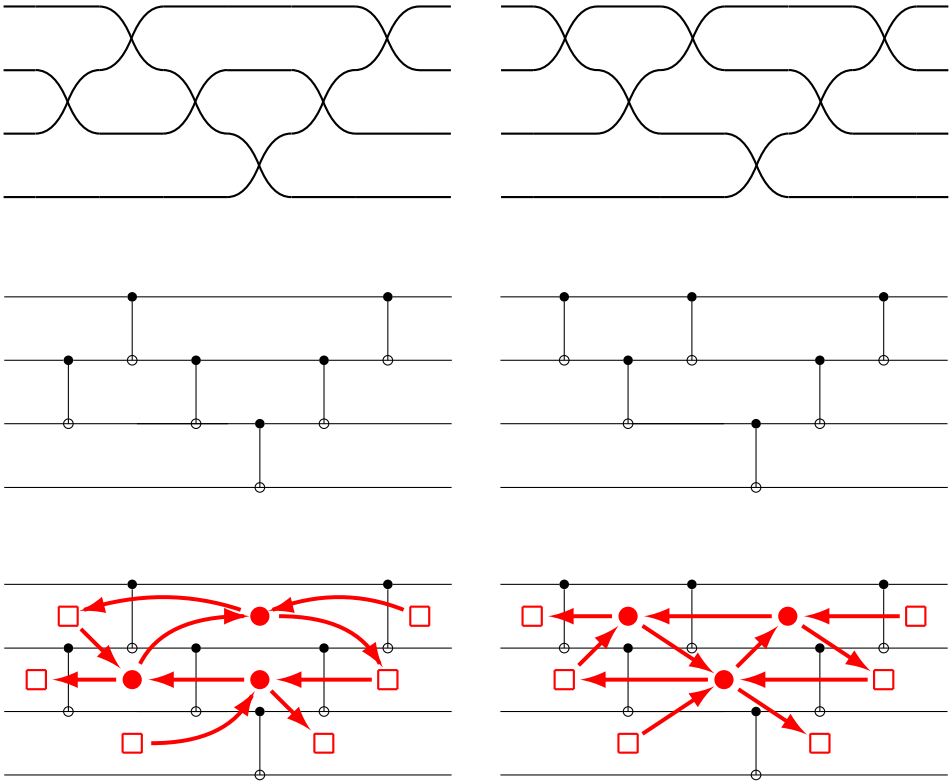

This construction applies to a more general version of wiring diagrams.
Let $s_i$ denote the simple transposition in the symmetric
	group $\mathcal{S}_n$ that exchanges $i$ and $i+1$.
Given a sequence $\mathbf{w}=s_{i_1} s_{i_2} \dots s_{i_m}$ of simple
transpositions, we associate to it a diagram $D(\mathbf{w})$
by concatenating 
$m$ graphs; here the graph associated to $s_{j}$ consists
of $n$ wires, of which $n-2$ are horizontal, while 
the $j$th and $({j}+1)$st wires cross over each other.
See \cref{fig:wiring-plabic0}. 
\end{example}

\begin{exercise}
Let $D$ be a wiring diagram.
Show that the trips starting~at the left side of~$G(D)$ follow the pattern determined by the strands
of~$D$, while the trips starting at the right side of~$G(D)$ proceed horizontally to the~left. 
\end{exercise}

%\pagebreak[3]

\begin{exercise}
\label{exercise:Q(G(D))=Q(D)}
Show that after removing isolated frozen vertices at the top and bottom,
the quiver $Q(G(D))$ 
associated to the plabic graph of a wiring diagram $D$ 
coincides with the quiver $Q(D)$ associated to~$D$, 
as in Definition~\ref{def:quiverwd}, up to a global reversal of arrows.
	\end{exercise}
	
\begin{remark}
If we changed our convention in \cref{def:wiringplabic}, swapping the 
	colors of the black and white vertices,
 we'd recover precisely the quiver $Q(D)$ associated to
the wiring diagram.  
	However, we prefer the convention used in 
\cref{def:wiringplabic} because it will
	 lead to a transparent algorithm for recovering the chamber minors, 
	 as shown in 
\cref{fig:wiring-plabic}.
	  And as noted in 
	 Remark~\ref{rem:opposite}, the cluster algebra associated to 
	 a given quiver is the same as the cluster algebra associated to the 
	 opposite quiver.
\end{remark}

\begin{remark}
\label{rem:braid-move=plabic-move}
If two wiring diagrams $D$ and $D'$ are related by a braid
move, then the corresponding plabic graphs $G(D)$ and $G(D')$
are related by a square move plus some (M3) moves, see 
\cref{fig:braidsquare0}.
	\end{remark}

\newsavebox{\ssone}
\setlength{\unitlength}{2pt} 
\savebox{\ssone}(10,20)[bl]{
\thicklines 
\qbezier(5,5)(7,10)(10,10)
\qbezier(5,5)(3,0)(0,0)
\qbezier(5,5)(3,10)(0,10)
\qbezier(5,5)(7,0)(10,0)
\put(0,20){\line(1,0){10}} 
}
\newsavebox{\sstwo}
\setlength{\unitlength}{2pt} 
\savebox{\sstwo}(10,20)[bl]{
\thicklines 
\qbezier(5,15)(7,20)(10,20)
\qbezier(5,15)(3,10)(0,10)
\qbezier(5,15)(3,20)(0,20)
\qbezier(5,15)(7,10)(10,10)
\put(0,0){\line(1,0){10}} 
}

\begin{figure}[ht]
\setlength{\unitlength}{2pt} 
	\begin{center}
\vspace{-15pt}
        \hspace{-0.1cm}
		\begin{picture}(30,20)(-7.5,-10) 
%\begin{picture}(36,17)(0,0) 
\put(0,0){\makebox(0,0){\usebox{\ssone}}} 
\put(10,0){\makebox(0,0){\usebox{\sstwo}}} 
\put(20,0){\makebox(0,0){\usebox{\ssone}}} 
\end{picture}
		\hspace{7.8cm}
\setlength{\unitlength}{2pt} 
\begin{picture}(30,20)(-2.5,-10) 
%\begin{picture}(60,40)(0,-18) 
\put(0,0){\makebox(0,0){\usebox{\sstwo}}} 
\put(10,0){\makebox(0,0){\usebox{\ssone}}} 
\put(20,0){\makebox(0,0){\usebox{\sstwo}}} 
\end{picture}
\end{center}
\vspace{.2cm}
\begin{center}
\setlength{\unitlength}{1pt}
\begin{picture}(60,40)(0,-10)
\thicklines
\put(0,40){\line(1,0){28.5}}
\put(60,40){\line(-1,0){28.5}}
\put(0,20){\line(1,0){8.5}}
\put(11.5,20){\line(1,0){37}}
\put(60,20){\line(-1,0){8.5}}
\put(0,0){\line(1,0){60}}
\put(10,1.5){\line(0,1){17}}
\put(30,21.5){\line(0,1){17}}
\put(50,1.5){\line(0,1){17}}
\put(10,0){\circle{3}}
\put(30,20){\circle{3}}
\put(50,0){\circle{3}}
\put(10,20){\circle*{3}}
\put(30,40){\circle*{3}}
\put(50,20){\circle*{3}}

\put(10,0){\white{\circle*{2.4}}}
\put(30,20){\white{\circle*{2.4}}}
\put(50,0){\white{\circle*{2.4}}}
\end{picture}
\setlength{\unitlength}{1pt}
\begin{picture}(20,10)(0,-10)
\put(5,23){\makebox(10,0){\large{$\overset{\scriptscriptstyle(\mathrm{M}3)^2}{\longleftrightarrow}$}}}
\end{picture}
\setlength{\unitlength}{1pt}
\begin{picture}(60,40)(0,-10)
\thicklines
\put(0,40){\line(1,0){28.5}}
\put(60,40){\line(-1,0){28.5}}
\put(0,20){\line(1,0){18.5}}
\put(60,20){\line(-1,0){18.5}}
\put(0,0){\line(1,0){60}}
\put(30,0){\line(0,1){8.5}}
\put(30,30){\line(0,1){8.5}}
\put(30,30){\line(1,-1){8.9}}
\put(30,30){\line(-1,-1){8.9}}
\put(30,10){\line(1,1){8.9}}
\put(30,10){\line(-1,1){8.9}}
\put(30,0){\circle{3}}
\put(30,30){\circle{3}}
\put(30,10){\circle{3}}
\put(20,20){\circle*{3}}
\put(40,20){\circle*{3}}
\put(30,40){\circle*{3}}

\put(30,0){\white{\circle*{2.4}}}
\put(30,30){\white{\circle*{2.4}}}
\put(30,10){\white{\circle*{2.4}}}
\end{picture}
\setlength{\unitlength}{1pt}
\begin{picture}(20,10)(0,-10)
\put(5,23){\makebox(10,0){\large{$\overset{\scriptscriptstyle\mathrm{M}1}{\longleftrightarrow}$}}}
\end{picture}
\setlength{\unitlength}{1pt}
\begin{picture}(60,40)(0,-10)
\thicklines
\put(0,40){\line(1,0){28.5}}
\put(60,40){\line(-1,0){28.5}}
\put(0,20){\line(1,0){18.5}}
\put(60,20){\line(-1,0){18.5}}
\put(0,0){\line(1,0){60}}
\put(30,0){\line(0,1){8.5}}
\put(30,31.5){\line(0,1){7}}
\put(20,20){\line(1,-1){8.9}}
\put(20,20){\line(1,1){8.9}}
\put(40,20){\line(-1,-1){8.9}}
\put(40,20){\line(-1,1){8.9}}
\put(30,0){\circle{3}}
\put(30,30){\circle*{3}}
\put(30,10){\circle*{3}}
\put(20,20){\circle{3}}
\put(40,20){\circle{3}}
\put(30,40){\circle*{3}}

\put(30,0){\white{\circle*{2.4}}}
\put(20,20){\white{\circle*{2.4}}}
\put(40,20){\white{\circle*{2.4}}}

\end{picture}
\setlength{\unitlength}{1pt}
\begin{picture}(20,10)(0,-10)
\put(5,23){\makebox(10,0){\large{$\overset{\scriptscriptstyle(\mathrm{M}3)^2}{\longleftrightarrow}$}}}
\end{picture}
\setlength{\unitlength}{1pt}
\begin{picture}(60,45)(0,-10)
\thicklines
\put(0,40){\line(1,0){8.5}}
\put(11.5,40){\line(1,0){37}}
\put(60,40){\line(-1,0){8.5}}
\put(0,20){\line(1,0){28.5}}
\put(60,20){\line(-1,0){28.5}}
\put(0,0){\line(1,0){60}}
\put(10,20){\line(0,1){18.5}}
\put(30,0){\line(0,1){18.5}}
\put(50,20){\line(0,1){18.5}}
\put(10,20){\circle{3}}
\put(30,0){\circle{3}}
\put(50,20){\circle{3}}
\put(10,40){\circle*{3}}
\put(30,20){\circle*{3}}
\put(50,40){\circle*{3}}

\put(10,20){\white{\circle*{2.4}}}
\put(30,0){\white{\circle*{2.4}}}
\put(50,20){\white{\circle*{2.4}}}

\end{picture}
\vspace{-18pt}
\end{center}
	\caption{
	A braid move on wiring diagrams, and a 
	corresponding sequence of moves on plabic graphs.
The first two (resp., the last two) plabic graphs are related by two (M3) moves;
the two plabic graphs in the middle are related by an (M1) move.}
\label{fig:braidsquare0}
\end{figure}
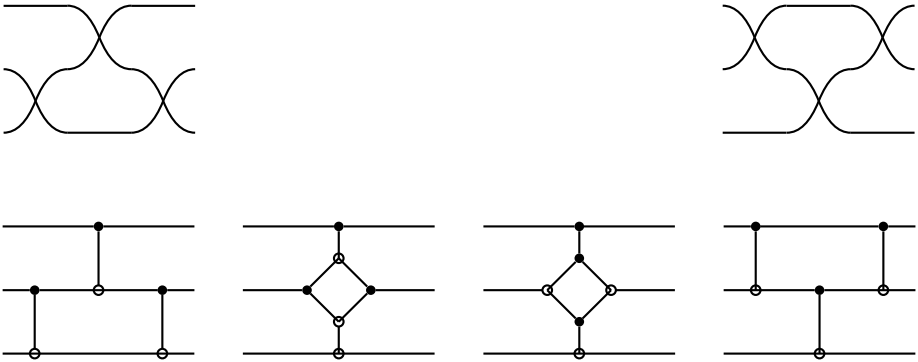

\vspace{-5pt}

\begin{exercise}
\label{exercise:braid-equivalence}
Consider sequences (or ``words'') $\mathbf{w} = s_{i_1} \cdots s_{i_m}$
of simple transpositions in a symmetric group, as in 
\cref{def:wiringplabic}.   We say that 
two such words are \emph{braid-equivalent} if they can be related to each other 
using the braid relations $s_{i} s_{i+1} s_i = s_{i+1} s_i s_{i+1}$
and $s_i s_j = s_j s_i$ for $|i-j| \geq 2$.  
A~word $\mathbf{w}$ is called a \emph{reduced expression} if no word 
in its braid equivalence class has two consecutive equal entries: $\cdots s_i s_i \cdots$.
Show that if 
$\mathbf{w}$ fails to be a reduced expression, then 
$G(D(\mathbf{w}))$ fails to be a reduced plabic graph.
\end{exercise}

\begin{remark}
Conversely, if $\mathbf{w}$ is reduced, then 
$G(D(\mathbf{w}))$ is reduced; this can be proved using 
\cref{thm:reduced} or~\cref{thm:resonance}. % below.
In this sense, reduced plabic graphs can be interpreted as generalizations
of reduced expressions in symmetric groups.
\end{remark}

\begin{example}[{\textbf{Double wiring diagrams}}]
	\label{ex:DWD}
Let $D$ be a double wiring diagram, as in Section~\ref{sec:matrices}. 
The plabic graph $G(D)$ associated to~$D$ is defined by adjusting 
the construction of Example~\ref{def:wiringplabic} in the following way: 
as before, we replace each crossing in the double wiring diagram 
by a pair of trivalent vertices connected vertically, 
and color one of these vertices white and the other black.
If the crossing is thin, the top vertex gets colored white and the bottom one black; if the crossing is thick, the colors of the two vertices are reversed.
%	We decide whether to color the top vertex white or black \linebreak[3]
%(thus, whether to color the bottom vertex black or white)
%depending on the color (or ``thickness'') of the strands involved in the crossing. 
See \cref{fig:double-wiring-plabic}.

\begin{figure}[ht]
\setlength{\unitlength}{1.2pt}
\begin{center}
\begin{picture}(180,45)(6,0)
\thicklines
\dark{

  \put(0,0){\line(1,0){40}}
  \put(60,0){\line(1,0){100}}
  \put(180,0){\line(1,0){10}}
  \put(0,20){\line(1,0){40}}
  \put(60,20){\line(1,0){10}}
  \put(90,20){\line(1,0){70}}
  \put(180,20){\line(1,0){10}}
  \put(0,40){\line(1,0){70}}
  \put(90,40){\line(1,0){100}}

  \put(40,0){\line(1,1){20}}
  \put(70,20){\line(1,1){20}}
  \put(160,0){\line(1,1){20}}

  \put(40,20){\line(1,-1){20}}
  \put(70,40){\line(1,-1){20}}
  \put(160,20){\line(1,-1){20}}
  \put(193,-6){$\mathbf{3}$}
  \put(193,14){$\mathbf{2}$}
  \put(193,34){$\mathbf{1}$}

  \put(-6,-6){$\mathbf{1}$}
  \put(-6,14){$\mathbf{2}$}
  \put(-6,34){$\mathbf{3}$}
}

\light{
% \thinlines
\thicklines

  \put(0,2){\line(1,0){100}}
  \put(120,2){\line(1,0){70}}
  \put(0,22){\line(1,0){10}}
  \put(30,22){\line(1,0){70}}
  \put(120,22){\line(1,0){10}}
  \put(150,22){\line(1,0){40}}
  \put(0,42){\line(1,0){10}}
  \put(30,42){\line(1,0){100}}
  \put(150,42){\line(1,0){40}}

  \put(10,22){\line(1,1){20}}
  \put(100,2){\line(1,1){20}}
  \put(130,22){\line(1,1){20}}

  \put(10,42){\line(1,-1){20}}
  \put(100,22){\line(1,-1){20}}
  \put(130,42){\line(1,-1){20}}

  \put(193,2){${1}$}
  \put(193,22){${2}$}
  \put(193,42){${3}$}

  \put(-6,2){${3}$}
  \put(-6,22){${2}$}
  \put(-6,42){${1}$}
}
\end{picture}
	\end{center}
			\vspace{0.1cm}
		\begin{center}
		\setlength{\unitlength}{1.5pt}
			\begin{picture}(140,45)(0,0)
				% \thinlines
                    \thicklines
				\put(0,20){\line(1,0){58.5}}
				\put(61.5,20){\line(1,0){17}}
				\put(81.5,20){\line(1,0){58.5}}
				\put(0,0){\line(1,0){38.5}}
				\put(140,0){\line(-1,0){18.5}}
				\put(118.5,0){\line(-1,0){77}}
				\put(0,40){\line(1,0){18.5}}
				\put(21.5,40){\line(1,0){37}}
				\put(61.5,40){\line(1,0){37}}
				\put(101.5,40){\line(1,0){38.5}}
				
				\put(20,21.5){\line(0,1){17}}
				\put(80,1.5){\line(0,1){17}}
				\put(60,21.5){\line(0,1){17}}
				\put(40,1.5){\line(0,1){17}}
				\put(100,21.5){\line(0,1){17}}
				\put(120,1.5){\line(0,1){17}}
				
				\put(20,20){\circle*{3}}
				\put(40,20){\circle*{3}}
				\put(60,20){\circle{3}}
				\put(80,20){\circle{3}}
				\put(100,20){\circle*{3}}
				\put(120,20){\circle*{3}}
				
				\put(20,40){\circle{3}}
				\put(60,40){\circle*{3}}
				\put(100,40){\circle{3}}
				
				\put(40,0){\circle{3}}
				\put(80,0){\circle*{3}}
				\put(120,0){\circle{3}}
				\thicklines
				\linethickness{1.5pt}
			\end{picture}
		\end{center}
			\vspace{.1cm}
		\begin{center}
			\setlength{\unitlength}{1.5pt}
			\begin{picture}(140,65)(0,0)
				% \thinlines
                    \thicklines
				\put(0,40){\line(1,0){58.5}}
				\put(61.5,40){\line(1,0){17}}
				\put(81.5,40){\line(1,0){58.5}}
				\put(0,20){\line(1,0){38.5}}
				\put(140,20){\line(-1,0){18.5}}
				\put(118.5,20){\line(-1,0){77}}
				\put(0,60){\line(1,0){18.5}}
				\put(21.5,60){\line(1,0){37}}
				\put(61.5,60){\line(1,0){37}}
				\put(101.5,60){\line(1,0){38.5}}
				
				\put(20,41.5){\line(0,1){17}}
				\put(80,21.5){\line(0,1){17}}
				\put(60,41.5){\line(0,1){17}}
				\put(40,21.5){\line(0,1){17}}
				\put(100,41.5){\line(0,1){17}}
				\put(120,21.5){\line(0,1){17}}
				
				\put(20,40){\circle*{3}}
				\put(40,40){\circle*{3}}
				\put(60,40){\circle{3}}
				\put(80,40){\circle{3}}
				\put(100,40){\circle*{3}}
				\put(120,40){\circle*{3}}
				
				\put(20,60){\circle{3}}
				\put(60,60){\circle*{3}}
				\put(100,60){\circle{3}}
				
				\put(40,20){\circle{3}}
				\put(80,20){\circle*{3}}
				\put(120,20){\circle{3}}
				
				\thicklines
				\put(77,7){\red{\line(0,1){6}}}
				\put(77,7){\red{\line(1,0){6}}}
				\put(83,13){\red{\line(0,-1){6}}}
				\put(83,13){\red{\line(-1,0){6}}}
				
				\put(17,27){\red{\line(0,1){6}}}
				\put(17,27){\red{\line(1,0){6}}}
				\put(23,33){\red{\line(0,-1){6}}}
				\put(23,33){\red{\line(-1,0){6}}}
				
				\put(127,27){\red{\line(0,1){6}}}
				\put(127,27){\red{\line(1,0){6}}}
				\put(133,33){\red{\line(0,-1){6}}}
				\put(133,33){\red{\line(-1,0){6}}}
				
				\put(7,47){\red{\line(0,1){6}}}
				\put(7,47){\red{\line(1,0){6}}}
				\put(13,53){\red{\line(0,-1){6}}}
				\put(13,53){\red{\line(-1,0){6}}}
				
				\put(117,47){\red{\line(0,1){6}}}
				\put(117,47){\red{\line(1,0){6}}}
				\put(123,53){\red{\line(0,-1){6}}}
				\put(123,53){\red{\line(-1,0){6}}}
				
				\put(40,50){\red{{\circle*{3}}}}
				\put(60,30){\red{{\circle*{3}}}}
				\put(100,30){\red{{\circle*{3}}}}
				\put(80,50){\red{{\circle*{3}}}}
				
				\put(57,67){\red{\line(0,1){6}}}
				\put(57,67){\red{\line(1,0){6}}}
				\put(63,73){\red{\line(0,-1){6}}}
				\put(63,73){\red{\line(-1,0){6}}}
				
				\linethickness{1pt}
\put(65,30){\red{\vector(1,0){30}}}
				\put(55,30){\red{\vector(-1,0){30}}}
				\put(75,50){\red{\vector(-1,0){30}}}
				\put(85,50){\red{\vector(1,0){30}}}
				\put(15,50){\red{\vector(1,0){20}}}
				\put(125,30){\red{\vector(-1,0){20}}}
				
				\put(43.5,53.5){\red{\vector(1,1){12.5}}}
				\put(64.5,65.5){\red{\vector(1,-1){12.5}}}
				
				\put(96.5,26.5){\red{\vector(-1,-1){12.5}}}
				\put(75.5,14.5){\red{\vector(-1,1){12.5}}}
				
				\put(43.5,46.5){\red{\vector(1,-1){13}}}
				\put(96.5,33.5){\red{\vector(-1,1){13}}}

				%\put(95,30){\red{\vector(-1,0){30}}}
				%\put(25,30){\red{\vector(1,0){30}}}
				%\put(45,50){\red{\vector(1,0){30}}}
				%\put(115,50){\red{\vector(-1,0){30}}}
				%\put(35,50){\red{\vector(-1,0){20}}}
				%\put(105,30){\red{\vector(1,0){20}}}
				
				%\put(55.5,65.5){\red{\vector(-1,-1){12.5}}}
				%\put(76.5,53.5){\red{\vector(-1,1){12.5}}}
				
				%\put(96.5,26.5){\red{\vector(-1,-1){12.5}}}
				%\put(75.5,14.5){\red{\vector(-1,1){12.5}}}
				
				%\put(56.5,33.5){\red{\vector(-1,1){13}}}
				%\put(84,46){\red{\vector(1,-1){13}}}
			\end{picture}
			\vspace{-.5cm}
\end{center}
\caption{%\LW{The red squares and circles are too big, compare with Figure 7.12} 
A double wiring diagram $D$, the 
plabic graph $G(D)$, and the corresponding quiver.  If one
removes the bottom frozen vertex, one recovers the quiver from
Figure~\ref{fig:chamber-quiver2} (up to a global reversal of arrows).}
\label{fig:double-wiring-plabic}
\end{figure}
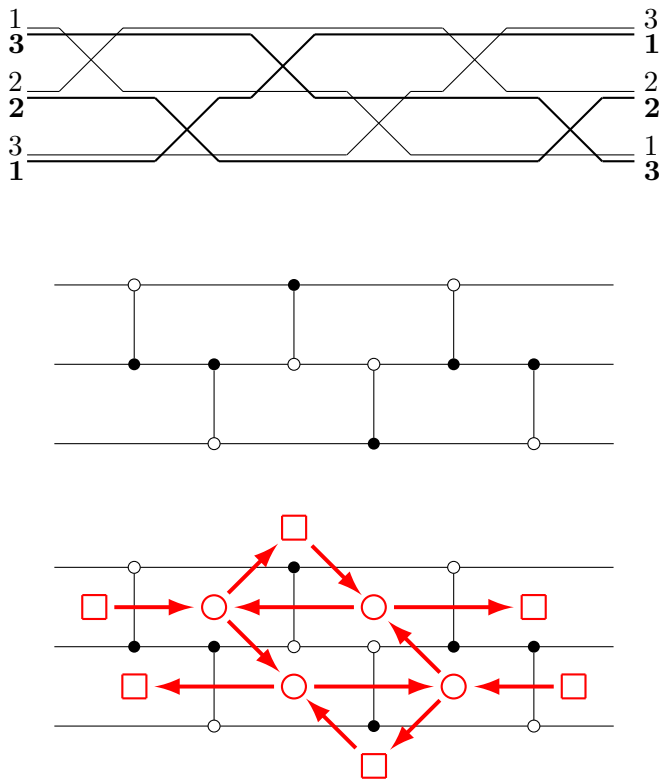

As in \cref{def:wiringplabic}, the above construction works for 
a more general version of double wiring diagrams.
Given two sequences $\mathbf{w} = s_{i_1} s_{i_2} \dots s_{i_m}$
and $\overline{\mathbf{w}} = \overline{s}_{j_1} \overline{s}_{j_2}
\dots \overline{s}_{j_{\ell}}$,  choose an arbitrary shuffle
of $\mathbf{w}$ and $\overline{\mathbf{w}}$.  Then we can 
	associate a (generalized) double wiring diagram 
to this shuffle, 
where thick crossings are associated to factors in $\mathbf{w}$,
and thin crossings are associated to factors in 
$\overline{\mathbf{w}}$.  So, e.g., the double wiring diagram in 
\cref{fig:double-wiring-plabic} is associated to the shuffle
$\overline{s}_2 s_1 s_2 \overline{s}_1 \overline{s}_2 s_1$.
\end{example}

\begin{exercise}
Describe the trips in a plabic graph associated to a double wiring diagram. 
Extend the statements of Exercise~\ref{exercise:Q(G(D))=Q(D)}
and Remark~\ref{rem:braid-move=plabic-move}
to the case of double wiring diagrams. 
\end{exercise}

In addition to triangulations and %(ordinary or double) 
wiring diagrams, 
plabic graphs can also be used to describe 
Fock-Goncharov cluster structures~\cite{fock-goncharov-ihes}:

\begin{exercise}
Construct a plabic graph whose associated quiver is the quiver shown in Figure~\ref{fig:Q_3}. 
How does this construction generalize to a quiver $Q_3(T)$ associated to an arbitrary triangulation~$T$
of a convex polygon, cf.\ Exercise~\ref{exercise:FG-sl3-d4}?
\end{exercise}

\clearpage

\section{Trivalent plabic graphs}
\label{sec:tri}

In \cref{sec:plabic}, we introduced plabic graphs and described 
local moves that generate an equivalence relation on them.
In this section, we focus on \emph{trivalent plabic graphs}, i.e., those plabic graphs
whose interior vertices are all trivalent. 
This will require working with an alternative set of moves that preserve the property of being trivalent.
As \cref{sec:tri} is not strictly necessary for the sections that follow, 
it can be skipped if desired. 

\begin{remark}
Trivalent plabic graphs arise naturally in the studies of 
\begin{itemize}[leftmargin=.2in]
\item 
soliton solutions to the KP equation \cite{kodwil}, 
\item
sections of fine zonotopal tilings of $3$-dimensional cyclic zonotopes~\cite{Galashin}, 
\item
combinatorics of planar divides and associated links~\cite{fpst}, 
and 
\item
subdivisions induced by projection of a hypersimplex onto a polygon~\cite{PostnikovICM}.
%$\pi$-induced subdivisions for a projection $\pi$ of a hypersimplex $\Delta_{k,n}$ onto an $n$-gon \cite{PostnikovICM}.
\end{itemize}
\end{remark}

%We begin with a simple observation: 

\begin{lemma}
\label{lem:bitri}
Any leafless plabic graph without lollipops
%Any plabic graph with no interior leaves 
%(i.e., no degree $1$ interior vertices) 
	can be transformed 
by a sequence of moves of type {\rm (M2)} and/or {\rm (M3)} into a plabic graph all of whose interior vertices are trivalent.
\end{lemma}

\begin{proof}
We can get rid of bivalent vertices using the moves~(M2). 
If there are any vertices of degree at least $4$, 
split those vertices using~(M3) until all internal vertices are trivalent. 
\end{proof}

The alternative set of moves for trivalent plabic graphs consists of the square move (M1)   
together with the \emph{flip move} (M4) defined below.
%As we will see in \cref{thm:newmoves1}, plabic graphs together 
%with moves (M1), (M2), (M3)
%are essentially equivalent to trivalent plabic graphs together with moves (M1) and (M4).

\begin{definition}
The \emph{flip move} (sometimes also called the \emph{Whitehead move})
for trivalent plabic graphs is defined as follows:
\begin{itemize}[leftmargin=.4in]
\item[(M4)] Replace a fragment containing two trivalent vertices of the same color connected by an edge 
by another such fragment, see \cref{fig:flipmove0}.
\end{itemize}
\end{definition}

\begin{figure}[ht]
\vspace{-15pt}
%\begin{center}
%\includegraphics[height=.8in]{FlipMove0.ps}
%\end{center}
\begin{center}
\begin{tikzpicture}[scale=0.80]
\node[circle, fill=black, draw=black, inner sep=0pt, minimum size=0pt] (1) at (8.15,0.75) {};
\node[circle, fill=black, draw=black, inner sep=0pt, minimum size=0pt] (2) at (9.85,0.75) {};
\node[circle, thick, fill=white, draw=black, inner sep=0pt, minimum size=4pt] (3) at (9,0.5) {};
\node[circle, fill=black, draw=black, inner sep=0pt, minimum size=0pt] (4) at (9,0) {};
\node[circle, thick, fill=white, draw=black, inner sep=0pt, minimum size=4pt] (5) at (9,-0.5) {};
\node[circle, fill=black, draw=black, inner sep=0pt, minimum size=0pt] (6) at (8.15,-0.75) {};
\node[circle, fill=black, draw=black, inner sep=0pt, minimum size=0pt] (7) at (9.85,-0.75) {};
\draw[thick] (1) -- (3); \draw[thick] (2) -- (3); \draw[thick] (3) -- (5);
\draw[thick] (6) -- (5); \draw[thick] (7) -- (5);
% \draw[>={Stealth[length=8pt]}, <->] (10.5,0) -- (12,0);
\end{tikzpicture}
\begin{picture}(30,20)(0,7)
\put(15,15){\makebox(0,0){\Large{$\longleftrightarrow$}}}
\end{picture}
\begin{tikzpicture}[scale=0.80]
\node[circle, fill=black, draw=black, inner sep=0pt, minimum size=0pt] (21) at (12.6,0.7) {};
\node[circle, fill=black, draw=black, inner sep=0pt, minimum size=0pt] (22) at (14.4,0.7) {};
\node[circle, thick, fill=white, draw=black, inner sep=0pt, minimum size=4pt] (23) at (12.9,0) {};
\node[circle, fill=black, draw=black, inner sep=0pt, minimum size=0pt] (24) at (13.5,0) {};
\node[circle, thick, fill=white, draw=black, inner sep=0pt, minimum size=4pt] (25) at (14.1,0) {};
\node[circle, fill=black, draw=black, inner sep=0pt, minimum size=0pt] (26) at (12.6,-0.7) {};
\node[circle, fill=black, draw=black, inner sep=0pt, minimum size=0pt] (27) at (14.4,-0.7) {};
\draw[thick] (21) -- (23); \draw[thick] (22) -- (25); \draw[thick] (23) -- (25);
\draw[thick] (26) -- (23); \draw[thick] (27) -- (25);
\end{tikzpicture}
\vspace{-10pt}
\end{center}
\caption{The flip move (or Whitehead move) for trivalent graphs.  The four vertices shown should 
	either all be white, or all be black.}
\label{fig:flipmove0}
\end{figure}
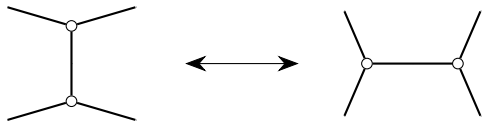
\begin{remark}
\label{rem:flip}
A flip move (M4) can be expressed as a composition of two moves of type~(M3).
\end{remark}

The main result of this section is the following.

\begin{theorem}
\label{thm:newmoves1}
Two trivalent plabic graphs  are related 
via a sequence of local moves of types {\rm (M1)}, {\rm (M2)}, and {\rm (M3)} if and only if 
they are related by a sequence of moves of types {\rm (M1)} and {\rm (M4)}.
\end{theorem}

\pagebreak[3]

The rest of this section is devoted to the proof of \cref{thm:newmoves1}. 

\pagebreak[3]

\begin{lemma} \label{lem:noM1}
If two %trivalent 
	plabic graphs  are related by a sequence of moves
of type {\rm(M2)} and/or {\rm (M3)}, then they are related by moves of type~{\rm(M3)}.
\end{lemma}

\begin{proof}
We first note that in many cases, move (M2) 
can be thought of as an instance of move~(M3):
instead of using (M2) to add or remove 
a bivalent vertex that is adjacent via an edge~$e$ to
a vertex of the same color, we can  use (M3) to (un)contract~$e$, to the same effect.

The only (M2) moves that are genuinely different from (M3) moves are the (M2) moves that 
add or remove a white (resp., black) vertex in the middle of a black-black 
or black-boundary
(resp., white-white or white-boundary) edge.
We  call them \emph{creative} or \emph{destructive} (M2) moves, see 
%	the left and middle of 
\cref{fig:trees}.
However, if one performs a creative (M2) move,  e.g. adding a white
vertex in the middle of a black-black or black-boundary edge, 
there is no further move that the new white vertex can participate in, 
except for a destructive (M2) move that removes it.
	Thus, such creative and destructive (M2) moves are unnecessary.
\end{proof}
%We need to show that it is never necessary to use a creative or destructive (M2) move
%to connect $G$ and $G'$ as above.

\begin{figure}[ht]
\vspace{-15pt}
\begin{center}
\begin{tikzpicture}[scale=1.2]
\node[circle, fill=black, draw=black, inner sep=0pt, minimum size=0pt] (1) at (-0.5,0.5) {};
\node[circle, fill=black, draw=black, inner sep=0pt, minimum size=0pt] (2) at (-0.5,-0.5) {};
\node[circle, fill=black, draw=black, inner sep=0pt, minimum size=4pt] (3) at (0,0) {};
%\node[circle, thick, fill=white, draw=black, inner sep=0pt, minimum size=4pt] (4) at (0.5,0) {};
\node[circle, fill=black, draw=black, inner sep=0pt, minimum size=4pt] (5) at (1,0) {};
\node[circle, fill=black, draw=black, inner sep=0pt, minimum size=0pt] (6) at (1.6,0.5) {};
\node[circle, fill=black, draw=black, inner sep=0pt, minimum size=0pt] (7) at (1.6,0) {};
\node[circle, fill=black, draw=black, inner sep=0pt, minimum size=0pt] (8) at (1.25,-0.55) {};
\draw[thick] (1) -- (3); \draw[thick] (2) -- (3); \draw[thick] (3) -- (5);
	%\draw[thick] (4) -- (5); 
\draw[thick] (6) -- (5); \draw[thick] (7) -- (5); \draw[thick] (8) -- (5); 
%\draw[thick] (4) -- (9); 
	%\draw[thick] (9) -- (10); \draw[thick] (9) -- (11); \draw[thick] (10) -- (12); 
%\draw[thick] (10) -- (13); 
%	\draw[thick] (11) -- (14); 
	%\draw[thick] (11) -- (15); 
\end{tikzpicture}
% \hspace{1cm}
\begin{picture}(10,20)(0,7)
\put(5,16.5){\makebox(0,0){\Large{$\overset{\scriptscriptstyle \mathrm{M}2}{\rightarrow}$}}}
\end{picture}
\begin{tikzpicture}[scale=1.2]
\node[circle, fill=black, draw=black, inner sep=0pt, minimum size=0pt] (1) at (-0.5,0.5) {};
\node[circle, fill=black, draw=black, inner sep=0pt, minimum size=0pt] (2) at (-0.5,-0.5) {};
\node[circle, fill=black, draw=black, inner sep=0pt, minimum size=4pt] (3) at (0,0) {};
\node[circle, thick, fill=white, draw=black, inner sep=0pt, minimum size=4pt] (4) at (0.5,0) {};
\node[circle, fill=black, draw=black, inner sep=0pt, minimum size=4pt] (5) at (1,0) {};
\node[circle, fill=black, draw=black, inner sep=0pt, minimum size=0pt] (6) at (1.6,0.5) {};
\node[circle, fill=black, draw=black, inner sep=0pt, minimum size=0pt] (7) at (1.6,0) {};
\node[circle, fill=black, draw=black, inner sep=0pt, minimum size=0pt] (8) at (1.25,-0.55) {};
%\node[circle, thick, fill=white, draw=black, inner sep=0pt, minimum size=4pt] (9) at (0.55,0.4) {};
%\node[circle, thick, fill=white, draw=black, inner sep=0pt, minimum size=4pt] (10) at (0.3,0.7) {};
%\node[circle, thick, fill=white, draw=black, inner sep=0pt, minimum size=4pt] (11) at (0.9,0.7) {};
%\node[circle, thick, fill=white, draw=black, inner sep=0pt, minimum size=4pt] (12) at (0.1,1) {};
%\node[circle, thick, fill=white, draw=black, inner sep=0pt, minimum size=4pt] (13) at (0.4,1) {};
%\node[circle, thick, fill=white, draw=black, inner sep=0pt, minimum size=4pt] (14) at (0.9,1) {};
%\node[circle, thick, fill=white, draw=black, inner sep=0pt, minimum size=4pt] (15) at (1.25,0.9) {};
\draw[thick] (1) -- (3); \draw[thick] (2) -- (3); \draw[thick] (3) -- (4); \draw[thick] (4) -- (5); 
\draw[thick] (6) -- (5); \draw[thick] (7) -- (5); \draw[thick] (8) -- (5); 
%\draw[thick] (4) -- (9); 
	%\draw[thick] (9) -- (10); \draw[thick] (9) -- (11); \draw[thick] (10) -- (12); 
%\draw[thick] (10) -- (13); 
%	\draw[thick] (11) -- (14); 
	%\draw[thick] (11) -- (15); 
\end{tikzpicture}
\vspace{-20pt}
\end{center}
\caption{A creative (M2) move.} % and then an (M3) move to grow a white leaf.}
\label{fig:trees}
\end{figure}

\pagebreak[3]

\begin{lemma} 
\label{lem:flipwhite}
Let $G$ and $G'$ be two trivalent plabic graphs such that
\begin{itemize}[leftmargin=.2in]
\item
each of the graphs $G$ and $G'$ is connected; 
\item
each of the graphs $G$ and $G'$ has $f$ interior faces, $b$ boundary vertices, 
		and $b$ boundary faces (the number of boundary vertices
		equals the number of boundary faces since the graphs are connected);
%and $f$ interior faces;  
\item
in each of the graphs $G$ and $G'$, all interior vertices have the same color, 
and this color is the same in both graphs.
\end{itemize}
Then $G$ and $G'$  can be connected by a sequence of flip moves~{\rm (M4)}.
\end{lemma}

\begin{proof}
	The \emph{dual graph} $G_\textup{dual}$ of a trivalent connected plabic graph~$G$ is obtained as follows. 
Place a vertex of $G_\textup{dual}$ in the interior of each face of~$G$. 
For each edge~$e$ of~$G$, introduce a (transversal) edge of $G_\textup{dual}$ 
	connecting the vertices of $G_\textup{dual}$ located in the faces of~$G$ on both sides of~$e$, see \cref{fig:dual}. 
(This new edge may be a loop.) 
Under the conditions of the lemma, the dual graph $G_\textup{dual}$ is a generalized triangulation~$T$ of a (dual) $b$-gon.
	(We note that $T$ may 
	contain \emph{self-folded} triangles -- loops with 
	an interior ``pendant'' edge --  
	coming from the faces of~$G$ enclosed by a loop in~$G$, as in \cref{fig:dual}.) 
The triangulation~$T$ has $b+f$ vertices: the $b$ vertices of the dual $b$-gon 
together with the $f$ interior points (``punctures'').  

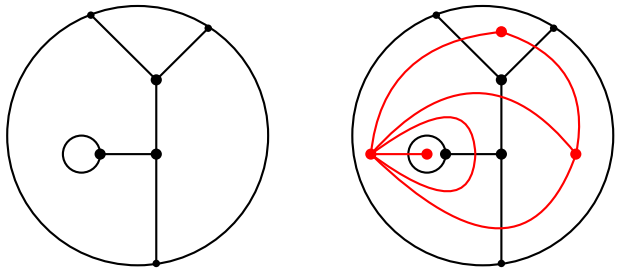
\begin{figure}[ht]
\begin{center}
\vspace{-.1in}
%graph G
\setlength{\unitlength}{1.2pt}
\begin{picture}(70,65)

\thicklines
\put(25,35){\circle{70}}

%vertices
\put(30,50){\circle*{3}}
\put(30,30){\circle*{3}}
\put(15,30){\circle*{3}}

%boundary vertices
\put(30,0.5){\circle*{2}}
\put(12.5,67.5){\circle*{2}}
\put(44,64){\circle*{2}}

%edges
\put(10,30){\circle{10}}
\put(30,50){\line(-1,1){17.5}}
\put(30,50){\line(1,1){14}}
\put(30,50){\line(0,-1){20}}
\put(30,30){\line(-1,0){15}}
\put(30,30){\line(0,-1){30}}

\end{picture}\hspace{0in}
%dual graph
\begin{picture}(70,65)

\thicklines
\put(45,35){\circle{70}}

%graph vertices
\put(50,50){\circle*{3}}
\put(50,30){\circle*{3}}
\put(35,30){\circle*{3}}

%graph edges
\put(30,30){\circle{10}}
\put(50,50){\line(-1,1){17.5}}
\put(50,50){\line(1,1){14}}
\put(50,50){\line(0,-1){20}}
\put(50,30){\line(-1,0){15}}
\put(50,30){\line(0,-1){30}}

%graph boundary edges
\put(50,0.5){\circle*{2}}
\put(32.5,67.5){\circle*{2}}
\put(64,64){\circle*{2}}

%dual vertices
\put(50,63){\red{\circle*{3}}}
\put(30,30){\red{\circle*{3}}}
\put(15,30){\red{\circle*{3}}}
\put(70,30){\red{\circle*{3}}}

%dual edges
\put(15,30){\red{\line(1,0){15}}}
\red{\qbezier(15,30)(42,10)(43,30)
\qbezier(15,30)(42,50)(43,30)
\qbezier(15,30)(44,63)(70,30)
\qbezier(15,30)(56,-10)(70,30)
\qbezier(15,30)(18,60)(50,63)
\qbezier(50,63)(75,55)(70,30)}

\end{picture}

\vspace{-.15in}
\end{center}
\caption{A trivalent plabic graph $G$ and its dual graph
$G_\textup{dual}$ (in red); the latter contains a self-folded triangle.  
Here $b=3$ and $f=1$.}
\vspace{-.05in}
\label{fig:dual}
\end{figure}

%\vspace{-.1in}

\cref{fig:fliptri} shows that a flip move in a trivalent plabic graph corresponds
to a flip in the corresponding triangulation.
The claim that $G$ and $G'$ are connected by flip moves 
can now be obtained from the well-known fact~\cite{harer-1986, hatcher} 
that any two triangulations of a $b$-gon (including triangulations involving
self-folded triangles) with  $f$ interior points are connected by flips.
\end{proof}

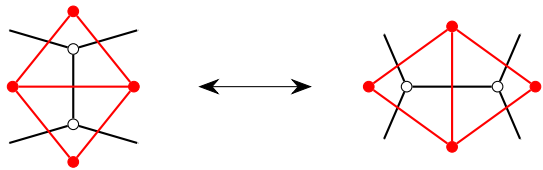
\begin{figure}[ht]
%\begin{center}
%\includegraphics[height=.8in]{FlipTriang.ps}
%\end{center}
\begin{center}
\vspace{-.2in}
\begin{tikzpicture}[scale=0.80]
% black
\node[circle, fill=black, draw=black, inner sep=0pt, minimum size=0pt] (1) at (8.15,0.75) {};
\node[circle, fill=black, draw=black, inner sep=0pt, minimum size=0pt] (2) at (9.85,0.75) {};
\node[circle, thick, fill=white, draw=black, inner sep=0pt, minimum size=4pt] (3) at (9,0.5) {};
\node[circle, fill=black, draw=black, inner sep=0pt, minimum size=0pt] (4) at (9,0) {};
\node[circle, thick, fill=white, draw=black, inner sep=0pt, minimum size=4pt] (5) at (9,-0.5) {};
\node[circle, fill=black, draw=black, inner sep=0pt, minimum size=0pt] (6) at (8.15,-0.75) {};
\node[circle, fill=black, draw=black, inner sep=0pt, minimum size=0pt] (7) at (9.85,-0.75) {};
\draw[thick] (1) -- (3); \draw[thick] (2) -- (3); \draw[thick] (3) -- (5);
\draw[thick] (6) -- (5); \draw[thick] (7) -- (5);
\node[circle, fill=red, draw=red, inner sep=0pt, minimum size=4pt] (31) at (9,1) {};
\node[circle, fill=red, draw=red, inner sep=0pt, minimum size=4pt] (32) at (8.2,0) {};
\node[circle, fill=red, draw=red, inner sep=0pt, minimum size=4pt] (33) at (9.8,0) {};
\node[circle, fill=red, draw=red, inner sep=0pt, minimum size=4pt] (34) at (9,-1) {};
\draw[red, thick] (31) -- (32); \draw[red, thick] (32) -- (33); \draw[red, thick] (31) -- (33); \draw[red, thick] (32) -- (34); \draw[red, thick] (33) -- (34);
\end{tikzpicture}
\begin{picture}(30,20)(0,10)
\put(15,20){\makebox(0,0){\Large{$\longleftrightarrow$}}}
\end{picture}
% \draw[>={Stealth[length=8pt]}, <->] (10.65,0) -- (12.15,0);
% red
\begin{tikzpicture}[scale=0.80]
\node[circle, fill=black, draw=black, inner sep=0pt, minimum size=0pt] (21) at (13.1,0.7) {};
\node[circle, fill=black, draw=black, inner sep=0pt, minimum size=0pt] (22) at (14.9,0.7) {};
\node[circle, thick, fill=white, draw=black, inner sep=0pt, minimum size=4pt] (23) at (13.4,0) {};
\node[circle, fill=black, draw=black, inner sep=0pt, minimum size=0pt] (24) at (14,0) {};
\node[circle, thick, fill=white, draw=black, inner sep=0pt, minimum size=4pt] (25) at (14.6,0) {};
\node[circle, fill=black, draw=black, inner sep=0pt, minimum size=0pt] (26) at (13.1,-0.7) {};
\node[circle, fill=black, draw=black, inner sep=0pt, minimum size=0pt] (27) at (14.9,-0.7) {};
\draw[thick] (21) -- (23); \draw[thick] (22) -- (25); \draw[thick] (23) -- (25);
\draw[thick] (26) -- (23); \draw[thick] (27) -- (25);
\node[circle, fill=red, draw=red, inner sep=0pt, minimum size=4pt] (41) at (14,0.8) {};
\node[circle, fill=red, draw=red, inner sep=0pt, minimum size=4pt] (42) at (12.9,0) {};
\node[circle, fill=red, draw=red, inner sep=0pt, minimum size=4pt] (43) at (15.1,0) {};
\node[circle, fill=red, draw=red, inner sep=0pt, minimum size=4pt] (44) at (14,-0.8) {};
\draw[red, thick] (41) -- (42); \draw[red, thick] (41) -- (44); \draw[red, thick] (41) -- (43); \draw[red, thick] (42) -- (44); \draw[red, thick] (43) -- (44);

\end{tikzpicture}
\end{center}
\vspace{-.15in}
\caption{The flip move for trivalent graphs corresponds to a flip of the 
corresponding dual triangulations.}
\label{fig:fliptri}
\end{figure}

\begin{definition}
\label{def:component}
A \emph{white component} $W$ of a plabic graph ${G}$ is obtained by taking a 
maximal (by inclusion) connected induced subgraph of~${G}$ all of whose internal vertices are white,
together with the half-edges extending from the (white) vertices of $W$ towards
black vertices outside $W$ or towards boundary vertices of~$G$.
\emph{Black components} of ${G}$ are defined in the same way, 
with the roles of black and white vertices reversed.
\end{definition}

\begin{remark}
\label{rem:component-as-plabic-graph}
Each black or white component $C$ of a plabic graph $G$ can itself be regarded  
as a (generalized) plabic graph.
To this end, enclose~$C$ by a simple closed curve~$\gamma$ passing through the endpoints of 
the half-edges on the outer boundary of~$C$. 
If the portion of~$G$ located inside~$\gamma$ is exactly~$C$, then we get a usual plabic graph.
It may however happen that $C$ contains ``holes,'' i.e., some of the half-edges on the boundary of~$C$
may be entirely contained in the interior of the disk enclosed by~$\gamma$. 
In that case, we need to draw simple closed curves through the endpoints of those half-edges,
so that $C$ becomes a generalized plabic graph inside a ``swiss-cheese'' shape 
	(a disk with some smaller disks removed), as in \cref{fig:cheese}.
The argument in the proof of \cref{lem:flipwhite} extends to this setting,
so \cref{lem:flipwhite} also holds for black/white components
	of trivalent plabic graphs. 
We~will use this generalization in the proof of \cref{prop:flipmoves} below. 
\end{remark}

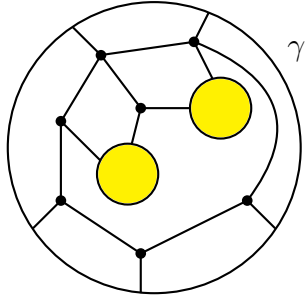
\begin{figure}[ht]
\vspace{-.1in}
\setlength{\unitlength}{0.8pt}
\begin{picture}(110,120)
\thicklines
%outer-curve
\put(55,60){\circle{110}}
\put(105,95){\large $\gamma$}

%swiss-holes
\put(80,75){\yellow{\circle*{22}}}
\put(80,75){\circle{23}}
\put(45,50){\yellow{\circle*{22}}}
\put(45,50){\circle{23}}

%vertices
\put(70,100){\circle*{4}}
\put(35,95){\circle*{4}}
\put(50,75){\circle*{4}}
\put(20,70){\circle*{4}}
\put(20,40){\circle*{4}}
\put(50,20){\circle*{4}}
\put(90,40){\circle*{4}}

%edges
\put(70,100){\line(1,2){5.5}}
\put(70,100){\line(-7,-1){35}}
\put(70,100){\line(1,-2){7}}
\put(35,95){\line(-1,1){10.5}}
\put(35,95){\line(-3,-5){15}}
\put(35,95){\line(3,-4){15}}
\put(50,75){\line(1,0){18.5}}
\put(50,75){\line(-1,-4){3.4}}
\put(20,70){\line(0,-1){30}}
\put(20,70){\line(1,-1){15}}
\put(20,40){\line(-1,-1){10.5}}
\put(20,40){\line(3,-2){30}}
\put(50,20){\line(0,-1){15}}
\put(50,20){\line(2,1){40}}
\put(90,40){\line(1,-1){10.5}}
\qbezier(90,40)(120,80)(70,100)

%vertices on edge
\put(75.5,111){\circle*{3}}
\put(9.5,29.5){\circle*{3}}
\put(100.5,29.5){\circle*{3}}
\put(24.5,105.5){\circle*{3}}
\put(50,5){\circle*{3}}

\end{picture}
\vspace{-.1in}
		\caption{A generalized plabic graph inside a ``swiss-cheese'' shape
		(in this case, a disk with two smaller disks removed).}
\label{fig:cheese}
\end{figure}
\vspace{-.1in}

\begin{proposition}
\label{prop:flipmoves}
If two trivalent plabic graphs are related to each other by moves {\rm(M2)}
and/or~{\rm (M3)}, then they are related via flip~moves~{\rm(M4)}.
\end{proposition}

\begin{proof}
%Let $G$ and $G'$ be the plabic graphs in question. 
Without loss of generality, we assume that the given plabic graphs $G, G'$ are connected.  
By \cref{lem:noM1}, they are related by moves~(M3).

Each of the graphs $G$ and $G'$ breaks into disjoint (white or black) components. 
Each (M3) move only affects a single component. 
It follows that the white (resp., black) components $W_1,\dots,W_{\ell}$ 
(resp., $B_1,\dots,B_m$) of~$G$ are in bijection with the components 
$W'_1,\dots,W'_{\ell}$ (resp., $B'_1,\dots,B'_m$) of~$G'$, 
so that each $W_i$ (resp., $B_j$) is related to $W_i'$ (resp., $B_j'$) 
via (M3) moves. 

Since an (M3) move preserves the number of boundary vertices and the number of faces,
both $W_i$ and $W'_i$ (respectively, $B_j$ and~$B'_j$) have the same
number of boundary vertices and the same number of faces.
It now follows from \cref{lem:flipwhite} 
(more precisely, from its extension to components~of plabic graphs, see \cref{rem:component-as-plabic-graph})
that each pair $W_i$ and $W'_i$ can be connected
by flip moves, and similarly for $B_j$ and $B'_j$.  The proposition follows.
\end{proof}
\vspace{-.05in}

\begin{proof}[Proof of \cref{thm:newmoves1}.]

The ``if'' direction follows from \cref{rem:flip}.

Suppose that $G$ and $G'$ are related via a sequence of 
(M1), (M2), and (M3) moves.
Let $k$ denote the number of square moves~(M1) in the sequence.
%	We want to show that they can 
%alternatively be related by a 
%Let us choose a sequence of (M1), (M2),
%and (M3) moves relating $G$ and $G'$ which minimizes 
%square (or (M1)) moves used.  Suppose this number is $k$.
We then have a sequence of move-equivalences
\[
%G=G_0 \sim G_1 \sim G_2 \sim \dots \sim G_k=G',
G=G_0' \sim G_1 \sim G_1' \sim G_2 \sim G_2'\sim \dots \sim G_k\sim G_k'\sim G_{k+1}=G',
\]
where for all $i$,  
$G_i$ is related to $G_i'$  by a single square move, whereas 
$G_i'$ is related to $G_{i+1}$ by a sequence of (M2) and (M3) moves. 
Since a square move only involves trivalent vertices, we may assume,
applying extra (M2) and (M3) moves as needed, 
that all plabic graphs $G_i$ and $G_i'$ are trivalent.
It then follows by \cref{prop:flipmoves} that for every~$i$, 
the graphs $G_i'$ and $G_{i+1}$ are related by flip moves alone, and we are done.
\end{proof}

\begin{remark}
Plabic graphs associated to wiring diagrams (ordinary or double), cf.\ Examples~\ref{def:wiringplabic}
and~\ref{ex:DWD}, are trivalent.
Consequently, one can express the 
transformations corresponding to braid moves using square moves and flip moves, 
as shown in \cref{fig:braidsquare0}.
On the other hand, plabic graphs associated to triangulations of a polygon
(see \cref{TriangulationA} and \cref{fig:triang-plabic0}) are \emph{not} trivalent.
\end{remark}

\section{Triple diagrams and normal plabic graphs}
\label{sec:triple-diagrams}

\emph{Triple diagrams} (or triple crossing diagrams), introduced by D.~Thurston in~\cite{thurston}, 
are planar topological gadgets closely related to plabic graphs.
Our treatment of triple diagrams in Sections~\ref{sec:triple-diagrams}--\ref{sec:mintriple}
is largely based on~\cite{thurston}.

\begin{definition}
\label{def:triple-diagram}
Let $\mathfrak{X}$ be a collection of oriented closed intervals 
and~oriented circles immersed into a closed disk~$\mathbf{D}$. 
The image of each interval~$I$ or circle~$C$ is called a \emph{strand}; 
it inherits its orientation from $I$ or~$C$. %the corresponding in\-ter\-val or circle. 
The strands that are immersed intervals (resp., circles)~are called \emph{arcs} (resp., \emph{closed strands}). 
%	Their images are collectively called \emph{strands}. 
%The images of oriented intervals (resp., circles) are the \emph{arcs} (resp., \emph{closed strands}) of~$\X$. 
A~\emph{face} of~$\mathfrak{X}$ is a connected component~of the complement of the union of all strands within~$\mathbf{D}$. 
We call  $\mathfrak{X}$ a \emph{triple diagram}  if 
\begin{itemize}[leftmargin=.2in]
\item 
the endpoints of arcs are distinct points located on the boundary $\partial \mathbf{D}$; 
each arc meets the boundary transversally;
%no other points map to the boundary; 
\item 
each closed strand is entirely contained in the interior of~$\mathbf{D}$; 
\item 
every point that lies on more than one local branch of a strand is a \emph{triple point} in the interior of~$\mathbf{D}$ 
where exactly three local branches meet, intersecting each other transversally;
\item
the union of the strands and the boundary $\partial \mathbf{D}$ is connected;  
this ensures that each face is homeomorphic to an open disk; 
\item 
the strand segments lying on the boundary of each face 
are oriented consistently (i.e., clockwise or counterclockwise); 
in particular, as we move along the boundary,
the \emph{sources} (endpoints where an arc runs away from~$\partial \mathbf{D}$) 
alternate with the \emph{targets} (where an arc runs into~$\partial \mathbf{D}$). 
 \end{itemize}

Each triple diagram, say with $b$ arcs, comes with a selection of $b$ distinguished points on~$\partial \mathbf{D}$
that are called \emph{boundary vertices}. 
There is one such boundary vertex within every other segment of~$\partial \mathbf{D}$
between two consecutive arc endpoints. 
Specifically, we place boundary vertices so that, moving clockwise along the boundary,
each boundary vertex follows (resp., precedes) a source (resp., a target). 
We label the boundary vertices $1,\dots,b$ in clockwise order. 
See \cref{fig:triple}. 
\end{definition}

\begin{figure}[ht]
\begin{center}
\vspace{-.2in}
\begin{tikzpicture}[scale=0.90]
% circle outlines
\filldraw[fill=white, thick] (12,0) circle (1.75cm);

% number labels
\node at (10.7,1.8) {\large \textbf{1}};
\node at (13.3,1.8) {\large \textbf{2}};
\node at (13.6,-1.4) {\large \textbf{3}};
\node at (12,-2.2) {\large \textbf{4}};
\node at (10.4,-1.4) {\large \textbf{5}};

\filldraw[color=black, fill=black, very thick](11.03,1.46) circle (0.1em);
\filldraw[color=black, fill=black, very thick](12.97,1.46) circle (0.1em);
\filldraw[color=black, fill=black, very thick](13.3,-1.17) circle (0.1em);
\filldraw[color=black, fill=black, very thick](12,-1.76) circle (0.1em);
\filldraw[color=black, fill=black, very thick](10.7,-1.17) circle (0.1em);

	% circle 3
\draw[red, thick] plot [smooth] coordinates {(11.15, 1.55) (11.68, 0.98) (11.75,0.8) (11.9, 0.45) (12.25,0.3) (12.5, 0.015) (12.5, -0.35) (12.1, -0.8) (12.1, -1.75)}; % 1 to 4
\draw[red, thick] plot [smooth] coordinates {(12.85, 1.55) (12.25, 1) (11.75,0.81) (11.66, 0.62) (11.65,0.3) (11.49,-0.1) (11.5,-0.25) (11.6,-0.45) (11.9,-0.8) (11.9,-1.75)}; % 2 to 4
\draw[red, thick] plot [smooth] coordinates {(13.08, 1.38) (12.4, 0.8) (12.33, 0.39) (12.25,0.3) (12.15, 0.25) (11.8, 0.15) (11.5,-0.25) (10.8, -1.3)}; % 2 to 5
\draw[red, thick] plot [smooth] coordinates {(13.4, -1.07) (12.58,-0.29) (12,-0.5) (11.5,-0.25) (11.25,-0.35) (10.6,-1.07)}; % 3 to 5
\draw[red, thick] plot [smooth] coordinates {(10.92,1.38) (11.75,0.8) (12.1,0.65) (12.25,0.3) (12.28,0.03) (12.38,-0.12) (12.5,-0.25) (12.63,-0.45) (13.2,-1.3)}; % 1 to 3
\draw[>={Stealth[length=6pt]}, red, semithick, ->] (10.95,1.35) -- (11.38,1.05); % 1 out
\draw[>={Stealth[length=6pt]}, red, semithick, ->] (11.53,1.15) -- (11.25,1.45); % 1 in
\draw[>={Stealth[length=6pt]}, red, semithick, ->] (12.85,1.55) -- (12.44,1.15); % 2 out
\draw[>={Stealth[length=6pt]}, red, semithick, ->] (12.6,1) -- (12.92,1.25); % 2 in
\draw[>={Stealth[length=6pt]}, red, semithick, ->] (13.4,-1.07) -- (13.1,-0.75); % 3 out
\draw[>={Stealth[length=6pt]}, red, semithick, ->] (12.76,-0.65) -- (13.07,-1.1); % 3 in
\draw[>={Stealth[length=6pt]}, red, semithick, ->] (12.1,-1.75) -- (12.08,-1.1); % 4 out
\draw[>={Stealth[length=6pt]}, red, semithick, ->] (11.91,-1.3) -- (11.91,-1.55); % 4 in
\draw[>={Stealth[length=6pt]}, red, semithick, ->] (10.8,-1.3) -- (11.08,-0.9); % 5 out
\draw[>={Stealth[length=6pt]}, red, semithick, ->] (11.1,-0.5) -- (10.78,-0.87); % 5 in

\filldraw[fill=none, thick] (12,0) circle (1.75cm);
\end{tikzpicture}
\vspace{-.3in}
\end{center}
	\caption{A triple diagram $\X$ with the strand permutation \hbox{$\pi_{\X}=(3, 4,5,1,2)$}.}
\label{fig:triple}
\end{figure}
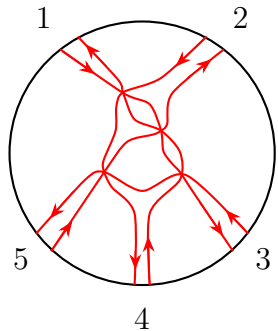

Triple diagrams are viewed up to smooth isotopy among such diagrams.
This makes them essentially combinatorial objects: 6-valent/uni\-valent directed graphs 
with some additional structure.

\begin{remark}
In \cite{thurston}, the definition of a triple diagram does not include the restriction
appearing in \cref{def:triple-diagram} 
that requires the union of the strands and the boundary~$\partial\mathbf{D}$ 
to be connected. 
In the terminology of~\cite{thurston}, all our triple diagrams are \emph{connected}. 
\end{remark}

\begin{remark}
In order to ensure consistent orientations along the face boundaries, 
the orientations of strands must alternate between ``in'' and ``out'' around each triple point.
Given a triple diagram with unoriented strands, we can 
satisfy this condition as follows: start anywhere and
propagate out by assigning alternating orientations around vertices.  
\end{remark}

\begin{definition}
\label{def:strandperm}
Let $\mathfrak{X}$ be a triple diagram
with $b$ boundary vertices (hence $b$ arcs).
For each boundary vertex~$i$, let $s_i$ (resp.,~$t_i$) denote the source (resp., target)
arc endpoint located next to~$i$ on the boundary of~$\mathbf{D}$. 

The \emph{strand permutation} $\pi_{\mathfrak{X}}$
is defined by setting $\pi_{\mathfrak{X}}(i)=j$ whenever the arc originating at~$s_i$
ends up at~$t_j$. 
Thus, the strand permutation describes the connectivity of the arcs. 
See \cref{fig:triple}. 
\end{definition}

We will soon see (cf.\ \cref{def:standard-triple} below) 
that any permutation can arise as a strand permutation of a triple diagram.

Just as the local moves on plabic graphs preserve the (decorated)
trip permutation (see \cref{exercise:forward}),
there is a notion of a local move on triple diagrams that keeps 
the strand permutation invariant.

\begin{definition}
\label{def:2-2-move-equivalence}
A \emph{swivel move} is a local transformation of triple diagrams that is shown in \cref{fig:2-2-move}.
(This move is called a \emph{$2\leftrightarrow2$ move} in~\cite{thurston}.)
We say that two triple diagrams $\X$ and $\X'$ are
\emph{move-equivalent} to each other, and write
$\X \sim \X'$, if one can get from $\X$ to~$\X'$ 
via a sequence of swivel moves.
\end{definition}

%%%%%%%%%%%%%%%%%%%%%%%%%%%%%%%%%%%%%%%%%%%%%%%%%%%%%%%%%%%%%%%%%%%%%%%%%%%
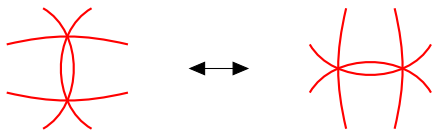
\begin{figure}[ht]
\vspace{-5pt}
\begin{tikzpicture}[scale=0.8]

\draw [red, thick] plot [smooth] coordinates {(0.35, 2.5) (0.4, 2.3) (0.6, 2) (1.1, 1.5) (1.1, 0.7) (0.6,0.2) (0.4, -0.1)  (0.35, -0.3)};

\draw [red, thick] plot [smooth] coordinates { (1.65, 2.5) (1.6, 2.3) (1.4, 2) (0.9, 1.5) (0.9, 0.7) (1.4,0.2) (1.6, -0.1)  (1.65, -0.3)};

\draw [red, thick] plot [smooth] coordinates {(2.5, 0.5) (2.3, 0.65) (2, 0.7) (1.4, 0.55) (0.6, 0.55) (0, 0.7) (-0.3, 0.65)  (-0.5, 0.5)};

\draw [red, thick] plot [smooth] coordinates {(2.5, 1.7) (2.3, 1.55) (2, 1.5) (1.4, 1.65) (0.6, 1.65) (0, 1.5) (-0.3, 1.55)  (-0.5, 1.7)};

\draw [line width=.5pt,dash pattern=on 1.5pt off 3pt, thick](1,1.1) circle(1cm);

\draw[>=triangle 45, <->] (3,1) -- (4,1);

\draw [red, thick] plot [smooth] coordinates {(5.35, 2.5) (5.4, 2.3) (5.6, 2) (5.45, 1.5) (5.45, 0.7) (5.6,0.2) (5.4, -0.1)  (5.35, -0.3)};

\draw [red, thick] plot [smooth] coordinates {(6.65, 2.5) (6.6, 2.3) (6.4, 2) (6.55, 1.5) (6.55, 0.7) (6.4,0.2) (6.6, -0.1)  (6.65, -0.3)};

\draw [red, thick] plot [smooth] coordinates {(7.5, 0.5) (7.3, 0.65) (7, 0.7) (6.8, 0.9) (6.4, 1.2) (5.6, 1.2) (5.2, 0.9) (5, 0.7) (4.7, 0.65)  (4.5, 0.5)};

\draw [red, thick] plot [smooth] coordinates {(7.5, 1.7) (7.3, 1.55) (7, 1.5) (6.8, 1.35) (6.4, 1.05) (5.6, 1.05) (5.2, 1.35) (5, 1.5) (4.7, 1.55)  (4.5, 1.7)};

\draw [line width=.5pt,dash pattern=on 1.5pt off 3pt, thick](6,1.1) circle(1cm);

\end{tikzpicture}
\caption{The swivel move
replaces one of these fragments of a triple diagram by the other fragment,
then smoothens out the strands.
The orientation of each strand on the left should match
the orientation of the strand on the right that has the same endpoints.}
\label{fig:2-2-move}
\end{figure}

%%%%%%%%%%%%%%%%%%%%%%%%%%%%%%%%%%%%%%%%%%%%%%%%%%%%%%%%%%%%%%%%%%%%%%%%%%

\begin{remark}
The connectivity of strands on both sides of \cref{fig:2-2-move}
is the same, so the strand permutation is invariant under the swivel move.
\end{remark}

We will soon see that triple diagrams are cryptomorphic to a variant
of plabic graphs that we call \emph{normal plabic graphs}, 
 defined below.

\pagebreak[3]

\begin{definition}
\label{def:normal}
Let $G$ be a plabic graph, where we allow leaves.
We say that $G$ is \emph{normal} if the coloring of its 
internal vertices %$G$ 
is bipartite, all white 
vertices in $G$ are trivalent, and each boundary vertex is adjacent 
to a black vertex. 
See \cref{fig:normal-plabic}. 
\end{definition}

\begin{figure}[ht]
\vspace{-15pt}
\begin{center}
%\vspace{-.2in}
\begin{tikzpicture}[scale=1]

% circle outlines
\filldraw[fill=white, thick] (0,0) circle (1.75cm);

% number labels
\node at (-1.2,1.7) {\large $\mathbf{1}$};
\node at (1.2,1.7) {\large $\mathbf{2}$};
\node at (1.5,-1.4) {\large $\mathbf{3}$};
\node at (0,-2.1) {\large $\mathbf{4}$};
\node at (-1.5,-1.4) {\large $\mathbf{5}$};

% circle 1
\node[circle, fill=black, draw=black, inner sep=0pt, minimum size=2.5pt] (1) at (-0.95,1.45) {};
\node[circle, fill=black, draw=black, inner sep=0pt, minimum size=2.5pt] (2) at (0.95,1.45) {};
\node[circle, fill=black, draw=black, inner sep=0pt, minimum size=4pt] (3) at (-0.6,1.15) {};
\node[circle, thick, fill=white, draw=black, inner sep=0pt, minimum size=4pt] (4) at (-0.25,0.8) {};
\node[circle, fill=black, draw=black, inner sep=0pt, minimum size=4pt] (5) at (0.25,0.8) {};
\node[circle, fill=black, draw=black, inner sep=0pt, minimum size=4pt] (6) at (-0.25,0.3) {};
\node[circle, thick, fill=white, draw=black, inner sep=0pt, minimum size=4pt] (7) at (0.25,0.3) {};
\node[circle, fill=black, draw=black, inner sep=0pt, minimum size=4pt] (8) at (0.375,0.015) {};
\node[circle, thick, fill=white, draw=black, inner sep=0pt, minimum size=4pt] (9) at (-0.5,-0.25) {};
\node[circle, thick, fill=white, draw=black, inner sep=0pt, minimum size=4pt] (10) at (0.5,-0.25) {};
\node[circle, fill=black, draw=black, inner sep=0pt, minimum size=4pt] (11) at (-0.88,-0.7) {};
\node[circle, fill=black, draw=black, inner sep=0pt, minimum size=4pt] (12) at (0,-0.7) {};
\node[circle, fill=black, draw=black, inner sep=0pt, minimum size=4pt] (13) at (0.88,-0.7) {};
\node[circle, fill=black, draw=black, inner sep=0pt, minimum size=2.5pt] (14) at (-1.28,-1.18) {};
\node[circle, fill=black, draw=black, inner sep=0pt, minimum size=2.5pt] (15) at (1.28,-1.18) {};
\node[circle, fill=black, draw=black, inner sep=0pt, minimum size=2.5pt] (16) at (0,-1.76) {};
\draw[thick] (1) -- (4);
\draw[thick] (2) -- (5);
\draw[thick] (4) -- (5);
\draw[thick] (4) -- (6);
\draw[thick] (5) -- (7);
\draw[thick] (6) -- (7);
\draw[thick] (6) -- (9);
\draw[thick] (7) -- (10);
\draw[thick] (9) -- (14);
\draw[thick] (10) -- (15);
\draw[thick] (9) -- (12);
\draw[thick] (10) -- (12);
\draw[thick] (12) -- (16);

\end{tikzpicture}
\vspace{-.2in}
\end{center}
\caption{A normal plabic graph~$G$. This graph was 
obtained from  the one in \cref{fig:plabic}(a) by inserting several bivalent black vertices.}
\label{fig:normal-plabic}
\end{figure}
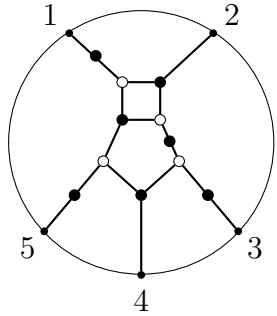

\begin{remark}
\label{rem:no-white-lollipop}
If a normal plabic graph 
$G$ has a leaf or a lollipop, it must be black, since 
 white vertices are required to be trivalent.
Therefore, in the case of normal graphs, there is no need to decorate the trip permutation. 
\end{remark}

\begin{remark}\label{rem:neverreduced}
If a normal plabic graph $G$ has a black leaf, then $G$ will 
fail to be reduced, see 
\cref{def:reduced-plabic}.
\end{remark}

\cref{def:GX} below associates a diagram $\X(G)$ to a normal plabic graph.  
We will then show in 
\cref{lem:normaltriple} that $\X(G)$ is indeed a triple diagram.

\begin{definition}
\label{def:GX}
Given a normal plabic graph $G$, 
we associate a diagram $\X(G)$ as follows.
To each trip in~$G$---either a one-way trip or a roundtrip---we associate a strand
in the ambient disk~$\mathbf{D}$ by slightly deforming the trip, 
as shown in \cref{fig:plabic-to-triple}, 
so that the strand 
\begin{itemize}[leftmargin=.2in]
\item 
runs along each edge of the trip, keeping the edge on its left,
\item
makes a U-turn at each black internal leaf (a vertex of degree~1),
\item
ignores black vertices of degree~2,
\item 
makes a right turn (as sharp as possible) at each other black vertex, and % of degree $>2$, and 
\item 
makes a left turn at each white vertex~$v$ along the trip, passing through~$v$.
\end{itemize}
%This collection of strands forms a triple diagram $\mathfrak{X}(G)$, 
\end{definition}

\begin{figure}[ht]
\vspace{-.1in}
\begin{center}
\begin{tabular}[h]{ccc}
\begin{tabular}[h]{c}
\ \\[-.9in]
\setlength{\unitlength}{1pt}
\begin{picture}(35,16)(0,0)
\thicklines
\put(30,10){\circle*{4}}
\put(0,10){\line(1,0){28}}
\red{
\put(0,16){\line(1,0){30}}
\put(0,4){\line(1,0){30}}
\put(19,4){\vector(1,0){1}}
\put(11,16){\vector(-1,0){1}}
\qbezier(30,16)(36,16)(36,10)
\qbezier(30,4)(36,4)(36,10)
}
\end{picture}
\\[.4in]
\setlength{\unitlength}{1pt}
\begin{picture}(60,16)(0,0)
\thicklines
\put(30,10){\circle*{4}}
\put(0,10){\line(1,0){28}}
\put(60,10){\line(-1,0){28}}
\red{
\put(0,4){\line(1,0){60}}
\put(0,16){\line(1,0){60}}
\put(33,16){\vector(-1,0){1}}
\put(27,4){\vector(1,0){1}}
}
\end{picture}
\end{tabular}
&
\ \ \quad
\setlength{\unitlength}{1pt}
\begin{picture}(70,60)(0,0)
\thicklines
\put(35,20){\circle*{4}}
\put(0,0){\line(7,4){33.25}}
\put(70,0){\line(-7,4){33.25}}
\put(35,22){\line(0,1){38}}
\red{\qbezier(2,-3.5)(35,16)(68,-3.5)
\qbezier(-2,3.5)(31.5,22)(31,60)
\qbezier(72,3.5)(38.5,22)(39,60)
\put(36,6){\vector(1,0){3}}
\put(48,25){\vector(-4,7){3}}
\put(21.5,25){\vector(-4,-7){3}}
}
\end{picture}
\qquad\quad
&
\setlength{\unitlength}{1pt}
\begin{picture}(70,60)(0,0)
\thicklines
\put(35,20){\circle{4}}
\put(0,0){\line(7,4){33.25}}
\put(70,0){\line(-7,4){33.25}}
\put(35,22){\line(0,1){38}}
\red{\qbezier(16,4.5)(42,15.5)(39,44)
\put(2,-3.5){\vector(7,4){16}}
\put(39,44){\vector(0,1){7}}
\put(39,44){\line(0,1){16}}
\qbezier(54,4.5)(28,15.5)(31,44)
\put(31,44){\line(0,1){16}}
\put(68,-3.5){\line(-7,4){16}}
\put(54,4.5){\vector(7,-4){3}}
\put(31,60){\vector(0,-1){17}}
\qbezier(12,11.5)(35,28.4)(58,11.5)
\put(-2,3.5){\line(7,4){16}}
\put(12,11.5){\vector(-7,-4){5}}
\put(58,11.5){\vector(-7,4){4}}
\put(72,3.5){\line(-7,4){16}}
}
\end{picture}
\end{tabular}
\vspace{-.1in}
\end{center}
\caption{Constructing a triple diagram  from a normal plabic graph.}
\label{fig:plabic-to-triple}
\end{figure}
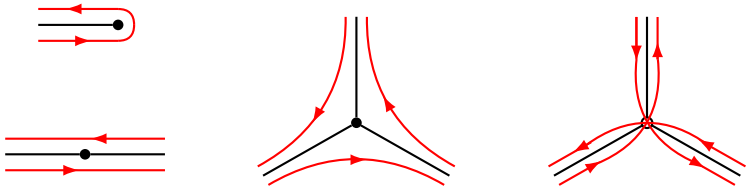
	See \cref{fig:normaltriple} for an example.  %and \cref{lem:normaltriple}.

\pagebreak[3]

\begin{figure}[ht]
\begin{center}
\vspace{-.05in}
\begin{tikzpicture}[scale=1]

% circle outlines
\filldraw[fill=none, thick] (6,0) circle (1.75cm);
\filldraw[fill=none, thick] (12,0) circle (1.75cm);

% number labels
\node at (4.7,1.8) {\large $\mathbf{1}$};
\node at (7.3,1.8) {\large $\mathbf{2}$};
\node at (7.6,-1.4) {\large $\mathbf{3}$};
\node at (6,-2.2) {\large $\mathbf{4}$};
\node at (4.4,-1.4) {\large $\mathbf{5}$};

\node at (10.7,1.8) {\large $\mathbf{1}$};
\node at (13.3,1.8) {\large $\mathbf{2}$};
\node at (13.6,-1.4) {\large $\mathbf{3}$};
\node at (12,-2.2) {\large $\mathbf{4}$};
\node at (10.4,-1.4) {\large $\mathbf{5}$};

% circle 2
\node[circle, fill=black, draw=black, inner sep=0pt, minimum size=2.5pt] (21) at (5.05,1.45) {};
\node[circle, fill=black, draw=black, inner sep=0pt, minimum size=2.5pt] (22) at (6.95,1.45) {};
\node[circle, fill=black, draw=black, inner sep=0pt, minimum size=4pt] (23) at (5.4,1.15) {};
\node[circle, thick, fill=white, draw=black, inner sep=0pt, minimum size=4pt] (24) at (5.75,0.8) {};
\node[circle, fill=black, draw=black, inner sep=0pt, minimum size=4pt] (25) at (6.25,0.8) {};
\node[circle, fill=black, draw=black, inner sep=0pt, minimum size=4pt] (26) at (5.75,0.3) {};
\node[circle, thick, fill=white, draw=black, inner sep=0pt, minimum size=4pt] (27) at (6.25,0.3) {};
\node[circle, fill=black, draw=black, inner sep=0pt, minimum size=4pt] (28) at (6.375,0.015) {};
\node[circle, thick, fill=white, draw=black, inner sep=0pt, minimum size=4pt] (29) at (5.5,-0.25) {};
\node[circle, thick, fill=white, draw=black, inner sep=0pt, minimum size=4pt] (210) at (6.5,-0.25) {};
\node[circle, fill=black, draw=black, inner sep=0pt, minimum size=4pt] (211) at (5.1,-0.7) {};
\node[circle, fill=black, draw=black, inner sep=0pt, minimum size=4pt] (212) at (6,-0.7) {};
\node[circle, fill=black, draw=black, inner sep=0pt, minimum size=4pt] (213) at (6.9,-0.7) {};
\node[circle, fill=black, draw=black, inner sep=0pt, minimum size=2.5pt] (214) at (4.72,-1.2) {};
\node[circle, fill=black, draw=black, inner sep=0pt, minimum size=2.5pt] (215) at (7.28,-1.2) {};
\node[circle, fill=black, draw=black, inner sep=0pt, minimum size=2.5pt] (216) at (6,-1.76) {};
\draw[thick] (21) -- (24);
\draw[thick] (22) -- (25);
\draw[thick] (24) -- (25);
\draw[thick] (24) -- (26);
\draw[thick] (25) -- (27);
\draw[thick] (26) -- (27);
\draw[thick] (26) -- (29);
\draw[thick] (27) -- (210);
\draw[thick] (29) -- (214);
\draw[thick] (210) -- (215);
\draw[thick] (29) -- (212);
\draw[thick] (210) -- (212);
\draw[thick] (212) -- (216);
\draw[red, thick] plot [smooth] coordinates {(5.15, 1.55) (5.75, 0.95) (24) (5.9, 0.45) (27) (6.5, 0.015) (6.5, -0.35) (6.15, -0.8) (6.1, -1.75)}; % 1 to 4
\draw[red, thick] plot [smooth] coordinates {(6.85, 1.55) (6.25, 0.95) (24) (5.65, 0.77) (5.65,0.3) (5.45,-0.1) (29) (5.6,-0.45) (5.85,-0.8) (5.9,-1.75)}; % 2 to 4
\draw[red, thick] plot [smooth] coordinates {(7.08, 1.38) (6.4, 0.8) (6.35, 0.35) (27) (6.15, 0.25) (5.8, 0.15) (29) (5.5, -0.4) (4.8, -1.3)}; % 2 to 5
\draw[red, thick] plot [smooth] coordinates {(7.4,-1.07) (6.7,-0.3) (210) (6,-0.5) (29) (5.25,-0.35) (4.6,-1.07)}; % 3 to 5
\draw[red, thick] plot [smooth] coordinates {(4.92,1.38) (5.65,0.77) (24) (6.1,0.65) (27) (6.23,0.03) (6.3,-0.12) (210) (6.5,-0.45) (7.2,-1.3)}; % 1 to 3
\draw[>={Stealth[length=6pt]}, red, semithick, ->] (4.95,1.35) -- (5.3,1.05); % 1 out
\draw[>={Stealth[length=6pt]}, red, semithick, ->] (5.56,1.15) -- (5.26,1.45); % 1 in
\draw[>={Stealth[length=6pt]}, red, semithick, ->] (6.85,1.55) -- (6.48,1.15); % 2 out
\draw[>={Stealth[length=6pt]}, red, semithick, ->] (6.6,1) -- (6.92,1.25); % 2 in
\draw[>={Stealth[length=6pt]}, red, semithick, ->] (7.4,-1.07) -- (7.12,-0.75); % 3 out
\draw[>={Stealth[length=6pt]}, red, semithick, ->] (6.65,-0.65) -- (7.03,-1.1); % 3 in
\draw[>={Stealth[length=6pt]}, red, semithick, ->] (6.1,-1.75) -- (6.11,-1.1); % 4 out
\draw[>={Stealth[length=6pt]}, red, semithick, ->] (5.89,-1.3) -- (5.9,-1.55); % 4 in
\draw[>={Stealth[length=6pt]}, red, semithick, ->] (4.8,-1.3) -- (5.12,-0.9); % 5 out
\draw[>={Stealth[length=6pt]}, red, semithick, ->] (5.1,-0.5) -- (4.78,-0.87); % 5 in
% circle 3
\draw[red, thick] plot [smooth] coordinates {(11.15, 1.55) (11.68, 0.98) (11.75,0.8) (11.9, 0.45) (12.25,0.3) (12.5, 0.015) (12.5, -0.35) (12.1, -0.8) (12.1, -1.75)}; % 1 to 4
\draw[red, thick] plot [smooth] coordinates {(12.85, 1.55) (12.25, 1) (11.75,0.81) (11.66, 0.62) (11.65,0.3) (11.49,-0.1) (11.5,-0.25) (11.6,-0.45) (11.9,-0.8) (11.9,-1.75)}; % 2 to 4
\draw[red, thick] plot [smooth] coordinates {(13.08, 1.38) (12.4, 0.8) (12.33, 0.39) (12.25,0.3) (12.15, 0.25) (11.8, 0.15) (11.5,-0.25) (10.8, -1.3)}; % 2 to 5
\draw[red, thick] plot [smooth] coordinates {(13.4, -1.07) (12.58,-0.29) (12,-0.5) (11.5,-0.25) (11.25,-0.35) (10.6,-1.07)}; % 3 to 5
\draw[red, thick] plot [smooth] coordinates {(10.92,1.38) (11.75,0.8) (12.1,0.65) (12.25,0.3) (12.28,0.03) (12.38,-0.12) (12.5,-0.25) (12.63,-0.45) (13.2,-1.3)}; % 1 to 3
\draw[>={Stealth[length=6pt]}, red, semithick, ->] (10.95,1.35) -- (11.38,1.05); % 1 out
\draw[>={Stealth[length=6pt]}, red, semithick, ->] (11.53,1.15) -- (11.25,1.45); % 1 in
\draw[>={Stealth[length=6pt]}, red, semithick, ->] (12.85,1.55) -- (12.44,1.15); % 2 out
\draw[>={Stealth[length=6pt]}, red, semithick, ->] (12.6,1) -- (12.92,1.25); % 2 in
\draw[>={Stealth[length=6pt]}, red, semithick, ->] (13.4,-1.07) -- (13.1,-0.75); % 3 out
\draw[>={Stealth[length=6pt]}, red, semithick, ->] (12.76,-0.65) -- (13.07,-1.1); % 3 in
\draw[>={Stealth[length=6pt]}, red, semithick, ->] (12.1,-1.75) -- (12.08,-1.1); % 4 out
\draw[>={Stealth[length=6pt]}, red, semithick, ->] (11.91,-1.3) -- (11.91,-1.55); % 4 in
\draw[>={Stealth[length=6pt]}, red, semithick, ->] (10.8,-1.3) -- (11.08,-0.9); % 5 out
\draw[>={Stealth[length=6pt]}, red, semithick, ->] (11.1,-0.5) -- (10.78,-0.87); % 5 in

% circle outlines
\filldraw[fill=none, thick] (6,0) circle (1.75cm);
\filldraw[fill=none, thick] (12,0) circle (1.75cm);
\end{tikzpicture}
\vspace{-.2in}
\end{center}
\caption{Left: 
A normal plabic graph~$G$ (cf.\ \cref{fig:normal-plabic}) 
together with the associated triple diagram~$\X=\mathfrak{X}(G)$.
Conversely, $G=G(\X)$. 
Right: The triple diagram~$\X=\mathfrak{X}(G)$.  
The trip permutation of $G$ and the strand permutation of $\X$
are equal: $\pi_G=\pi_\X=(3,4,5,1,2)$.}
\label{fig:normaltriple}
\end{figure}

\begin{lemma}
\label{lem:normaltriple}
	The diagram $\X(G)$ 
	associated to a normal plabic graph~$G$ as in  \cref{def:GX} is a triple diagram.
\end{lemma}
\begin{proof}
Since the white vertices in $G$ are trivalent, 
$\X(G)$ has a triple point for every white vertex in~$G$, and no other crossings.
We need to check  that 
$\X=\X(G)$ is connected, or more precisely, 
that the union of the strands and the boundary of the disk is connected,
as required by \cref{def:triple-diagram}.  
%Clearly
%every one-way trip in $G$ becomes a strand which is connected
%to the boundary of the disk, so we just need to check that there
%is no union of closed strands (obtained from roundtrips in $G$)
%which is disconnected from the collection of strands
%connected to the boundary of $G$.

Let us ignore any component consisting of a single black vertex 
which is adjacent only to boundary vertices
(e.g. a black lollipop), as the corresponding strands
are clearly connected to the boundary of the disk. 
Consider any other strand $S$ in~$\X$.
By construction, $S$~passes through at least one white vertex
of the (bipartite) graph~$G$, which is a~triple point on~$S$. 
It therefore suffices to show that every triple point in~$\X$ is connected 
to the boundary within~$\X$ (i.e., via strand segments of~$\X$). 

Let $u$ be a $k$-valent black vertex in~$G$ and let $v_1,\dots,v_k$
be the white or boundary vertices adjacent to~$u$. 
The strands of~$\X$ that run along the $k$ edges of~$G$ incident to~$u$
cyclically connect the triple points $v_1,\dots,v_k$ to each other.
(If the list $v_1,\dots,v_k$ includes boundary vertices, then the corresponding
strand segments are connected via the boundary.) 
We conclude that for any two-edge path $v-\!\!\!-u-\!\!\!-v'$ in~$G$
connecting two white or boundary vertices $v$ and~$v'$ 
via a black vertex~$u$,
the triple (or nearby boundary) points $v$ and~$v'$ are connected within~$\X$. 
It follows that for any path in the bipartite graph~$G$ connecting a white vertex~$v$ 
to the boundary, there is a path in~$\X$ that connects the triple point~$v$ to the boundary. 

It remains to note that by \cref{def:plabic}, any white vertex $v$ in~$G$
is connected by a path in~$G$ to some boundary vertex.
\end{proof}

We now go in the opposite direction, from a triple diagram to a normal plabic graph. 

\begin{definition}
\label{def:G(X)}
The normal plabic graph $G\!=\!G(\X)$ associated 
to a triple  diagram~$\X$ is constructed as follows. 
Place a white vertex of~$G$ at each triple crossing in~$\X$.
Treat each boundary vertex of~$\X$ as a boundary vertex of~$G$. 
For each region~$R$ of~$\X$ whose boundary is oriented counter\-clock\-wise, 
place a black vertex in the interior of~$R$ and connect it to the white and boundary vertices 
lying on the boundary of~$R$, so that each white (resp., boundary) vertex 
is trivalent (resp., univalent). 
The resulting plabic graph $G\!=\!G(\mathfrak{X})$ is normal by construction.
\end{definition}

\begin{proposition}
\label{pr:bijplabictriple}
The maps $G \mapsto \X(G)$ and $\X\mapsto G(\X)$ 
described in Definitions \ref{def:GX} and~\ref{def:G(X)} are mutually inverse bijections between 
normal plabic graphs and  triple diagrams with the same number of boundary vertices.
The trip permutation $\pi_G$ of a normal graph~$G$ is equal to the strand permutation 
$\pi_{\mathfrak{X}}$ of the corresponding triple diagram $\X=\X(G)$.
\end{proposition}

\cref{fig:bijreg} illustrates the bijection between normal plabic graphs and triple diagrams
in the case of reduced normal plabic graphs on three nodes.

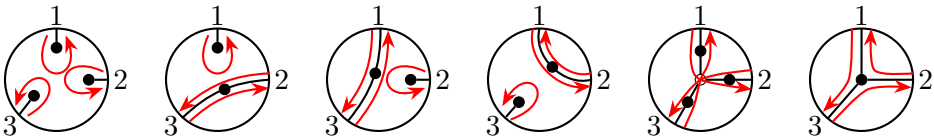
\begin{figure}[ht]
\begin{center}
\vspace{-.2in}
\begin{tikzpicture}[scale=0.85]

\filldraw[fill=white, thick] (0,0) circle (0.8cm);
\filldraw[fill=white, thick] (2.5,0) circle (0.8cm);
\filldraw[fill=white, thick] (5,0) circle (0.8cm);
\filldraw[fill=white, thick] (7.5,0) circle (0.8cm);
\filldraw[fill=white, thick] (10,0) circle (0.8cm);
\filldraw[fill=white, thick] (12.5,0) circle (0.8cm);

\node at (0,1) (1) {$\mathbf{1}$}; \node at (1,0) (2) {$\mathbf{2}$}; \node at (-0.72,-0.72) (3) {$\mathbf{3}$};
\node at (2.5,1) (21) {$\mathbf{1}$}; \node at (3.5,0) (22) {$\mathbf{2}$}; \node at (1.78,-0.72) (23) {$\mathbf{3}$};
\node at (5,1) (31) {$\mathbf{1}$}; \node at (6,0) (32) {$\mathbf{2}$}; \node at (4.28,-0.72) (33) {$\mathbf{3}$};
\node at (7.5,1) (41) {$\mathbf{1}$}; \node at (8.5,0) (42) {$\mathbf{2}$}; \node at (6.78,-0.72) (43) {$\mathbf{3}$};
\node at (10,1) (51) {$\mathbf{1}$}; \node at (11,0) (52) {$\mathbf{2}$}; \node at (9.28,-0.72) (53) {$\mathbf{3}$};
\node at (12.5,1) (61) {$\mathbf{1}$}; \node at (13.5,0) (62) {$\mathbf{2}$}; \node at (11.78,-0.72) (63) {$\mathbf{3}$};

% circle 1
\node[circle, fill=black, draw=black, inner sep=0pt, minimum size=4pt] at (0,0.5) (n1) {};
\node[circle, fill=black, draw=black, inner sep=0pt, minimum size=4pt] at (0.5,0) (n2) {};
\node[circle, fill=black, draw=black, inner sep=0pt, minimum size=4pt] at (-0.35,-0.25) (n3) {};
\draw[thick] (0,0.8) -- (n1);
\draw[thick] (0.8,0) -- (n2);
\draw[thick] (-0.59,-0.56) -- (n3);
\draw [>=Stealth, ->, red, thick] (1) to[out=245, in=295, distance=25pt] (1);
\draw [>=Stealth, ->, red, thick] (2) to[out=155, in=205, distance=25pt] (2);
\draw [>=Stealth, ->, red, thick] (3) to[out=30, in=75, distance=25pt] (3);

% circle 2
\node[circle, fill=black, draw=black, inner sep=0pt, minimum size=4pt] at (2.5,0.5) (n21) {};
\node[circle, fill=black, draw=black, inner sep=0pt, minimum size=4pt] at (2.61,-0.15) (n22) {};
\draw[thick] (2.5,0.8) -- (n21);
\draw[thick] (1.95,-0.6) to[bend left=20] (3.3,0);
\draw[>=Stealth, ->, red, thick] (21) to[out=245, in=295, distance=25pt] (21);
\draw[>=Stealth, ->, red, thick] (3.3,0.1) to[bend right=20] (1.9,-0.5);
\draw[>=Stealth, ->, red, thick] (2.07,-0.65) to[bend left=20] (3.3,-0.13);

% circle 3
\node[circle, fill=black, draw=black, inner sep=0pt, minimum size=4pt] at (4.95,0.1) (n31) {};
\node[circle, fill=black, draw=black, inner sep=0pt, minimum size=4pt] at (5.5,0) (n32) {};
\draw[thick] (5.8,0) -- (n32);
\draw[thick] (4.52,-0.65) to[bend right=20] (5.03,0.8);
\draw[>=Stealth, ->, red, thick] (32) to[out=155, in=205, distance=25pt] (32);
\draw[>=Stealth, ->, red, thick] (4.9,0.78) to[bend left=20] (4.42,-0.53);
\draw[>=Stealth, ->, red, thick] (4.65,-0.7) to[bend right=20] (5.15,0.77);

% circle 4
\node[circle, fill=black, draw=black, inner sep=0pt, minimum size=4pt] at (7.7,0.2) (n41) {};
\node[circle, fill=black, draw=black, inner sep=0pt, minimum size=4pt] at (7.18,-0.35) (n42) {};
\draw[thick] (6.98,-0.58) -- (n42);
\draw[thick] (7.5,0.8) to[bend right=60] (8.3,0);
\draw[>=Stealth, ->, red, thick] (43) to[out=20, in=65, distance=25pt] (43);
\draw[>=Stealth, ->, red, thick] (7.4,0.8) to[bend right=60] (8.3,-0.1);
\draw[>=Stealth, ->, red, thick] (8.3,0.1) to[bend left=60] (7.6,0.8);

% circle 5
\node[circle, fill=black, draw=black, inner sep=0pt, minimum size=4pt] at (10,0.45) (n51) {};
\node[circle, fill=black, draw=black, inner sep=0pt, minimum size=4pt] at (10.45,0) (n52) {};
\node[circle, fill=black, draw=black, inner sep=0pt, minimum size=4pt] at (9.8,-0.35) (n53) {};
\draw[thick] (10,0.8) -- (10,0);
\draw[thick] (10.8,0) -- (10,0);
\draw[thick] (9.6,-0.7) -- (10,0);
\node[circle, thick, fill=white, draw=black, inner sep=0pt, minimum size=4pt] at (10,0) (n54) {};
\draw[>=Stealth, ->, red, thick] plot [smooth] coordinates {(9.88,0.8) (9.85,0.3) (10,0) (10.8,-0.15)};
\draw[>=Stealth, ->, red, thick] plot [smooth] coordinates {(10.8,0.15) (10,0) (9.5,-0.6)};
\draw[>=Stealth, ->, red, thick] plot [smooth] coordinates {(9.75,-0.75) (9.9,-0.4) (10,0) (10.1,0.2) (10.15,0.38) (10.15,0.78)};

% circle 6
\node[circle, fill=black, draw=black, inner sep=0pt, minimum size=4pt] at (12.5,0) (n61) {};
\draw[thick] (n61) -- (12.5,0.8);
\draw[thick] (n61) -- (13.3,0);
\draw[thick] (n61) -- (11.98,-0.6);
\draw[>=Stealth, ->, red, thick] plot [smooth] coordinates {(13.3,0.1) (12.7,0.15) (12.65,0.8)};
\draw[>=Stealth, ->, red, thick] plot [smooth] coordinates {(12.35,0.8) (12.3,0) (11.9,-0.5)};
\draw[>=Stealth, ->, red, thick] plot [smooth] coordinates {(12.05,-0.66) (12.6,-0.15) (13.28,-0.12)};

\end{tikzpicture}
\vspace{-.3in}
\end{center}
\caption{The six reduced normal plabic graphs with three boundary
vertices, shown together with the corresponding triple diagrams, cf.\ \cref{def:GX}. 
The associated trip (resp., strand) permutations are precisely the six 
permutations of $\{1,2,3\}$.
}
\vspace{-10pt}
\label{fig:bijreg}
\end{figure}

\begin{proof}
Starting from a normal graph~$G$, let us decompose it into star-shaped subgraphs~$S_v$ 
each of which includes a black vertex~$v$, all the edges incident~to~$v$, and the endpoints of those edges. 
Each of these stars will give rise to a fragment of the triple diagram~$\X(G)$ that ``hugs'' the edges of~$S_v$
and whose boundary is oriented counterclockwise (looking from~$v$). 
Moreover, $\X(G)$ is obtained by stitching these fragments together. 
Applying the map $\X\mapsto G(\X)$ to $\X(G)$ will recover the original graph~$G$. 

One similarly shows that if we start from a triple diagram~$\X$, construct the normal graph~$G(\X)$,
and then apply the map $G\mapsto \X(G)$ to $G(\X)$, then we recover the original triple diagram~$\X$. 
The key property to keep in mind is that each face of~$\X$ is homeomorphic to an open disk. 

The strands of $\X$ run alongside the trips of~$G$, 
implying that $\pi_\X\!=\!\pi_G$. 
\end{proof}

In \cref{red-minimal}, we will characterize triple diagrams that correspond, 
under the bijection of \cref{pr:bijplabictriple}, to \emph{reduced} normal plabic graphs. 

\pagebreak[3]

The bijective correspondence between triple diagrams and normal plabic graphs
can be used to translate 
the equivalence of triple diagrams 
under swivel moves, cf.\ \cref{def:2-2-move-equivalence}, 
into a variant of move equivalence of plabic graphs 
(cf.\ \cref{def:moves}) adapted to the setting of normal plabic graphs:

\begin{definition}
\label{def:urban-normal}
The \emph{(normal) spider} move is a local transformation of a normal plabic graph~$G$  
that replaces one of the fragments shown in \cref{fig:urban2} (see also \cref{fig:urban2-bivalent}) by the other. 
To be more precise, assume that $G$ contains a quadrilateral face with black vertices $A, A'$
and white vertices $K, K'$. 
Let $K$ (resp.,~$K'$) be adjacent to the black vertices $A, A', B$ (resp., $A, A',B'$).
We then replace the white vertices $K,K'$ (resp., the 6~edges adjacent to them)
by the new white vertices $L,L'$ (resp., the edges $AL, BL, B'L, A'L', BL', B'L'$). 
%where we require that each black vertex of the quadrilateral face has degree at least $2$.
%See \cref{fig:urban2-bivalent} for an example in which the quadrilateral face has  a bivalent black vertex.
\end{definition}

\begin{figure}[ht]
\begin{center}
\vspace{-.2in}
\setlength{\unitlength}{1pt}
\begin{picture}(30,36)(0,0)
\thicklines
\multiput(2,0)(0,30){2}{\line(1,0){26}}
\multiput(0,2)(30,0){2}{\line(0,1){26}}
\put(-1.4,-1.4){\line(-1,-1){8}}
\put(31.8,-0.9){\line(2,-1){8}}
\put(30.9,-1.8){\line(1,-2){4}}
\put(-1.4,31.4){\line(-1,1){7}}
\put(-0.9,31.8){\line(-1,2){4}}
\put(-1.8,30.9){\line(-2,1){8}}
\put(31.4,31.4){\line(1,1){8}}
\put(0,0){\circle{4}}
\put(-5,-5){\circle*{4}}
\put(0,30){\circle*{4}}
\put(30,30){\circle{4}}
\put(35,35){\circle*{4}}
\put(30,0){\circle*{4}}
\end{picture}
	\quad
\begin{picture}(30,36)(0,0)
\put(15,15){\makebox(0,0){\Large{$\longleftrightarrow$}}}
\end{picture}
\quad
\begin{picture}(30,36)(0,0)
\thicklines
\put(10,8){\line(1,0){10}}
\put(10,22){\line(1,0){10}}
\put(8,10){\line(0,1){10}}
\put(22,10){\line(0,1){10}}
%\put(10,8)(22,8){2}{\line(0,1){10}} 
%\put(0,0){\circle{4}}
\put(-1.3,31.3){\circle*{4}}
%\put(30,30){\circle{4}}
\put(31.3,-1.3){\circle*{4}}
\put(8,8){\circle*{4}}
\put(8,22){\circle{4}}
\put(22,22){\circle*{4}}
\put(22,8){\circle{4}}

\put(8,8){\line(-1,-1){12}}
\put(22,22){\line(1,1){12}}
\put(6.6,23.4){\line(-1,1){8}}
\put(23.4,6.6){\line(1,-1){8}}

\put(0,0){\line(-1,-1){8}}
%\put(31,-1){\line(1,-1){7}}
\put(31,-1){\line(2,-1){8}}
\put(31,-1){\line(1,-2){4}}
\put(-1,31){\line(-1,1){7}}
\put(-1,31){\line(-1,2){4}}
\put(-1,31){\line(-2,1){8}}
\put(30,30){\line(1,1){8}}

\end{picture}

\end{center}
\vspace{-.1in} 
\caption{A normal spider move. 
%The normal spider move replaces one of these configurations by  the other. 
%Each black vertex of the quadrilateral face has degree at least $2$.
%, as long as that number agrees with its
%counterpart in the other local configuration.
%Each black vertex on the left can have any
%nonnegative number of edges incident to it and leading outside the 
%configuration, as long as its counterpart on the right has the same number of incident edges 
%leading outside the configuration.
%Cf.\ also \cref{fig:urban2-bivalent}.
}
\label{fig:urban2}
\end{figure}
\begin{figure}[ht]
\begin{center}
\vspace{-.1in}
\setlength{\unitlength}{1pt}
\begin{picture}(30,36)(0,0)
\thicklines
\multiput(2,0)(0,30){2}{\line(1,0){26}}
\multiput(0,2)(30,0){2}{\line(0,1){26}}
\put(-1.4,-1.4){\line(-1,-1){8}}
%\put(31.8,-0.9){\line(2,-1){8}}
%\put(30.9,-1.8){\line(1,-2){4}}
\put(-1.4,31.4){\line(-1,1){7}}
\put(-0.9,31.8){\line(-1,2){4}}
\put(-1.8,30.9){\line(-2,1){8}}
\put(31.4,31.4){\line(1,1){8}}
\put(0,0){\circle{4}}
\put(-5,-5){\circle*{4}}
\put(0,30){\circle*{4}}
\put(30,30){\circle{4}}
\put(35,35){\circle*{4}}
\put(30,0){\circle*{4}}
\end{picture}
	\quad
\begin{picture}(30,36)(0,0)
\put(15,15){\makebox(0,0){\Large{$\longleftrightarrow$}}}
\end{picture}
\quad
\begin{picture}(30,36)(0,0)
\thicklines
\put(10,8){\line(1,0){10}}
\put(10,22){\line(1,0){10}}
\put(8,10){\line(0,1){10}}
\put(22,10){\line(0,1){10}}
%\put(10,8)(22,8){2}{\line(0,1){10}} 
%\put(0,0){\circle{4}}
\put(-1.3,31.3){\circle*{4}}
%\put(30,30){\circle{4}}
\put(31.3,-1.3){\circle*{4}}
\put(8,8){\circle*{4}}
\put(8,22){\circle{4}}
\put(22,22){\circle*{4}}
\put(22,8){\circle{4}}

\put(8,8){\line(-1,-1){12}}
\put(22,22){\line(1,1){12}}
\put(6.6,23.4){\line(-1,1){8}}
\put(23.4,6.6){\line(1,-1){8}}

\put(0,0){\line(-1,-1){8}}
%\put(31,-1){\line(1,-1){7}}
%\put(31,-1){\line(2,-1){8}}
%\put(31,-1){\line(1,-2){4}}
\put(-1,31){\line(-1,1){7}}
\put(-1,31){\line(-1,2){4}}
\put(-1,31){\line(-2,1){8}}
\put(30,30){\line(1,1){8}}

\end{picture}

\end{center}
%\vspace{-.1in} 
\caption{A normal spider move for a quadrilateral face incident to a bivalent vertex. 
%one of the black vertices of the quadrilateral has degree $2$.
By \cref{rem:neverreduced}, this may only occur for non-reduced normal plabic graphs.}
\label{fig:urban2-bivalent}
\end{figure}

\begin{remark}
The normal spider moves introduced in 
\cref{def:urban-normal} are slightly different from 
the square moves described in \cref{fig:M1} (resp., Definition~\ref{def:urban-renewal}), 
which required each vertex of the quadrilateral face to have degree~$3$ (resp., degree at least~$3$).  
%It is also different from the spider move 
%of Definition~\ref{def:urban-renewal}, which required each vertex of 
%the quadrilateral face to have degree at least $3$.
%In this chapter, we will consistently use the new definition.) 
%Unlike the square move~(M1), the normal spider move 
%does not require the vertices of the square to be trivalent.
%They can even be bivalent, see \cref{fig:urban2-bivalent}. 
\end{remark}

\begin{definition}
\label{def:flip-normal}
The \emph{normal flip move} is the local transformation 
shown in \cref{fig:flipmove}. 
Ignoring the bivalent black vertices, this local move is the same as 
the (white) flip move for trivalent plabic graphs.
\end{definition}

\begin{figure}[ht]
\vspace{-.5in}
%\begin{center}
%\includegraphics[height=.8in]{FlipMove.ps}
%\end{center}
\begin{center}
\begin{tikzpicture}[scale=.7]

% left
\node[circle, fill=black, draw=black, inner sep=0pt, minimum size=4pt] (1) at (-0.85,0.75) {};
\node[circle, fill=black, draw=black, inner sep=0pt, minimum size=4pt] (2) at (0.85,0.75) {};
\node[circle, thick, fill=white, draw=black, inner sep=0pt, minimum size=4pt] (3) at (0,0.5) {};
\node[circle, fill=black, draw=black, inner sep=0pt, minimum size=4pt] (4) at (0,0) {};
\node[circle, thick, fill=white, draw=black, inner sep=0pt, minimum size=4pt] (5) at (0,-0.5) {};
\node[circle, fill=black, draw=black, inner sep=0pt, minimum size=4pt] (6) at (-0.85,-0.75) {};
\node[circle, fill=black, draw=black, inner sep=0pt, minimum size=4pt] (7) at (0.85,-0.75) {};
\draw[thick] (1) -- (3); \draw[thick] (2) -- (3); \draw[thick] (3) -- (5);
\draw[thick] (6) -- (5); \draw[thick] (7) -- (5);
\end{tikzpicture}
% \draw[>={Stealth[length=8pt]}, <->] (1.5,0) -- (3,0);
\begin{picture}(20,30)(0,0)
\put(10,6){\makebox(0,0){\Large{$\longleftrightarrow$}}}
\end{picture}
\begin{tikzpicture}[scale=.7]
\node[circle, fill=black, draw=black, inner sep=0pt, minimum size=4pt] (21) at (3.6,0.7) {};
\node[circle, fill=black, draw=black, inner sep=0pt, minimum size=4pt] (22) at (5.4,0.7) {};
\node[circle, thick, fill=white, draw=black, inner sep=0pt, minimum size=4pt] (23) at (3.9,0) {};
\node[circle, fill=black, draw=black, inner sep=0pt, minimum size=4pt] (24) at (4.5,0) {};
\node[circle, thick, fill=white, draw=black, inner sep=0pt, minimum size=4pt] (25) at (5.1,0) {};
\node[circle, fill=black, draw=black, inner sep=0pt, minimum size=4pt] (26) at (3.6,-0.7) {};
\node[circle, fill=black, draw=black, inner sep=0pt, minimum size=4pt] (27) at (5.4,-0.7) {};
\draw[thick] (21) -- (23); \draw[thick] (22) -- (25); \draw[thick] (23) -- (25);
\draw[thick] (26) -- (23); \draw[thick] (27) -- (25);
\end{tikzpicture}
\vspace{-.2in}
\end{center}
	\caption{The normal flip move.}
	\label{fig:flipmove}
\end{figure}

We want to relate the normal spider move and the normal flip move to the moves
(M1), (M2), (M3) that we saw in \cref{sec:overview}.  In order to do so, we 
need the following variant of (M3).

\begin{definition}\label{def:M3'}
The move (M3$'$) on plabic graphs contracts or uncontracts two internal
vertices of the same color, and is defined as in \cref{fig:M3}, except
that if both vertices are black, then 
the number of ``hanging'' edges on either side can be any 
\emph{nonnegative} integer.  In particular, (M3$'$) can create or remove
a black leaf.
\end{definition}

\def\osim{\stackrel{\raisebox{-3pt}{$\scriptstyle\bullet$}}{\sim}}
\def\omove{$\raisebox{-0.3pt}{$\circ$}$-move}

\begin{definition}
\label{def:omove}
Let $G$ and $G'$ be plabic graphs.
% $G$ and $G'$ are \emph{\omove-equiv\-a\-lent} to each other, and 
We write $G\osim G'$~if $G$ and $G'$ can be related to each other 
	via local moves (M1), (M2), and/or~(M3$'$).
%an instance of the (M3) move shown in \cref{fig:M3-white-leaf},
%cf.\ \cref{rem:problematic-X(G)}. 
\end{definition}

\begin{lemma}
\label{lem:normal-moves-via-M1M2M3}
Let $G$ and $G'$ be normal plabic graphs
related via a sequence of normal spider moves and normal flip moves.
Then $G\osim G'$. 
\end{lemma}

\begin{proof}
Degenerate versions of the normal spider move, 
as in \cref{fig:urban2-bivalent}, 
 can be expressed as a square move (M1)
together with (M2) and/or~(M3$'$) moves.
Nondegenerate versions of the normal spider move, as well as 
the normal flip move, can be expressed as a combination of (M1), (M2)
	and (M3) moves. (In particular, (M3$'$) is not needed for these.)
\end{proof}

By \cref{rem:neverreduced} and \cref{def:reduced-plabic},
a normal plabic graph which is reduced must be leafless.

%\begin{lemma}\label{lem:star}
%Suppose that $G$ and $G'$ are normal plabic graphs which are reduced.
%Then $G \osim G'$ if and only if $G \sim G'$.
%That is, $G$ and $G'$ can be related by moves (M1), (M2), and (M3$'$)
%if and only if $G$ and $G'$ can be related by moves (M1), (M2), and (M3).
%\end{lemma}
%\begin{proof}
%By definition, $G\sim G'$ implies that $G\osim G'$.
%If $G\osim G'$, then since $G$ and $G'$ are reduced, they  are leafless,
%which implies that we will never need to create/remove leaves in order to 
%	connect them. \LW{clear?}
%\end{proof}

\begin{lemma}
\label{lem:star}
Let $G$ and $G'$ be normal plabic graphs which are reduced, and which are 
related via a sequence of normal spider moves and normal flip moves.
Then $G\sim G'$. 
\end{lemma}
\begin{proof}
Since $G$ and $G'$ are reduced, 
we will never need to use a degenerate spider move
to relate them to each other.  Both the nondegenerate normal 
spider move and the normal flip 
move can be expressed in terms of (M1), (M2), and~(M3).
\end{proof}

\begin{theorem}
\label{thm:moves-moves}
Let $G$ and $G'$ be normal plabic graphs
and let $\X=\X(G)$ 
	and $\X'=\X(G')$ be the corresponding triple diagrams. 
Then the following are equivalent:
\begin{itemize}[leftmargin=.2in]
\item 
$G$ and $G'$ are related via a sequence of normal spider moves and/or normal flip moves; 
\item
$\X$ and $\X'$ are move-equivalent (i.e., related via swivel moves).
%(in the sense of \cref{def:2-2-move-equivalence}). 
\end{itemize}
\end{theorem}
\begin{figure}[ht]
%\begin{center}
%\includegraphics[height=.8in]{FlipsPlabicTriple.ps}
%\end{center}

\begin{proof}
\cref{fig:flipsplabictriple} shows that each swivel move in a triple diagram $\X(G)$ 
corresponds to---depending on the orientations of the strands---either a normal spider move 
or a normal flip move in the normal plabic graph $G$.
The statement of the theorem follows.
%how one can translate back-and-forth between 
%\begin{itemize}[leftmargin=.2in]
%\item 
%an arbitrary swivel move in a triple diagram and 
%\item
%either a normal spider move 
%or a normal flip move in the corresponding normal plabic graph. 
%\end{itemize}
%The latter choice depends on the orientations of the strands involved.
\end{proof}

\begin{center}
\vspace{-.05in}
\begin{tikzpicture}[scale=0.8]

% black
\node[circle, fill=black, draw=black, inner sep=0pt, minimum size=3pt] (1) at (1.35,2) {};
\node[circle, thick, fill=white, draw=black, inner sep=0pt, minimum size=4pt] (2) at (1.35,1.62) {};
\node[circle, fill=black, draw=black, inner sep=0pt, minimum size=0pt] (3) at (0.4,1) {};
\node[circle, fill=black, draw=black, inner sep=0pt, minimum size=3pt] (4) at (0.8,1) {};
\node[circle, fill=black, draw=black, inner sep=0pt, minimum size=3pt] (5) at (1.9,1) {};
\node[circle, fill=black, draw=black, inner sep=0pt, minimum size=0pt] (6) at (2.3,1) {};
\node[circle, thick, fill=white, draw=black, inner sep=0pt, minimum size=4pt] (7) at (1.35,0.4) {};
\node[circle, fill=black, draw=black, inner sep=0pt, minimum size=3pt] (8) at (1.35,0) {};
\draw[thick] (1) -- (2); \draw[thick] (3) -- (4); \draw[thick] (5) -- (6); \draw[thick] (7) -- (8);
\draw[thick] (2) -- (4); \draw[thick] (2) -- (5); \draw[thick] (4) -- (7); \draw[thick] (5) -- (7);
\node[circle, fill=black, draw=black, inner sep=0pt, minimum size=0pt] (21) at (6,2) {};
\node[circle, fill=black, draw=black, inner sep=0pt, minimum size=3pt] (22) at (6,1.6) {};
\node[circle, fill=black, draw=black, inner sep=0pt, minimum size=3pt] (23) at (5,1) {};
\node[circle, thick, fill=white, draw=black, inner sep=0pt, minimum size=4pt] (24) at (5.3,1) {};
\node[circle, thick, fill=white, draw=black, inner sep=0pt, minimum size=4pt] (25) at (6.7,1) {};
\node[circle, fill=black, draw=black, inner sep=0pt, minimum size=3pt] (26) at (7,1) {};
\node[circle, fill=black, draw=black, inner sep=0pt, minimum size=3pt] (27) at (6,0.4) {};
\node[circle, fill=black, draw=black, inner sep=0pt, minimum size=0pt] (28) at (6,0) {};
\draw[thick] (21) -- (22); \draw[thick] (23) -- (24); \draw[thick] (25) -- (26); \draw[thick] (27) -- (28);
\draw[thick] (22) -- (24); \draw[thick] (22) -- (25); \draw[thick] (24) -- (27); \draw[thick] (25) -- (27);
\node[circle, fill=black, draw=black, inner sep=0pt, minimum size=3pt] (31) at (9.9,1.8) {};
\node[circle, fill=black, draw=black, inner sep=0pt, minimum size=3pt] (32) at (10.8,1.8) {};
\node[circle, thick, fill=white, draw=black, inner sep=0pt, minimum size=4pt] (33) at (10.35,1.63) {};
\node[circle, fill=black, draw=black, inner sep=0pt, minimum size=3pt] (34) at (10.35,1) {};
\node[circle, thick, fill=white, draw=black, inner sep=0pt, minimum size=4pt] (35) at (10.35,0.4) {};
\node[circle, fill=black, draw=black, inner sep=0pt, minimum size=3pt] (36) at (9.9,0.2) {};
\node[circle, fill=black, draw=black, inner sep=0pt, minimum size=3pt] (37) at (10.8,0.2) {};
\draw[thick] (31) -- (33); \draw[thick] (32) -- (33); \draw[thick] (33) -- (35);
\draw[thick] (35) -- (36); \draw[thick] (35) -- (37);
\node[circle, fill=black, draw=black, inner sep=0pt, minimum size=3pt] (41) at (14.1,1.55) {};
\node[circle, fill=black, draw=black, inner sep=0pt, minimum size=3pt] (42) at (15.9,1.55) {};
\node[circle, thick, fill=white, draw=black, inner sep=0pt, minimum size=4pt] (43) at (14.35,1.05) {};
\node[circle, fill=black, draw=black, inner sep=0pt, minimum size=3pt] (44) at (15,1.05) {};
\node[circle, thick, fill=white, draw=black, inner sep=0pt, minimum size=4pt] (45) at (15.65,1.05) {};
\node[circle, fill=black, draw=black, inner sep=0pt, minimum size=3pt] (46) at (14.1,0.5) {};
\node[circle, fill=black, draw=black, inner sep=0pt, minimum size=3pt] (47) at (15.9,0.5) {};
\draw[thick] (41) -- (43); \draw[thick] (42) -- (45); \draw[thick] (43) -- (45);
\draw[thick] (43) -- (46); \draw[thick] (45) -- (47);

% red
\draw[red,thick, postaction={decorate, decoration = {markings, mark = between positions 0.15 and 0.95 step 0.4 with {\arrow{Stealth}}}}] (1,2) to[bend left=60] (1,0);
\draw[red,thick, postaction={decorate, decoration = {markings, mark = between positions 0.15 and 0.95 step 0.4 with {\arrow{Stealth}}}}] (1.7,0) to[bend left=60] (1.7,2);
\draw[red,thick, postaction={decorate, decoration = {markings, mark = between positions 0.25 and 0.8 step 0.55 with {\arrow{Stealth}}}}] (2.2,1.4) to[bend right=30] (0.5,1.4);
\draw[red,thick, postaction={decorate, decoration = {markings, mark = between positions 0.25 and 0.8 step 0.55 with {\arrow{Stealth}}}}] (0.5,0.65) to[bend right=30] (2.2,0.65);
\draw[>={Stealth[length=8pt]}, black, <->] (3,1) -- (4.2,1);
\draw[red,thick, postaction={decorate, decoration = {markings, mark = between positions 0.15 and 0.95 step 0.4 with {\arrow{Stealth}}}}] (5,0.65) to[bend left=60] (7,0.65);
\draw[red,thick, postaction={decorate, decoration = {markings, mark = between positions 0.15 and 0.95 step 0.4 with {\arrow{Stealth}}}}] (7,1.35) to[bend left=60] (5,1.35);
\draw[red,thick, postaction={decorate, decoration = {markings, mark = between positions 0.25 and 0.8 step 0.55 with {\arrow{Stealth}}}}] (5.6,1.9) to[bend right=30] (5.6,0.1);
\draw[red,thick, postaction={decorate, decoration = {markings, mark = between positions 0.25 and 0.8 step 0.55 with {\arrow{Stealth}}}}] (6.4,0.1) to[bend right=30] (6.4,1.9);

\draw[red,thick, postaction={decorate, decoration = {markings, mark = between positions 0.15 and 0.95 step 0.4 with {\arrow{Stealth}}}}] (10,0) to[bend right=60] (10,2);
\draw[red,thick, postaction={decorate, decoration = {markings, mark = between positions 0.15 and 0.95 step 0.4 with {\arrow{Stealth}}}}] (10.7,2) to[bend right=60] (10.7,0);
\draw[red,thick, postaction={decorate, decoration = {markings, mark = between positions 0.25 and 0.8 step 0.55 with {\arrow{Stealth}}}}] (9.5,1.4) to[bend left=30] (11.2,1.4);
\draw[red,thick, postaction={decorate, decoration = {markings, mark = between positions 0.25 and 0.8 step 0.55 with {\arrow{Stealth}}}}] (11.2,0.65) to[bend left=30] (9.5,0.65);
\draw[>={Stealth[length=8pt]}, black, <->] (12,1) -- (13.2,1);
\draw[red,thick, postaction={decorate, decoration = {markings, mark = between positions 0.15 and 0.95 step 0.4 with {\arrow{Stealth}}}}] (16,0.7) to[bend right=60] (14,0.7);
\draw[red,thick, postaction={decorate, decoration = {markings, mark = between positions 0.15 and 0.95 step 0.4 with {\arrow{Stealth}}}}] (14,1.4) to[bend right=60] (16,1.4);
\draw[red,thick, postaction={decorate, decoration = {markings, mark = between positions 0.25 and 0.8 step 0.55 with {\arrow{Stealth}}}}] (14.6,0.1) to[bend left=30] (14.6,1.9);
\draw[red,thick, postaction={decorate, decoration = {markings, mark = between positions 0.25 and 0.8 step 0.55 with {\arrow{Stealth}}}}] (15.4,1.9) to[bend left=30] (15.4,0.1);

% black
\node[circle, fill=black, draw=black, inner sep=0pt, minimum size=3pt] (1) at (1.35,2) {};
\node[circle, thick, fill=none, draw=black, inner sep=0pt, minimum size=4pt] (2) at (1.35,1.62) {};
\node[circle, fill=black, draw=black, inner sep=0pt, minimum size=0pt] (3) at (0.4,1) {};
\node[circle, fill=black, draw=black, inner sep=0pt, minimum size=3pt] (4) at (0.8,1) {};
\node[circle, fill=black, draw=black, inner sep=0pt, minimum size=3pt] (5) at (1.9,1) {};
\node[circle, fill=black, draw=black, inner sep=0pt, minimum size=0pt] (6) at (2.3,1) {};
\node[circle, thick, fill=none, draw=black, inner sep=0pt, minimum size=4pt] (7) at (1.35,0.4) {};
\node[circle, fill=black, draw=black, inner sep=0pt, minimum size=3pt] (8) at (1.35,0) {};
\draw[thick] (1) -- (2); \draw[thick] (3) -- (4); \draw[thick] (5) -- (6); \draw[thick] (7) -- (8);
\draw[thick] (2) -- (4); \draw[thick] (2) -- (5); \draw[thick] (4) -- (7); \draw[thick] (5) -- (7);
\node[circle, fill=black, draw=black, inner sep=0pt, minimum size=0pt] (21) at (6,2) {};
\node[circle, fill=black, draw=black, inner sep=0pt, minimum size=3pt] (22) at (6,1.6) {};
\node[circle, fill=black, draw=black, inner sep=0pt, minimum size=3pt] (23) at (5,1) {};
\node[circle, thick, fill=none, draw=black, inner sep=0pt, minimum size=4pt] (24) at (5.3,1) {};
\node[circle, thick, fill=none, draw=black, inner sep=0pt, minimum size=4pt] (25) at (6.7,1) {};
\node[circle, fill=black, draw=black, inner sep=0pt, minimum size=3pt] (26) at (7,1) {};
\node[circle, fill=black, draw=black, inner sep=0pt, minimum size=3pt] (27) at (6,0.4) {};
\node[circle, fill=black, draw=black, inner sep=0pt, minimum size=0pt] (28) at (6,0) {};
\draw[thick] (21) -- (22); \draw[thick] (23) -- (24); \draw[thick] (25) -- (26); \draw[thick] (27) -- (28);
\draw[thick] (22) -- (24); \draw[thick] (22) -- (25); \draw[thick] (24) -- (27); \draw[thick] (25) -- (27);
\node[circle, fill=black, draw=black, inner sep=0pt, minimum size=3pt] (31) at (9.9,1.8) {};
\node[circle, fill=black, draw=black, inner sep=0pt, minimum size=3pt] (32) at (10.8,1.8) {};
\node[circle, thick, fill=none, draw=black, inner sep=0pt, minimum size=4pt] (33) at (10.35,1.63) {};
\node[circle, fill=black, draw=black, inner sep=0pt, minimum size=3pt] (34) at (10.35,1) {};
\node[circle, thick, fill=none, draw=black, inner sep=0pt, minimum size=4pt] (35) at (10.35,0.4) {};
\node[circle, fill=black, draw=black, inner sep=0pt, minimum size=3pt] (36) at (9.9,0.2) {};
\node[circle, fill=black, draw=black, inner sep=0pt, minimum size=3pt] (37) at (10.8,0.2) {};
\draw[thick] (31) -- (33); \draw[thick] (32) -- (33); \draw[thick] (33) -- (35);
\draw[thick] (35) -- (36); \draw[thick] (35) -- (37);
\node[circle, fill=black, draw=black, inner sep=0pt, minimum size=3pt] (41) at (14.1,1.55) {};
\node[circle, fill=black, draw=black, inner sep=0pt, minimum size=3pt] (42) at (15.9,1.55) {};
\node[circle, thick, fill=none, draw=black, inner sep=0pt, minimum size=4pt] (43) at (14.35,1.05) {};
\node[circle, fill=black, draw=black, inner sep=0pt, minimum size=3pt] (44) at (15,1.05) {};
\node[circle, thick, fill=none, draw=black, inner sep=0pt, minimum size=4pt] (45) at (15.65,1.05) {};
\node[circle, fill=black, draw=black, inner sep=0pt, minimum size=3pt] (46) at (14.1,0.5) {};
\node[circle, fill=black, draw=black, inner sep=0pt, minimum size=3pt] (47) at (15.9,0.5) {};
\draw[thick] (41) -- (43); \draw[thick] (42) -- (45); \draw[thick] (43) -- (45);
\draw[thick] (43) -- (46); \draw[thick] (45) -- (47);
\end{tikzpicture}
\\
(a) \hspace{2.6in}(b) 
\vspace{-.2in}
\end{center}

\caption{Depending on the orientations of the strands involved, 
a swivel move in a triple diagram 
may correspond to (a) a normal spider move
or (b) a normal flip move in the associated normal plabic graph.}
\label{fig:flipsplabictriple}
\end{figure}

%Our next goal is to strengthen \cref{thm:moves-moves} 
%by replacing the equivalence of normal plabic graphs 
%under the urban renewal and normal flip moves   
%(see Definitions~\ref{def:urban-normal}--\ref{def:flip-normal}) 
%by a slightly modified version of the ordinary move-equivalence utilizing the moves (M1)--(M3). 
%This will require some preparation. 

Our next goal is to
give an analogue of  \cref{def:GX} for arbitrary leafless plabic graphs, not
necessarily normal.

\begin{definition}
\label{lem:biptri2}
Let $G$ be a (leafless) plabic graph, not necessarily normal. 
The~\emph{normal form} associated to~$G$, is a (non-unique) plabic graph
$N(G)$ which is
move-equivalent to $G$, constructed as follows
	(see \cref{fig:make-Bip}).
%(To be precise, we define $N(G)$ up to move-equivalence.) 
%Let $G$ be a %reduced 
%	leafless 
%plabic graph that does not contain a white lollipop.
%Then $G$ is move-equivalent to a (leafless) normal plabic graph, which we
%can construct as follows:
\begin{itemize}[leftmargin=.2in]
	\item Use (M2) to omit degree $2$ white vertices.
	\item Use (M3) to contract all edges with both endpoints black.
		(If this results in a loop, then the graph is not reduced.)
	\item Use (M3) to turn white vertices of degree $\geq 4$
	into subgraphs whose vertices are white and trivalent.
	\item Use (M2) to add degree $2$ black vertices so that
		that the resulting graph is bipartite and each boundary vertex  is adjacent to 
		a black vertex.
\end{itemize}
Note that if $G$ is reduced and has no white lollipops, then 
$N(G)$ is normal.  
%Otherwise, $N(G)$ is normal up to 
%the removal of lollipops.

\begin{figure}[h]
\begin{center}
\vspace{-.1in}
\begin{tikzpicture}[scale=1.05]
% circle outlines           
\filldraw[fill=white, thick] (0,0) circle (0.9cm);
\filldraw[fill=white, thick] (2.5,0) circle (0.9cm);
\filldraw[fill=white, thick] (5,0) circle (0.9cm);
\filldraw[fill=white, thick] (7.5,0) circle (0.9cm);
\filldraw[fill=white, thick] (10,0) circle (0.9cm);

% circle 1
\node[circle, fill=black, draw=black, inner sep=0pt, minimum size=2.25pt] at (0.85,0.25) (n1) {};
\node[circle, fill=black, draw=black, inner sep=0pt, minimum size=2.25pt] at (-0.85,0.25) (n2) {};
\draw[thick] (n1) -- (n2);
\node[circle, fill=white, draw=black, inner sep=0pt, minimum size=3.25pt, thick] at (0,0.25) (n3) {};
\node[circle, fill=black, draw=black, inner sep=0pt, minimum size=3.25pt, thick] at (-0.4,0.25) (n4) {};
\node[circle, fill=black, draw=black, inner sep=0pt, minimum size=3.25pt, thick] at (0.4,0.25) (n5) {};

\node[circle, fill=black, draw=black, inner sep=0pt, minimum size=2.25pt] at (0.85,-0.25) (n6) {};
\node[circle, fill=black, draw=black, inner sep=0pt, minimum size=2.25pt] at (-0.85,-0.25) (n7) {};
\draw[thick] (n6) -- (n7);
\node[circle, fill=black, draw=black, inner sep=0pt, minimum size=3.25pt, thick] at (0,-0.25) (n8) {};
\node[circle, fill=white, draw=black, inner sep=0pt, minimum size=3.25pt, thick] at (-0.4,-0.25) (n9) {};
\node[circle, fill=white, draw=black, inner sep=0pt, minimum size=3.25pt, thick] at (0.4,-0.25) (n10) {};

\node[circle, fill=black, draw=black, inner sep=0pt, minimum size=2.25pt, thick] at (-0.55,-0.7) (n11) {};
\draw[thick] (n9) -- (n11);

\node[circle, fill=black, draw=black, inner sep=0pt, minimum size=2.25pt, thick] at (0,-0.9) (n12) {};
\draw[thick] (n8) -- (n12);

\draw[thick] (n4) -- (n9);
\draw[thick] (n5) -- (n10);

%%%%%%%%%%%%%%%%%%%%%%%%%%%%%%%%%%%%%%%%%%%%%%%%%%%%%%%%%%%%%%%%
% circle 2
\node[circle, fill=black, draw=black, inner sep=0pt, minimum size=2.25pt] at (3.35,0.25) (n1) {};
\node[circle, fill=black, draw=black, inner sep=0pt, minimum size=2.25pt] at (1.65,0.25) (n2) {};
\draw[thick] (n1) -- (n2);
\node[circle, fill=black, draw=black, inner sep=0pt, minimum size=3.25pt, thick] at (2.1,0.25) (n4) {};
\node[circle, fill=black, draw=black, inner sep=0pt, minimum size=3.25pt, thick] at (2.9,0.25) (n5) {};

\node[circle, fill=black, draw=black, inner sep=0pt, minimum size=2.25pt] at (3.35,-0.25) (n6) {};
\node[circle, fill=black, draw=black, inner sep=0pt, minimum size=2.25pt] at (1.65,-0.25) (n7) {};
\draw[thick] (n6) -- (n7);
\node[circle, fill=black, draw=black, inner sep=0pt, minimum size=3.25pt, thick] at (2.5,-0.25) (n8) {};
\node[circle, fill=white, draw=black, inner sep=0pt, minimum size=3.25pt, thick] at (2.1,-0.25) (n9) {};
\node[circle, fill=white, draw=black, inner sep=0pt, minimum size=3.25pt, thick] at (2.9,-0.25) (n10) {};

\node[circle, fill=black, draw=black, inner sep=0pt, minimum size=2.25pt, thick] at (1.95,-0.7) (n11) {};
\draw[thick] (n9) -- (n11);

\node[circle, fill=black, draw=black, inner sep=0pt, minimum size=2.25pt, thick] at (2.5,-0.9) (n12) {};
\draw[thick] (n8) -- (n12);

\draw[thick] (n4) -- (n9);
\draw[thick] (n5) -- (n10);

%%%%%%%%%%%%%%%%%%%%%%%%%%%%%%%%%%%%%%%%%%%%%%%%%%%%%%%%%%%%%%%%
% circle 3
\node[circle, fill=black, draw=black, inner sep=0pt, minimum size=2.25pt] at (5.85,0.25) (n1) {};
\node[circle, fill=black, draw=black, inner sep=0pt, minimum size=2.25pt] at (4.15,0.25) (n2) {};
\draw[thick] (n1) -- (n2);
\node[circle, fill=black, draw=black, inner sep=0pt, minimum size=3.25pt, thick] at (5,0.25) (n4) {};
% \node[circle, fill=black, draw=black, inner sep=0pt, minimum size=3.25pt, thick] at (2.9,0.25) (n5) {};

\node[circle, fill=black, draw=black, inner sep=0pt, minimum size=2.25pt] at (5.85,-0.25) (n6) {};
\node[circle, fill=black, draw=black, inner sep=0pt, minimum size=2.25pt] at (4.15,-0.25) (n7) {};
\draw[thick] (n6) -- (n7);
\node[circle, fill=black, draw=black, inner sep=0pt, minimum size=3.25pt, thick] at (5,-0.25) (n8) {};
\node[circle, fill=white, draw=black, inner sep=0pt, minimum size=3.25pt, thick] at (4.6,-0.25) (n9) {};
\node[circle, fill=white, draw=black, inner sep=0pt, minimum size=3.25pt, thick] at (5.4,-0.25) (n10) {};

\node[circle, fill=black, draw=black, inner sep=0pt, minimum size=2.25pt, thick] at (4.45,-0.7) (n11) {};
\draw[thick] (n9) -- (n11);

\node[circle, fill=black, draw=black, inner sep=0pt, minimum size=2.25pt, thick] at (5,-0.9) (n12) {};
\draw[thick] (n4) -- (n9);
\draw[thick] (n4) -- (n10);
\draw[thick] (n8) -- (n12);

%%%%%%%%%%%%%%%%%%%%%%%%%%%%%%%%%%%%%%%%%%%%%%%%%%%%%%%%%%%%%%%%
% circle 4
\node[circle, fill=black, draw=black, inner sep=0pt, minimum size=2.25pt] at (8.35,0.25) (n1) {};
\node[circle, fill=black, draw=black, inner sep=0pt, minimum size=2.25pt] at (6.65,0.25) (n2) {};
\draw[thick] (n1) -- (n2);
\node[circle, fill=black, draw=black, inner sep=0pt, minimum size=3.25pt, thick] at (7.5,0.25) (n4) {};
% \node[circle, fill=black, draw=black, inner sep=0pt, minimum size=3.25pt, thick] at (2.9,0.25) (n5) {};

\node[circle, fill=black, draw=black, inner sep=0pt, minimum size=2.25pt] at (8.35,-0.25) (n6) {};
\node[circle, fill=black, draw=black, inner sep=0pt, minimum size=2.25pt] at (6.65,-0.25) (n7) {};
\draw[thick] (n6) -- (n7);
\node[circle, fill=black, draw=black, inner sep=0pt, minimum size=3.25pt, thick] at (7.5,-0.25) (n8) {};
\node[circle, fill=white, draw=black, inner sep=0pt, minimum size=3.25pt, thick] at (7,-0.25) (n9) {};

\node[circle, fill=white, draw=black, inner sep=0pt, minimum size=3.25pt, thick] at (7.25,-0.25) (nn9) {};

\node[circle, fill=white, draw=black, inner sep=0pt, minimum size=3.25pt, thick] at (8,-0.25) (n10) {};

\node[circle, fill=black, draw=black, inner sep=0pt, minimum size=2.25pt, thick] at (6.95,-0.7) (n11) {};
\draw[thick] (n9) -- (n11);

\node[circle, fill=black, draw=black, inner sep=0pt, minimum size=2.25pt, thick] at (7.5,-0.9) (n12) {};
\draw[thick] (n4) -- (nn9);
\draw[thick] (n4) -- (n10);
\draw[thick] (n8) -- (n12);

%%%%%%%%%%%%%%%%%%%%%%%%%%%%%%%%%%%%%%%%%%%%%%%%%%%%%%%%%%%%%%%%
% circle 5
\node[circle, fill=black, draw=black, inner sep=0pt, minimum size=2.25pt] at (10.85,0.25) (n1) {};
\node[circle, fill=black, draw=black, inner sep=0pt, minimum size=2.25pt] at (9.15,0.25) (n2) {};
\draw[thick] (n1) -- (n2);
\node[circle, fill=black, draw=black, inner sep=0pt, minimum size=3.25pt, thick] at (10,0.25) (n4) {};
% \node[circle, fill=black, draw=black, inner sep=0pt, minimum size=3.25pt, thick] at (2.9,0.25) (n5) {};

\node[circle, fill=black, draw=black, inner sep=0pt, minimum size=2.25pt] at (10.85,-0.25) (n6) {};
\node[circle, fill=black, draw=black, inner sep=0pt, minimum size=2.25pt] at (9.15,-0.25) (n7) {};
\draw[thick] (n6) -- (n7);
\node[circle, fill=black, draw=black, inner sep=0pt, minimum size=3.25pt, thick] at (10,-0.25) (n8) {};
\node[circle, fill=white, draw=black, inner sep=0pt, minimum size=3.25pt, thick] at (9.45,-0.25) (n9) {};

\node[circle, fill=black, draw=black, inner sep=0pt, minimum size=3.25pt, thick] at (9.28,-0.25) {};
\node[circle, fill=black, draw=black, inner sep=0pt, minimum size=3.25pt, thick] at (9.62,-0.25) {};
\node[circle, fill=black, draw=black, inner sep=0pt, minimum size=3.25pt, thick] at (10.67,-0.25) {};

\node[circle, fill=white, draw=black, inner sep=0pt, minimum size=3.25pt, thick] at (9.8,-0.25) (nn9) {};

\node[circle, fill=white, draw=black, inner sep=0pt, minimum size=3.25pt, thick] at (10.5,-0.25) (n10) {};

\node[circle, fill=black, draw=black, inner sep=0pt, minimum size=3.25pt, thick] at (9.45,-0.48) {};

\node[circle, fill=black, draw=black, inner sep=0pt, minimum size=2.25pt, thick] at (9.45,-0.7) (n11) {};
\draw[thick] (n9) -- (n11);

\node[circle, fill=black, draw=black, inner sep=0pt, minimum size=2.25pt, thick] at (10,-0.9) (n12) {};
\draw[thick] (n4) -- (nn9);
\draw[thick] (n4) -- (n10);
\draw[thick] (n8) -- (n12);
\end{tikzpicture}
\vspace{-.25in}
\end{center}
 \caption{A reduced plabic graph $G$ (left) %and the corresponding bipartite graph~$\Bip(G)$.
        and a normal form $N(G)$ (right). 
}
\label{fig:make-Bip}
\end{figure}

%\begin{figure}
%\begin{center}
%\includegraphics[height=1.3in]{Fig7-30.ps}
%\vspace{-.2in}
%\end{center}
%	\caption{A reduced plabic graph $G$ (left) %and the corresponding bipartite graph~$\Bip(G)$.
%	and a normal form $N(G)$ (right). \LW{Figure to be made.}
%}
%\label{fig:make-Bip}
%\end{figure}

Now we define the \emph{generalized triple diagram}
$\X(G)=\X(N(G))$ associated to $G$
by applying \cref{def:GX} to the normal form $N(G)$,
with the following additional rule 
dealing with white lollipops:
\begin{itemize}[leftmargin=.2in]
\item 
at a  white lollipop in~$G$, make a U-turn: \setlength{\unitlength}{0.7pt}
\begin{picture}(35,16)(0,5)
\thicklines
\put(30,10){\circle{4}}
\put(0,10){\line(1,0){28}}
\red{
\put(0,16){\line(1,0){30}}
\put(0,4){\line(1,0){30}}
\put(19,4){\vector(1,0){1}}
\put(11,16){\vector(-1,0){1}}
\qbezier(30,16)(36,16)(36,10)
\qbezier(30,4)(36,4)(36,10)
}
\end{picture}
.
%see \cref{fig:plabic-low-degree-to-triple}.
\end{itemize}
\end{definition}

\iffalse
\begin{figure}[ht]
\vspace{-.15in}
\begin{center}
\setlength{\unitlength}{0.7pt}
\begin{picture}(35,16)(0,0)
\thicklines
\put(30,10){\circle{4}}
\put(0,10){\line(1,0){28}}
\red{
\put(0,16){\line(1,0){30}}
\put(0,4){\line(1,0){30}}
\put(19,4){\vector(1,0){1}}
\put(11,16){\vector(-1,0){1}}
\qbezier(30,16)(36,16)(36,10)
\qbezier(30,4)(36,4)(36,10)
}
\end{picture}
%\qquad\qquad
%\setlength{\unitlength}{0.7pt}
%\begin{picture}(60,16)(0,0)
%\thicklines
%\put(30,10){\circle{4}}
%\put(0,10){\line(1,0){28}}
%\put(60,10){\line(-1,0){28}}
%\red{
%\put(0,4){\line(1,0){60}}
%\put(0,16){\line(1,0){60}}
%\put(33,16){\vector(-1,0){1}}
%\put(27,4){\vector(1,0){1}}
%}
%\end{picture}
\vspace{-.25in}
\end{center}
	\caption{Constructing a triple diagram around a white lollipop.}
	%or a white vertex of degree~2 in a general plabic graph.}
\label{fig:plabic-low-degree-to-triple}
\end{figure}
\fi

\pagebreak[3]

\begin{remark}
Given $G$, there are many 
possible choices for $N(G)$,
since the trivalent tree replacing~$v$ is not unique.  Nevertheless,
all these trees are related to each other by flip moves, 
cf.\ \cref{fig:flipmove0}.  
Hence all triple diagrams constructed from them are move-equivalent to each other,
cf.\ \cref{fig:flipsplabictriple}(b)   
(remove the black vertex in the center).  
\end{remark}

The following statement is immediate from the definitions. 

\begin{lemma}
\label{lem:generalized-triple-diagram-connected}
Let $G$ be a plabic graph. If the union of the strands in~$\X(G)$ 
and the boundary~$\partial\mathbf{D}$ is connected, then $\X(G)$~is a triple diagram
in the sense of \cref{def:triple-diagram}. 
\end{lemma}

%\begin{remark}
%\label{rem:generalized-triple-diagram}
The connectedness condition in \cref{lem:generalized-triple-diagram-connected} does not hold in general. 
To be concrete, if~$G$ contains a cycle~$C$ all of whose vertices are black,
then the strands located inside~$C$ are disconnected from the rest of~$\X(G)$.    

\begin{lemma}
\label{lem:M1M2M3-to-triple}
Let $G$ and $G'$ be (leafless) plabic graphs such that $G\osim G'$. 
Then the corresponding (generalized) triple diagrams $\X(G)$ and~$\X(G')$ are move-equivalent 
(i.e., related to each other via swivel moves). 
\end{lemma}

We note that $\X(G)$ and $\X(G')$ are defined up to move-equivalence, so the statement
that they are move-equivalent to each other makes sense. 

%We first observe that the plabic graphs $G$ and $G'$ can be connected to each other via a sequence
%of (M1)--(M3) moves that do not go through graphs with white leaves---essentially because
%growing white leaves cannot create any additional opportunities for local moves.
%Although adding a white leaf may create a new configuration containing an ``alternating square''
%as in \cref{fig:M1}, an (M1) move would not be allowed at that location 
%due to the restriction illustrated in \cref{fig:square-move-tricky}.
%
%We then 
%\begin{proof}

\noindent\textbf{Proof.}
It is straightforward to 
verify, case by case, that each of the local moves (M1)--(M3$'$)
%with the exception of the move shown in \cref{fig:M3-white-leaf}, 
either leaves the associated (generalized) triple diagram intact
or applies a swivel move to it 
(more precisely, to any of the possible diagrams
obtained via the construction in 
	\cref{lem:biptri2}).
	%\cref{def:triple-non-normal}). 
To be specific:
\begin{itemize}[leftmargin=.2in]
\item 
a square move (M1) translates into a swivel move, see \cref{fig:flipsplabictriple}(a);  
\item
both the move (M2) 
and a black (de)contraction move (M3$'$) leave the triple diagram invariant (up to isotopy);
\item
a white (de)contraction move (M3$'$)
		%, other than the instance shown in \cref{fig:M3-white-leaf},
translates into a swivel move, see \cref{fig:flipsplabictriple}(b)
(remove the black vertex in the center).  
\qed
\end{itemize}
%\end{proof}
%The last statement relies on the absence of white leaves:  
%it actually fails in the special case of contracting/decontracting an edge connecting 
%a white leaf to a white vertex, see \cref{fig:M3-white-leaf}. 

%The ``normal'' local moves introduced in Definitions~\ref{def:urban-normal} and~\ref{def:flip-normal}
%can be applied to any (not necessarily normal) plabic graph~$G$ 
%that contains a local configuration shown in \cref{fig:urban2} or \cref{fig:flipmove}.  
%But if $G$ is normal, then the transformed graph will be normal as well.

\begin{corollary}
\label{cor:newmoves}
Let $G$ and $G'$ be normal plabic graphs. 
%and let $\X=\X(G)$ and $\X'=\X(G')$ be the corresponding triple diagrams. 
The following are equivalent:
\begin{enumerate}[leftmargin=.3in]
\item[{\rm(1)}]
$G\osim G'$; % are move-equivalent (in the sense of \cref{def:moves}); 
\item[{\rm(2)}]
$G$ and $G'$ are related via a sequence of normal spider moves and normal flip moves; 
\item[{\rm(3)}]
$\X(G)$ and $\X(G')$ are move-equivalent (in the sense of \cref{def:2-2-move-equivalence}). 
\end{enumerate}
\end{corollary}

\begin{proof}
The implication (2)$\Rightarrow$(1) is %an easy enhancement 
 \cref{lem:normal-moves-via-M1M2M3}. 
The equivalence (2)$\Leftrightarrow$(3) was established in \cref{thm:moves-moves}. 
The implication (1)$\Rightarrow$(3) was proved in \cref{lem:M1M2M3-to-triple}. 
\end{proof}

We note the similarity between \cref{cor:newmoves} and \cref{thm:newmoves1}. 

\newpage

\section{Minimal triple diagrams}
\label{sec:mintriple}

\begin{definition}
A triple diagram  is called \emph{minimal} if %it is connected and 
it has no more triple points than any other triple diagram with the same strand permutation.
\end{definition}
%(The condition of connectivity merely rules out small loops with no crossings
%disconnected from the rest of the diagram.)

We will show in Section~\ref{minred} that minimal triple diagrams 
are the natural counterparts of reduced normal plabic graphs.  

Much of this section is devoted to the proof of the following key result.

\begin{theorem}
\label{thm:domino-flip}
 Any two minimal triple diagrams with the same strand permutation
are %related to each other via a sequence of swivel moves.
move-equivalent to each other. 
\end{theorem}

\begin{lemma}
\label{lem:min-diag-move-equiv}
If a triple diagram $\X$ is  minimal, then so is 
every triple diagram move-equivalent to~$\X$. 
\end{lemma}

\begin{proof}
It is easy to see that a swivel move preserves both the number of triple points
and the strand permutation. The claim follows. 
\end{proof}

\iffalse
We begin by constructing, for each permutation $\pi$, a particular
triple diagram whose strand permutation is $\pi$.
We will call these triple diagrams \emph{standard}; 
we will see in \cref{cor:standardminimal}
that they are all minimal.
This will be helpful in the proof of \cref{thm:domino-flip}.
\fi

We next describe certain ``bad features'' (of a triple diagram) 
and show that they cannot occur in a minimal triple diagram. 
	
\begin{definition}
\label{def:monogon}
A strand in a triple diagram that intersects itself forms a \emph{monogon}.  
A pair of strands that intersect at two points $x$ and $y$ form either a \emph{parallel} 
or \emph{anti-parallel digon}, 
depending on whether their segments connecting $x$ and~$y$ run in the same or opposite direction, 
see \cref{fig:monogondigon}.
We use the term \emph{badgon} to refer to either a monogon or a parallel digon.
\end{definition}

\begin{figure}[ht]
\begin{center}
\vspace{-.25in}
\begin{tikzpicture}[scale=0.8]
\draw [thick, postaction={decorate, decoration = {markings, mark = between positions 0.3 and 0.5 step 0.5 with {\arrow[rotate=-9]{Stealth}}}}] (-0.7,-1) to[out=20,in=265] (0.5,0) to[out=92,in=355] (0,0.8) to[out=185,in=89] (-0.5,0) to[out=274,in=160] (0.7,-1);
\draw[thick, postaction={decorate, decoration = {markings, mark = between positions 0.58 and 0.8 step 0.5 with {\arrow{Stealth}}}}] (2.5,0.75) to[bend left=55] (2.5,-1);
\draw[thick, postaction={decorate, decoration = {markings, mark = between positions 0.58 and 0.8 step 0.5 with {\arrow{Stealth}}}}] (2.95,0.75) to[bend right=55] (2.95,-1);
%\node[circle, fill=black, draw=black, inner sep=0pt, minimum size=3pt] (37) at (2.73,0.53) {};
%\node[circle, fill=black, draw=black, inner sep=0pt, minimum size=3pt] (37) at (2.73,-0.77) {};
\node at (2.73,0.89) {\small $x$};
\node at (2.73,-1.15) {\small $y$};
%\draw[thick, postaction={decorate, decoration = {markings, mark = between positions 0.58 and 0.8 step 0.5 with {\arrow{Stealth}}}}] (5,0.75) to[bend left=55] (5,-1);
%		\draw[thick, postaction={decorate, decoration = {markings, mark = between positions 0.58 and 0.8 step 0.5 with {\arrow{Stealth}}}}] (5.45,0.75) to[bend right=55] (5.45,-1);
%\node[circle, fill=black, draw=black, inner sep=0pt, minimum size=3pt] (37) at (5.23,0.53) {};
\draw[thick, postaction={decorate, decoration = {markings, mark = between positions 0.56 and 0.8 step 0.5 with {\arrow{Stealth}}}}] (5,-1) to[bend right=55] (5,0.75);
\draw[thick, postaction={decorate, decoration = {markings, mark = between positions 0.58 and 0.8 step 0.5 with {\arrow{Stealth}}}}] (5.45,0.75) to[bend right=55] (5.45,-1);
%\node[circle, fill=black, draw=black, inner sep=0pt, minimum size=3pt] (37) at (5.23,0.53) {};
\node at (5.23,0.89) {\small $x$};
\node at (5.23,-1.15) {\small $y$};
\end{tikzpicture}
\vspace{-.25in}
\end{center}
\caption{A monogon, a parallel digon, and an anti-parallel digon. 
The actual picture will contain %(potentially many) 
additional strands and intersections.}
\label{fig:monogondigon}
\end{figure}
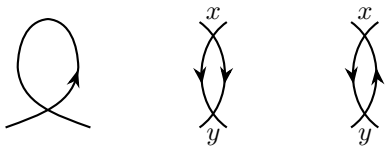

\begin{lemma}
\label{lem:no-closed-strands}
A triple diagram without badgons has no closed strands.
\end{lemma}

\begin{proof}
Let $\X$ be a triple diagram without badgons.  Since $\X$ does not 
contain monogons, no strand of $\X$ can intersect itself.
Suppose that $\X$ 
contains a closed strand $S$.
%Suppose that $S$ is a (non-self-intersecting) closed strand in such 
%	a triple diagram~$\X$.
%Since $S$ does not contain monogons, $S$~does not intersect itself.
Let $T$ be another strand of $\X$ intersecting~$S$ at points~$x$ and~$y$; 
such~$T$ exists since $\X$ must be connected to the boundary~$\partial\mathbf{D}$. 
%Suppose $T$ runs from a point $x\in S$ to another point $y\in S$.
Then the segment of~$T$ between $x$ and~$y$ 
together with one of the segments of~$S$ connecting $x$ and~$y$ form a parallel digon, which is a contradiction.
\end{proof}

%\pagebreak[3]

\begin{lemma}
\label{lem:nomonogondigon}
A minimal triple diagram does not contain badgons.
Therefore (cf.\ \cref{lem:no-closed-strands}) it does not contain closed strands.
	%, and neither can any 
%diagram move-equivalent to $\X$.
%no strand of $\X$ (or any diagram move-equivalent to it) 
%can form a monogon. 
%And no pair of strands of $\X$ (or any diagram move-equivalent to it)
%can form a parallel digon.
\end{lemma}

\begin{proof}
Let $\X$ be a triple diagram containing a monogon, 
i.e., a strand~$S$ with a self-intersection at a triple point~$v$.  
Construct the triple diagram $\X'$ by deforming $\X$ around~$v$
so that $S$ ``spins off'' a closed strand while the triple point disappears, see \cref{fig:fewer}. 
(If the spun-off portion is disconnected from the rest of~$\X$, then remove it altogether.)
The triple diagram $\X'$ has~the same strand permutation as~$\X$
but fewer triple points; thus $\X$ is not minimal.

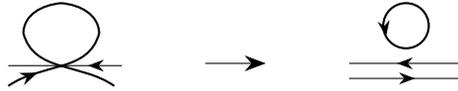
\begin{figure}[ht]
\begin{center}
\vspace{-.1in}
\begin{tikzpicture}
\draw[>={Stealth[length=8pt]}, ->] (-0.4,0) -- (0.4,0);
\draw[postaction={decorate, decoration = {markings, mark = between positions 0.3 and 0.5 step 0.5 with {\arrow{Stealth[length=6pt]}}}}] (-1.5,-0.03) -- (-3,-0.03);
\draw [thick, postaction={decorate, decoration = {markings, mark = between positions 0.1 and 0.9 step 0.9 with {\arrow[rotate=5]{Stealth}}}}] (-3,-0.3) to[out=35,in=255] (-1.8,0.4) to[out=92,in=355] (-2.3,0.8) to[out=185,in=89] (-2.8,0.4) to[out=285,in=145] (-1.6,-0.3);
\draw[postaction={decorate, decoration = {markings, mark = between positions 0.58 and 0.6 step 0.5 with {\arrow{Stealth[length=6pt]}}}}] (1.5,-0.2) -- (3,-0.2);
\draw[postaction={decorate, decoration = {markings, mark = between positions 0.58 and 0.6 step 0.5 with {\arrow{Stealth[length=6pt]}}}}] (3,0) -- (1.5,0);
\draw[thick, decoration={markings, mark=at position 0.58 with {\arrow[rotate=-20]{Stealth[length=6pt]}}},postaction={decorate}] (2.25,0.5) circle (0.3);
\end{tikzpicture}
\vspace{-.15in}
\end{center}
\caption{In the presence of a monogon, we can reduce the number of triple points 
while keeping the same strand permutation.
The triple diagram may contain additional strands intersecting the monogon,
as well as additional points of self-intersection.}
\label{fig:fewer}
\end{figure}

\noindent
Now suppose that $\X$ does not contain monogons but does 
contain two strands $S$ and $T$ that form a parallel digon. 
Say, $S$ and $T$ contain segments $\overline S$ and~$\overline T$ 
that run from a triple point~$x$ to a triple point~$y$. 
Let $U$ (resp.,~$V$) be the third strand passing through~$x$ (resp.,~$y$). 
We then deform $\X$ around both $x$ and~$y$ by 
smoothing each of the two triple points: the strands $U$ and $V$ 
continue to go straight through, whereas the endpoints of~$\overline S$ (resp.,~$\overline T$)
get connected to~$T$ (resp.,~$S$). 
Thus, the strands $S$ and $T$ swap their segments $\overline S$ and $\overline T$
with each other (with appropriate smoothings), the overall connectivity (i.e., the strand permutation)
is preserved, and the triple points at $x$ and~$y$ disappear. 
(If the diagram becomes disconnected from~$\partial\mathbf{D}$, then remove the disconnected portion.) 
We then conclude that $\X$ was not minimal. 
\end{proof}

\begin{definition}
\label{def:boundary-parallel}
Let $S$ be an arc in a triple diagram; its endpoints $s$ and~$t$ lie on the boundary 
of the ambient disk~$\mathbf{D}$. 
We call $S$ \emph{boundary-parallel}
if it runs along a segment~$I$ of the boundary~$\partial\mathbf{D}$ between $s$ and~$t$
(in either direction), 
so that every other strand with an endpoint inside~$I$  runs
directly to or from~$S$, without any triple crossings in between. 
See \cref{fig:standard}.
\end{definition}

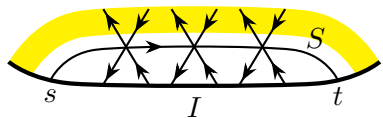
\begin{figure}[ht]
\begin{center}
\vspace{-.2in}
\begin{tikzpicture}

\draw[line width=10pt,yellow] plot [smooth] coordinates {(-2.3,-0.3) (-1.5,0.3) (1.5,0.3) (2.3,-0.3)};

\draw[thick, postaction={decorate, decoration = {markings, mark = between positions 0.3 and 1 step 0.65 with {\arrow{Stealth}}}}] (-0.6,-0.5) -- (-1.2,0.5);
\draw[thick, postaction={decorate, decoration = {markings, mark = between positions 0.3 and 1 step 0.65 with {\arrow{Stealth}}}}] (-0.6,0.5) -- (-1.2,-0.5);
\draw[thick, postaction={decorate, decoration = {markings, mark = between positions 0.3 and 1 step 0.65 with {\arrow{Stealth}}}}] (0.3,0.5) -- (-0.3,-0.5);
\draw[thick, postaction={decorate, decoration = {markings, mark = between positions 0.3 and 1 step 0.65 with {\arrow{Stealth}}}}] (0.3,-0.5) -- (-0.3,0.5);
\draw[thick, postaction={decorate, decoration = {markings, mark = between positions 0.3 and 1 step 0.65 with {\arrow{Stealth}}}}] (1.2,0.5) -- (0.6,-0.5);
\draw[thick, postaction={decorate, decoration = {markings, mark = between positions 0.3 and 1 step 0.65 with {\arrow{Stealth}}}}] (1.2,-0.5) -- (0.6,0.5);

\draw[thick, postaction={decorate, decoration = {markings, mark = between positions 0.4 and 0.5 step 0.55 with {\arrow{Stealth}}}}] plot [smooth] coordinates {(-1.9,-0.45) (-1.3,-0.04) (1.3,-0.04) (1.9,-0.45)};
%{ (1.9,-0.45) (1.3,-0.04) (-1.3,-0.04) (-1.9,-0.45)};

\draw[ultra thick] plot [smooth] coordinates {(-2.45, -0.2) (-1.6,-0.49) (1.6,-0.49) (2.45, -0.2)};

\node at (1.6,0.12) {$S$};
\node at (0,-0.8) {$I$};
\node at (-1.9,-0.65) {$s$};
\node at (1.9,-0.65) {$t$};

\end{tikzpicture}
\vspace{-.35in}
\end{center}
\caption{A boundary-parallel strand $S$ in a triple diagram.}
\label{fig:standard}
\end{figure}

We next describe a particular way to construct, for any given permutation~$\pi$, 
a triple diagram whose strand permutation is~$\pi$.

\begin{definition}
\label{def:standard-triple}
Let $\pi$ be a permutation of $b$ letters $1,\dots,b$. 
A~triple diagram in the disk~$\mathbf{D}$ is called \emph{standard} (for~$\pi$) 
if it can be constructed using the following recursive process. 
(The process involves some choices, so a standard diagram for~$\pi$ is not unique.) 

\pagebreak[3]

We place $b$ boundary vertices on the boundary $\partial\mathbf{D}$
and label them $1,\dots,b$ clockwise. 
Next to each boundary vertex~$v$, we mark two endpoints of the future strands:
a source endpoint that precedes $v$ in the clockwise order
and a target endpoint that follows~$v$ in this order. 
We know which source is to be matched to which target
by the strand permutation~$\pi$. 
The source and target of a given strand divides the circle~$\partial\mathbf{D}$
into two intervals. 
Let us partially order these $2b$ intervals by inclusion and select a \emph{minimal
interval~$I$} with respect to this partial order.

We start constructing the triple diagram by running a boundary-parallel strand~$S$
along the interval~$I$, introducing a triple crossing for each pair of strands that
need to terminate in the interior of~$I$, as shown in \cref{fig:standard}.
There will always be an even number (possibly zero) of strands to cross over,
so the construction will proceed without a hitch.

Let $\mathbf{D'}$ be the disk obtained from~$\mathbf{D}$ by 
removing the region between the boundary  segment~$I$ and the strand~$S$ 
together with a small neighborhood of~$S$; so $\mathbf{D'}$
is the shaded region in \cref{fig:standard}.
We accordingly remove $S$ and its endpoints 
from the original pairing of the in- and out-endpoints, and swap each pair that $S$ crossed over.
This yields $2(b-1)$ endpoints on the boundary of~$\mathbf{D'}$;
note that the in- and out-endpoints alternate, as before. 
We then determine the new pairing of these endpoints 
(thus, a new strand permutation, after an appropriate renumbering) 
and recursively continue the process in~$\mathbf{D'}$
	until the desired (standard) triple diagram is constructed.  
	See \cref{fig:standardtriple}.
\end{definition}
	
%\begin{figure}[h]
%\includegraphics[height=1.5in]{StandardTriple.ps}
%\caption{\LW{Figure to make} Constructing a standard triple diagram associated to the permutation
%$\pi=(4,6,5,3,1,2)$.  The figure shows the stages in the construction after:
%adding boundary-parallel strands $4\to 3$ and $3\to 5$; then $1\to 4$ and $6\to 2$; then
%$2\to 6$ and $5 \to 1$.\label{fig:standardtriple}}
%\end{figure}

\tikzset{
mid arrow/.style={postaction={decorate,decoration={
        markings,
        mark=at position .5 with {\arrow[#1]{stealth}}
      }}},
}

\begin{figure}[h]
\centering
\begin{tikzpicture}[scale = 0.62]
% Define common styles
\tikzset{vertex/.style={circle,draw,minimum size=0.6cm}, arrow/.style={-Stealth}}

% Loop to create 4 diagrams
\def \i {1}
    \begin{scope}[xshift=5*\i cm]
        % Draw circle
        \draw[thick](0,0) circle (2cm);
        % Place numbers around the circle
        \foreach \n in {1,...,6} {
            \node at ({210-60*\n}:2.3cm) {\n};
            % Two inward arrows per point
            \draw[arrow, thick] ({60*(\n-1)-83}:2cm) -- ({60*(\n-1)-83}:1.65cm);
            \draw[thick]({60*(\n-1)-83}:1.75cm) -- ({60*(\n-1)-83}:1.55cm);
            \draw[arrow, thick] ({60*(\n-1)-97}:1.55cm) -- ({60*(\n-1)-97}:1.88cm);
            \draw[thick] ({60*(\n-1)-97}:1.85cm) -- ({60*(\n-1)-97}:2cm);
        }
    \end{scope}

\def \i {2}
    \begin{scope}[xshift=5*\i cm]
        % Draw circle
        \draw[thick](0,0) circle (2cm);
        % Place numbers around the circle
        \foreach \n in {1,...,6} {
            \node at ({210-60*\n}:2.3cm) {\n};

            \node (10*\n+1) at ({217-60*\n}:2cm){};
            \node (10*n\n+2) at ({203-60*\n}:2cm){};
        }

        \foreach \n in {4,5,6}{
        \draw[arrow, thick] ({60*(\n-1)-83}:2cm) -- ({60*(\n-1)-83}:1.65cm);
        \draw[thick]({60*(\n-1)-83}:1.75cm) -- ({60*(\n-1)-83}:1.55cm);
        \draw[arrow, thick] ({60*(\n-1)-97}:1.55cm) -- ({60*(\n-1)-97}:1.88cm);
        \draw[thick] ({60*(\n-1)-97}:1.85cm) -- ({60*(\n-1)-97}:2cm);
        };

        \draw[thick]plot [smooth, tension = 1.5] coordinates {({217-60*4}:2cm) ({210-60*3.5}:1.5cm) ({203-60*3}:2cm)};

        \draw[arrow, thick] ({210-60*3.41}:1.518cm) -- ({210-60*3.4}:1.52cm);

        \draw[arrow, thick] ({217-60*3.015}:1.8cm) -- ({217-60*3.053}:1.65cm); 

        \draw[arrow, thick] ({203-60*4.965}:1.65cm) --({203-60*5}:1.8cm) ; 

        \draw[arrow, thick] ({217-60*4.91}:1.82cm)  --({217-60*4.85}:1.7cm); 

        \draw[arrow, thick] ({203-60*4.15}:1.65cm) --({203-60*4.08}:1.8cm) ;

        \draw[thick]plot [smooth, tension = 0.7] coordinates {({217-60*3}:2cm)  ({203-60*(2.9)}:1.5cm) ({203-60*4}:1.3cm) ({203-60*4.9}:1.5cm)({203-60*5}:2cm)};

        \draw[thick]plot [smooth, tension = 0.7] coordinates {({217-60*5}:2cm) ({203-60*4.5}:1.5cm) ({203-60*4.3}:1cm)};

        \draw[thick]plot [smooth, tension = 0.7] coordinates {({203-60*4}:2cm) ({203-60*4.43}:1.3cm) ({203-60*4.7}:1.1cm)};

        % \draw[arrow, thick] plot [smooth, tension = 1.5] coordinates {({217-60*4}:2cm) ({210-60*3.75}:1.55cm)  ({210-60*3.5}:1.45cm)};

        % \draw[thick]plot [smooth, tension = 1.5] coordinates {({210-60*3.5}:1.45cm) ({210-60*3.25}:1.55cm)  ({203-60*3}:2cm)};

    \end{scope}

\def \i {3}
    \begin{scope}[xshift=5*\i cm]
        % Draw circle
        \draw[thick](0,0) circle (2cm);
        % Place numbers around the circle
        \foreach \n in {1,...,6} {
            \node at ({210-60*\n}:2.3cm) {\n};
        }
        \draw[thick]plot [smooth, tension = 1.5] coordinates {({217-60*4}:2cm) ({210-60*3.5}:1.5cm) ({203-60*3}:2cm)};

        \draw[thick]plot [smooth, tension = 0.7] coordinates {({217-60*3}:2cm)  ({203-60*(2.9)}:1.5cm) ({203-60*4}:1.3cm) ({203-60*4.9}:1.5cm)({203-60*5}:2cm)};

        \draw[thick]plot [smooth, tension = 0.7] coordinates {({217-60*5}:2cm) ({203-60*4.5}:1.5cm) ({203-60*4.3}:1cm)};

        \draw[thick]plot [smooth, tension = 0.7] coordinates {({203-60*4}:2cm) ({203-60*4.43}:1.3cm) ({203-60*5.5}:1.3cm)  ({210-60*0.8}:1.6cm) ({217-60*1}:2cm)};

        \draw[thick]plot [smooth, tension = 0.7] coordinates {({203-60*2}:2cm) ({210-60*2.05}:1.5cm) ({210-60*1.3}:1.2cm) ({210-60*5.95}:1.5cm)({217-60*6}:2cm)};

        \draw[thick]plot [smooth, tension = 0.7] coordinates {({210-60*6.15}:1cm) ({203-60*6}:2cm)};

        \draw[thick]plot [smooth, tension = 0.7] coordinates {({210-60*6.15}:1cm) ({203-60*6}:2cm)};

        \draw[thick]plot [smooth, tension = 0.7] coordinates {({210-60*1.9}:1.1cm) ({210-60*1.5}:1.3cm)({210-60*1.2}:1.6cm)  ({203-60*1}:2cm)};

        \draw[thick]plot [smooth, tension = 1] coordinates {({217-60*2}:2cm) ({210-60*1.73}:1.7cm) ({210-60*1.27}:1cm)};

        \draw[arrow, thick] ({210-60*3.41}:1.518cm) -- ({210-60*3.4}:1.52cm);

        \draw[arrow, thick] ({217-60*3.015}:1.8cm) -- ({217-60*3.053}:1.65cm); 

        \draw[arrow, thick] ({203-60*4.965}:1.65cm) --({203-60*5}:1.8cm) ; 

        \draw[arrow, thick] ({217-60*4.91}:1.82cm)  --({217-60*4.85}:1.7cm); 

        \draw[arrow, thick] ({203-60*4.15}:1.65cm) --({203-60*4.08}:1.8cm) ; 

        \draw[arrow, thick] ({203-60*1.985}:1.65cm) --({203-60*2.01}:1.8cm) ; 

        \draw[arrow, thick] ({217-60*1.895}:1.82cm)  --({217-60*1.85}:1.71cm); 

        \draw[arrow, thick] ({203-60*1.07}:1.64cm) --({203-60*1.02}:1.8cm) ; 

        \draw[arrow, thick] ({217-60*.99}:1.78cm)  --({217-60*0.93}:1.6cm); 

        \draw[arrow, thick] ({203-60*6.01}:1.64cm) --({203-60*6.008}:1.8cm) ; 

        \draw[arrow, thick] ({217-60*5.99}:1.78cm)  --({217-60*6.035}:1.55cm); 

        \draw[arrow, thick] ({210-60*0.9}:1.064cm)  --({210-60*0.935}:1.07cm);

    \end{scope}

\def \i {4}
    \begin{scope}[xshift=5*\i cm]
        % Draw circle
        \draw[thick](0,0) circle (2cm);
        % Place numbers around the circle
        \foreach \n in {1,...,6} {
            \node at ({210-60*\n}:2.3cm) {\n};
        }
        \draw[thick]plot [smooth, tension = 1.5] coordinates {({217-60*4}:2cm) ({210-60*3.5}:1.5cm) ({203-60*3}:2cm)};

        \draw[thick]plot [smooth, tension = 0.7] coordinates {({217-60*3}:2cm)  ({203-60*(2.9)}:1.5cm) ({203-60*4}:1.3cm) ({203-60*4.9}:1.5cm)({203-60*5}:2cm)};

        \draw[thick]plot [smooth, tension = 0.7] coordinates {({217-60*5}:2cm)  ({203-60*4.1}:1cm) ({210-60*2.5}:0.6cm) ({210-60*1.56}:1.3cm) ({203-60*1}:2cm)};

        \draw[thick]plot [smooth, tension = 0.7] coordinates {({203-60*4}:2cm) ({203-60*4.43}:1.3cm) ({203-60*5.5}:1.3cm)  ({210-60*0.8}:1.6cm) ({217-60*1}:2cm)};

        \draw[thick]plot [smooth, tension = 0.7] coordinates {({203-60*2}:2cm) ({210-60*2.05}:1.5cm) ({210-60*1.3}:1.2cm) ({210-60*5.95}:1.5cm)({217-60*6}:2cm)};

        \draw[thick]plot [smooth, tension = 1] coordinates {({217-60*2}:2cm) ({210-60*1.73}:1.7cm) ({210-60*1.27}:1cm) ({210-60*6.15}:1.25cm) ({203-60*6}:2cm)};

            \draw[arrow, thick] ({210-60*3.41}:1.518cm) -- ({210-60*3.4}:1.52cm);

        \draw[arrow, thick] ({217-60*3.015}:1.8cm) -- ({217-60*3.053}:1.65cm); 

        \draw[arrow, thick] ({203-60*4.965}:1.65cm) --({203-60*5}:1.8cm) ; 

        \draw[arrow, thick] ({217-60*4.93}:1.82cm)  --({217-60*4.83}:1.6cm); 

        \draw[arrow, thick] ({203-60*4.15}:1.65cm) --({203-60*4.08}:1.8cm) ; 

        \draw[arrow, thick] ({203-60*1.985}:1.65cm) --({203-60*2.01}:1.8cm) ; 

        \draw[arrow, thick] ({217-60*1.895}:1.82cm)  --({217-60*1.85}:1.71cm); 

        \draw[arrow, thick] ({203-60*1.2}:1.64cm) --({203-60*1.113}:1.8cm) ; 

        \draw[arrow, thick] ({217-60*.99}:1.78cm)  --({217-60*0.93}:1.6cm); 

        \draw[arrow, thick] ({203-60*5.96}:1.64cm) --({203-60*5.95}:1.8cm) ; 

        \draw[arrow, thick] ({217-60*5.99}:1.78cm)  --({217-60*6.035}:1.55cm); 

        \draw[arrow, thick] ({210-60*0.9}:1.064cm)  --({210-60*0.935}:1.07cm); 

        \draw[arrow, thick] ({210-60*1}:0.885cm)  --({210-60*0.935}:0.87cm); 

        \draw[arrow, thick] ({210-60*3}:0.531cm)-- ({210-60*2.7}:0.55cm); 

        \draw[arrow, thick] ({210-60*5.35}:1.27cm)  --({210-60*5.3}:1.27cm); 
    \end{scope}

\end{tikzpicture}
\caption{Constructing a standard triple diagram associated to the permutation
$\pi=(4,6,5,3,1,2)$.  The figure shows the stages in the construction after:
adding boundary-parallel strands $4\to 3$ and $3\to 5$; then $1\to 4$ and $6\to 2$; then 
$2\to 6$ and $5 \to 1$.\label{fig:standardtriple}}
\end{figure}

\iffalse
	\begin{figure}[h]
	\includegraphics[height=1.3in]{StandardTriple2.ps}
	\caption{\LW{alternative figure} Constructing a standard triple diagram associated to the permutation
	$\pi=(3,5,4,1,2)$.  The figure shows the stages in the construction after:
	adding boundary-parallel strands $1\to 3$;  $3\to 4$;  $4 \to 1$; $5 \to 2$; and $2\to 5$.
		\label{fig:standardtriple2}}
\end{figure}
\fi

We shall keep in mind that a standard triple diagram is constructed 
by choosing a sequence of minimal intervals. 

\begin{exercise}
\label{lem:useful-moves}
For each of the three pairs of triple diagrams shown in \cref{fig:7.6.1-7.6.3},
demonstrate that the two diagrams are move-equivalent to each other, 
i.e., are related via a sequence of swivel moves. 
(In each of the three cases, the central section can involve an arbitrary number of repetitions.)
\end{exercise}

\begin{figure}[ht]
\vspace{-.2in}
\begin{align}	
\label{it:move-1}
&\begin{tikzpicture}[baseline={(0,-0.1)}]
\draw[thick, postaction={decorate, decoration = {markings, mark = between positions 0.55 and 0.8 step 0.5 with {\arrow{Stealth}}}}] (-4,0.15) to[bend right=30] (-1,0.15);
\draw[thick, postaction={decorate, decoration = {markings, mark = between positions 0.55 and 0.8 step 0.5 with {\arrow{Stealth}}}}] (-1,-0.15) to[bend right=30] (-4,-0.15);
\draw[thick, postaction={decorate, decoration = {markings, mark = between positions 0.45 and 0.9 step 0.5 with {\arrow{Stealth}}}}] (-3.72,-0.3) -- (-3.72,0.3);
\draw[thick, postaction={decorate, decoration = {markings, mark = between positions 0.45 and 0.9 step 0.7 with {\arrow{Stealth}}}}] (-1.28,0.3) -- (-1.28,-0.3);
\draw[thick, postaction={decorate, decoration = {markings, mark = between positions 0.65 and 0.8 step 0.5 with {\arrow{Stealth}}}}] (-1.95,-0.45) to[bend left=55] (-1.95,0.45);
\draw[thick, postaction={decorate, decoration = {markings, mark = between positions 0.65 and 0.8 step 0.5 with {\arrow{Stealth}}}}] (-2.2,0.45) to[bend left=55] (-2.2,-0.45);
\draw[thick, postaction={decorate, decoration = {markings, mark = between positions 0.65 and 0.8 step 0.5 with {\arrow{Stealth}}}}] (-2.8,-0.45) to[bend left=55] (-2.8,0.45);
\draw[thick, postaction={decorate, decoration = {markings, mark = between positions 0.65 and 0.8 step 0.5 with {\arrow{Stealth}}}}] (-3.05,0.45) to[bend left=55] (-3.05,-0.45);
\draw[>={Stealth[length=6pt]}, <->] (-0.45,0) -- (0.55,0);
\draw[thick, postaction={decorate, decoration = {markings, mark = between positions 0.65 and 0.8 step 0.5 with {\arrow{Stealth}}}}] (1.35,-0.45) to[bend right=55] (1.35,0.45);
\draw[thick, postaction={decorate, decoration = {markings, mark = between positions 0.65 and 0.8 step 0.5 with {\arrow{Stealth}}}}] (1.6,0.45) to[bend right=55] (1.6,-0.45);
\draw[thick, postaction={decorate, decoration = {markings, mark = between positions 0.65 and 0.8 step 0.5 with {\arrow{Stealth}}}}] (2.05,-0.45) to[bend right=55] (2.05,0.45);
\draw[thick, postaction={decorate, decoration = {markings, mark = between positions 0.65 and 0.8 step 0.5 with {\arrow{Stealth}}}}] (2.3,0.45) to[bend right=55] (2.3,-0.45);
\draw[thick, postaction={decorate, decoration = {markings, mark = between positions 0.65 and 0.8 step 0.5 with {\arrow{Stealth}}}}] (2.75,-0.45) to[bend right=55] (2.75,0.45);
\draw[thick, postaction={decorate, decoration = {markings, mark = between positions 0.65 and 0.8 step 0.5 with {\arrow{Stealth}}}}] (3,0.45) to[bend right=55] (3,-0.45);
\draw[thick, postaction={decorate, decoration = {markings, mark = between positions 0.43 and 0.8 step 0.5 with {\arrow{Stealth}}}}] (1,0.3) -- (3.3,0.3);
\draw[thick, postaction={decorate, decoration = {markings, mark = between positions 0.7 and 0.8 step 0.5 with {\arrow{Stealth}}}}] (3.3,-0.3) -- (1,-0.3);
\end{tikzpicture}
\\[.1in]
& \label{it:move-2}
\begin{tikzpicture}[baseline={(0,-0.1)}]
\draw[thick, postaction={decorate, decoration = {markings, mark = between positions 0.35 and 0.9 step 0.7 with {\arrow{Stealth}}}}] (-3.71,-0.35) -- (-3.71,0.35);
\draw[thick, postaction={decorate, decoration = {markings, mark = between positions 0.65 and 0.8 step 0.5 with {\arrow{Stealth}}}}] (-1.3,-0.45) to[bend left=55] (-1.3,0.45);
\draw[thick, postaction={decorate, decoration = {markings, mark = between positions 0.65 and 0.8 step 0.5 with {\arrow{Stealth}}}}] (-1.55,0.45) to[bend left=55] (-1.55,-0.45);
%\draw[thick, postaction={decorate, decoration = {markings, mark = between positions 0.65 and 0.8 step 0.5 with {\arrow{Stealth}}}}] (-1.55,-0.45) to[bend right=55] (-1.55,0.45);
%\draw[thick, postaction={decorate, decoration = {markings, mark = between positions 0.65 and 0.8 step 0.5 with {\arrow{Stealth}}}}] (-1.3,0.45) to[bend right=55] (-1.3,-0.45);
\draw[thick, postaction={decorate, decoration = {markings, mark = between positions 0.65 and 0.8 step 0.5 with {\arrow{Stealth}}}}] (-2.15,-0.45) to[bend left=55] (-2.15,0.45);
\draw[thick, postaction={decorate, decoration = {markings, mark = between positions 0.65 and 0.8 step 0.5 with {\arrow{Stealth}}}}] (-2.4,0.45) to[bend left=55] (-2.4,-0.45);
%\draw[thick, postaction={decorate, decoration = {markings, mark = between positions 0.65 and 0.8 step 0.5 with {\arrow{Stealth}}}}] (-2.4,-0.45) to[bend right=55] (-2.4,0.45);
%\draw[thick, postaction={decorate, decoration = {markings, mark = between positions 0.65 and 0.8 step 0.5 with {\arrow{Stealth}}}}] (-2.15,0.45) to[bend right=55] (-2.15,-0.45);
\draw[thick, postaction={decorate, decoration = {markings, mark = between positions 0.65 and 0.8 step 0.5 with {\arrow{Stealth}}}}] (-3,-0.45) to[bend left=55] (-3,0.45);
\draw[thick, postaction={decorate, decoration = {markings, mark = between positions 0.65 and 0.8 step 0.5 with {\arrow{Stealth}}}}] (-3.25,0.45) to[bend left=55] (-3.25,-0.45);
%\draw[thick, postaction={decorate, decoration = {markings, mark = between positions 0.65 and 0.8 step 0.5 with {\arrow{Stealth}}}}] (-3.25,-0.45) to[bend right=55] (-3.25,0.45);
%\draw[thick, postaction={decorate, decoration = {markings, mark = between positions 0.65 and 0.8 step 0.5 with {\arrow{Stealth}}}}] (-3,0.45) to[bend right=55] (-3,-0.45);
\draw[thick, postaction={decorate, decoration = {markings, mark = between positions 0.55 and 0.6 step 0.55 with {\arrow{Stealth}}}}] plot [smooth] coordinates {(-1,0.33) (-3,0.33) (-3.7,0) (-3.95,-0.26)};
\draw[thick, postaction={decorate, decoration = {markings, mark = between positions 0.52 and 0.6 step 0.55 with {\arrow{Stealth}}}}] plot [smooth] coordinates {(-3.95,0.26) (-3.7,0) (-3,-0.33) (-1,-0.33)};
\draw[>={Stealth[length=6pt]}, <->] (-0.45,0) -- (0.55,0);
\draw[thick, postaction={decorate, decoration = {markings, mark = between positions 0.65 and 0.8 step 0.5 with {\arrow{Stealth}}}}] (1.35,-0.45) to[bend right=55] (1.35,0.45);
\draw[thick, postaction={decorate, decoration = {markings, mark = between positions 0.65 and 0.8 step 0.5 with {\arrow{Stealth}}}}] (1.6,0.45) to[bend right=55] (1.6,-0.45);
\draw[thick, postaction={decorate, decoration = {markings, mark = between positions 0.65 and 0.8 step 0.5 with {\arrow{Stealth}}}}] (2.05,-0.45) to[bend right=55] (2.05,0.45);
\draw[thick, postaction={decorate, decoration = {markings, mark = between positions 0.65 and 0.8 step 0.5 with {\arrow{Stealth}}}}] (2.3,0.45) to[bend right=55] (2.3,-0.45);
\draw[thick, postaction={decorate, decoration = {markings, mark = between positions 0.65 and 0.8 step 0.5 with {\arrow{Stealth}}}}] (2.75,-0.45) to[bend right=55] (2.75,0.45);
\draw[thick, postaction={decorate, decoration = {markings, mark = between positions 0.65 and 0.8 step 0.5 with {\arrow{Stealth}}}}] (3,0.45) to[bend right=55] (3,-0.45);
\draw[thick, postaction={decorate, decoration = {markings, mark = between positions 0.575 and 0.7 step 0.55 with {\arrow{Stealth}}}}] plot [smooth] coordinates {(1,0.33) (3,0.3) (3.7,-0.26)};
\draw[thick, postaction={decorate, decoration = {markings, mark = between positions 0.51 and 0.6 step 0.55 with {\arrow{Stealth}}}}] plot [smooth] coordinates {(3.7,0.26) (3,-0.3) (1,-0.33)};
\draw[thick, postaction={decorate, decoration = {markings, mark = between positions 0.35 and 0.9 step 0.7 with {\arrow{Stealth}}}}] (3.48,-0.35) -- (3.48,0.35);
\end{tikzpicture}
\\[.1in]
&\label{it:move-3}
\begin{tikzpicture}[baseline={(0,-0.1)}]
\draw[thick, postaction={decorate, decoration = {markings, mark = between positions 0.5 and 0.7 step 0.55 with {\arrow{Stealth}}}}] plot [smooth] coordinates {(-3.9,0.33) (-1.7,0.3) (-1,-0.26)};
\draw[thick, postaction={decorate, decoration = {markings, mark = between positions 0.36 and 0.6 step 0.55 with {\arrow{Stealth}}}}] plot [smooth] coordinates {(-3.9,-0.33) (-1.7,-0.3) (-1,0.26)};
\draw[thick, postaction={decorate, decoration = {markings, mark = between positions 0.65 and 0.9 step 0.7 with {\arrow{Stealth}}}}] (-3.2,-0.5) -- (-3.55,0.5);
\draw[thick, postaction={decorate, decoration = {markings, mark = between positions 0.65 and 0.9 step 0.7 with {\arrow{Stealth}}}}] (-2.5,-0.5) -- (-2.85,0.5);
\draw[thick, postaction={decorate, decoration = {markings, mark = between positions 0.65 and 0.9 step 0.7 with {\arrow{Stealth}}}}] (-1.8,-0.5) -- (-2.15,0.5);
\draw[thick, postaction={decorate, decoration = {markings, mark = between positions 0.6 and 0.9 step 0.7 with {\arrow{Stealth}}}}] (-2.7,0.5) -- (-3.33,-0.5);
\draw[thick, postaction={decorate, decoration = {markings, mark = between positions 0.6 and 0.9 step 0.7 with {\arrow{Stealth}}}}] (-2,0.5) -- (-2.63,-0.5);
\draw[thick, postaction={decorate, decoration = {markings, mark = between positions 0.95 and 0.99 step 0.5 with {\arrow{Stealth}}}}] (-3.4,0.5) to[bend left=30] (-3.9,0);
\draw[thick, postaction={decorate, decoration = {markings, mark = between positions 0.62 and 0.88 step 0.55 with {\arrow{Stealth}}}}] plot [smooth] coordinates {(-1,0) (-1.4,-0.01) (-1.6,-0.05) (-1.8,-0.2) (-2,-0.5)};
\draw[>={Stealth[length=6pt]}, <->] (-0.45,0) -- (0.55,0);
\draw[thick, postaction={decorate, decoration = {markings, mark = between positions 0.72 and 0.8 step 0.55 with {\arrow{Stealth}}}}] plot [smooth] coordinates {(1,-0.26) (1.3,0) (2,0.33) (4,0.33)};
\draw[thick, postaction={decorate, decoration = {markings, mark = between positions 0.57 and 0.6 step 0.55 with {\arrow{Stealth}}}}] plot [smooth] coordinates {(1,0.26) (1.3,0) (2,-0.33) (4,-0.33)};
\draw[thick, postaction={decorate, decoration = {markings, mark = between positions 0.65 and 0.9 step 0.7 with {\arrow{Stealth}}}}] (2.2,-0.5) -- (1.85,0.5);
\draw[thick, postaction={decorate, decoration = {markings, mark = between positions 0.65 and 0.9 step 0.7 with {\arrow{Stealth}}}}] (2.9,-0.5) -- (2.55,0.5);
\draw[thick, postaction={decorate, decoration = {markings, mark = between positions 0.65 and 0.9 step 0.7 with {\arrow{Stealth}}}}] (3.6,-0.5) -- (3.25,0.5);
\draw[thick, postaction={decorate, decoration = {markings, mark = between positions 0.6 and 0.9 step 0.7 with {\arrow{Stealth}}}}] (2.68,0.5) -- (2.04,-0.5);
\draw[thick, postaction={decorate, decoration = {markings, mark = between positions 0.6 and 0.9 step 0.7 with {\arrow{Stealth}}}}] (3.38,0.5) -- (2.75,-0.5);
\draw[thick, postaction={decorate, decoration = {markings, mark = between positions 0.45 and 0.49 step 0.5 with {\arrow{Stealth}}}}] (4,0) to[bend right=30] (3.5,-0.5);
\draw[thick, postaction={decorate, decoration = {markings, mark = between positions 0.5 and 0.88 step 0.55 with {\arrow{Stealth}}}}] plot [smooth] coordinates {(2,0.5) (1.8,0.2) (1.55,0.05) (1,0)};
\end{tikzpicture}
\end{align}
\vspace{-.2in}
\caption{Move equivalence of triple diagrams.}
\label{fig:7.6.1-7.6.3}
\vspace{-.1in}
\end{figure}

\begin{exercise}
Use \eqref{it:move-3} to prove the move-equivalence \eqref{eq:parallel-digon-to-monogon} below: 
\vspace{-.25in}
\begin{equation}
\label{eq:parallel-digon-to-monogon}
\setlength{\unitlength}{0.35pt}
\begin{picture}(435,90)(0,30)
\thicklines
	\put(70,85){$S$}
	\put(95,-20){$U$}
\put(70,10){\line(1,0){160}}
\put(70,10){\line(4,7){46}}
\put(70,10){\line(-4,-7){6}}
\put(150,10){\line(4,7){46}}
\put(150,10){\line(-4,-7){6}}
\put(110,80){\line(1,0){80}}
\put(110,80){\line(4,-7){46}}
\put(110,80){\line(-4,7){6}}
\put(190,80){\line(4,-7){46}}
\put(190,80){\line(-4,7){6}}
\put(230,10){\line(-2,-3){7}}
\put(70,10){\line(2,-3){7}}
\put(280,40){\line(1,0){20}}
\put(280,40){\line(2,3){20}}
\put(280,40){\line(5,-4){20}}
\put(20,40){\line(-1,0){20}}
\put(20,40){\line(-2,3){20}}
\put(20,40){\line(-5,-4){20}}
\qbezier(190,80)(230,80)(280,40)
\qbezier(110,80)(70,80)(20,40)
\qbezier(230,10)(260,10)(280,40)
\qbezier(70,10)(40,10)(20,40)
\qbezier(230,10)(250,40)(280,40)
\qbezier(70,10)(50,40)(20,40)
\put(110,10){\vector(1,0){10}}
\put(150,80){\vector(1,0){10}}
\put(190,10){\vector(1,0){10}}
\put(90,45){\vector(-4,-7){5}}
\put(170,45){\vector(-4,-7){5}}
\put(130,45){\vector(-4,7){5}}
\put(210,45){\vector(-4,7){5}}
\put(240,66){\vector(2,-1.1){1}}
\put(272,30){\vector(1.1,1){1}}
\put(250,32){\vector(-2.8,-1){1}}
\put(60,66){\vector(2,1.1){1}}
\put(13,51){\vector(4,-7){1}}
\put(53,30){\vector(-1,1.2){1}}
\put(370,45){\makebox(0,0){$\longleftrightarrow$}}
\end{picture}
\setlength{\unitlength}{0.35pt}
\begin{picture}(300,90)(0,30)
\thicklines
\put(70,10){\line(1,0){230}}
\put(70,10){\line(-1,0){70}}
\put(70,10){\line(4,7){46}}
\put(70,10){\line(-4,-7){6}}
\put(150,10){\line(4,7){46}}
\put(150,10){\line(-4,-7){6}}
\put(110,80){\line(1,0){190}}
\put(110,80){\line(4,-7){46}}
\put(110,80){\line(-4,7){6}}
\put(190,80){\line(4,-7){46}}
\put(190,80){\line(-4,7){6}}
\put(230,10){\line(-2,-3){7}}
\put(70,10){\line(2,-3){7}}
\put(280,40){\line(1,0){20}}
%\put(280,40){\line(2,3){20}}
%\put(280,40){\line(5,-4){20}}
\put(20,40){\line(-1,0){20}}
\put(20,40){\line(-2,3){20}}
%\put(20,40){\line(-5,-4){20}}
%\qbezier(190,80)(230,80)(280,40)
\qbezier(110,80)(70,80)(20,40)
%\qbezier(230,10)(260,10)(280,40)
%\qbezier(70,10)(40,10)(20,40)
\qbezier(230,10)(250,40)(280,40)
\qbezier(70,10)(50,40)(20,40)
\qbezier(20,40)(30,25)(20,10)
\qbezier(20,40)(0,25)(20,10)

\qbezier(20,10)(15,5)(15,0)
\qbezier(20,10)(30,5)(30,0)
\qbezier(15,0)(15,-10)(22.5,-10)
\qbezier(30,0)(30,-10)(22.5,-10)

\put(25,16){\vector(0,-1){1}}
%\put(30,5){\vector(0,1){1}}
\put(53,30){\vector(-1,1.2){1}}

\put(45,10){\vector(1,0){10}}
\put(110,10){\vector(1,0){10}}
\put(150,80){\vector(1,0){10}}
\put(190,10){\vector(1,0){10}}
\put(90,45){\vector(-4,-7){5}}
\put(170,45){\vector(-4,-7){5}}
\put(130,45){\vector(-4,7){5}}
\put(210,45){\vector(-4,7){5}}
%\put(240,66){\vector(2,-1.1){1}}
%\put(272,30){\vector(1.1,1){1}}
\put(250,32){\vector(-2.8,-1){1}}
\put(60,66){\vector(2,1.1){1}}
%\put(28,30){\vector(-1.1,1){1}}

\put(270,80){\vector(1,0){10}}
\put(270,10){\vector(1,0){10}}
\put(12.5,51){\vector(4,-7){1}}

\end{picture}
\end{equation}
\end{exercise}
\vspace{.05in}

\begin{lemma}
\label{lem:straighten2}
Let $\X$ be %a minimal triple diagram.
a triple diagram such that no triple diagram move-equivalent to~$\X$ contains a monogon.
Then the following statements~hold:
\begin{itemize}[leftmargin=.3in]
\item[{\rm (i)}]
No triple diagram move-equivalent to $\X$ has a badgon or a closed strand. 
\item[{\rm (ii)}]
Let $I$ be a minimal interval for the strand permutation associated with~$\X$. 
%see \cref{def:standard-triple}.  
Then $\X$ is move-equivalent to a diagram $\X'$ in which the strand 
connecting the endpoints of~$I$ 
is boundary-parallel along~$I$. 
%see \cref{def:boundary-parallel}.
\end{itemize}
\vspace{-.1in}
\end{lemma}

\begin{proof}
We will simultaneously prove statements (i) and~(ii) 
by induction on the number of triple points in~$\X$. 
Thus, we assume that both (i) and~(ii) hold for triple diagrams 
that have fewer triple points than~$\X$. 

We first prove~(i). 
Suppose that a triple diagram $\X'\sim\X$ contains (non-self-intersecting) strands~$S$ and~$U$ forming
a parallel digon. 
The strand~$S$ cuts the disk~$\mathbf{D}$ into two regions.  
Let $R$ be the region containing the digon, with a small neighbourhood of~$S$ removed. 
Since the boundaries of the faces of~$\X'$ are consistently oriented, the same is true 
for the portion of~$\X'$ contained inside~$R$, so this portion can be viewed as a (smaller) triple diagram. 
Suppose that $U$ bounds a minimal interval within~$S$ (viewed as a portion of the boundary of~$R$).
Then by the induction assumption, $U$~can be moved to be boundary-parallel to~$S$.
Since $S$ and $U$ are co-oriented, 
we get the picture on the left-hand side of~\eqref{eq:parallel-digon-to-monogon}
(with $U$ running horizontally at the bottom). 
Applying~\eqref{eq:parallel-digon-to-monogon}, we obtain a monogon, a contradiction. 

If the subinterval of~$S$ cut out by~$U$ is not minimal, 
then there is a strand $T$ that cuts across~$S$ twice, creating a minimal interval within~$S$ 
and forming a digon inside~$R$. 
We may assume that this digon is anti-parallel (or else replace $U$ by~$T$ and repeat). 
By the induction assumption, we can apply swivel moves inside~$R$ to make $T$ boundary-parallel to~$S$. 
We then apply~\eqref{it:move-1} to remove the digon, as shown in \cref{fig:step-1}. 
Repeating this operation if necessary, we obtain a triple diagram in which 
$U$ bounds a minimal interval within~$S$;
we then argue as above to arrive at a contradiction. 

\pagebreak[3]

\begin{figure}[ht]
\begin{center}
\vspace{-.1in}
\begin{tikzpicture}

\filldraw[fill=white, ultra thick] (-4,0) circle (1.2cm);
\filldraw[fill=white, thick] (0,0) circle (1.2cm);
\filldraw[fill=white, thick] (4,0) circle (1.2cm);
\draw[>={Stealth[length=7pt]}, ->] (-2.3,0) -- (-1.6,0);
\draw[>={Stealth[length=7pt]}, ->] (1.7,0) -- (2.4,0);

% 1st circle
\def\first{(-4,0) circle (1.2cm)}
\def\second{(-5.19,-0.2) to[bend left=40] (-2.81,-0.2)}
  \fill[fill=white] \first;
\begin{scope}
    \clip \first;
    \fill[fill=yellow] \second;
    \path[fill=yellow] (-6,-0.2)--(-2,-0.2)--(-4,-3)--cycle;
\end{scope}

\draw[thick, postaction={decorate, decoration = {markings, mark = between positions 0.55 and 0.79 step 0.5 with {\arrow{Stealth}}}}] (-5.19,-0.2) to[bend left=40] (-2.81,-0.2);

\draw[thick] (-5,0.15) -- (-4.7,-0.1);
\draw[thick] (-4.9,-0.2) -- (-4.8,0.25);
\draw[thick] (-4.5,0.35) -- (-4.25,0.05);
\draw[thick] (-4.45,0) -- (-4.3,0.4);
\draw[thick] (-3.75,0.4) -- (-3.55,0);
\draw[thick] (-3.8,0.05) -- (-3.55,0.38);
\draw[thick] (-3.2,0.25) -- (-3.05,-0.2);
\draw[thick] (-3.25,-0.1) -- (-3,0.17);
\draw [thick, postaction={decorate, decoration = {markings, mark = between positions 0.55 and 0.69 step 0.35 with {\arrow[rotate=20]{Stealth}}}}] (-3.25,-0.1) to[out=235,in=25] (-3.75,-0.7) to[out=220,in=0] (-4.1,-0.8) to[out=160,in=280] (-4.6,-0.3) to[out=100,in=320] (-4.7,-0.1);
\node at (-5.03,-0.3) {\footnotesize $S$};
\node at (-3.6,-0.4) {\footnotesize $T$};
\node at (-3.1,-0.45) {\footnotesize $R$};
\node at (-4,-1.4) {\footnotesize $I$};

% 2nd circle
\draw[thick, postaction={decorate, decoration = {markings, mark = between positions 0.55 and 0.79 step 0.5 with {\arrow{Stealth}}}}] (-1.19,-0.2) to[bend left=40] (1.19,-0.2);
\draw[thick, postaction={decorate, decoration = {markings, mark = between positions 0.55 and 0.79 step 0.5 with {\arrow{Stealth}}}}] (1,0.2) to[bend left=60] (-1,0.2);
\draw[thick] (-0.5,-0.35) to[bend right=45] (-0.5,0.35);
\draw[thick] (-0.3,-0.37) to[bend left=45] (-0.3,0.36);
\draw[thick] (0.5,-0.35) to[bend left=45] (0.5,0.35);
\draw[thick] (0.3,-0.37) to[bend right=45] (0.3,0.36);
\draw[thick] (-0.82,-0.2) to[bend left=25] (-0.82,0.2);

% 3rd circle
\draw[thick, postaction={decorate, decoration = {markings, mark = between positions 0.45 and 0.59 step 0.5 with {\arrow{Stealth}}}}] (2.9,-0.45) to[bend left=30] (5.1,-0.45);
\draw[thick, postaction={decorate, decoration = {markings, mark = between positions 0.4 and 0.59 step 0.5 with {\arrow{Stealth}}}}] (4.9,0.2) -- (3.1,0.2);
\draw[thick] (4.1,-0.29) to[bend left=45] (4.1,0.36);
\draw[thick] (3.9,-0.29) to[bend right=45] (3.9,0.36);
\draw[thick] (4.7,-0.37) to[bend left=45] (4.7,0.33);
\draw[thick] (4.5,-0.33) to[bend right=45] (4.5,0.33);
\draw[thick] (3.5,-0.33) to[bend left=45] (3.5,0.33);
\draw[thick] (3.3,-0.37) to[bend right=45] (3.3,0.33);
\end{tikzpicture}
\vspace{-.25in}
\end{center}
\caption{Removing double intersections with $S$.}
\label{fig:step-1}
\end{figure}
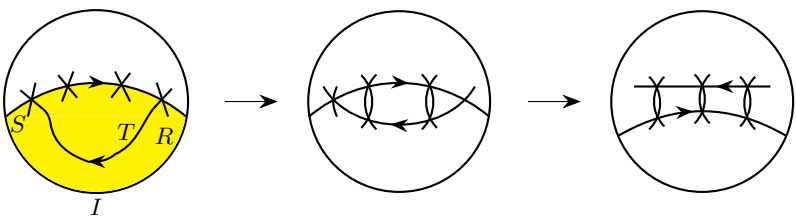

Thus, no triple diagram $\X'\sim\X$ contains badgons. 
By \cref{lem:no-closed-strands}, we conclude that any such $\X'$ does not contain closed strands either.
\iffalse
Now suppose that $\X'\sim\X$ contains a closed (necessarily non-self\-inter\-secting) strand~$S$.
%Without loss of generality, we may assume that $\X'=\X$. 
Let $R$ be the region enclosed by~$S$, with a small neigbourhood of~$S$ removed. 
The strands intersecting~$R$ form a triple diagram inside~$R$.
Since this diagram has fewer triple points than~$\X'$ (or equivalently~$\X$),
the induction assumption applies.
Therefore we can move some strand~$T$ in~$R$ so that it becomes boundary-parallel (to~$\partial R$). 
If $T$ is parallel to the corresponding segment of~$S$, 
then we can apply~\eqref{eq:parallel-digon-to-monogon} to create a monogon. 
Thus, $T$~must be anti-parallel to~$S$. 
We can then apply \eqref{it:move-1} to move this segment of~$T$ outside~$S$,
thereby decreasing the number of triple points inside~$S$.
Repeating this step if necessary, 
we arrive at a situation with no triple points inside~$S$. 
Applying further swivel moves, we move all strands crossing~$S$ away from~$S$ 
until there are only two of them left, 
and then a swivel move creates a self-intersection. 
\fi
This completes the induction step for statement~(i). 

We now proceed to proving statement~(ii). In addition to the induction assumption for~(ii), 
%
%repeatedly applying 
%transformations \eqref{it:move-1}--\eqref{it:move-3} to appropriately chosen fragments of~$\X$.  
%
%By \cref{lem:nomonogondigon}, %since $\X$ is minimal,
we may assume that neither~$\X$ 
nor any triple diagram move-equivalent to~$\X$ contains a badgon or a closed strand. 

Let $S$ be the strand connecting the endpoints of the minimal interval~$I$. 

\noindent
\emph{Step 1: Removing double intersections with $S$, see \cref{fig:step-1}.} 
Let $R$ be the region between the strand~$S$ and the interval~$I$, 
with a small neigborhood of~$S$ removed. 
%not including the triple crossings along $S$ itself.  
Suppose there is a strand that intersects $S$ more than once.  
Among such strands, take one that 
cuts out a minimal interval along the boundary of~$R$. 
Let $T$ denote the segment of this strand contained in~$R$. 
The portion of $\X$ contained inside~$R$ has fewer triple crossings than~$\X$, 
so by the induction assumption, we can make $T$ boundary-parallel to~$S$
by applying swivel moves inside~$R$.  
Now $T$ and $S$ form a (necessarily anti-parallel) digon, 
which we then remove using~\eqref{it:move-1}.  
We repeat this procedure until there are no strands left that intersect~$S$ more than once.  
Since the number of triple points along $S$ decreases each time, the process terminates.

\noindent
\emph{Step 2: Combing out the triple crossings.}
At this stage, no strand crosses $S$ more than once.
Since $I$ is minimal, no strand has both ends at~$I$. 
Since $\X$ contains no closed strands, 
every (non-self-intersecting) strand appearing between~$S$ and~$I$ 
must start or end at a point in~$I$ and cross~$S$.  
Suppose that $S$ is not boundary-parallel. 
Then there exists a strand~$T$ with an endpoint at~$I$
that passes through a triple point before hitting~$S$. 
Among all such~$T$, choose the one with the leftmost endpoint along~$I$, 
cf.\ Figures~\ref{fig:step-2} and~\ref{fig:step-2p} on the left.  
Let $R$ be the part of the region between $I$ and $S$ 
that lies to the right of any strand $T'$ located to the left~of~$T$.  
(By our choice of~$T$, all such strands $T'$ run directly from $I$ to~$S$, with no crossings in between.)  
As we have eliminated all double intersections with~$S$, 
the inter\-val corresponding to~$T$ (looking to the left) is minimal inside~$R$.
We can therefore use the induction assumption inside~$R$ to make~$T$ boundary-parallel.

What we do next depends on the orientation of~$T$ relative to~$S$. 
If $T$ is anti-parallel to~$S$, as in \cref{fig:step-2}, 
then we apply~\eqref{it:move-2} to make $T$ run directly to~$S$.
If $T$ is parallel to~$S$, as in \cref{fig:step-2p},  
then we apply~\eqref{it:move-3}.

We repeat this step until $S$ is boundary-parallel.  
%(One might wonder whether there could be disconnected components
%    of the diagram between $S$ and the boundary, but this is impossible,
%    since we started with a connected diagram, and the swivel move
%    preserves connectivity.)
\end{proof}

\begin{figure}[ht]
\begin{center}
\vspace{-.15in}
\begin{tikzpicture}

\filldraw[fill=white, ultra thick] (-4,0) circle (1.2cm);
\filldraw[fill=white, thick] (0,0) circle (1.2cm);
\filldraw[fill=white, thick] (4,0) circle (1.2cm);
\draw[>={Stealth[length=7pt]}, ->] (-2.3,0) -- (-1.6,0);
\draw[>={Stealth[length=7pt]}, ->] (1.7,0) -- (2.4,0);

% 1st circle
\def\first{(-4,0) circle (1.2cm)}
\fill[fill=white] \first;
%\begin{scope}
%    \clip \first;
%    \path[fill=yellow] (-5,-0.7) -- (-4.8,0.1) 
%    to[out=-10,in=180] (-4.1,0) 
%    to[out=0,in=210] (-3.7,0.05) 
%    to[out=260,in=50] (-4,-0.8) 
%    to[out=220,in=15] (-4.8,-1.2) 
%    -- cycle;
%\end{scope}
\draw[thick, postaction={decorate, decoration = {markings, mark = between positions 0.5 and 0.59 step 0.5 with {\arrow{Stealth}}}}] (-5,-0.68) -- (-4.75,0.3);
\draw[thick, postaction={decorate, decoration = {markings, mark = between positions 0.45 and 0.59 step 0.5 with {\arrow{Stealth}}}}] (-5.2,0.04) to[out=15,in=180] (-4.8,0.1) to[out=-10,in=180] (-4.1,0) to[out=0,in=210] (-3.7,0.05) to[out=30,in=180] (-3,0.2) to[out=0,in=155] (-2.8,0.15);
\draw[thick, postaction={decorate, decoration = {markings, mark = between positions 0.4 and 0.59 step 0.5 with {\arrow{Stealth}}}}] (-3.6,0.25) to[out=250,in=80] (-3.7,0.05) to[out=260,in=50] (-4,-0.8) to[out=220,in=15] (-4.5,-1.1);
\draw[thick, postaction={decorate, decoration = {markings, mark = between positions 0.7 and 0.79 step 0.5 with {\arrow{Stealth}}}}] (-4.6,0.3) to[bend right=20] (-5.1,-0.53);
\draw[thick] (-4,-0.55) -- (-3.7,-0.6);
\draw[thick] (-3.9,-0.45) -- (-3.8,-0.7);
\draw[thick] (-4.2,-0.8) -- (-3.95,-0.9);
\draw[thick] (-4.1,-0.7) -- (-4.05,-1);
\node at (-4.3,-0.4) {\footnotesize $R$};
\node at (-5,0.3) {\footnotesize $S$};
\node at (-4.4,-0.9) {\footnotesize $T$};

% 2nd circle
\draw[thick, postaction={decorate, decoration = {markings, mark = between positions 0.5 and 0.59 step 0.5 with {\arrow{Stealth}}}}] (-1,-0.68) -- (-0.75,0.3);
\draw[thick, postaction={decorate, decoration = {markings, mark = between positions 0.3 and 0.59 step 0.5 with {\arrow{Stealth}}}}] (-1.2,0.04) to[out=15,in=180] (-0.8,0.1) to[out=-10,in=180] (-0.1,0) to[out=0,in=210] (0.3,0.05) to[out=30,in=180] (1,0.2) to[out=0,in=155] (1.2,0.15);
\draw[thick, postaction={decorate, decoration = {markings, mark = between positions 0.5 and 0.59 step 0.5 with {\arrow{Stealth}}}}] (0.6,0.5) to[out=210,in=90] (0.3,0.05) to[out=260,in=50] (0,-0.8) to[out=220,in=15] (-0.5,-1.1);
\draw[thick, postaction={decorate, decoration = {markings, mark = between positions 0.7 and 0.79 step 0.5 with {\arrow{Stealth}}}}] (-0.6,0.3) to[bend right=20] (-1.1,-0.51);
\draw[thick] (0.45,-0.05) to[bend left=10] (0.15,0.15);
\draw[thick] (0.3,-0.7) to[bend left=30] (-0.1,0.2);
\draw[thick] (0.14,-0.8) to[bend right=30] (-0.25,0.15);
\draw[thick] (0,-1) to[bend left=25] (-0.36,0.2);
\draw[thick] (-0.15,-1.1) to[bend right=25] (-0.5,0.18);
\node at (-0.95,0.3) {\footnotesize $S$};
\node at (-0.39,-0.9) {\footnotesize $T$};

% 3rd circle
\draw[thick, postaction={decorate, decoration = {markings, mark = between positions 0.5 and 0.59 step 0.5 with {\arrow{Stealth}}}}] (3,-0.68) -- (3.25,0.3);
\draw[thick, postaction={decorate, decoration = {markings, mark = between positions 0.7 and 0.79 step 0.5 with {\arrow{Stealth}}}}] (3.4,0.3) to[bend right=20] (2.9,-0.5);
\draw[thick] (4.7,-0.5) to[bend left=30] (4.3,0.4);
\draw[thick] (4.54,-0.6) to[bend right=30] (4.15,0.35);
\draw[thick] (4.2,-0.9) to[bend left=25] (3.84,0.3);
\draw[thick] (4.05,-1) to[bend right=25] (3.7,0.28);
\draw[thick] (3.4,-1.05) to[out=30,in=240] (3.65,-0.1) to[out=70,in=170] (4.32,0.23) to[out=10,in=250] (4.7,0.5);
\draw[thick, postaction={decorate, decoration = {markings, mark = between positions 0.25 and 0.59 step 0.5 with {\arrow[rotate=9]{Stealth}}}}] (2.8,-0.05) to[out=60,in=180] (3.25,0.15) to[out=0,in=170] (3.9,-0.78) to[out=0,in=250] (4.8,0.15) to[out=70,in=190] (5.1,0.44);
\draw[thick] (3.7,-0.83) to[bend left=15] (3.62,0.12);
\node at (3,0.3) {\footnotesize $S$};
\node at (3.75,-1) {\footnotesize $T$};
\end{tikzpicture}
\vspace{-.2in}
\end{center}
\caption{Combing out the triple crossings: the anti-parallel case.}
\label{fig:step-2}
\end{figure}
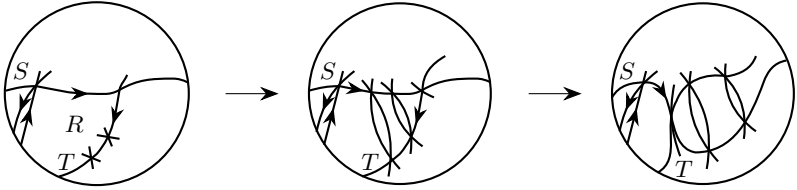

\begin{figure}[ht]
\begin{center}
\vspace{-.2in}
\begin{tikzpicture}

\filldraw[fill=white, thick] (-4,0) circle (1.2cm);
\filldraw[fill=white, thick] (0,0) circle (1.2cm);
\filldraw[fill=white, thick] (4,0) circle (1.2cm);
\draw[>={Stealth[length=7pt]}, ->] (-2.3,0) -- (-1.6,0);
\draw[>={Stealth[length=7pt]}, ->] (1.7,0) -- (2.4,0);

% 1st circle
\draw[thick, postaction={decorate, decoration = {markings, mark = between positions 0.15 and 0.39 step 0.5 with {\arrow{Stealth}}}}] (-5.2,0.04) to[out=15,in=180] (-4.8,0.1) to[out=-10,in=180] (-4.1,0) to[out=0,in=210] (-3.7,0.05) to[out=30,in=180] (-3,0.2) to[out=0,in=155] (-2.8,0.15);
\draw[thick, postaction={decorate, decoration = {markings, mark = between positions 0.85 and 0.89 step 0.5 with {\arrow[rotate=20]{Stealth}}}}] (-4.6,0.2) to[out=200,in=60] (-4.8,0) to[out=250,in=0] (-5.15,-0.3);
\draw[thick, postaction={decorate, decoration = {markings, mark = between positions 0.25 and 0.59 step 0.5 with {\arrow{Stealth}}}}] (-4.5,-1.1) to[out=15,in=220] (-4,-0.8) to[out=50,in=260] (-3.7,0.05) to[out=80,in=250] (-3.6,0.25);
\draw[thick] (-3.95,-0.4) -- (-3.65,-0.45);
\draw[thick] (-3.85,-0.3) -- (-3.75,-0.55);
\draw[thick] (-4.1,-0.7) -- (-3.85,-0.8);
\draw[thick] (-4.0,-0.6) -- (-3.95,-0.9);
\node at (-4.9,0.3) {\footnotesize $S$};
%\node at (-4.03,-1.05) {\footnotesize $T$};
\node at (-4.3,-0.4) {\footnotesize $R$};
\node at (-4.4,-0.9) {\footnotesize $T$};

\draw[thick] (-4.78,0.25) to[bend right=15] (-4.64,-0.1);

% 2nd circle
\draw[thick, postaction={decorate, decoration = {markings, mark = between positions 0.15 and 0.39 step 0.5 with {\arrow{Stealth}}}}] (-1.2,0.04) to[out=15,in=180] (-0.8,0.1) to[out=-10,in=180] (-0.1,0) to[out=0,in=210] (0.3,0.05) to[out=30,in=180] (1,0.2) to[out=0,in=155] (1.2,0.15);
\draw[thick, postaction={decorate, decoration = {markings, mark = between positions 0.85 and 0.89 step 0.5 with {\arrow[rotate=20]{Stealth}}}}] (-0.6,0.2) to[out=200,in=60] (-0.8,0) to[out=250,in=0] (-1.15,-0.3);
\draw[thick, postaction={decorate, decoration = {markings, mark = between positions 0.25 and 0.59 step 0.5 with {\arrow{Stealth}}}}] (-0.5,-1.1) to[out=15,in=220] (0,-0.8) to[out=50,in=260] (0.3,0.05) to[out=80,in=250] (0.4,0.35);
\node at (-0.9,0.3) {\footnotesize $S$};
\node at (-0.01,-1.05) {\footnotesize $T$};
\draw[thick, postaction={decorate, decoration = {markings, mark = between positions 0.6 and 0.69 step 0.5 with {\arrow{Stealth[length=6pt]}}}}] (0.49,0.2) to[bend right=50] (0.25,-0.75);
\draw[thick, postaction={decorate, decoration = {markings, mark = between positions 0.6 and 0.69 step 0.5 with {\arrow{Stealth[length=6pt]}}}}] (0.35,-0.65) to[bend left=15] (-0.4,0.15);
\draw[thick, postaction={decorate, decoration = {markings, mark = between positions 0.6 and 0.69 step 0.5 with {\arrow{Stealth[length=6pt]}}}}] (0.05,-0.91) to[bend left=10] (-0.8,0.2);
\draw[thick, postaction={decorate, decoration = {markings, mark = between positions 0.75 and 0.79 step 0.5 with {\arrow{Stealth[length=6pt]}}}}] (-0.3,0.25) to[out=255,in=110] (-0.1,-0.65) to[out=300,in=85] (-0.05,-0.94);

% 3rd circle
\draw[thick, postaction={decorate, decoration = {markings, mark = between positions 0.6 and 0.69 step 0.5 with {\arrow{Stealth[length=6pt]}}}}] (4.7,0.35) to[bend right=50] (4.55,-0.6);
\draw[thick, postaction={decorate, decoration = {markings, mark = between positions 0.6 and 0.69 step 0.5 with {\arrow{Stealth[length=6pt]}}}}] (4.65,-0.55) to[bend left=15] (3.9,0.35);
\draw[thick, postaction={decorate, decoration = {markings, mark = between positions 0.6 and 0.69 step 0.5 with {\arrow{Stealth[length=6pt]}}}}] (4.35,-0.71) to[bend left=10] (3.5,0.4);
\draw[thick, postaction={decorate, decoration = {markings, mark = between positions 0.75 and 0.79 step 0.5 with {\arrow{Stealth[length=6pt]}}}}] (4,0.45) to[out=255,in=110] (4.2,-0.45) to[out=300,in=85] (4.25,-0.74);
	\draw[thick, postaction={decorate, decoration = {markings, mark = between positions 0.2 and 0.49 step 0.5 with {\arrow{Stealth[length=6pt]}}}}] (3.5,-1.1) to[out=90,in=240] (3.25,-0.15) to[out=40,in=200] (3.45,0.3) to[out=10,in=180] (4.3,0.1) to[out=40,in=210](4.6,0.5); % to[out=0,in=220] (4.8,0.3);
	\draw[thick, postaction={decorate, decoration = {markings, mark = between positions 0.15 and 0.39 step 0.5 with {\arrow{Stealth[length=6pt]}}}}] (2.82,0.2) to[out=40,in=180] (3,0.25) to[out=350,in=130] (3.2,-0.15) to[out=330,in=180] (3.95,-0.8) to[out=10,in=280] (4.5,-0.1)   to[out=90,in=155] (5.2,0.15);
\draw[thick, postaction={decorate, decoration = {markings, mark = at position 0.93 with {\arrow[rotate=-20]{Stealth[length=6pt]}}}}] (3.7,0.45) to[out=200,in=0] (3.25,-0.17) to[out=200,in=45] (2.8,-0.1);
\node at (3.05,0.4) {\footnotesize $S$};
\node at (3.7,-1) {\footnotesize $T$};

\end{tikzpicture}
\vspace{-.25in}
\end{center}
\caption{Combing out the triple crossings: the parallel case.
}
%\vspace{-.1in}
\label{fig:step-2p}
\end{figure}
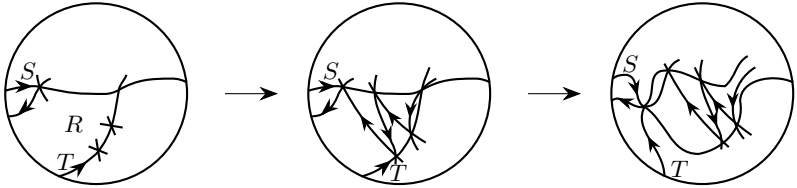

\vspace{-.1in}
\begin{lemma}
\label{lem:straighten}
For a triple diagram~$\X$, the following are equivalent:
\begin{itemize}[leftmargin=.3in]
\item[{\rm(a)}]
Any diagram $\X'$ move-equivalent to~$\X$ does not contain a monogon. 
\item[{\rm(b)}]
$\X$ is move-equivalent to any standard triple diagram with the same strand permutation. 
\item[{\rm(c)}]
$\X$ is minimal. 
\end{itemize}
In particular, any standard triple diagram is minimal. 
\end{lemma} 

\begin{proof}
The implication (c)$\Rightarrow$(a) follows from Lemmas \ref{lem:min-diag-move-equiv}
and~\ref{lem:nomonogondigon}. 
To prove the implication (a)$\Rightarrow$(b), 
choose a sequence of minimal intervals and repeatedly apply \cref{lem:straighten2}. 
We have now established (c)$\Rightarrow$(b), so any minimal triple diagram is
move-equivalent to any standard triple diagram with the same strand permutation. 
It follows by \cref{lem:min-diag-move-equiv} that any standard triple diagram is minimal,
hence so is any diagram move-equivalent to a standard one. 
Thus (b)$\Rightarrow$(c) is proved. 
\end{proof}

\begin{proof}[Proof of \cref{thm:domino-flip}]
By \cref{lem:straighten}, any two minimal triple diagrams with strand permutation~$\pi$
are move-equivalent to any standard diagram with strand permutation~$\pi$,
and therefore to each other. 
\end{proof}

\begin{lemma}
\label{lem:propagate}
Let $\X$ and $\X'$ be triple diagrams related by a swivel move.  
If $\X$ contains a badgon, then so does~$\X'$.
\end{lemma}

%My proof is different from Dylan's proof of Lemma 15. 
\begin{proof}
We label the strands and the triple points 
involved in this swivel move 
by $a,b, c, d$, and $x, y$,  as shown in \cref{fig:22}.

\begin{figure}[ht]
\begin{center}
\vspace{-.1in}
\begin{tikzpicture}

% left
\draw[thick, postaction={decorate, decoration = {markings, mark = between positions 0.3 and 0.99 step 0.55 with {\arrow{Stealth[length=6pt]}}}}] (-1,0.6) -- (1,0.6);
\draw[thick, postaction={decorate, decoration = {markings, mark = between positions 0.3 and 0.99 step 0.55 with {\arrow{Stealth[length=6pt]}}}}] (1,-0.6) -- (-1,-0.6);
\draw[thick, postaction={decorate, decoration = {markings, mark = between positions 0.17 and 0.99 step 0.37 with {\arrow{Stealth[length=6pt]}}}}] (-0.7,-1.3) to[out=45,in=270] (0.2,0) to[out=90,in=315] (-0.7,1.3);
\draw[thick, postaction={decorate, decoration = {markings, mark = between positions 0.17 and 0.99 step 0.37 with {\arrow{Stealth[length=6pt]}}}}] (0.7,1.3) to[out=225,in=90] (-0.2,0) to[out=270,in=135] (0.7,-1.3);
\node at (0.85,1.1) {\footnotesize $a$};
\node at (0.85,-1.1) {\footnotesize $a$};
\node at (-0.85,1.1) {\footnotesize $b$};
\node at (-0.85,-1.1) {\footnotesize $b$};
\node at (1.17,0.54) {\footnotesize $c$};
\node at (-1.17,0.54) {\footnotesize $c$};
\node at (1.17,-0.54) {\footnotesize $d$};
\node at (-1.17,-0.54) {\footnotesize $d$};
\node at (0,0.82) {\footnotesize $x$};
\node at (0,-0.86) {\footnotesize $y$};
\node[circle, fill=black, draw=black, inner sep=0pt, minimum size=3pt] at (0,0.6) {};
\node[circle, fill=black, draw=black, inner sep=0pt, minimum size=3pt] at (0,-0.6) {};

% right
\draw[thick, postaction={decorate, decoration = {markings, mark = between positions 0.3 and 0.99 step 0.55 with {\arrow{Stealth[length=6pt]}}}}] (3.46,-1) -- (3.46,1);
\draw[thick, postaction={decorate, decoration = {markings, mark = between positions 0.3 and 0.99 step 0.55 with {\arrow{Stealth[length=6pt]}}}}] (4.54,1) -- (4.54,-1);
\draw[thick, postaction={decorate, decoration = {markings, mark = between positions 0.17 and 0.99 step 0.4 with {\arrow{Stealth[length=6pt]}}}}] (2.7,0.4) to[out=0,in=170] (4,-0.26) to[out=10,in=180] (5.3,0.4);
\draw[thick, postaction={decorate, decoration = {markings, mark = between positions 0.17 and 0.99 step 0.4 with {\arrow{Stealth[length=6pt]}}}}] (5.3,-0.4) to[out=170,in=0] (4,0.26) to[out=180,in=10] (2.7,-0.4);
\node at (4.54,1.2) {\footnotesize $a$};
\node at (4.54,-1.2) {\footnotesize $a$};
\node at (3.46,1.2) {\footnotesize $b$};
\node at (3.46,-1.2) {\footnotesize $b$};
\node at (2.5,0.4) {\footnotesize $c$};
\node at (5.5,0.4) {\footnotesize $c$};
\node at (2.5,-0.4) {\footnotesize $d$};
\node at (5.5,-0.4) {\footnotesize $d$};
\end{tikzpicture}
\vspace{-.2in}
\end{center}
\caption{A swivel move relating $\X$ and $\X'$.}
\label{fig:22}
\end{figure}

If $\X$ contains a badgon that involves neither~$x$ nor~$y$, 
then this badgon persists in~$\X'$.

Suppose $\X$ contains a monogon whose self-intersection point is (say)~$x$.
Thus, two of the strands $\{a, b, c\}$ coincide.
If $a=c$ (resp., $b=c$), 
then the same monogon persists in~$\X'$ because in \cref{fig:22}, 
strands $a$ and~$c$ (resp., $b$ and $c$) intersect in both $\X$ and~$\X'$.

If, on the other hand, $a=b$, then $\X'$ has a parallel digon, see \cref{fig:13}.

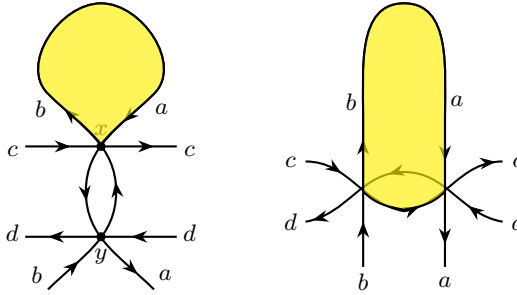
\begin{figure}[ht]
\begin{center}
%\vspace{-.2in}
\begin{tikzpicture}
\draw [thick, draw=black, fill=yellow, fill opacity=0.65] (3.46,0) -- (3.46,1) to[out=90,in=180] (4,2.5) to[out=0,in=90] (4.54,1) -- (4.54,0) to[out=225,in=10] (4,-0.24) to[out=170,in=315] (3.46,0);
\draw [thick, draw=black, fill=yellow, fill opacity=0.65] (-0.7,1.3) to[out=135,in=180] (0,2.5) to[out=0,in=45] (0.7,1.3) to[out=225,in=50] (0,0.63) to[out=130,in=315] (-0.7,1.3);
% left
\draw[thick, postaction={decorate, decoration = {markings, mark = between positions 0.3 and 0.99 step 0.55 with {\arrow{Stealth[length=6pt]}}}}] (-1,0.6) -- (1,0.6);
\draw[thick, postaction={decorate, decoration = {markings, mark = between positions 0.3 and 0.99 step 0.55 with {\arrow{Stealth[length=6pt]}}}}] (1,-0.6) -- (-1,-0.6);
\draw[thick, postaction={decorate, decoration = {markings, mark = between positions 0.17 and 0.99 step 0.37 with {\arrow{Stealth[length=6pt]}}}}] (-0.7,-1.3) to[out=45,in=270] (0.2,0) to[out=90,in=315] (-0.7,1.3);
\draw[thick, postaction={decorate, decoration = {markings, mark = between positions 0.17 and 0.99 step 0.37 with {\arrow{Stealth[length=6pt]}}}}] (0.7,1.3) to[out=225,in=90] (-0.2,0) to[out=270,in=135] (0.7,-1.3);
\node at (0.8,1.1) {\footnotesize $a$};
\node at (0.85,-1.1) {\footnotesize $a$};
\node at (-0.8,1.1) {\footnotesize $b$};
\node at (-0.85,-1.1) {\footnotesize $b$};
\node at (1.17,0.54) {\footnotesize $c$};
\node at (-1.17,0.54) {\footnotesize $c$};
\node at (1.17,-0.54) {\footnotesize $d$};
\node at (-1.17,-0.54) {\footnotesize $d$};
\node at (0,0.82) {\footnotesize $x$};
\node at (0,-0.86) {\footnotesize $y$};
\node[circle, fill=black, draw=black, inner sep=0pt, minimum size=3pt] at (0,0.6) {};
\node[circle, fill=black, draw=black, inner sep=0pt, minimum size=3pt] at (0,-0.6) {};
\draw[thick] (-0.7,1.3) to[out=135,in=180] (0,2.5) to[out=0,in=45] (0.7,1.3);

% right
\draw[thick, postaction={decorate, decoration = {markings, mark = between positions 0.3 and 0.99 step 0.55 with {\arrow{Stealth[length=6pt]}}}}] (3.46,-1) -- (3.46,1);
\draw[thick, postaction={decorate, decoration = {markings, mark = between positions 0.3 and 0.99 step 0.55 with {\arrow{Stealth[length=6pt]}}}}] (4.54,1) -- (4.54,-1);
\draw[thick, postaction={decorate, decoration = {markings, mark = between positions 0.17 and 0.99 step 0.4 with {\arrow{Stealth[length=6pt]}}}}] (2.7,0.4) to[out=0,in=170] (4,-0.26) to[out=10,in=180] (5.3,0.4);
\draw[thick, postaction={decorate, decoration = {markings, mark = between positions 0.17 and 0.99 step 0.4 with {\arrow{Stealth[length=6pt]}}}}] (5.3,-0.4) to[out=170,in=0] (4,0.26) to[out=180,in=10] (2.7,-0.4);
\node at (4.7,1.2) {\footnotesize $a$};
\node at (4.54,-1.2) {\footnotesize $a$};
\node at (3.3,1.2) {\footnotesize $b$};
\node at (3.46,-1.2) {\footnotesize $b$};
\node at (2.5,0.4) {\footnotesize $c$};
\node at (5.5,0.4) {\footnotesize $c$};
\node at (2.5,-0.4) {\footnotesize $d$};
\node at (5.5,-0.4) {\footnotesize $d$};
\draw[thick] (3.46,1) to[out=90,in=180] (4,2.5) to[out=0,in=90] (4.54,1);
\end{tikzpicture}
\vspace{-.2in}
\end{center}
\caption{A monogon in $\X$ results in a parallel digon in $\X'$.}
\label{fig:13}
\end{figure}
From now on, we can assume that there is no monogon in~$\X$.
Suppose $\X$ has a parallel digon whose two intersection points
include $x$ but not~$y$.  The sides of this parallel digon are 
either $\{a,b\}$ or $\{a,c\}$ or $\{b,c\}$.
The last two cases are easy because such a parallel digon will persist in~$\X'$,
since the strands $a$ and~$c$ (resp., $b$ and~$c$) intersect in both $\X$ and~$\X'$.

Now suppose that our parallel digon has sides $a$ and~$b$, 
see  \cref{fig:14} on the left.
(If the strands $a$ and $b$ go to the left and meet again there, then we \linebreak[3]
get the same picture 
but with the roles of $x$ and $y$  interchanged.)
Note that the end of strand $a$ shown inside the digon must 
extend outside of it, but it cannot intersect~$a$, as this would create a monogon.  
So strand $a$ must intersect strand~$b$ again,
see \cref{fig:14} in the middle.  
Then, after the swivel move, we get a parallel digon
	as shown in \cref{fig:14} on the right.

\begin{figure}[ht]
\begin{center}
\vspace{-.2in}
\begin{tikzpicture}[scale=0.85]
% left
\draw [thick, draw=black, fill=yellow, fill opacity=0.65] (0.7,1.3) to[out=30,in=145] (1.25,1.2) to[out=315,in=135] (2,-0.4) to [out=320,in=150] (2.07,-0.47) to[out=210,in=50] (2,-0.5) to[out=235,in=0] (0.2,-1.8) to[out=180,in=260] (-0.7,-1.3) to[out=45,in=270] (0.2,0) to[out=90,in=300] (0,0.6) to[out=58,in=225] (0.7,1.3);
\draw[thick, postaction={decorate, decoration = {markings, mark = between positions 0.3 and 0.99 step 0.55 with {\arrow{Stealth[length=6pt]}}}}] (-1,0.6) -- (1,0.6);
\draw[thick, postaction={decorate, decoration = {markings, mark = between positions 0.3 and 0.99 step 0.55 with {\arrow{Stealth[length=6pt]}}}}] (1,-0.6) -- (-1,-0.6);
\draw[thick, postaction={decorate, decoration = {markings, mark = between positions 0.17 and 0.99 step 0.37 with {\arrow{Stealth[length=6pt]}}}}] (-0.7,-1.3) to[out=45,in=270] (0.2,0) to[out=90,in=315] (-0.7,1.3);
\draw[thick, postaction={decorate, decoration = {markings, mark = between positions 0.17 and 0.99 step 0.37 with {\arrow{Stealth[length=6pt]}}}}] (0.7,1.3) to[out=225,in=90] (-0.2,0) to[out=270,in=135] (0.7,-1.3);
\node at (0.85,1.1) {\footnotesize $a$};
\node at (0.85,-1.1) {\footnotesize $a$};
\node at (-0.85,1.1) {\footnotesize $b$};
\node at (-0.85,-1.1) {\footnotesize $b$};
\node at (1.17,0.54) {\footnotesize $c$};
\node at (-1.17,0.54) {\footnotesize $c$};
\node at (1.17,-0.54) {\footnotesize $d$};
\node at (-1.17,-0.54) {\footnotesize $d$};
\node at (0,0.82) {\footnotesize $x$};
\node at (0,-0.86) {\footnotesize $y$};
\draw[thick] (-0.7,-1.3) to[out=260,in=180] (0.2,-1.8) to[out=0,in=235] (2,-0.5) to [out=50,in=205] (2.4,-0.1);
\draw[thick] (0.7,1.3) to[out=30,in=145] (1.25,1.2) to[out=315,in=135] (2,-0.4) to [out=320,in=180] (2.3,-0.5);
\node[circle, fill=black, draw=black, inner sep=0pt, minimum size=3pt] at (0,0.6) {};
\node[circle, fill=black, draw=black, inner sep=0pt, minimum size=3pt] at (0,-0.6) {};

% middle
\draw[thick, postaction={decorate, decoration = {markings, mark = between positions 0.3 and 0.99 step 0.55 with {\arrow{Stealth[length=6pt]}}}}] (4,0.6) -- (6,0.6);
\draw[thick, postaction={decorate, decoration = {markings, mark = between positions 0.3 and 0.99 step 0.55 with {\arrow{Stealth[length=6pt]}}}}] (6,-0.6) -- (4,-0.6);
\draw[thick, postaction={decorate, decoration = {markings, mark = between positions 0.17 and 0.99 step 0.37 with {\arrow{Stealth[length=6pt]}}}}] (4.3,-1.3) to[out=45,in=270] (5.2,0) to[out=90,in=315] (4.3,1.3);
\draw[thick, postaction={decorate, decoration = {markings, mark = between positions 0.17 and 0.99 step 0.37 with {\arrow{Stealth[length=6pt]}}}}] (5.7,1.3) to[out=225,in=90] (4.8,0) to[out=270,in=135] (5.7,-1.3);
\node at (5.85,1.1) {\footnotesize $a$};
\node at (5.85,-1.1) {\footnotesize $a$};
\node at (4.15,1.1) {\footnotesize $b$};
\node at (4.15,-1.1) {\footnotesize $b$};
\node at (6.17,0.54) {\footnotesize $c$};
\node at (3.83,0.54) {\footnotesize $c$};
\node at (6.17,-0.54) {\footnotesize $d$};
\node at (3.83,-0.54) {\footnotesize $d$};
\node at (5,0.82) {\footnotesize $x$};
\node at (5,-0.86) {\footnotesize $y$};
\node[circle, fill=black, draw=black, inner sep=0pt, minimum size=3pt] at (5,0.6) {};
\node[circle, fill=black, draw=black, inner sep=0pt, minimum size=3pt] at (5,-0.6) {};
\draw[thick] (4.3,-1.3) to[out=260,in=180] (5.2,-1.8) to[out=0,in=235] (7,-0.5) to [out=50,in=205] (7.4,-0.1);
\draw[thick] (5.7,1.3) to[out=30,in=145] (6.25,1.2) to[out=315,in=135] (7,-0.4) to [out=320,in=180] (7.3,-0.5);
\draw[thick] (5.7,-1.3) to[out=325,in=180] (6.5,-1.7);

% right
\draw [thick, draw=black, fill=yellow, fill opacity=0.65] (10.54,1) to[out=30,in=145] (11.4,1.1) to[out=315,in=135] (12,-0.4) to [out=320,in=145] (12.19,-0.55) to[out=235,in=37] (11.45,-1.36) to[out=135,in=350] (11.3,-1.3) to[out=180,in=300] (10.54,-1);
\draw[thick, postaction={decorate, decoration = {markings, mark = between positions 0.3 and 0.99 step 0.55 with {\arrow{Stealth[length=6pt]}}}}] (9.46,-1) -- (9.46,1);
\draw[thick, postaction={decorate, decoration = {markings, mark = between positions 0.3 and 0.99 step 0.55 with {\arrow{Stealth[length=6pt]}}}}] (10.54,1) -- (10.54,-1);
\draw[thick, postaction={decorate, decoration = {markings, mark = between positions 0.17 and 0.99 step 0.4 with {\arrow{Stealth[length=6pt]}}}}] (8.7,0.4) to[out=0,in=170] (10,-0.26) to[out=10,in=180] (11.3,0.4);
\draw[thick, postaction={decorate, decoration = {markings, mark = between positions 0.17 and 0.99 step 0.4 with {\arrow{Stealth[length=6pt]}}}}] (11.3,-0.4) to[out=170,in=0] (10,0.26) to[out=180,in=10] (8.7,-0.4);
\node at (10.54,1.2) {\footnotesize $a$};
\node at (10.54,-1.2) {\footnotesize $a$};
\node at (9.46,1.2) {\footnotesize $b$};
\node at (9.3,-1.1) {\footnotesize $b$};
\node at (8.5,0.4) {\footnotesize $c$};
\node at (11.5,0.4) {\footnotesize $c$};
\node at (8.5,-0.4) {\footnotesize $d$};
\node at (11.5,-0.4) {\footnotesize $d$};
\draw[thick, postaction={decorate, decoration = {markings, mark = between positions 0.25 and 0.39 step 0.6 with {\arrow{Stealth[length=6pt]}}}}] (12.4,-0.4) to[out=205,in=50] (12,-0.8) to[out=235,in=0] (10.2,-1.7) to [out=180,in=270] (9.46,-1);
\draw[thick] (10.54,1) to[out=30,in=145] (11.4,1.1) to[out=315,in=135] (12,-0.4) to [out=320,in=180] (12.4,-0.6);
\draw[thick] (10.54,-1) to[out=300,in=180] (11.3,-1.3) to[out=350,in=135] (11.6,-1.5);
\node at (11.45,-1.55) {\footnotesize $z$};
\node at (12.2,-0.3) {\footnotesize $w$};
\node[circle, fill=black, draw=black, inner sep=0pt, minimum size=3pt] at (12.2,-0.55) {};
\node[circle, fill=black, draw=black, inner sep=0pt, minimum size=3pt] at (11.45,-1.35) {};
\end{tikzpicture}
\vspace{-.25in}
\end{center}
\caption{Persistence of parallel digons under swivel moves.}
\label{fig:14}
\vspace{-.1in}
\end{figure}
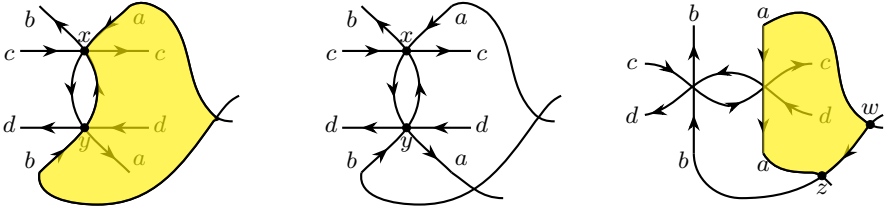

Finally, suppose there is a parallel digon in $\X$ whose two intersection points are
$x$ and $y$.  We can assume it is oriented from $x$ to $y$.
The two sides of the parallel digon should come from the following list:
\begin{itemize}[leftmargin=.4in]
\item[{\rm (aa)}]
	the (portion of the) arc along $a$ from $x$ to $y$; 
\item[{\rm (bb)}] 
the arc along $b$ from $x$ to $y$; 
\item[{\rm (cd)}]
an arc leaving $x$ along $c$, and returning to $y$
along $d$ (so $c=d$);
\item[{\rm (cb)}] 
an arc leaving $x$ along $c$, and returning to $y$ along $b$ (so $c=b$); 
\item[{\rm (bd)}]
an arc leaving $x$ along $b$, and returning to $y$ along $d$ (so $b=d$).
\end{itemize}
In  case (bb), we get a closed strand; it will persist in~$\X'$ and yield a badgon by
 \cref{lem:no-closed-strands}. 
In  cases (cb) and~(bd), we get a monogon, contradicting our assumption.
The remaining case is when the parallel digon has sides~(aa) and (cd);   
we then get a monogon in~$\X'$. (The picture is like 
	\cref{fig:13}, with the roles of $\X$ and $\X'$ swapped and some strands relabeled.)
\end{proof}

\begin{theorem}
\label{thm:minimal}
A  triple diagram is minimal if and only if it has no badgons. 
%-- in other words, it has no monogons or parallel digons.
\end{theorem}

\begin{proof}
The ``only if'' direction is \cref{lem:nomonogondigon}. 
%if $\X$ is a minimal triple diagram, then $\X$ cannot  contain a badgon.
The ``if'' direction follows from \cref{lem:propagate}
and \cref{lem:straighten} (implication (a)$\Rightarrow$(c)). 
%
%Suppose that $\X$ does not contain a badgon. By \cref{lem:propagate}, 
%badgons propagate under swivel-moves, so no $\X'$ move-equivalent to $\X$ can 
%contain a monogon or parallel digon. But then by  \cref{lem:straighten}, $\X$ must be minimal.
\end{proof}

%We note the similarity between the criteria in \cref{thm:minimal} (minimality of triple diagrams) 
%and \cref{thm:reduced} (reducedness of plabic graphs).

\begin{lemma}
\label{cor:Dyl}
Assume that a triple diagram $\mathfrak{X}$ is not minimal.
Then there exists a diagram $\mathfrak{X}'$ move-equivalent to 
$\mathfrak{X}$ that contains a hollow monogon.
\end{lemma}

\begin{proof}
We will argue by induction on the number of faces in~$\X$. 
If this number is 1 or~2, then the claim is vacuously true. 

By \cref{lem:straighten}, there exists $\X'\sim\X$ such that $\X'$ has a monogon. 
Let $M$ be the segment of a strand in~$\X'$ that forms a monogon;
we may assume  $M$ does not intersect itself except at its endpoints
(or else replace $M$ by its sub-segment). 
If the monogon encircled by~$M$ is hollow, we are done. 
Otherwise, consider the disk~$\mathbf{D}_\circ$ obtained by 
removing a small neighborhood of~$M$ from the interior of the monogon.
Let $\X'_\circ$ denote the portion of~$\X'$ contained in~$\mathbf{D}_\circ$;
this is a triple diagram with fewer faces than~$\X'$ (or equivalently~$\X$). 

The rest of the argument proceeds by showing that either we can 
apply local moves to $\X'_\circ$ to create a hollow monogon
inside $\mathbf{D}_\circ$ or we can apply moves to 
reduce the number of faces inside the monogon encircled by~$M$
(eventually producing a hollow monogon).
If $\X'_\circ$ is not minimal, then the induction assumption applies, so
we can transform~$\X'_\circ$ (thus~$\X'$ or~$\X$) into a 
move-equivalent triple diagram
containing a hollow monogon. 
Therefore, we may assume that  $\X'_\circ$ is minimal. 
Let $M_\circ$ denote the interval obtained from the boundary of~$\mathbf{D}_\circ$
by removing a point located near the vertex of our monogon. 
Let $I\subset M_\circ$ be a minimal interval of the triple diagram~$\X'_\circ$.
Since this triple diagram is minimal, we can, by \cref{lem:straighten}
(or~\cref{lem:straighten2}), 
apply local moves inside~$\mathbf{D}_\circ$
to transform $\X'_\circ$ into a triple diagram in which the strand~$T$ connecting 
the endpoints of~$I$ is boundary-parallel to~$I$. 
Let us now look at the digon~$D$ formed by~$T$ and the portion of~$M$ that runs along~$I$. 
If $D$ is anti-parallel, then we can push~$T$ outside the monogon as in 
\cref{fig:step-1}, 
reducing the number of faces enclosed by~$M$.
%then invoke the induction assumption. 
If, on the other hand, $D$ is parallel,
then we can use \eqref{eq:parallel-digon-to-monogon} to create a hollow monogon. 
\end{proof}

\section{From minimal triple diagrams to reduced plabic graphs}
\label{minred}

In this section, we use the machinery of triple diagrams and normal plabic graphs to prove 
\cref{prop:fixedlollipop}, 
\cref{thm:moves}, and 
\cref{cor:reduced=min-faces}.
In particular, we will be working with normal plabic graphs which are reduced and hence leafless, 
see \cref{rem:neverreduced}.  It follows that 
all the machinery that we have developed for leafless reduced plabic
graphs will apply here.
%\end{lemma}

%\begin{proof}
%All internal leaves in a normal plabic graph~$G$ are black.
%Whatever (M2)--(M3) moves we apply to a tree,
%it will always have a black leaf (so it can't collapse to 
%a white root or lollipop), and it will always have a white 
%vertex of degree at least $3$ (so it can't collapse to a black 
%root or lollipop).
%\end{proof}

\iffalse
\begin{lemma}
\label{lem:not-reduced=>osim}
Let $G$ be a non-reduced normal plabic graph.
Then there exists a plabic graph $G'\osim G$ (cf.\ \cref{def:omove})
containing one of the forbidden configurations shown in \cref{fig:fail}. 
\end{lemma}

\begin{proof}
Suppose $G$ has an internal (necessarily black) leaf~$u$ that is not a lollipop. 
Let $v$ be the unique (white, trivalent) vertex adjacent to~$u$. 
This gives us a forbidden configuration as in~\cref{fig:fail}(d). 
(Note that by \cref{lem:normal-no-collapse}, $G$~has no collapsible trees.) 

If, on the other hand, $G$ has no such internal leaves, 
then by \cref{pr:reduced-collapse}, $G$~can be transformed, 
via local moves that do not create leaves, into a graph containing a hollow digon.
\end{proof}
\fi

Recall from 
\cref{def:GX} 
and \cref{pr:bijplabictriple} 
that the map $G\to \mathfrak{X}(G)$ 
gives a bijection 
between  normal plabic graphs and  triple diagrams 
with the same number of boundary vertices; 
moreover, this bijection preserves the associated (resp., trip or strand) permutation. 

\begin{theorem}
\label{red-minimal}
A normal plabic graph $G$ is reduced if and only if 
the triple diagram  $\mathfrak{X}(G)$ %defined in \cref{def:GX} 
is minimal.
  Thus the map 
 $G \mapsto \mathfrak{X}(G)$ restricts to a bijection 
 between reduced normal plabic graphs and minimal triple diagrams.
\vspace{-5pt}
\end{theorem}

\begin{proof}
Suppose $\mathfrak{X}(G)$ is not minimal.  
By \cref{cor:Dyl}, there is a triple diagram $\mathfrak{X'} \!\sim\! \mathfrak{X}(G)$ 
such that $\mathfrak{X'}$ has a hollow monogon.  
By \cref{pr:bijplabictriple}, $\mathfrak{X'} \!=\! \mathfrak{X}(G')$ 
for some normal plabic graph~$G'$. 
Moreover,  \cref{cor:newmoves} implies that $G\osim G'$. 
The hollow monogon in~$\X'$ corresponds in the normal graph~$G'$ 
to one of the configurations shown in \cref{fig:18}: 
either a hollow digon (in which case by definition $G'$ is not reduced)
or 
a black leaf adjacent to a white trivalent vertex (in which case, as $G'$
is normal, it is not reduced, by 
\cref{rem:neverreduced}).
Either way, $G'$ is not reduced, so $G$ is not reduced either.
%	contains one of the forbidden configurations from \cref{fig:fail} 
%(cf.\ \cref{lem:normal-no-collapse}),
%so $G$ is not reduced.

Going in the other direction, let $G$ be a non-reduced normal plabic~graph.
Then either $G$ has a (necessarily) black leaf %(the only kind of leaf it can have),
or  $G\sim G'$ where $G'$ contains a hollow digon.
%	By \cref{lem:not-reduced=>osim}, 
%there exists $G'\osim G$ containing a forbidden configuration. 
%shown in \cref{fig:fail}.
%We may assume that $G'$ has no white leaves, 
%since growing those leaves does not help with producing ``forbidden configurations.''
By \cref{lem:M1M2M3-to-triple}, 
the triple diagram~$\X(G)$ is move-equivalent
to the (generalized) triple diagram~$\X(G')$. 
Since $\X(G)$ is connected, so is~$\X(G')$. 
It follows by \cref{lem:generalized-triple-diagram-connected} that $\X(G')$ is an honest triple diagram. 

If $G$ contains a black leaf, then $\X(G)$ contains a monogon, cf.\ \cref{fig:18},
hence is not minimal.
%The remaining argument depends on the type of a forbidden configuration present in~$G'$. 
%Since $G$ is normal and $G'\osim G$, it follows that $G'$ has no white leaves. 
If $G'$ contains a hollow digon with vertices of the same color,
then $\X(G')$ has a closed strand; hence $\X(G')$ is not minimal 
(by \cref{lem:nomonogondigon}) and neither is~$\X(G)$. 
If the vertices of the digon have different colors, cf.\ \cref{fig:18},
then $\X(G')$ contains a monogon, hence is not minimal. 
\end{proof}

\begin{figure}[ht]
\begin{center}
\vspace{-.2in}
\begin{tikzpicture}

\draw[thick] (0,0) -- (0.65,0);
\draw[thick] (1.6,0) -- (2.25,0);
\draw[thick] (0.65,0) to[out=80,in=180] (1.125,0.45) to[out=0,in=100] (1.6,0);
\draw[thick] (0.65,0) to[out=-80,in=180] (1.125,-0.45) to[out=0,in=-100] (1.6,0);
\node[circle, fill=black, draw=black, inner sep=0pt, minimum size=4pt] at (0.65,0) {};
\node[circle, thick, fill=white, draw=black, inner sep=0pt, minimum size=4pt] at (1.6,0) {};
\draw[red,thick, postaction={decorate, decoration = {markings, mark = between positions 0.98 and 0.99 step 0.6 with {\arrow{Stealth[length=6pt]}}}}] (2.25,0.12) to[out=180,in=20] (1.7,0.08) to[out=215,in=0] (1.125,-0.25) to[out=180,in=270] (0.85,0) to[out=90,in=180] (1.125,0.25) to[out=0,in=135] (1.7,-0.08) to[out=-20,in=180] (2.25,-0.12);
\node at (1.7,-0.35) {\small $x$};
\node at (2,0.26) {\small $e$};

\draw[thick] (4,0) -- (4.6,0);
\draw[thick] (4,0) -- (3.6,0.4);
\draw[thick] (4,0) -- (3.6,-0.4);
\node[circle, thick, fill=white, draw=black, inner sep=0pt, minimum size=4pt] at (4,0) {};
\node[circle, fill=black, draw=black, inner sep=0pt, minimum size=4pt] at (4.6,0) {};
\draw[red,thick, postaction={decorate, decoration = {markings, mark = between positions 0.11 and 0.29 step 0.6 with {\arrow{Stealth[length=6pt]}}}}] (3.5,0.35) to[out=-45,in=175] (4,0) to[out=-20,in=180] (4.6,-0.15) to[out=0,in=270] (4.8,0) to[out=90,in=0] (4.6,0.15) to[out=180,in=20] (4,0) to[out=185,in=45] (3.5,-0.35);
\node at (4.1,-0.25) {\small $x$};

\end{tikzpicture}
\vspace{-.2in}
\end{center}	
\caption{A hollow monogon in a triple diagram yields a forbidden configuration
	in the corresponding normal plabic graph.} %, cf.\ \cref{fig:fail}.}
\label{fig:18}
\end{figure}

\vspace{-.1in}

\begin{proof}[Proof of \cref{prop:fixedlollipop}]
Let $G$ be a reduced leafless plabic graph such that $\pi_G(i)\!=\!i$.
We need to show that the connected component of~$G$ containing 
the boundary vertex~$i$ %is a tree that 
%collapses to 
	is a lollipop at~$i$.
%cf.\ \cref{def:collapsible-tree}.

Suppose otherwise, that $G$ has no lollipop at $i$.  
Without loss of generality we can assume that $G$ has no lollipops at 
any other boundary vertex, since they don't affect which moves we can apply.
%Without loss of generality, we can assume that $G$ has no trees  
%collapsing to other lollipops either. 
By \cref{lem:biptri2}, 
%we can apply moves (M2) and (M3) to $G$ to get 
 $G$ is move-equivalent to a normal plabic graph~$G'$.
The trip permutations of $G$ and~$G'$ coincide with each other 
(by \cref{exercise:trip-invariant})~and with 
the strand permutation of the triple diagram~$\X(G')$ (by \cref{pr:bijplabictriple}).
Since $G$ is reduced, so is~$G'$; hence $\X(G')$ is minimal by \cref{red-minimal}.

%Consider the boundary vertex $i$ in the normal plabic graph~$G'$. 
Let $d$ be the degree of the black vertex adjacent to the boundary vertex~$i$ in~$G'$. 
	It is impossible that $d=1$, since moves (M1), (M2), (M3)
	never create degree $1$ vertices.
%
%If $d\!=\!1$, then there is a sequence of local moves relating the
%component~of~$G$ containing~$i$ 
%to the black lollipop at~$i$~in~$G'$. % via some sequence of moves.
%Since local moves preserve the number of internal faces, and a black lollipop
%has no internal faces, no (M1) move appears in the sequence of moves.
%Therefore this component must be a (collapsible) tree and we are done.
%
If $d=2$ (see \cref{fig:baddigon} on the left), 
then $\pi_G(i)\!=\!i$ implies that $\X(G')$ has a monogon,  
so it cannot be minimal, cf.\ \cref{lem:nomonogondigon}. 
If $d\ge 3$, then we get a par\-al\-lel digon  
(see \cref{fig:baddigon} on the right), again contradicting the minimality of~$\X(G')$. 
\end{proof}

\begin{figure}[ht]
\begin{center}
\vspace{-.25in}
\begin{tikzpicture}

% left
\draw[thick] (-2,0) -- (-4,0);
\node[circle, thick, fill=white, draw=black, inner sep=0pt, minimum size=4pt] (1) at (-3,0.5) {};
\node[circle, fill=black, draw=black, inner sep=0pt, minimum size=4pt] (15) at (-3,0.25) {};
\draw[thick] (1) -- (-3,0);
\draw[thick] (1) -- (-3.5,0.9);
\draw[thick] (1) -- (-2.5,0.9);
\draw[red,thick, postaction={decorate, decoration = {markings, mark = between positions 0.24 and 0.99 step 0.65 with {\arrow{Stealth[length=5pt]}}}}] (-2.7,0) to[out=90,in=-60] (-3,0.5) to[out=120,in=290] (-3.5,1.2);
\draw[red,thick, postaction={decorate, decoration = {markings, mark = between positions 0.28 and 0.99 step 0.65 with {\arrow{Stealth[length=5pt]}}}}] (-2.5,1.2) to[out=250,in=60] (-3,0.5) to[out=-120,in=90] (-3.3,0);
\draw[red, thick, densely dotted] (-3.5,1.2) to[out=90,in=180] (-3,2) to[out=0,in=90] (-2.5,1.2);

% right
\draw[thick] (2,0) -- (4,0);
\node[circle, fill=black, draw=black, inner sep=0pt, minimum size=4pt] (11) at (3,0.5) {};
\node[circle, thick, fill=white, draw=black, inner sep=0pt, minimum size=4pt] (12) at (3.4,0.9) {};
\node[circle, thick, fill=white, draw=black, inner sep=0pt, minimum size=4pt] (13) at (2.6,0.9) {};
\draw[thick] (11) -- (3,0);
\draw[thick] (11) -- (12);
\draw[thick] (11) -- (13);
\draw[thick] (13) -- (2.2,1.05);
\draw[thick] (13) -- (2.6,1.3);
\draw[thick] (12) -- (3.8,1.05);
\draw[thick] (12) -- (3.4,1.3);
\draw[red,thick, postaction={decorate, decoration = {markings, mark = between positions 0.88 and 0.99 step 0.65 with {\arrow{Stealth[length=5pt]}}}}] (2.5,1.3) to[out=270,in=120] (2.6,0.9) to[out=-60,in=90] (2.9,0);
\draw[red,thick, postaction={decorate, decoration = {markings, mark = between positions 0.28 and 0.99 step 0.68 with {\arrow{Stealth[length=5pt]}}}}] (3.1,0) to[out=90,in=-120] (3.4,0.9) to[out=60,in=270] (3.5,1.3);
\draw[red,thick, densely dotted, postaction={decorate, decoration = {markings, mark = between positions 0.55 and 0.99 step 0.68 with {\arrow[rotate=-10]{Stealth[length=5pt]}}}}] (3.5,1.3) to[out=90,in=0] (3,2) to[out=180,in=90] (2.5,1.3);
\draw[red,thick, postaction={decorate, decoration = {markings, mark = between positions 0.55 and 0.99 step 0.68 with {\arrow[rotate=10]{Stealth[length=5pt]}}}}] (3.9,1.2) to[out=-150,in=30] (3.4,0.9) to[out=-140,in=0] (3,0.7) to[out=180,in=-40] (2.6,0.9) to[out=150,in=-30] (2.1,1.2);

\end{tikzpicture}
\vspace{-.2in}
\end{center}
	\caption{The vicinity of $i$ in~$G'$.}
	\label{fig:baddigon}
\vspace{-.1in}
\end{figure}
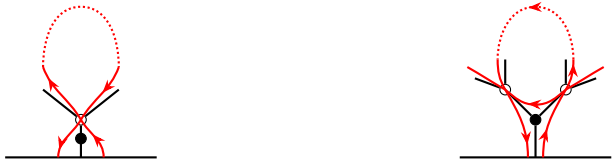
%\begin{figure}
%\includegraphics[height=2in]{IMG_0105.eps}
%	\caption{The vicinity of $i$}
%	\label{fig:baddigon}
%\end{figure}

\begin{proof}[Proof of \cref{thm:moves}]
	Let $G$ and $G'$ be reduced (leafless) plabic graphs.  
If $G\sim G'$, then $\dpi_G = \dpi_{G'}$ by \cref{exercise:forward}.
We need to show the converse. 

	Let $G$ and $G'$ be reduced (leafless) plabic graphs such that
$\dpi_G = \dpi_{G'}$.
If this decorated permutation has a fixed point at some vertex~$i$, 
then by \cref{prop:fixedlollipop}, applying local moves if needed,
both $G$ and $G'$ have a lollipop
of the same color in position~$i$.
Delete this lollipop in both graphs; the resulting
graphs are still reduced, and their decorated trip permutations coincide.  
So without loss of generality, we may assume that
$\dpi_G \!=\! \dpi_{G'}$ has no fixed points and accordingly
 $G$ and $G'$ have no %trees collapsing to 
	lollipops.  
Applying local moves as needed,  we can furthermore assume,
in light of \cref{lem:biptri2}, that both $G$ and $G'$ are normal.
Since they are reduced, \cref{red-minimal} implies 
that the triple diagrams $\mathfrak{X}(G)$ and $\mathfrak{X}(G')$ are minimal.  
By \cref{pr:bijplabictriple}, we moreover have 
$\pi_{\mathfrak{X}(G)} = \pi_G = \pi_{G'} = \pi_{\mathfrak{X}(G')}$.
Invoking \cref{thm:domino-flip}, we conclude that 
$\mathfrak{X}(G)$ and $\mathfrak{X}(G')$ are move-equivalent. 
By \cref{thm:moves-moves}
	and \cref{lem:star}, the same is true for $G$ and~$G'$.
\end{proof}

\vspace{-3pt}

\begin{proof}[Proof of \cref{cor:reduced=min-faces}]
Local moves do not change the number of faces.
It follows by \cref{thm:moves} that all reduced plabic graphs 
with a given decorated trip permutation have the same number of faces. 

Changing the color of a lollipop transforms a reduced plabic graph 
into another reduced graph with the same number of faces
and the same trip permutation (but with different decoration). 
Therefore all reduced plabic graphs~$G$ with $\pi(G)=\pi$ 
have the same number of faces. 

It remains to show that if $G$ is not reduced and has no internal leaves other than lollipops,
then there exists a plabic graph $G'$ with $\pi(G')=\pi$ and with fewer faces than~$G$. 
Since $G$ is not reduced, 
%under our assumptions on~$G$, 
%\cref{pr:reduced-collapse} applies, 
$G$ can be transformed by local moves 
that do not create internal leaves into a plabic graph~$G''$ containing a hollow digon.  
%
%We begin by removing bivalent vertices and collapsing all unicolored edges. 
%If we encounter a loop, then remove it together with everything that the loop encloses. 
%The resulting graph has fewer faces and the same trip permutation. 
%
%Now the graph is bipartite, with no internal leaves and no bivalent vertices. 
%Insert bivalent vertices near the boundary to make it normal. 
%(This isn't really necessary, but it's convenient to use the terminology.) 
%
We claim that there exists a plabic graph $G'''$
(\emph{not} move-equivalent to~$G''$) 
that has the same trip permutation as~$G''$, 
but fewer faces compared to~$G''$. 
The graph $G'''$ is constructed as follows.
If the vertices of the hollow digon in~$G''$ are of the same color, 
then remove one of the sides of the digon (keeping its vertices) to get~$G'''$. 
If, on the other hand, the vertices of the digon have different colors,
then remove both sides of the digon; 
if one of the vertices was bivalent, then remove it as well. 
It is straightforward to check that in each case, the %(decorated) 
trip permutation
does not change, whereas the number of faces decreases by~1~or~2. 
\end{proof}

\vspace{-2pt}

\begin{remark}
As we have seen, A.~Postnikov's theory of plabic graphs %from his 2006 preprint
\cite{postnikov} is closely related to 
D.~Thurston's theory of triple diagrams \cite{thurston}. %first posted to the \texttt{arXiv} in 2004.
In particular, reduced plabic graphs are essentially minimal triple diagrams in disguise.  
	If one starts with a non-reduced (leafless) plabic graph, one can 
	apply the moves together with the reduction move (R1) in order to 
	transform the graph into a reduced one.
	Similarly,
	one can apply reduction moves to a 
%	While we have not discussed it here, there are some \emph{reduction \linebreak[3]
%moves}
%that can be repeatedly applied to a non-reduced plabic graph 
 a non-minimal triple diagram
in order to eventually make it minimal.
%Similarly, there are local reduction moves that can be used to transform 
%non-minimal triple diagrams into minimal ones. 
Here, however, the two theories diverge: %when it comes to reduction moves:
reduction moves for triple diagrams preserve the strand permutation, 
but reduction moves for plabic graphs do not preserve the trip permutation.
In spite of that, reduction for plabic graphs fits 
into the theory of the totally nonnegative Grassmannian,
as it is compatible with its cell decomposition, cf.\ \cite[Section~12]{postnikov}. We will discuss this in a subsequent chapter.
\end{remark}

%\clearpage
\newpage

\section{The bad features criterion}\label{sec:bad}

In this section, we provide a criterion for deciding whether a (leafless) 
plabic graph is reduced or not. 
%We first explain (see \cref{def:normalize}) how to 
%transform an arbitrary plabic graph~$G$ into a normal plabic graph $N(G)$ move-equivalent to~$G$ 
%(or conclude that $G$ is not reduced). 
%We then use a criterion based on \cref{red-minimal} to determine 
%whether $N(G)$ (hence~$G$) is reduced or not. 

%Recall the notion of a roundtrip from \cref{def:trip}. 

\begin{lemma}
\label{lem:roundtrips} 
A reduced leafless plabic graph has no roundtrips.
%Also, $G$ has no loops, i.e., edges whose endpoints coincide. 
\end{lemma}

\begin{proof}
We may assume that our plabic graph $G$ does not contain white lollipops. 
%	or 
%trees that collapse to white lollipops.
%(Collapsing such trees and 
(Removing lollipops 
does not affect whether a graph is reduced or whether it has a roundtrip.) 
Since $G$ is leafless, by~\cref{lem:biptri2} 
it is move-equivalent 
%(up to the removal of some white lollipops) 
to a normal plabic graph~$G'$.
Since $G'$ is reduced, $\X(G')$ is minimal (see \cref{red-minimal}).
Hence $\X(G')$ has no closed strands (see \cref{lem:nomonogondigon}),  
so $G'$ has no roundtrips. 
Since roundtrips persist under local moves, 
$G$ has no roundtrips either. 
%Suppose $G$ has a loop~$e$ based at a black (resp., white) vertex. 
%(Some edges and vertices might be enclosed by~$e$.)
%Then the trip that traverses~$e$ clockwise (resp., counterclockwise) is a roundtrip,
%a contradiction. 
\end{proof}

\begin{definition}
\label{def:bad-features}
If a trip passes through an edge~$e$ of a plabic graph 
twice (in the opposite directions), 
we call this an \emph{essential self-intersection}.  

If for two edges $e_1$ and $e_2$, there are two distinct trips each of~which
passes first through~$e_1$ and then through~$e_2$,
we call this a \emph{bad double crossing}.

We use the term \emph{bad features} 
to collectively refer to 
\begin{itemize}[leftmargin=.2in]
\item 
roundtrips (see \cref{def:trip}), 
\item
essential self-intersections, and
\item
bad double crossings.  
\end{itemize}
These notions are illustrated in \cref{fig:bad-features}. 
\end{definition}

\begin{figure}[ht]
\begin{center}
%\vspace{.2in}
\begin{tabular}{ccc}
\begin{tabular}{c}
\ \\[-.6in]
\setlength{\unitlength}{.6pt}
%\begin{picture}(60,65)(0,-3)
\begin{picture}(50,55)(10,-3)
\thicklines
\put(20,0){\circle*{6}}
\put(50,0){\circle{6}}
\put(0,20){\circle{6}}
\put(20,40){\circle*{6}}
\put(50,40){\circle{6}}
\put(70,20){\circle*{6}}

\put(23,0){\line(1,0){24}}
\put(23,40){\line(1,0){24}}
\put(2,18){\line(1,-1){16}}
\put(2,22){\line(1,1){16}}
\put(68,18){\line(-1,-1){16}}
\put(68,22){\line(-1,1){16}}
\put(-2,20){\line(-1,0){15}}
\put(68,20){\line(-1,0){15}}
\put(52,-2){\line(1,-1){10}}
\put(52,42){\line(1,1){10}}
\put(22,2){\line(1,1){10}}
\put(22,38){\line(1,-1){10}}
\end{picture}
\\[.1in]
\setlength{\unitlength}{1pt}
\begin{picture}(45,20)(0,10)
\thicklines
\put(15,15){\circle*{4}}
\put(30,15){\circle*{4}}
\put(13,15){\line(-1,0){10}}
\put(32,15){\line(1,0){10}}
%\put(34,17){$e$}
\qbezier(15,15)(22,30)(30,17)
\qbezier(15,15)(22,0)(30,13)
\end{picture}
\end{tabular}
&
\qquad {\ }
\begin{tabular}{c}
\ \\[-.5in]
\setlength{\unitlength}{1pt}
\begin{picture}(45,20)(0,10)
\thicklines
\put(15,15){\circle*{4}}
\put(30,15){\circle{4}}
\put(13,15){\line(-1,0){10}}
\put(32,15){\line(1,0){10}}
\put(34,17){$e$}
\qbezier(15,15)(22,30)(30,17)
\qbezier(15,15)(22,0)(30,13)
\end{picture}
\\[.1in]
\setlength{\unitlength}{1pt}
\begin{picture}(45,20)(-5,10)
\thicklines
\put(15,15){\circle{4}}
\put(30,15){\circle*{4}}
\put(20,17){$e$}
\put(17,15){\line(1,0){11}}
\put(13,15){\line(-2,1){9}}
\put(13,15){\line(-2,-1){9}}
\end{picture}
\end{tabular}
\qquad {\ }
&
\setlength{\unitlength}{.6pt}
\begin{picture}(50,55)(7,-3)
\thicklines

\put(0,0){\circle{6}}
\put(0,30){\circle*{6}}
\put(30,0){\circle*{6}}
\put(30,30){\circle{6}}
\put(60,0){\circle{6}}
\put(60,30){\circle*{6}}

\put(2.5,0){\line(1,0){25}}
\put(32,0){\line(1,0){26}}
\put(2,30){\line(1,0){26}}
\put(32,30){\line(1,0){26}}

\put(0,2){\line(0,1){26}}
\put(30,2){\line(0,1){26}}
\put(60,2){\line(0,1){26}}

\put(-2,-2){\line(-1,-1){15}}
\put(-2,32){\line(-1,1){15}}
\put(62,-2){\line(1,-1){15}}
\put(62,32){\line(1, 1){15}}

\put(-20,14){$e_1$}
\put(65,14){$e_2$}

\end{picture}
\\[.3in]
(a) & (b) & (c) 
\end{tabular}
\end{center}
\vspace{-.2in} 
\caption{Plabic graph fragments representing ``bad features:'' (a)~a~roundtrip;
(b) essential self-intersection; (c) bad double crossing.
The fragment at the bottom of column (b) may not appear
in a leafless plabic graph, but may occur in a normal plabic graph.}
\label{fig:bad-features}
\end{figure}

\begin{lemma}
\label{lem:badfeatures}
A normal plabic graph $G$ has a bad feature %(see \cref{def:bad-features}) 
if and only if the associated triple diagram $\X(G)$ has a badgon.
%
%More specifically,
%\begin{itemize}[leftmargin=.2in]
%\item 
%$G$ has a roundtrip if and only if $\X(G)$ has a closed strand; 
%\item
%$G$ has an essential self-intersection if and only if $\X(G)$ has a monogon; 
%\item
%$G$ has a bad double crossing %(which might involve a roundtrip) 
%if and only if $\X(G)$ has a parallel digon.
%\end{itemize}
\end{lemma}

\begin{proof}
Let $G$ be a normal plabic graph. 
The strands in the triple diagram $\X=\X(G)$ closely follow the trips in~$G$. 
Therefore $\X$ has a closed strand if and only if $G$ has a roundtrip. 

If $G$ has an essential self-intersection (resp., a bad double crossing),
then $\X$ has a monogon (resp., a parallel digon).
To see that, take each  edge~$e$ involved in a bad feature
and consider the white end~$v$ of~$e$. 
The strands corresponding to the trips involved in the bad feature 
will intersect at~$v$; thus $v$ will be a vertex of the corresponding badgon. 
Cf.\ Figures~\ref{fig:18} and~\ref{fig:17}.

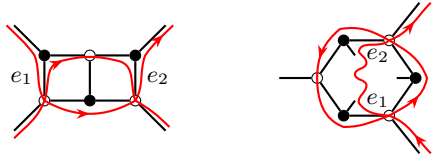
\begin{figure}[ht]
\begin{center}
\vspace{-.1in}
\begin{tikzpicture}

%\draw[thick] (-6.1,0) circle (0.28cm);
%\draw[thick] (-5.8,0) -- (-5.3,0);
%\draw[thick] (-6.4,0) -- (-6.9,0);
%\node[circle, thick, fill=white, draw=black, inner sep=0pt, minimum size=4pt] at (-5.8,0) {};
%\node[circle, fill=black, draw=black, inner sep=0pt, minimum size=4pt] at (-6.4,0) {};
%\node at (-5.5,0.25) {\small $e$};
%\draw[red,thick, postaction={decorate, decoration = {markings, mark = between positions 0.97 and 0.99 step 0.6 with {\arrow{Stealth[length=5pt]}}}}] (-5.3,0.1) to[out=185,in=45] (-5.8,0) to[out=215,in=0] (-6.1,-0.2) to[out=180,in=270] (-6.3,0) to[out=90,in=180] (-6.1,0.2) to[out=0,in=135] (-5.8,0) to[out=-45,in=175] (-5.3,-0.1);

%\draw[thick] (-4,0) -- (-3.4,0);
%\draw[thick] (-4,0) -- (-4.4,0.4);
%\draw[thick] (-4,0) -- (-4.4,-0.4);
%\node[circle, thick, fill=white, draw=black, inner sep=0pt, minimum size=4pt] at (-4,0) {};
%\node[circle, fill=black, draw=black, inner sep=0pt, minimum size=4pt] at (-3.4,0) {};
%\draw[red,thick, postaction={decorate, decoration = {markings, mark = between positions 0.11 and 0.29 step 0.6 with {\arrow{Stealth[length=6pt]}}}}] (-4.5,0.35) to[out=-45,in=175] (-4,0) to[out=-20,in=180] (-3.4,-0.15) to[out=0,in=270] (-3.2,0) to[out=90,in=0] (-3.4,0.15) to[out=180,in=20] (-4,0) to[out=185,in=45] (-4.5,-0.35);
%\node at (-3.8,0.25) {\small $e$};

\node[circle, fill=black, draw=black, inner sep=0pt, minimum size=4pt] (1) at (-0.6,0.3) {};
\node[circle, thick, fill=white, draw=black, inner sep=0pt, minimum size=4pt] (2) at (0,0.3) {};
\node[circle, fill=black, draw=black, inner sep=0pt, minimum size=4pt] (3) at (0.6,0.3) {};
\node[circle, thick, fill=white, draw=black, inner sep=0pt, minimum size=4pt] (4) at (-0.6,-0.3) {};
\node[circle, fill=black, draw=black, inner sep=0pt, minimum size=4pt] (5) at (0,-0.3) {};
\node[circle, thick, fill=white, draw=black, inner sep=0pt, minimum size=4pt] (6) at (0.6,-0.3) {};
\draw[thick] (1) -- (2); \draw[thick] (2) -- (3);
\draw[thick] (4) -- (5); \draw[thick] (5) -- (6);
\draw[thick] (1) -- (4); \draw[thick] (2) -- (5);
\draw[thick] (3) -- (6);
\draw[thick] (1) -- (-1,0.7); 
\draw[thick] (3) -- (1,0.7);
\draw[thick] (4) -- (-1,-0.7); 
\draw[thick] (6) -- (1,-0.7);
\draw[red,thick, postaction={decorate, decoration = {markings, mark = between positions 0.5 and 0.99 step 0.45 with {\arrow{Stealth[length=5pt]}}}}] (-1.1,0.7) to[out=-45,in=125] (-0.73,0.3) to[out=290,in=120] (-0.6,-0.3) to[out=-30,in=210] (0.6,-0.3) to[out=60,in=250] (0.73,0.3) to[out=55,in=225] (1.1,0.7);
\draw[red,thick, postaction={decorate, decoration = {markings, mark = between positions 0.4 and 0.79 step 0.45 with {\arrow[rotate=10]{Stealth[length=5pt]}}}}] (-1,-0.8) to[out=45,in=250] (-0.6,-0.3) to[out=70,in=200] (-0.4,0.22) to[out=15,in=165] (0.4,0.22) to[out=-20,in=110] (0.6,-0.3) to[out=330,in=135] (1.05,-0.65);
\node at (-0.9,0) {\footnotesize $e_1$};
\node at (0.9,0) {\footnotesize $e_2$};

\node[circle, thick, fill=white, draw=black, inner sep=0pt, minimum size=4pt] (11) at (3,0) {};
\node[circle, fill=black, draw=black, inner sep=0pt, minimum size=4pt] (12) at (3.35,0.5) {};
\node[circle, thick, fill=white, draw=black, inner sep=0pt, minimum size=4pt] (13) at (3.95,0.5) {};
\node[circle, fill=black, draw=black, inner sep=0pt, minimum size=4pt] (14) at (4.3,0) {};
\node[circle, thick, fill=white, draw=black, inner sep=0pt, minimum size=4pt] (15) at (3.95,-0.5) {};
\node[circle, fill=black, draw=black, inner sep=0pt, minimum size=4pt] (16) at (3.35,-0.5) {};
\draw[thick] (11) -- (12); \draw[thick] (12) -- (13);
\draw[thick] (13) -- (14); \draw[thick] (14) -- (15);
\draw[thick] (15) -- (16); \draw[thick] (16) -- (11);
\draw[thick] (11) -- (2.5,0);
\draw[thick] (12) -- (3.5,0.3);
\draw[thick] (13) -- (4.35,1);
\draw[thick] (14) -- (4.05,0);
\draw[thick] (15) -- (4.35,-1);
\draw[thick] (16) -- (3.5,-0.3);
\draw[red,thick, postaction={decorate, decoration = {markings, mark = between positions 0.1 and 0.99 step 0.8 with {\arrow{Stealth[length=5pt]}}}}] (4.5,-1) to[out=135,in=-20] (3.95,-0.5) to[out=175,in=-55] (3.55,-0.33) to[out=100,in=270] (3.65,-0.1) to[out=90,in=270] (3.5,0.05) to[out=90,in=270] (3.65,0.2) to[out=90,in=-60] (3.55,0.35) to[out=80,in=-135] (3.95,0.5) to[out=20,in=-135] (4.5,1);
\draw[red,thick, postaction={decorate, decoration = {markings, mark = between positions 0.44 and 0.99 step 0.5 with {\arrow{Stealth[length=5pt]}}}}] (3,0) to[out=270,in=160] (3.35,-0.65) to[out=0,in=-160] (3.95,-0.5) to[out=30,in=260] (4.45,0) to[out=110,in=-30] (3.95,0.5) to[out=150,in=10] (3.3,0.65) to[out=-140,in=90] (3,0);
\node at (3.8,-0.3) {\footnotesize $e_1$};
\node at (3.8,0.3) {\footnotesize $e_2$};

\end{tikzpicture}
\vspace{-.3in}
\end{center}
\caption{
%An essential self-intersection, a 
A bad double crossing in $G$ yields a
parallel digon in~$\X(G)$.}
\label{fig:17}
\end{figure}

Conversely, suppose that $\X$ has a monogon with self-intersection
corresponding to the white vertex~$v$ of~$G$. There are three 
strand segments of $\X$ that pass through~$v$, each running along 
two distinct 
edges incident to~$v$; 
 because we have a self-intersection, two of these
strands segments are part of the same strand $s$.  
%Each of these two strand segments includes two distinct edges incident to~$v$.
Since $v$ is trivalent, the pigeonhole principle implies that 
two of the four edges that $s$ runs along
must coincide.
This yields an essential self-intersection in~$G$.
A~similar argument shows that if  $\X$ has a parallel digon,
then $G$ has a bad double crossing. 
%
%This means that $\X$ is not minimal (see \cref{thm:minimal}). 
%Therefore, by \cref{cor:Dyl}, $\X$~is move-equivalent to 
%a triple diagram $\X'$ containing an hollow monogon.
%The corresponding normal plabic graph $G'=G(\X')$ contains a bad feature
%(either a bad double crossing or an essential self-intersection), 
%see \cref{fig:18}. 
%Since $G$ and $G'$ are related via a sequence of urban renewal or normal flip moves
%(see \cref{thm:moves-moves}), 
%it follows by \cref{lem:bad-features-normal-moves}
%that $G$ also contains a bad feature. 
\end{proof}

\begin{corollary}
\label{cor:bad<=>bad}
Let $G$ be a normal plabic graph. 
Let $\X=\X(G)$ be the corresponding triple diagram.
Then the following are equivalent:
\begin{itemize}[leftmargin=.2in]
\item 
$G$ is reduced;
\item
$\X$ is minimal;
\item 
$G$ has no bad features;
\item
$\X$ has no badgons. 
\end{itemize}
\end{corollary}

\begin{proof}
By \cref{red-minimal}, %the normal plabic graph 
$G$ is reduced if and only if %the triple diagram 
$\X$~is minimal. 
By virtue of \cref{thm:minimal}, $\X$ is minimal if and only if $\X$ has no badgons. 
By \cref{lem:badfeatures}, $\X$ has no badgons if and only if $G$ has no bad features. 
\end{proof}

The following result  is a version of \cite[Theorem 13.2]{postnikov}. 

\begin{theorem}
\label{thm:reduced}
A normal plabic graph %without internal leaves. % (other than lollipops).
%cf.\ \cref{def:collapsible-tree} and \cref{lem:collapse-tree}.
%(Lollipops are allowed.) 
%Then $G$ 
 is reduced if and only if it does not contain any bad features. 
 A leafless plabic graph is reduced if and only if it does not contain any bad features. 
\end{theorem}
\begin{proof}
	The first statement is a consequence of 
\cref{cor:bad<=>bad}.
The second statement follows from the first, using 
\cref{lem:biptri2}, plus the fact that the moves relating a leafless plabic graph
to a normal plabic graph neither add nor remove  bad features.
\end{proof}

For example, any plabic graph containing one of the fragments shown 
in \cref{fig:bad-features} is necessarily not reduced. 

\pagebreak[3]

%\begin{remark}
%\label{rem:}
%For \emph{any} plabic graph~$G$, 
%\cref{thm:reduced} can be used in conjunction with the procedure 
%described in \cref{def:normalize}
%to determine whether $G$ is reduced or not. 
%\end{remark}

%The following examples illustrate \cref{thm:reduced}: 

\begin{remark}
Recall from Remark~\ref{rem:reduceddecomp} that an expression (possibly non-reduced) 
of an element of a symmetric group
as a product of simple reflections can be represented by %(a version of) 
a wiring diagram.  
%In this case the reduced decompositions 
% correspond to collections of $n$ piecewise-straight lines
%such that each pair of lines intersects at most once.
As plabic graphs can be viewed as generalizations of wiring diagrams (see \cref{def:wiringplabic}), 
reduced plabic graphs may be viewed as a generalization of 
reduced expressions.
In this context, the criterion of \cref{thm:reduced} corresponds 
to the condition that each pair of lines in the wiring diagram
intersect at most once.
\end{remark}

\section{Affine permutations}
\label{sec:affine}

By \cref{thm:moves}, move-equivalence classes of reduced plabic graphs 
are labeled by decorated permutations. 
An alternative labeling utilizes 
($(a,b)$-bounded) \emph{affine permutations}, 
introduced and studied in this section. 

\begin{definition}
\label{def:anti}
For a decorated permutation $\dpi$ on $b$ letters, we say that $i\in \{1,\dots,b\}$ is
an \emph{anti-excedance} of $\dpi$ if either $\dpi^{-1}(i)>i$ or  $\dpi(i)=\overline{i}$.  
%The number of anti-excedances of~$\dpi$ (which we usually denote by~$a$)
%is equal to the number of values 
%$i\in \{1,\dots,b\}$ such that $\dpi(i)<i$ or $\dpi(i)=\overline{i}$.  
\end{definition}

We will usually let $a$ denote the number of anti-excedances. % of a decorated permutation.

\begin{example}
\label{example:523641}
The decorated permutation $\dpi = (5,\underline{2},\overline{3},6,4,1)$ on $b=6$ letters 
(cf.\ \cref{fig:plabic3}) has $a=3$ anti-excedances, namely, $3$, $4$, and $1$.
%	$1$, $4$, and~$\overline{3}$. 
%	Indeed, $\dpi^{-1}(1)=6>1$, and $\dpi^{-1}(4)=5>4$. 
\end{example}

\begin{definition}
\label{def:affinization}
Let $\dpi$ be a decorated permutation on $b$ letters with $a$ anti-excedances. 
The \emph{affinization} of $\dpi$ 
is the map ~$\affpi:\ZZ\to\ZZ$ constructed as follows.
For $i\in\{1,\dots,b\}$, we set 
%If $\dpi(i)<i$, set $\affpi(i):=\dpi(i)+b$.
%If $\dpi(i)=\overline{i}$  set
%$\affpi(i):=i$.
%If $\dpi(i)=\underline{i}$ set
%$\affpi(i):=i+b$.
%Finally, if $\dpi(i)>i$, set $\affpi(i):=\dpi(i)$.
	%\begin{equation*}
	%	\affpi(i)= \begin{cases}
	%		\dpi(i) &\text{ if }\dpi(i)>i \text{ or }\dpi(i)=\underline{i}, \\
	%		\dpi(i)+b &\text{ if } \dpi(i)<i \text{ or }\dpi(i)=\overline{i}.
	%	\end{cases}
	%\end{equation*}
	\begin{equation*}
		\affpi(i)= \begin{cases}
			\dpi(i) &\text{ if }\dpi(i)>i,\\ 
			i &\text{ if }\dpi(i)=\underline{i}, \\
			\dpi(i)+b &\text{ if } \dpi(i)<i,\\
			i+b &\text{ if }\dpi(i)=\overline{i}.
		\end{cases}
	\end{equation*}
%to define a map $\{1,2,\dots,b\} \to \{1, 2,\dots, 2b\}$ such that $i \leq \affpi(i) \leq i+b$; 
We then extend $\affpi$ to~$\ZZ$ so that it satisfies
\begin{equation}
\label{eq:affpi+b}
\affpi(i+b)=\affpi(i)+b \quad (i\in\ZZ). 
\end{equation}
We note that 
\begin{align}
\label{eq:i<f(i)<i+b}
&i \leq \affpi(i) \leq i+b \quad (i\in\ZZ)  , \\
\label{eq:sum-aex}
&\sum_{i=1}^b (\affpi(i)-i) = b\cdot \#\{ i\in\{1,\dots,b\} \mid \dpi(i)<i \text{\ or\ } \dpi(i)=\overline{i}\}
= ab. 
\end{align}
\end{definition}

\begin{example}
\label{ex:affinization}
Continuing with $\dpi \!=\! (5,\underline{2},\overline{3},6,4,1)$ 
from \cref{example:523641}, we get 
%the affinization $\affpi$ given by 
$\affpi(1)\!=\!5$, $\affpi(2)\!=\!2$, $\affpi(3)\!=\!9$, $\affpi(4)\!=\!6$, $\affpi(5)\!=\!10$, $\affpi(6)\!=\!7$,
or more succinctly,
\begin{equation*}
\affpi=(\dots,5,2,9,6,10,7,\dots)=(\cdots\  \boxed{5\ 2\ 9\ 6\ 10\ 7}\ 11\ 8\ 15\ 12\ 16\ 13\ \cdots).
\end{equation*} 
(The boxed terms are the values at $1,\dots,b$. 
They determine the rest of the sequence by virtue of~\eqref{eq:affpi+b}.) 
In accordance with \eqref{eq:sum-aex}, we have
\begin{equation*}
(5+2+9+6+10+7)-(1+\cdots+6)=39-21=18=3\cdot 6=ab. 
\end{equation*}
\end{example}

With the above construction in mind, we introduce the following notion. 

\begin{definition}
\label{def:affine}
Let $a$ and $b$ be positive integers. 
An~\emph{$(a,b)$-bounded affine permutation} 
is a bijection $f:\ZZ \to \ZZ$ satisfying the following conditions: 
\begin{itemize}[leftmargin=.2in]
\item 
$f(i+b)=f(i)+b\,$ for all $i\in \ZZ$; 
\item
$i \leq f(i) \leq i+b\,$ for all $i\in\ZZ$; 
\item
$\sum_{i=1}^b (f(i)-i) = ab$.
\end{itemize}
\end{definition}

\begin{lemma}\cite{KLS}
\label{lem:affinization-type}
The correspondence $\dpi\mapsto \affpi$ (see \cref{def:affinization}) restricts to 
 a bijection between decorated permutations on $b$ letters with $a$ anti-excedances and 
the $(a,b)$-bounded affine permutations. 
\end{lemma}

\cref{lem:affinization-type} is illustrated in \cref{fig:affinization} (the first two columns). 

\begin{figure}[ht]
\vspace{-.1in}
\[
\begin{array}{|c|c|c|}
                \hline
\dpi & \affpi & \ell(\affpi)  \\
\hline 
&& \\[-.1in]
\overline{1}\ \underline{2}\ \underline{3} & \cdots \boxed{4\ 2\ 3}\ 7\ 5\ 6\ \cdots & 2 \\[.02in]
\underline{1}\ \overline{2}\ \underline{3} & \cdots \boxed{1\ 5\ 3}\ 4\ 8\ 6\ \cdots & 2 \\[.02in]
\underline{1}\ \underline{2}\ \overline{3} & \cdots \boxed{1\ 2\ 6}\ 4\ 5\ 9\ \cdots & 2 \\[.02in]
2\ 1\ \underline{3} & \cdots \boxed{2\ 4\ 3}\ 5\ 7\ 6\ \cdots & 1 \\[.02in]
\underline{1}\ 3\ 2 & \cdots \boxed{1\ 3\ 5}\ 4\ 6\ 8\ \cdots & 1 \\[.02in]
3\ \underline{2}\ 1 & \cdots \boxed{3\ 2\ 4}\ 6\ 5\ 7\ \cdots & 1 \\[.02in]
2\ 3\ 1 & \cdots \boxed{2\ 3\ 4}\ 5\ 6\ 7\ \cdots & 0 \\[.02in]
\hline
\end{array}
\]
\vspace{-.15in}
\caption{Decorated permutations $\dpi$ on $b\!=\!3$ letters with $a\!=\!1$ anti-excedance; 
the corresponding $(a,b)$-bounded affine permutations~$\affpi$;
and the lengths $\ell(\affpi)$ of these affine permutations, cf.
	\cref{def:inv-affpi}.}
\label{fig:affinization}
\vspace{-.15in}
\end{figure}

\begin{proof}
If $\dpi$ is a decorated permutation on $b$ letters with $a$ anti-excedances,
then \eqref{eq:affpi+b}--\eqref{eq:sum-aex}
show that $\affpi$ is an $(a,b)$-bounded affine permutation.

	Conversely, given an $(a,b)$-bounded affine permutation $f:\ZZ\to\ZZ$,
we can define the decorated permutation $\dpi$ on $b$ letters by
\begin{equation*}
\dpi(i)=\begin{cases}
\underline{i} & \text{if $f(i)=i$;} \\
\overline{i} & \text{if $f(i)=i+b$;} \\
	f(i) & \text{if $f(i)\le b$ and $f(i)\neq i$;}\\
	f(i)-b & \text{if $f(i)> b$ and $f(i) \neq i+b$.} 
\end{cases}
\end{equation*}
We claim that~$\dpi$ has $a$ anti-excedances.
Using the inequality $i\le f(i)\le i+b$, 
we conclude that the anti-excedances of~$\dpi$ 
are in bijection with the values $i\in\{1,\dots,b\}$ such that $f(i)>b$.
The claim follows from the observation that 
	%On~the other hand, 
$ab=\sum_1^b (f(i)-i)= 
	b \cdot \#\{i\in\{1,\dots,b\} \mid f(i)>b\}$.
	%there are exactly $a$ such values. 
\end{proof}

%\pagebreak[3]

Recall from \cref{ex:enumeration} that the number of decorated 
permutations on $b$ letters is equal to $b!\sum_{k=0}^b \frac{1}{k!}$. 
We next refine this formula by taking into account the number of anti-excedances. 

Let $D_{a,b}$ denote the number of decorated permutations on $b$ letters with $a$ anti-excedances 
(or the number of $(a,b)$-bounded affine permutations, cf.\ \cref{lem:affinization-type}). 
The following result is reproduced here without a proof. 

\begin{proposition}[{\rm 
	\cite[Proposition~23.1]{postnikov}}] 
\label{pr:postnikov-23.1}
%	Let $D_{a,b}$ be as in \cref{pr:A_{a,b}(1)}.  Then
We have
\begin{equation*}
\sum_{0\le a\le b} D_{a,b} \,\, x^a \, \frac{y^b}{b!} = e^{xy}\, \frac{x-1}{x-e^{y(x-1)}}.
\end{equation*}
\end{proposition}

\begin{definition}
\label{def:inv-affpi}
An \emph{inversion} of 
$\affpi$ is a pair of integers $(i,j)$ such that $i<j$ and $\affpi(i)>\affpi(j)$.  
Two inversions $(i,j)$ and $(i',j')$ are \emph{equivalent}
if $i'-i=j'-j\in b\ZZ$.  
The \emph{length}
$\ell(\affpi)$ of $\affpi$ is the number of equivalence classes of inversions.
(We note that $\ell(\affpi)$ equals the number
$\algn(\dpi)$ of \emph{alignments} 
 of~$\dpi$, as defined in~\cite{postnikov}.)
This number is finite since for any inversion $(i,j)$, we have $i<j<i+b$. 
Indeed, if $j\ge i+b$, then 
$\affpi(j)\ge j\ge i+b\ge \affpi(i)$. 
See~\cref{fig:affinization}. 
\end{definition}

We will now state, without proof, a refinement of Proposition~\ref{pr:postnikov-23.1} 
that enumerates decorated 
permutations on $b$ letters with respect to both the number of 
anti-excedances and the number of inversions. 
To this end, we~set
\begin{equation*}
D_{a,b}(q)= \sum_{\dpi} q^{a(b-a)-\algn(\dpi)} = \sum_{\affpi} q^{a(b-a) - \ell(\affpi)},
\end{equation*}
where the 
first sum is over all decorated permutations on $b$ letters with $a$ anti-excedances, and the second 
sum is over all $(a,b)$-bounded affine permutations.  
The significance of this polynomial is that the coefficient of $q^r$ in $D_{a,b}(q)$ is 
the number of $r$-dimensional positroid cells in the totally nonnegative Grassmannian
	$\Gr_{a,b}^{\geq 0}$, see~\cite{Williams}.

%Let $[j]=1+q+q^2+ \dots + q^{j-1}$ denote the \emph{$q$-analogue} of~$j$.

\begin{theorem}[{\rm 
	\cite[Theorem 4.1]{Williams}}]
\label{th:LW}
We have
%	{\small
\begin{equation*}
	D_{a,b}(q)=q^{-a^2} \sum_{i=0}^{a-1} (-1)^i \binom{b}{ i} (q^{ai} [a-i]^i  [a-i+1]^{b-i} - q^{(a+1)i}[a-i-1]^i [a-i]^{b-i}), 
\end{equation*} 
where we use the ``$q$-analogue'' notation $[j]=1+q+q^2+ \dots + q^{j-1}$. 
\end{theorem}

We next introduce an important special class of bounded affine permutations. 

\begin{definition}
Let $\affpi$ be an $(a,b)$-bounded affine permutation,
an affinization of a decorated permutation~$\dpi$, cf.\ \cref{lem:affinization-type}. 
We refer to a position $i\in\ZZ$ such that $\affpi(i)\equiv i\bmod b$ 
	(in other words, $\affpi(i)\in\{i,i+b\}$; and if $1\leq i \leq b$ then
 $\dpi(i  )\in\{\underline{i},\overline{i}\}$) 
as a \emph{fixed point} of~$\affpi$.
If every $i\in\ZZ$ is a fixed point of~$\affpi$,
then we say that $\affpi$ is \emph{equivalent to the identity modulo~$b$} 
(or that $\dpi$ is a \emph{decoration of the identity}). 
\end{definition}

\begin{lemma}
\label{lem:inv-affpi-bound}
Let $\affpi$ be an $(a,b)$-bounded affine permutation
that is equivalent to the identity modulo~$b$. 
Then $\ell(\affpi)= a(b-a)$. 
\end{lemma}

\begin{proof}
Let
$I=\{i\in\{1,\dots,b\}\mid \affpi(i)=i+b\}$ and 
$\underline{I}
=\{i\in\{1,\dots,b\}\mid \affpi(i)=i\}
%=\{1,\dots,b\}\setminus I
$. 
Then $|I|=a$ %(that is, $I$ has $a$ elements) 
and $|\underline{I}|=b-a$. 
The equivalence classes of inversions of~$\affpi$ are described
by the following list of representatives:
\begin{equation*}
\{(i,j)\in I\times\underline{I} \mid 1\le i<j\le b\} \cup \{(i,j+b)\mid (i,j)\in I\times\underline{I},  1\le j<i\le b\}. 
\end{equation*}
The cardinality $\ell(\affpi)$ of this set is equal to $|I\times\underline{I}|=a(b-a)$. 
\end{proof}

\begin{lemma}
\label{lem:affpi-adjacent}
If $\affpi$ is not equivalent to the identity modulo~$b$, then there exist $i,j\in\ZZ$
such that 
\begin{align}
\label{eq:ij-pair-1}
&1\le i<j\le b, \\
\label{eq:ij-pair-2}
&\affpi(i)<\affpi(j), \\
\label{eq:ij-pair-3}
&\text{every position $h$ such that $i<h<j$ is a fixed point of~$\affpi$, and} \\
\label{eq:ij-pair-4}
&\text{neither $i$ nor $j$ are fixed points of~$\affpi$.} 
\end{align}
\end{lemma}

\begin{proof}
Suppose such a pair $(i,j)$ does not exist.
Let $i_1<\dots<i_m$ be the elements of $\{1,\dots,b\}$ that are not fixed points of~$\affpi$.
Then %$\affpi(i_1)\ge\cdots\ge\affpi(i_m)$, so that
\begin{equation*}
i_1<\dots<i_m<\affpi(i_m)\le\cdots\le\affpi(i_1). 
\end{equation*}
We conclude that none of the values $\affpi(i_j)$ 
is of the form~$i_\ell$ and consequently is of the form $i_\ell+b$. 
In particular, $\affpi(i_j)=i_m+b$ for some~$j\neq m$. 
This implies  $\affpi(i_j) > i_j+b$, a contradiction. 
\end{proof}

%\pagebreak

We next describe an algorithm for factoring affine permutations
that will be used in Section~\ref{sec:bridge}. 

\begin{definition}
\label{def:BCFW0}
Let $\affpi$ be an $(a,b)$-bounded affine permutation. 
\begin{itemize}[leftmargin=.2in]
\item If $\affpi$ is not equivalent to the identity modulo~$b$, 
then use \cref{lem:affpi-adjacent} to find positions $i,j\in\ZZ$ 
satisfying \eqref{eq:ij-pair-1}--\eqref{eq:ij-pair-4}. 
\item Swap the values of~$\affpi$ in positions $i$ and~$j$ 
(and more generally, in positions $i+mb$ and $j+mb$, for all~$m\in\ZZ$). 
%Any entries in the resulting permutation
%which are fixed points are designated as \emph{frozen}, and henceforth ignored.
\item Repeat this procedure until we obtain an affine permutation 
that is equivalent to the identity modulo~$b$.
\end{itemize}
The ordered list of transpositions $(ij)$ produced by the above algorithm is 
called the \emph{bridge factorization} of $\affpi$.
\end{definition}

An example of a bridge factorization is shown in \cref{fig:factorization-bridges}. 

%The resulting permutation will be equivalent to the identity modulo~$b$. 

%        \renewcommand{\arraystretch}{1.2}
\begin{figure}[ht]
\vspace{-.1in}
\[
\begin{array}{| cccccc | c | c |}
                \hline
 1 & 2 & 3 & 4 & 5 & 6 & (i,j) & \begin{array}{c} \text{number of}\\ \text{inversions}\end{array}  \\
%(i,j) & \downarrow & \downarrow & \downarrow & \downarrow & \downarrow & \downarrow   \\
                \hline 
                &&&&&& & \\[-10pt]
               4 & 6 & 5 & 7 & 8 & 9 & & 1 \\
               &&&&&& (34) & \\
               4  & 6 &  7 & 5 & 8 & 9 & & 2 \\
               &&&&&&  (23) & \\
               4 & 7 & 6 & 5 & 8 & 9 & & 3\\
               &&&&&&  (12) & \\
                \boxed{7}  & 4 & 6 & 5 & 8 & 9 & & 4\\
               &&&&&& (56) & \\
               \boxed{7}  & 4 & 6 & 5 & 9 & 8 & & 5\\
               &&&&&&  (45) & \\
               \boxed{7}  & 4 & 6 & 9 & \boxed{5} & 8 & & 6\\
               &&&&&& (34) & \\
               \boxed{7}  & 4 & \boxed{9}  & 6 & \boxed{5}  & 8 & & 7\\
               &&&&&&  (46) & \\
               \boxed{7}  & 4 & \boxed{9}  & 8 & \boxed{5} & \boxed{6} & & 8\\
               &&&&&&  (24) & \\
               \boxed{7}  & \boxed{8} & \boxed{9}  & \boxed{4} & \boxed{5} & \boxed{6} & & 9\\[-10pt]
                &&&&&& & \\
 \hline
\end{array}
\]
\vspace{-5pt}
\caption{Applying the algorithm described  in \cref{def:BCFW0} to the $(a,b)$-bounded affine permutation
$\affpi = (4,6,5,7,8,9)$, with $b=6$ and $a=3$. 
%The associated decorated permutation is $\dpi=(4,6,5,1,2,3)$.
The resulting bridge factorization is the sequence
$(34), (23), (12), (56), (45), (34), (46), (24)$.
The entries corresponding to fixed points are boxed.
%The bridge factorization of $\affpi$ is 
%}
	\label{fig:factorization-bridges}}
        \end{figure}

\begin{remark}
\label{rem:l(affpi)-bound}
In view of~\eqref{eq:sum-aex},  
the affine permutation at hand remains $(a,b)$-bounded 
after each step of the algorithm in \cref{def:BCFW0}.  
Moreover, the  algorithm in \cref{def:BCFW0}
terminates because each swap increases 
the length of the affine permutation by~$1$;
this number is bounded by \cref{def:inv-affpi}. 
\end{remark}

\begin{proposition}\label{prop:maxlength}
Among all $(a,b)$-bounded affine permutations~$\affpi$,
the ones that have the maximal possible length $\ell(\affpi)=a(b-a)$
are precisely the ones that are equivalent to the identity modulo~$b$. 
\end{proposition}
\begin{proof}
This follows from \cref{lem:inv-affpi-bound} and 
\cref{rem:l(affpi)-bound}.
\end{proof}

\newpage

\section{Bridge decompositions}
\label{sec:bridge}

Bridge decompositions \cite[Section 3.2]{amplitudes} 
provide a useful 
recursive construction of reduced plabic graphs with a given decorated trip permutation.  
%This construction will find applications 
%in Sections~\ref{sec:edge-labels} and~\ref{sec:face-labels}. 

\begin{definition} %[\emph{The BCFW-bridge construction}]
\label{def:BCFW}
A \emph{bridge} is a graph fragment shown
in \cref{fig:bridges} on the left.
Let $\dpi$ be a decorated permutation on $b$ letters that has $a$ anti-excedances,
and let $\affpi$ be the corresponding affine permutation.
To build a plabic graph associated to $\affpi$, 
we begin by introducing a white (resp., black) lollipop
in each position $i$ with $\dpi(i)=\overline{i}$ (resp., $\dpi(i)=\underline{i}$).
If $\dpi$ is a decoration of the identity, we are done.  Otherwise,
we 
%	generate a sequence of transpositions $(i,j)$ %giving a factorization of~$\affpi$, 
%following 
	use %the algorithm in 
	\cref{def:BCFW0} to produce a \emph{bridge 
	factorization}, 
then attach successive bridges in the corresponding positions,  
as in \cref{fig:bridges}.
The resulting graph is called a \emph{bridge decomposition} of~$\affpi$, 
or sometimes a BCFW bridge decomposition, due to its relation with 
the Britto-Cachazo-Feng-Witten recursion in quantum field theory, see~\cite{amplitudes}.
\end{definition}

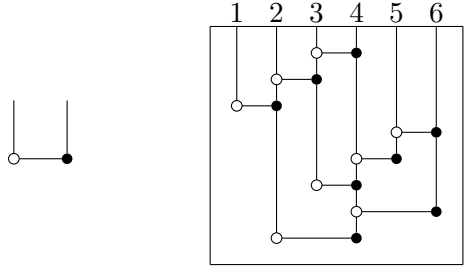
\begin{figure}[ht]
\begin{center}
\vspace{-.15in}
\setlength{\unitlength}{1pt}
%single bridge
\begin{picture}(80,95)

\thicklines
%vertices
\put(10, 40){\circle{4}}
\put(30, 40){\circle*{4}}

%edges
\put(12,40){\line(1,0){16}}
\put(10,42){\line(0,1){20}}
\put(30,42){\line(0,1){20}}

\end{picture}\hspace{0in}
\begin{picture}(95,95)

\thicklines
%square
\put(0,0){\line(1, 0){95}}
\put(0,0){\line(0, 1){90}}
\put(0,90){\line(1, 0){95}}
\put(95,0){\line(0,1){90}}

%vertices grouped by column
\put(10, 60){\circle{4}}

\put(25, 10){\circle{4}}
\put(25, 60){\circle*{4}}
\put(25, 70){\circle{4}}

\put(40, 30){\circle{4}}
\put(40, 70){\circle*{4}}
\put(40, 80){\circle{4}}

\put(55, 10){\circle*{4}}
\put(55, 20){\circle{4}}
\put(55, 30){\circle*{4}}
\put(55, 40){\circle{4}}
\put(55, 80){\circle*{4}}

\put(70, 40){\circle*{4}}
\put(70, 50){\circle{4}}

\put(85, 20){\circle*{4}}
\put(85, 50){\circle*{4}}

%vertical edges grouped by column
\put(10,90){\line(0,-1){28}}

\put(25,90){\line(0,-1){18}}
\put(25,68){\line(0,-1){6}}
\put(25,58){\line(0,-1){46}}

\put(40,90){\line(0,-1){8}}
\put(40,78){\line(0,-1){6}}
\put(40,68){\line(0,-1){36}}

\put(55,90){\line(0,-1){8}}
\put(55,78){\line(0,-1){36}}
\put(55,38){\line(0,-1){6}}
\put(55,28){\line(0,-1){6}}
\put(55,18){\line(0,-1){6}}

\put(70,90){\line(0,-1){38}}
\put(70,48){\line(0,-1){6}}

\put(85,90){\line(0,-1){38}}
\put(85,48){\line(0,-1){26}}

%horizontal edges from top to bottom
\put(42,80){\line(1,0){11}}
\put(27,70){\line(1,0){11}}
\put(12,60){\line(1,0){11}}
\put(72,50){\line(1,0){11}}
\put(57,40){\line(1,0){11}}
\put(42,30){\line(1,0){11}}
\put(57,20){\line(1,0){26}}
\put(27,10){\line(1,0){26}}

%labels on top
\put(10,95){\makebox(0,0){$\mathbf{1}$}}
\put(25,95){\makebox(0,0){$\mathbf{2}$}}
\put(40,95){\makebox(0,0){$\mathbf{3}$}}
\put(55,95){\makebox(0,0){$\mathbf{4}$}}
\put(70,95){\makebox(0,0){$\mathbf{5}$}}
\put(85,95){\makebox(0,0){$\mathbf{6}$}}

\end{picture}
\vspace{-.1in}
\end{center}
\caption{\emph{Left}: a single bridge. 
\emph{Right}: the bridge decomposition associated to the 
factorization constructed in \cref{fig:factorization-bridges}.  
The resulting plabic graph has trip permutation $\dpi= (4,6,5,1,2,3)$. 
Moreover, we have $\dpi=(24)(46)(34)(45)(56)(12)(23)(34)$,
the product of the transpositions $(i,j)$ generated by the algorithm
(reading right to left).}
\vspace{-.1in}
\label{fig:bridges}
\end{figure}

\begin{proposition}
\label{pr:number-of-bridges}
A bridge decomposition of an $(a,b)$-bounded affine permutation~$\affpi$ 
uses $a(b-a)-\ell(\affpi)$ bridges.
\end{proposition}

\begin{proof}
See \cref{lem:inv-affpi-bound} and 
Definitions \ref{def:BCFW0} and~\ref{def:BCFW}. 
\end{proof}

\begin{theorem}
\label{thm:bridge}
Let $\dpi$ be a decorated permutation on $b$ letters that has $a$ anti-excedances. 
Let $\affpi$ be the associated $(a,b)$-bounded affine permutation. 
Then any bridge decomposition of $\affpi$ is a reduced plabic graph 
with the decorated trip permutation~$\dpi$.
%The construction of this  bridge decomposition  involves $a(b-a)-\ell(\affpi)$ bridges.
\end{theorem}

\begin{proof}
%Let $\affpi$ be the affine permutation associated to $\dpi$.
We use induction on the number of bridges  $\beta=a(b-a)-\ell(\affpi)$.
If $\beta=0$, then $\dpi$ is a decoration of the identity
(see \cref{prop:maxlength}), 
so the bridge decomposition consists entirely of lollipops,
and we are done. 

Now suppose that $\dpi$ is not a decoration of the identity.  
Proceeding as in \cref{def:BCFW0}, we construct a bridge factorization
%	sequence
%of transpositions 
	$\sigma_1, \sigma_2,\dots, \sigma_\beta$, 
where $\sigma_1 = (ij)$ satisfies \eqref{eq:ij-pair-1}--\eqref{eq:ij-pair-4}. 
Let $G$ be the plabic graph obtained by attaching bridges
according to $\sigma_1,\dots, \sigma_{\beta}$ (from top to bottom).
By the induction assumption, attaching bridges according to 
$\sigma_2, \dots, \sigma_\beta$ as in \cref{def:BCFW} produces
a reduced plabic graph $G'$ with the trip permutation~$\dpi'= \sigma_\beta \cdots \sigma_2$. 
This graph has $\beta-1$ bridges and is obtained by removing the topmost horizontal edge~$e$ from~$G$
and applying local moves (M2) to remove the endpoints of~$e$. 

Conversely, $G$ is obtained from $G'$ by attaching a bridge in position $(i,j)$ at the top of~$G'$. 
(To illustrate, in \cref{fig:bridges} we have $(i,j)=(3,4)$.) 
When we add this bridge to~$G'$, the trips starting at $i$ and $j$ get their ``tails'' swapped: 
the trip $T_i$ (resp.,~$T_j$) in~$G$
that begins at~$i$ (resp., at~$j$) traverses~$e$ and continues along the trip that 
used to begin at~$j$ (resp., at~$i$) in~$G'$; all other trips remain the same.
Hence the trip permutation of $G$ is $\dpi' \sigma_1 = \dpi$.

It remains to show that $G$ is reduced.
One option is to use 
the ``bad features'' criterion of \cref{thm:reduced}.
%However, this theorem 
%requires the plabic graph to be normal, so we will need to %adjust the strategy
	%replace $G$ by a suitable normal graph~$N(G)$. 
However, it will  be more convenient for us to utilize the triple diagram version 
of the criterion, cf.\ \cref{cor:bad<=>bad}. 

We use \cref{lem:biptri2} to construct 
a normal graph $N(G)$ and the associated triple diagram
 $\X(G)=\X(N(G))$. 
	%the algorithm in 
	%\cref{def:normalize}. 
%(Note that stages 1, 4 and 6 of the algorithm are not 
%needed since every vertex in~$G$ has degree 2 or~3, or is a lollipop 
%	that gets removed.) 
	%Up to the addition/removal of lollipops,
 %	the triple diagram $\X(N(G))$ 
%is isotopic to the triple diagram~$\X(G)$ constructed directly from~$G$ as 
%in %\cref{def:triple-non-normal}. 
%\cref{lem:biptri2}.
The graph $G$ is reduced if and only if $N(G)$ is reduced, 
which is 
	%in turn 
	equivalent to the triple diagram $\X=\X(G)$ being minimal, 
or to $\X$ having no badgons, see \cref{cor:bad<=>bad}. 
Thus, our goal is to show $\X$ has no badgons.

%We first observe that $\X$ has no closed strands since each of the three strands in~$\X$
%that pass through~$w$ either begins or ends near one of the boundary vertices $i$ and~$j$. 

By the induction assumption, %we can rely on the fact that 
	the triple diagram $\X'=\X(G')$ has no badgons. 
It follows that any potential badgon in $\X$ must involve the white endpoint~$w$ of  edge~$e$; 
otherwise this feature would have already been present in~$\X'$. 
In particular, this means that $w$ is trivalent. %, i.e., not an elbow of~$G$. 

A monogon in $\X$ would have to have its vertex at~$w$.
Three of the six half-strands at~$w$ run straight to or from the boundary,
so we need to use two of the remaining three; 
moreover, those two half-strands have to be oppositely oriented.
There are two such cases to consider.
In one case, $i$ would be a fixed point of~$\dpi$, contradicting our choice of~$(i,j)$. 
In the other case, $i$ would be a fixed point of~$\dpi'=\dpi(G')$,
which can also be ruled out since in that case, 
$i$~would not participate in any bridge in~$G'$,
making it impossible to produce the bottom vertex 
of the vertical edge pointing downwards from~$w$. 

Finally, suppose that $\X$ contains a parallel digon. 
One of the vertices of the digon has to be~$w$; let $w'$ denote the other vertex.
Since the two sides of the digon are oriented in the same way at~$w$,
it follows that these sides lie on the strands $S_i$ and~$S_j$
that start near the vertices $i$ and~$j$, respectively. 

\pagebreak[3]

By our choice of bridge $(i,j)$, 
every $h$ with $i<h<j$ is a fixed point of $\dpi$, but
$i$ and $j$ are not fixed points.  
It follows that $S_i$ 
(resp. $S_j$) does not terminate at $i$
(resp. $j$),
%, $S_j$ does not 
%terminate at $j$, 
and neither terminates between 
the boundary vertices $i$ and~$j$.
%(since every $h$ with $i<h<j$ is a fixed point of~$\dpi$). 
We explained in the monogon case that $S_j$ cannot terminate at~$i$;
one can similarly argue that $S_i$ cannot terminate at~$j$.
Also, neither $S_i$ nor $S_j$ intersects itself. 
Moreover the ``tails'' of $S_i$ and~$S_j$ 
that start at the second vertex~$w'$ of the digon
do not intersect each other
(since otherwise a parallel digon would have been present in~$\X'$). 
It follows that either
$\dpi(i)<i$ and $\dpi(j)>j$, or 
$\dpi(i)>\dpi(j)>j$, or $\dpi(j)<\dpi(i)<i$.  
In each case, we get $\affpi(i)>\affpi(j)$, which 
contradicts the way we chose $i$ and~$j$.
\end{proof}

\begin{corollary}
\label{permtoG}
Let $\dpi$ be a decorated permutation on $b$ letters.  Then there exists
a reduced plabic graph whose decorated trip permutation is $\dpi$.
\end{corollary}

\begin{proof}
Use either \cref{thm:bridge} or the construction in \cref{def:standard-triple}
(together with \cref{pr:bijplabictriple} and \cref{red-minimal}).
\end{proof}

\begin{corollary}
\label{cor:number-faces}
Let $G$ be a reduced plabic graph with the decorated trip permutation~$\dpi$. 
If $\dpi$ has $b$ letters and $a$ anti-excedances, then the number of faces in~$G$
is $a(b-a)-\ell(\affpi)+1$. 
\end{corollary}

\begin{proof}
The number of faces is invariant under local moves. 
Therefore, by \cref{thm:moves}, it suffices to establish this formula for 
a particular reduced plabic graph with decorated trip permutation~$\dpi$. 
By \cref{thm:bridge}, we can use a bridge decomposition of~$\affpi$.
Since each bridge adds one face to the graph, the claim follows by \cref{pr:number-of-bridges}. 
\end{proof}

Let $\dpi_{a,b}$ denote the decorated permutation on $b$ letters defined by
\begin{equation}
\label{eq:dpi-ab}
	\dpi_{a,b}=\begin{cases}
		(a+1,a+2,\dots,b,1,2,\dots,a) &\text{ if }1 \leq a \leq b-1\\
		(1,2,\dots, a) &\text{ if }a=0\\
		(\overline{1},\overline{2},\dots,\overline{a}) & \text{ if }a=b
	\end{cases}
\end{equation}

\begin{exercise}
\label{exer:length=0}
Let $\dpi$ be a decorated permutation on $b$ letters that has $a$~anti-excedances.
Show that if $\ell(\affpi)=0$, then $\dpi=\dpi_{a,b}$. 
\end{exercise}

\begin{corollary}
\label{cor:max-number-faces}
Let $G$ be a reduced plabic graph whose decorated trip permutation $\dpi_G$ 
has $b$ letters and $a$ anti-excedances.
Then $G$ has at most $a(b-a)+1$ faces. 
Moreover it has $a(b-a)+1$ faces if and only if $\dpi_G=\dpi_{a,b}$. 
\end{corollary}

\begin{proof}
This is immediate from \cref{cor:number-faces} and \cref{exer:length=0}. 
\end{proof}

\begin{remark}
The \emph{permutohedron} $\mathcal{P}_n$ \cite[Exercise~4.64a]{ec1}
is a  polytope whose $n!$ vertices are labeled by permutations in the symmetric group~$\mathcal{S}_n$. 
Shortest paths in the $1$-skeleton of~$\mathcal{P}_n$ encode 
reduced expressions %of the longest element 
in~$\mathcal{S}_n$, and its $2$-dimensional faces correspond to their local (braid) transformations, 
cf.\ \cref{exercise:braid-equivalence}. 
Similarly, 
paths in the $1$-skeleton of the \emph{bridge polytope}~\cite{WilliamsBridge} 
encode bridge decompositions of the decorated permutation~$\tilde{\pi}_{a,b}$; 
its $2$-dimensional faces correspond to local moves in plabic graphs.
\end{remark}

\newpage

\section{Edge labels of reduced plabic graphs}
\label{sec:edge-labels}

%We conclude this section by presenting two criteria 
%for testing whether a given plabic graph is reduced.

\begin{definition}
\label{def:edgelabeling}
Let $G$ be a leafless plabic graph.
Let us label the edges of~$G$ by subsets of integers 
that indicate which one-way trips traverse a given edge; 
more precisely, for each boundary vertex~$i$, 
we include~$i$ in the label of every edge contained in the trip that starts at~$i$.  
By~\cref{rem:atmost2}, each edge will be labeled by at most two integers.
See \cref{fig:plabicresonance}. 

We say that $G$ has the \emph{resonance property} 
if after labeling the edges of~$G$ as in \cref{def:edgelabeling}, 
the following condition is satisfied at each internal vertex~$v$ that is not a lollipop: 
\begin{itemize}[leftmargin=.2in]
\item 
there exist numbers $i_1 < \dots < i_m$
such that the edges incident to~$v$ are labeled by the two-element sets
$\{i_1,i_2\}, \{i_2,i_3\}, \dots, \{i_{m-1},i_m\}, \{i_1,i_m\}$,
appearing in clockwise order.
\end{itemize}
In particular, each edge of $G$ that is not incident to a lollipop is labeled by 
a two-element subset.
See \cref{fig:plabicresonance}. 
\end{definition}

%\begin{example}
%In \cref{fig:plabicresonance}, 
%the three edge labels around the vertex~$v$ (resp., at~$v'$), 
%listed in the clockwise order, are:
%$\{1,2\}, \{2,4\}, \{1,4\}$ (resp., $\{1,2\}, \{2,5\}, \{1,5\}$). 
%\end{example}

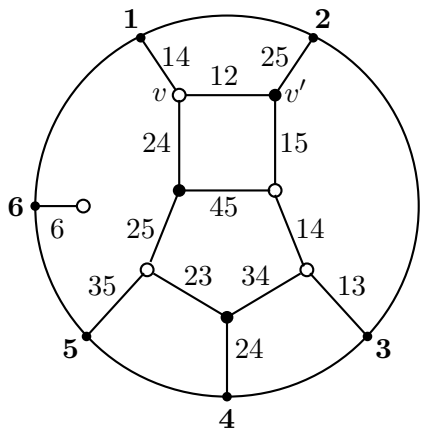
\begin{figure}[ht]
\begin{center}
\vspace{-.1in} 
\setlength{\unitlength}{1.4pt}
\begin{picture}(100,125)(-30,-70)
\thicklines
\multiput(1,0)(1,30){2}{\line(1,0){27.5}}
\multiput(0,1)(30,1){2}{\line(0,1){27.5}}
\put(-30,-5){\circle{4}}
\put(-45,-5){\circle*{3}}
\put(-45,-5){\line(1,0){13}}
\put(0,30){\circle{4}}
\put(30,30){\circle*{4}}
\put(30,0){\circle{4}}
\put(0,0){\circle*{4}}
\put(-10,-25){\circle{4}}
\put(15,-40){\circle*{4}}
\put(40,-25){\circle{4}}
\put(0,0){\line(-10,-25){9.5}}
\put(30,-1.5){\line(10,-25){9}}
\put(15,-40){\line(-25,15){23}}
\put(15,-40){\line(25,15){23}}
\put(15,-5){\circle{120}}
\put(15,-65){\circle*{3}}
\put(15,-40){\line(0,-1){24}}
\put(-12,48){\circle*{3}}
\put(42,48){\circle*{3}}
\put(-12,48){\line(12,-18){11}}
\put(42,48){\line(-12,-18){11}}
\put(-29,-46){\circle*{3}}
\put(59,-46){\circle*{3}}
\put(-29,-46){\line(19,21){18}}
\put(59,-46){\line(-19,21){18}}
\put(-15,54){\makebox(0,0){$\mathbf{1}$}}
\put(45,54){\makebox(0,0){$\mathbf{2}$}}
\put(64,-50){\makebox(0,0){$\mathbf{3}$}}
\put(15,-72){\makebox(0,0){$\mathbf{4}$}}
\put(-34,-50){\makebox(0,0){$\mathbf{5}$}}
\put(-51,-5){\makebox(0,0){$\mathbf{6}$}}
\put(-1,42){\makebox(0,0){${14}$}}
\put(14,36){\makebox(0,0){${12}$}}
\put(-6,30){\makebox(0,0){$v$}}
\put(36.5,31.5){\makebox(0,0){$v'$}}
\put(30,42){\makebox(0,0){${25}$}}
\put(-38,-12){\makebox(0,0){$6$}}
			\put(-12,-12){\makebox(0,0){$25$}}
			\put(-7,15){\makebox(0,0){$24$}}
			\put(36,15){\makebox(0,0){$15$}}
			\put(14,-5){\makebox(0,0){$45$}}
			\put(41,-12){\makebox(0,0){$14$}}
			\put(54,-30){\makebox(0,0){$13$}}
			\put(24,-27){\makebox(0,0){$34$}}
			\put(6,-27){\makebox(0,0){$23$}}
			\put(-24,-30){\makebox(0,0){$35$}}
			\put(22,-50){\makebox(0,0){$24$}}

\put(0,30){\white{\circle*{3.3}}}
\put(30,0){\white{\circle*{3.3}}}
\put(-10,-25){\white{\circle*{3.3}}}
\put(40,-25){\white{\circle*{3.3}}}
\end{picture}
\end{center}
\vspace{-.1in} 
\caption{A reduced plabic graph  
from Figures~\ref{fig:plabic}(b) and~\ref{fig:plabic-move-equiv}. 
Its edge labeling 
%Here $\dpi_{G'} = (3,4,5,1,2,\overline{6})$.
exhibits the resonance property, see \cref{def:edgelabeling}. 
For example, the edge labels around the vertex~$v$ (resp.,~$v'$), 
listed in clockwise order, are
$\{1,2\}, \{2,4\}, \{1,4\}$ (resp., $\{1,2\}, \{2,5\}, \{1,5\}$). 
}
\label{fig:plabicresonance}
\end{figure}

\begin{remark}
\label{rem:resonance-1-2}
%If a plabic graph has an internal leaf that is not a lollipop, 
%then the resonance property fails. 
%
At a bivalent vertex~$v$, the resonance condition is satisfied if and only if 
the two trips passing through~$v$ are distinct and none of them is a roundtrip. 
\end{remark}

\begin{remark}
\label{rem:trivalentresonance}
If a plabic graph $G$ is trivalent (apart from lollipops), then 
the resonance property is equivalent to the following requirement
at each interior vertex~$v$ (other than a lollipop): 
\begin{itemize}[leftmargin=.2in]
\item 
the three edges incident to~$v$ have labels
$\{a,b\}$, $\{a,c\}$, and $\{b,c\}$, for some $a<b<c$, 
and moreover this (lexicographic) ordering of labels corresponds to 
the counterclockwise direction around~$v$.
\end{itemize}
For example, in \cref{fig:plabicresonance},  
the edge labels around the vertex~$v$ (resp.,~$v'$) 
are, in lexicographic order, 
$\{1,2\}, \{1,4\}, \{2,4\}$ (resp., $\{1,2\}, \{1,5\}, \{2,5\}$). 
The three edges carrying these labels
appear  in the counterclockwise order around~$v$ (resp.,~$v'$).
\end{remark}

\begin{exercise}
Verify that none of the plabic graphs shown in 
\cref{fig:bad-features} (draw a disk around each of the fragments)
satisfy the resonance property.
\end{exercise}

\begin{theorem}[{\cite[Theorem 10.5]{kodwil}}]
\label{thm:resonance}
Let $G$ be a leafless plabic graph. % without internal leaves other than lollipops. 
Then $G$ is reduced if and only if it has the resonance property.
\end{theorem}

\cref{thm:resonance} is proved below in this section,
following a few remarks and auxiliary lemmas.  

%\begin{remark}
%\cref{thm:resonance} can be used to test whether a given plabic graph
%(potentially having internal leaves) is reduced or not. 
%To this end, use the moves (M2) and (M3) 
%to get rid of collapsible trees. 
%	If an internal leaf (not a lollipop) remains, then the graph is not reduced. 
%Otherwise, one can apply the criterion in \cref{thm:resonance}. 
%\end{remark}

\begin{remark}
We find the resonance criterion of \cref{thm:resonance} easier to check than
the ``bad features'' criterion of \cref{thm:reduced}.
\end{remark}

\begin{remark}
Certain reduced plabic graphs were realized 
as tropical curves in \cite{kodwil}, where it was shown that 
the resonance property corresponds to the \emph{balancing condition}
for tropical curves.  
\end{remark}

\begin{lemma}
\label{lem:moves-resonance}
The resonance property is preserved under the local moves {\rm(M1)--(M3)}.
%except when the decontraction move {\rm(M3)} creates a new leaf.  
\end{lemma}

\begin{proof}
The square move~(M1) only changes the labels
of the sides of the square, see \cref{fig:M1labels}. 
Moreover, the labels around each vertex match the labels around the opposite vertex
after the square move, with the same cyclic order. 
Hence this move preserves the resonance property. 

The case of the local move~(M2) is easy, cf.\ \cref{rem:resonance-1-2}. 

For the case of the local move (M3) around a degree $4$ black vertex, 
see \cref{fig:M3labels}.
(The case of white vertices and of higher degree vertices of both colors is similar.)
\end{proof}

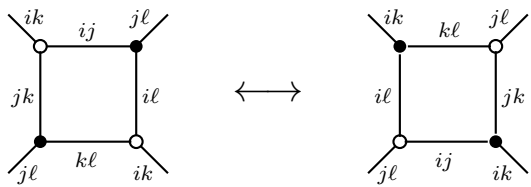
\begin{figure}[ht]%[htbp]
\begin{center}
%\vspace{-.1in}
\setlength{\unitlength}{1.2pt}
\begin{picture}(30,36)(0,0)
\thicklines
\multiput(1,0)(1,30){2}{\line(1,0){27.1}}
\multiput(0,1)(30,1){2}{\line(0,1){27.1}}
\put(0,0){\circle*{4}}
\put(0,30){\circle{4}}
\put(30,30){\circle*{4}}
\put(30,0){\circle{4}}
\put(-10,40){\line(1,-1){8.5}}
\put(-10,-10){\line(1,1){9}}
\put(40,-10){\line(-1,1){8.5}}
\put(40,40){\line(-1,-1){9}}
\put(-5,37){$\scriptstyle ik$}
\put(28,37){$\scriptstyle j\ell$}
\put(12,33){$\scriptstyle ij$} %kl
\put(-9,12){$\scriptstyle jk$} %il
\put(32,12){$\scriptstyle i\ell$} %jk
\put(11,-8){$\scriptstyle k\ell$} % ij
\put(-7,-12){$\scriptstyle j\ell$}
\put(29,-12){$\scriptstyle ik$}
\end{picture}
\qquad
\begin{picture}(40,36)(0,0)
\put(20,15){\makebox(0,0){\Large{$\longleftrightarrow$}}}
\end{picture}
\qquad
\begin{picture}(30,36)(0,0)
\put(-5,37){$\scriptstyle ik$}
\put(28,37){$\scriptstyle j\ell$}
\put(12,33){$\scriptstyle k\ell$} %kl
\put(-8,12){$\scriptstyle i\ell$} %il
\put(32,12){$\scriptstyle jk$} %jk
\put(11,-8){$\scriptstyle ij$} % ij
\put(-7,-12){$\scriptstyle j\ell$}
\put(29,-12){$\scriptstyle ik$}

\thicklines
\multiput(2,0)(-1,30){2}{\line(1,0){26.5}}
\multiput(0,1.8)(30,-0.5){2}{\line(0,1){26.5}}
\put(0,0){\circle{4}}
\put(0,30){\circle*{4}}
\put(30,30){\circle{4}}
\put(30,0){\circle*{4}}
\put(-10,40){\line(1,-1){9}}
\put(-10,-10){\line(1,1){8.5}}
\put(40,-10){\line(-1,1){9}}
\put(40,40){\line(-1,-1){8.5}}
\end{picture}
\end{center}
%\vspace{-.1in}
\caption{Transformation of edge labels under a square move~(M1).
%Here we have that $i,j,k,\ell$ are ordered in cyclic order, i.e. , $i<j<k<\ell$ or $j<k<\ell<i$ or $k<\ell<i<j$ or $\ell<i<j<k$.
}
\label{fig:M1labels}
\end{figure}

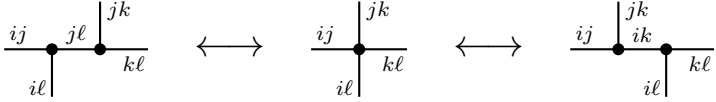
\begin{figure}[ht]%[htbp]
\begin{center}
\vspace{-.1in} 
\setlength{\unitlength}{1.8pt}
\begin{picture}(30,20)(0,0)
\thicklines
\put(10,10){\circle*{2.5}}
\put(20,10){\circle*{2.5}}

\put(10,0){\line(0,1){10}}
\put(20,10){\line(0,1){10}}
\put(0,10){\line(1,0){30}}

\put(3,13){\makebox(0,0){$\scriptstyle ij$}}
\put(15,13){\makebox(0,0){$\scriptstyle j\ell$}}
\put(27,7){\makebox(0,0){$\scriptstyle k\ell$}}
\put(24,18){\makebox(0,0){$\scriptstyle jk$}}
\put(7,2){\makebox(0,0){$\scriptstyle i\ell$}}

\end{picture}
\begin{picture}(30,20)(0,0)
\put(15,10){\makebox(0,0){\Large{$\longleftrightarrow$}}}
\end{picture}
\begin{picture}(20,20)(0,0)
\thicklines
\put(10,10){\circle*{2.5}}

\put(10,0){\line(0,1){20}}
\put(0,10){\line(1,0){20}}

\put(3,13){\makebox(0,0){$\scriptstyle ij$}}
\put(17,7){\makebox(0,0){$\scriptstyle k\ell$}}
\put(14,18){\makebox(0,0){$\scriptstyle jk$}}
\put(7,2){\makebox(0,0){$\scriptstyle i\ell$}}

\end{picture}
\begin{picture}(30,20)(0,0)
\put(15,10){\makebox(0,0){\Large{$\longleftrightarrow$}}}
\end{picture}
\begin{picture}(30,20)(0,0)
\thicklines
\put(10,10){\circle*{2.5}}
\put(20,10){\circle*{2.5}}

\put(20,0){\line(0,1){10}}
\put(10,10){\line(0,1){10}}
\put(0,10){\line(1,0){30}}

\put(3,13){\makebox(0,0){$\scriptstyle ij$}}
\put(15,13){\makebox(0,0){$\scriptstyle ik$}}
\put(27,7){\makebox(0,0){$\scriptstyle k\ell$}}
\put(14,18){\makebox(0,0){$\scriptstyle jk$}}
\put(17,2){\makebox(0,0){$\scriptstyle i\ell$}}

\end{picture}
\vspace{-.1in} 
\end{center}
\caption{Transformation of edge labels under a local move~(M3) at a 4-valent black vertex.
(Alternatively, make all the vertices white.)}
\label{fig:M3labels}
\end{figure}

\pagebreak[3]

\begin{lemma}
\label{lem:resonance}
Any plabic graph obtained via the bridge decomposition construction
(see \cref{def:BCFW}) 
%of $\affpi$ 
has the resonance property. 
\end{lemma}

\begin{proof}
We will show that, more concretely, the edge labels around  
trivalent vertices in such a plabic graph~$G$ 
follow one of the patterns described in \cref{fig:resonance}.
	%, with the 
%labels around white vertices satisfying
%$r<i<j$ or $i<j<r$ and the labels around 
%black vertices satisfying $s<i<j$ or $i<j<s$.
We will establish this result by induction on~$\beta$, the number of bridges, 
following the strategy used in the proof of \cref{thm:bridge}. 

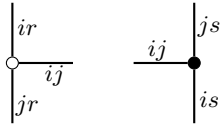
\begin{figure}[ht]
%\centering
\begin{center}
\begin{tikzpicture}[scale=0.8]

\node[circle, thick, fill=white, draw=black, inner sep=0pt, minimum size=4.5pt] (1) at (0,0) {};
\node[circle, fill=black, draw=black, inner sep=0pt, minimum size=4.5pt] (2) at (3,0) {};
\draw[thick] (1) -- (0,1);
\draw[thick] (1) -- (1,0);
\draw[thick] (1) -- (0,-1);
\draw[thick] (2) -- (3,1);
\draw[thick] (2) -- (2,0);
\draw[thick] (2) -- (3,-1);
\node at (0.25,0.6) {\footnotesize $ir$};
\node at (0.7,-0.2) {\footnotesize $ij$};
\node at (0.25,-0.7) {\footnotesize $jr$};
\node at (2.4,0.2) {\footnotesize $ij$};
\node at (3.25,0.6) {\footnotesize $js$};
\node at (3.25,-0.6) {\footnotesize $is$};

\end{tikzpicture}
\vspace{-.1in}
\end{center}
\caption{Edge labels near trivalent vertices in a bridge decomposition.  
At a white vertex, shown on the left, either $r<i<j$ or $i<j<r$.  
At a black vertex, shown on the right, either $s<i<j$ or $i<j<s$.}
\label{fig:resonance}
\end{figure}

Let $G$ be a bridge decomposition of $\affpi$, associated to 
the sequence of transpositions $\sigma_1,\dots, \sigma_\beta$.  
Thus $\dpi=\dpi_G=\sigma_\beta \cdots \sigma_1$.
Here $\sigma_1 = (i j)$, where $i$ and $j$ satisfy \eqref{eq:ij-pair-1}--\eqref{eq:ij-pair-4}.
Let $G'$ be the bridge decomposition associated to $\sigma_2,\dots, \sigma_\beta$, so that 
$G$ is obtained from $G'$ by adding
a single bridge in position $(i,j)$ at the top of~$G'$. 

Suppose the result is true for 
$G'$.  We need to verify it for $G$.
Let $r = \dpi^{-1}(i)$ and $s=\dpi^{-1}(j)$.
Adding the bridge in position $(i,j)$ at the top of~$G'$
adds at most two trivalent vertices:
it adds a white (respectively, black)
trivalent vertex provided that $r\neq j$
(respectively, $s\neq i$).  
The cases when one of the vertices on the bridge is bivalent are easy to verify,
so we are going to assume that $r\neq j$ and $s\neq i$.
In this case, the local configuration around positions $i$ and $j$
in $G'$ and $G$ is as shown in \cref{fig:resonance2}.

\begin{figure}[ht]
\begin{center}
\vspace{-.2in}
\begin{tikzpicture}

% left
\draw[thick] (0,0) -- (4,0);
\draw[thick] (1.5,0) -- (1.5,-1.5); 
\draw[thick] (2.5,0) -- (2.5,-1.5);
\node at (2,0.85) {$G'$};
\node at (0.2,0.3) {$1$};
\node at (0.5,0.3) {$2$};
\node at (1,0.3) {$...$};
\node at (1.5,0.3) {$i$};
\node at (2.5,0.3) {$j$};
\node at (3.1,0.3) {$...$};
\node at (3.7,0.3) {$b$};
\node at (1.2,-0.8) {$ir$};
\node at (2.8,-0.8) {$js$};

% right
\draw[thick] (6,0) -- (10,0);
\draw[thick] (7.5,0) -- (7.5,-1.5); 
\draw[thick] (8.5,0) -- (8.5,-1.5);
\draw[thick] (7.5,-0.7) -- (8.5,-0.7);
\node[circle, thick, fill=white, draw=black, inner sep=0pt, minimum size=4pt] at (7.5,-0.7) {};
\node[circle, fill=black, draw=black, inner sep=0pt, minimum size=4pt] at (8.5,-0.7) {};
\node at (8,0.85) {$G$};
\node at (6.2,0.3) {$1$};
\node at (6.5,0.3) {$2$};
\node at (7,0.3) {$...$};
\node at (7.5,0.3) {$i$};
\node at (8.5,0.3) {$j$};
\node at (9.1,0.3) {$...$};
\node at (9.7,0.3) {$b$};
\node at (7.2,-0.3) {$ir$};
\node at (8.8,-0.3) {$js$};
\node at (7.2,-1.2) {$jr$};
\node at (8.8,-1.2) {$is$};
\node at (8,-0.5) {$ij$};

\end{tikzpicture}
\vspace{-.2in}
\end{center}
\caption{The local configuration around positions $i$ and $j$
	in  $G'$ and $G$.}
\label{fig:resonance2}
\end{figure}
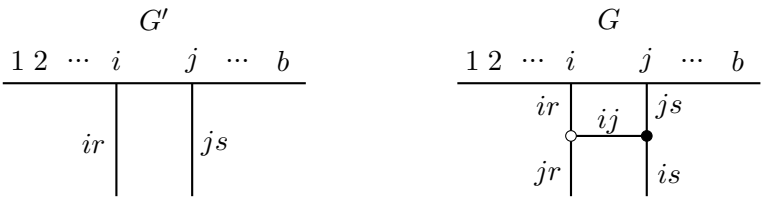

Recall that $i<j$ and moreover any $h$ such that $i<h<j$ is a fixed point of~$\dpi$.
For the reasons indicated in the proof of \cref{thm:bridge}, 
adding the bridge $(i,j)$ at the top of $G'$ 
has the effect of replacing the label $i$ (resp.,~$j$) by~$j$ (resp.,~$i$) 
in every edge label outside of the bridge.

If $G'$ has an edge with the label $ij$, 
then $G$ has a bad double crossing involving the trips originating at~$i$ and~$j$.
This however is impossible since $G$ is reduced, by \cref{thm:bridge}.
Therefore $G'$ has no edge with label~$ij$. 
Furthermore, $G'$ has no edge with a label $h$ for $i<h<j$.
Since all trivalent vertices of $G'$ satisfy the resonance condition of 
\cref{fig:resonance}, the same remains true after 
switching the labels $i$ and~$j$.
Thus, all trivalent vertices of $G$ that were present
in $G'$ satisfy this resonance condition.

Finally, the two new trivalent vertices in $G$ satisfy this condition because 
$i<j$ and we can exclude $i<r<j$ and $i<s<j$ because of \eqref{eq:ij-pair-3}. 
\end{proof}

\begin{proof}[Proof of \cref{thm:resonance}]
We first establish the ``if'' direction. % of the theorem. 
Let $G$ be a leafless plabic graph.
Suppose that $G$ has the resonance property.
We want to show that $G$ is reduced. 

Assume the contrary. 
By definition,
%\cref{pr:reduced-collapse}, 
$G$ can be transformed by local moves  (that is, (M1), (M2), (M3))
%that don't create internal leaves 
into a plabic graph~$G'$ containing a hollow monogon or a hollow digon.  
Since $G$ has the resonance property, so does~$G'$, by \cref{lem:moves-resonance}. 
This yields a contradiction because the labels around a hollow monogon or 
digon do not satisfy
the resonance property, see \cref{fig:fail2}.

\begin{figure}[ht]
\begin{center}
\vspace{-.05in}
%\includegraphics[height=1in]{FailureReduced.ps}
%\\
\setlength{\unitlength}{1.1pt}
\begin{picture}(35,20)(0,10)
\thicklines
\put(15,15){\circle*{4.5}}
\put(13,15){\line(-1,0){10}}
\cbezier(15,17)(30,40)(30,-10)(15,13)
\put(5,18){$\scriptstyle ii$}
\put(20,25){$\scriptstyle i$}
\end{picture}
\quad
\begin{picture}(45,20)(0,10)
\thicklines
\put(15,15){\circle{4}}
\cbezier(15,17)(30,40)(30,-10)(15,13)
\put(5,18){$\scriptstyle ii$}
\put(20,25){$\scriptstyle i$}
\put(15,15){\color{white}\circle*{3}}
\put(13,15){\line(-1,0){10}}
\end{picture}
\quad
\begin{picture}(45,20)(0,10)
\thicklines
\put(4,18){$\scriptstyle ij$}
\put(32,18){$\scriptstyle ij$}
\put(20,25){$\scriptstyle i$}
\put(19,-2){$\scriptstyle j$}
\put(15,15){\circle{4}}
\put(30,15){\circle{4}}
\put(13,15){\line(-1,0){10}}
\put(32,15){\line(1,0){10}}
\qbezier(15,17)(22,30)(30,17)
\qbezier(15,13)(22,0)(30,13)
\end{picture}
\quad
\begin{picture}(45,20)(0,10)
\thicklines
\put(6,17){$\scriptstyle i$}
\put(34,18){$\scriptstyle j$}
\put(20,26){$\scriptstyle ij$}
\put(19,-2){$\scriptstyle ij$}
\put(15,15){\circle*{4}}
\put(30,15){\circle{4}}
\put(13,15){\line(-1,0){10}}
\put(32,15){\line(1,0){10}}
\qbezier(15,15)(22,30)(30,17)
\qbezier(15,15)(22,0)(30,13)
\end{picture}
\quad
\begin{picture}(45,20)(0,10)
\put(4,18){$\scriptstyle ij$}
\put(32,18){$\scriptstyle ij$}
\put(20,27){$\scriptstyle j$}
\put(19,-2){$\scriptstyle i$}
\thicklines
\put(15,15){\circle*{4}}
\put(30,15){\circle*{4}}
\put(13,15){\line(-1,0){10}}
\put(32,15){\line(1,0){10}}
\qbezier(15,15)(22,30)(30,17)
\qbezier(15,15)(22,0)(30,13)
\end{picture}
%{\ }
%\begin{picture}(30,20)(-10,10)
%\thicklines
%\put(7,17){$\scriptstyle i$}
%\put(15,22){$\scriptstyle ij$}
%\put(15,2){$\scriptstyle ij$}
%\put(15,15){\circle{4}}
%\put(1,15){\circle*{4}}
%\put(13,15){\line(-1,0){10}}
%\put(16.6,16.2){\line(4,3){12}}
%\put(16.6,13.8){\line(4,-3){12}}
%\end{picture}
%\quad
%\begin{picture}(30,20)(-10,10)
%\thicklines
%\put(7,17){$\scriptstyle i$}
%\put(15,22){$\scriptstyle ij$}
%\put(15,2){$\scriptstyle ij$}
%\put(15,15){\circle*{4}}
%\put(1,15){\circle{4}}
%\put(13,15){\line(-1,0){10}}
%\put(16.6,16.2){\line(4,3){12}}
%\put(16.6,13.8){\line(4,-3){12}}
%\end{picture}
\end{center}
\caption{A plabic graph containing a hollow digon fails to satisfy the resonance property.}
\label{fig:fail2}
\end{figure}
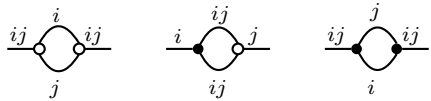

We next establish the ``only if'' direction.
Suppose that $G$ is reduced, with $\dpi_G=\dpi$.  We know from 
\cref{thm:bridge} that there is a bridge
decomposition $G'$---a reduced plabic graph---with trip permutation~$\dpi$. 
By \cref{lem:resonance}, $G'$ has the resonance property.
By \cref{thm:moves}, $G\sim G'$.
But now by  \cref{lem:moves-resonance}, the resonance property is 
preserved under the local moves, so 
$G$ has the resonance property as well.
%$G$ has the resonance property if and only if $G_3$ does.
%Since $G$ has no internal leaves, there exists a trivalent
%plabic graph $G_3\sim G$, see \cref{lem:bitri}.
%Moreover it can be seen from the proof of this lemma that $G$ and $G_3$
%are related via local moves, % that do not create internal leaves,
%so by \cref{lem:moves-resonance}, $G$ has the resonance property if and only if $G_3$ does.
%Similarly, by removing bivalent vertices from~$G'$,
%we obtain a trivalent graph $G'_3\sim G'$
%that has the resonance property.
%Since $G_3$ and $G'_3$ are  move-equivalent to each other, 
%by \cref{lem:moves-resonance}, these local moves 
%preserve the resonance property, so $G_3$ has the resonance property
%and therefore $G$ has it as well.
\end{proof}

\newpage

\section{Face labels of reduced plabic graphs}
\label{sec:face-labels}

In this section, we use the notion of a trip introduced in \cref{def:trip}
to label each face of a reduced (leafless) plabic graph by a collection of positive integers.
These face labels generalize the labeling 
of diagonals in a polygon by Pl\"ucker coordinates (cf.\ Section~\ref{sec:Ptolemy}) 
as well as the labeling of faces in 
(double or ordinary) wiring diagrams by chamber minors   
(cf.\ Sections~\hbox{\ref{sec:baseaffine}--\ref{sec:matrices}}).
In a subsequent chapter, we will
%Chapter~\ref{ch:Grassmannians}, we will 
relate the face labels of reduced plabic graphs to Pl\"ucker coordinates that form an extended cluster 
for the standard cluster structure on a Grassmannian or, more generally, on a Schubert or positroid 
subvariety within it.

\begin{remark}
%\label{rem:}
	Let $G$ be a reduced (leafless) plabic graph. 
Let $T_i$ be the one-way trip in~$G$ that begins at a boundary vertex~$i$
and ends at a boundary vertex~$j$. 

If $i\neq j$, then we claim that there are two kinds of faces in~$G$: 
those on the left side of the trip~$T_i$ and those on the right side of it.
%If $G$ is normal, then 
	This claim follows from the fact (see \cref{thm:reduced}) 
that $G$ does not contain essential self-intersections. 
%For a general reduced plabic graph,
%the claim can be deduced from the case of normal graphs 
%using the procedure described in \cref{def:normalize}. 

If $i=j$, then by \cref{prop:fixedlollipop}, the boundary vertex $i$ is 
incident to
%	the root of a tree 
%that collapses to 
	a lollipop. 
If this lollipop is white (resp., black), then we declare that all faces of $G$ 
lie on the left (resp., right) side of the trip~$T_i$.
\end{remark}

\begin{definition}
\label{def:faces}
	Let $G$ be a reduced (leafless) plabic graph with boundary vertices $1,\dots,b$.
We define two natural face labelings of $G$, cf.\ \cref{fig:plabic3}: 
\begin{itemize}[leftmargin=.2in]
\item 
in the \emph{source labeling} $\mathcal{F}_{\source}(G)$,  each face $f$ of $G$ is labeled by the set 
\[
I_{\source}(f)=\{i \mid \text{$f$ lies to the left of the trip starting at vertex $i$} \}; 
\]
\item 
in the \emph{target labeling} $\mathcal{F}_{\target}(G)$, each face $f$ of $G$ is labeled by the set 
\[
I_{\target}(f)=\{i \mid
			\text{$f$ is to the left of the trip ending at vertex~$i$}\}. 
\]
\end{itemize}
\end{definition}

\begin{figure}[ht]%[htbp]
\begin{center}
\begin{tabular}{cc}
\hspace{3mm}
\setlength{\unitlength}{1pt}
\begin{picture}(100,120)(-20,-70)
\thicklines
%\multiput(1,0)(1,30){2}{\line(1,0){27.5}}
%\multiput(0,1)(30,1){2}{\line(0,1){27.5}}
%\put(0,30){\circle{4}}
\put(30,30){\circle*{5}}
\put(50,3){\circle{5}}
%\put(0,0){\circle*{4}}
\put(-10,0){\circle{5}}
\put(30,-25){\circle*{5}}
%\put(40,-25){\circle{4}}
%\put(0,0){\line(-10,-25){9}}
%\put(30,-1.5){\line(10,-25){9}}
%\put(15,-40){\line(-25,15){23}}
%\put(15,-40){\line(25,15){23}}
\put(15,-5){\circle{120}}
\put(15,-65){\circle*{3}}
%\put(15,-40){\line(0,-1){24}}
\put(-12,48){\circle*{3}}
\put(42,48){\circle*{3}}
\put(74,3){\circle*{3}}
%\put(-12,48){\line(12,-18){11}}
\put(42,48){\line(-12,-18){11}}
\put(-29,-46){\circle*{3}}
\put(59,-46){\circle*{3}}
%\put(-29,-46){\line(19,21){18}}
\put(59,-46){\line(-29,21){30}}
\put(15,-65){\line(15,40){15}}
\put(-29,-46){\line(19,46){18}}
\put(-10,1){\line(-2,48){2}}
\put(74,3){\line(-1,0){23}}
\put(-9,0){\line(40,-25){38}}
\put(-14,54){\makebox(0,0){$\mathbf{1}$}}
\put(44,54){\makebox(0,0){$\mathbf{2}$}}
\put(79,3){\makebox(0,0){$\mathbf{3}$}}
\put(64,-50){\makebox(0,0){$\mathbf{4}$}}
\put(15,-72){\makebox(0,0){$\mathbf{5}$}}
\put(-33,-50){\makebox(0,0){$\mathbf{6}$}}
\put(25,5){\makebox(0,0){$\mathbf{135}$}}
\put(35,-45){\makebox(0,0){$\mathbf{134}$}}
\put(5,-25){\makebox(0,0){$\mathbf{345}$}}
\put(-30,0){\makebox(0,0){$\mathbf{356}$}}

\put(50,3){\white{\circle*{4.25}}}
\put(-10,0){\white{\circle*{4.25}}}
\end{picture}
\hspace{3mm}
&
\hspace{3mm}
\setlength{\unitlength}{1pt}
\begin{picture}(100,120)(-40,-70)
\thicklines
%\multiput(1,0)(1,30){2}{\line(1,0){27.5}}
%\multiput(0,1)(30,1){2}{\line(0,1){27.5}}
%\put(0,30){\circle{4}}
\put(30,30){\circle*{5}}
\put(50,3){\circle{5}}
%\put(0,0){\circle*{4}}
\put(-10,0){\circle{5}}
\put(30,-25){\circle*{5}}
%\put(40,-25){\circle{4}}
%\put(0,0){\line(-10,-25){9}}
%\put(30,-1.5){\line(10,-25){9}}
%\put(15,-40){\line(-25,15){23}}
%\put(15,-40){\line(25,15){23}}
\put(15,-5){\circle{120}}
\put(15,-65){\circle*{3}}
%\put(15,-40){\line(0,-1){24}}
\put(-12,48){\circle*{3}}
\put(42,48){\circle*{3}}
\put(74,3){\circle*{3}}
%\put(-12,48){\line(12,-18){11}}
\put(42,48){\line(-12,-18){11}}
\put(-29,-46){\circle*{3}}
\put(59,-46){\circle*{3}}
%\put(-29,-46){\line(19,21){18}}
\put(59,-46){\line(-29,21){30}}
\put(15,-65){\line(15,40){15}}
\put(-29,-46){\line(19,46){18}}
\put(-10,1){\line(-2,48){2}}
\put(74,3){\line(-1,0){23}}
\put(-9,0){\line(40,-25){38}}
\put(-14,54){\makebox(0,0){$\mathbf{1}$}}
\put(44,54){\makebox(0,0){$\mathbf{2}$}}
\put(79,3){\makebox(0,0){$\mathbf{3}$}}
\put(64,-50){\makebox(0,0){$\mathbf{4}$}}
\put(15,-72){\makebox(0,0){$\mathbf{5}$}}
\put(-33,-50){\makebox(0,0){$\mathbf{6}$}}
\put(25,5){\makebox(0,0){$\mathbf{345}$}}
\put(35,-45){\makebox(0,0){$\mathbf{356}$}}
\put(5,-25){\makebox(0,0){$\mathbf{346}$}}
\put(-30,0){\makebox(0,0){$\mathbf{134}$}}

\put(50,3){\white{\circle*{4.25}}}
\put(-10,0){\white{\circle*{4.25}}}
\end{picture}
\hspace{3mm}
\\[.15in]
(a) \hspace{.3in} & \hspace{.3in} (b) 
\end{tabular}
\vspace{-.2in}
\end{center}
\caption{(a) The source labeling $\mathcal{F}_{\source}(G)$
of a reduced plabic graph~$G$.
(b) The target labeling $\mathcal{F}_{\target}(G)$. 
Here $\dpi_G = (5,\underline{2},\overline{3},6,4,1)$.
Every face is labeled by a subset of cardinality~$3$, in agreement with 
\cref{thm:faces}, cf.\ \cref{example:523641}.
}
\label{fig:plabic3}
\end{figure}
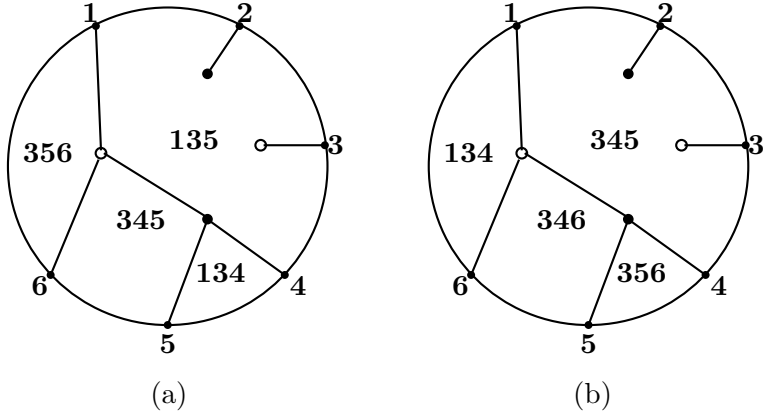
\begin{remark} 
\label{rem:edge-vs-face}
The edge labeling and the face labeling
of a reduced plabic graph $G$ are related as follows:
if two faces $f$ and $f'$ of $G$ are separated by a single edge
whose edge label is $\{i,j\}$, then the face label of $f'$
is obtained from that of $f$ by either removing $i$ and 
adding $j$, or removing $j$ and adding~$i$.
\end{remark}

%\begin{theorem}\label{thm:samesize}
%Let $G$ be a reduced plabic graph with $n$ boundary vertices,
%whose trip permutation has precisely $k$ anti-excedances.
%Then 
%  its face labeling $\mathcal{F}_{\source}(G)$ 
%	(or $\mathcal{F}_{\target}(G)$) assigns a subset of size $k$
%	to every face.
%\end{theorem}
\begin{theorem}
\label{thm:faces}
	Let $G$ be a reduced (leafless) plabic graph with $b$ boundary vertices.
Let $a$ denote the number of anti-excedances in the trip permutation~$\pi_G$.
Let us label the faces of $G$ using either the source or the target labeling. % see \cref{def:faces}.
Then every face of $G$ will be labeled by an $a$-element subset of $\{1,\dots,b\}$.
\end{theorem}

\pagebreak[3]

\begin{proof}
	By \cref{thm:resonance}, every reduced (leafless) plabic graph $G$ has the resonance
property, which in particular means that every edge label of $G$ 
consists of two distinct numbers.  It then follows from \cref{rem:edge-vs-face}
that every face label of $G$ has the same cardinality.  It remains
to show that this cardinality is~$a$, the number of anti-excedances of~$\dpi=\dpi_G$.

Furthermore, it is sufficient to establish the latter claim for one particular
reduced plabic graph with the trip permutation~$\dpi_G$,
e.g., for a bridge decomposition of~$\affpi$. 
Indeed, any two reduced plabic graphs with the same trip permutation
are related by local moves, and any such move 
preserves all labels except at most one, see \cref{ex:face-labels-square}.

To prove the theorem for bridge decompositions,
we use induction on the number of bridges~$\beta$.
In the base case $\beta=0$, 
the bridge decomposition~$G$ consists of $a$
white lollipops, $b-a$ black lollipops, and no bridges.  
Thus $G$ has a single face, labeled by the 
$a$-element subset indicating the positions of the white lollipops.

Consider a bridge decomposition $G$ built 
from a bridge decomposition $G'$ by adding a bridge
in position $(i,j)$ at the top of~$G'$, as in \cref{fig:resonance2}.  
Both $G$ and $G'$ have trip permutations with $a$ anti-excedances
(cf.\ \cref{rem:l(affpi)-bound}), 
so by the induction assumption, the faces in $G'$ have cardinality~$a$. 
Since $G$ inherits most of its faces from~$G'$, 
and all face labels of $G$ have the same cardinality, this cardinality is equal to~$a$. 
\end{proof}

%It is natural to examine how the face labels change under the various local moves in a plabic graph:  

\begin{exercise}
\label{ex:face-labels-square}
%Let $G$ be a plabic graph.
Verify that applying a move
{\rm (M2)} or {\rm (M3)} does not affect the face labels of a plabic graph, 
whereas applying the square move {\rm (M1)}  changes the  face labels
as shown in \cref{fig:M1Plucker}.  
%	(The target face labels change in 
%a very similar way.)
\end{exercise}

\begin{figure}[ht]%[htbp]
\begin{center}
\vspace{.1in}
\setlength{\unitlength}{1pt}
\begin{picture}(40,36)(0,0)
\thicklines
\multiput(1,0)(1,30){2}{\line(1,0){27.2}}
\multiput(0,1)(30,1){2}{\line(0,1){27.2}}
\put(45,45){\makebox(0,0){$\mathbf{i}$}}
\put(45,-15){\makebox(0,0){$\mathbf{j}$}}
\put(-15,-15){\makebox(0,0){$\mathbf{k}$}}
\put(-15,45){\makebox(0,0){$\mathbf{l}$}}
\put(42,15){\makebox(0,0){$ilS$}}
\put(15,-10){\makebox(0,0){$ijS$}}
\put(-14,15){\makebox(0,0){$jkS$}}
\put(15,40){\makebox(0,0){$klS$}}
\put(15,15){\makebox(0,0){$ikS$}}

\put(0,0){\circle*{4}}
\put(0,30){\circle{4}}
\put(30,30){\circle*{4}}
\put(30,0){\circle{4}}
\put(-10,40){\line(1,-1){8.5}}
\put(-10,-10){\line(1,1){9}}
\put(40,-10){\line(-1,1){8.5}}
\put(40,40){\line(-1,-1){9}}
\end{picture}
\qquad
\begin{picture}(40,36)(0,0)
\put(15,15){\makebox(0,0){\Large{$\longleftrightarrow$}}}
\end{picture}
\qquad
\begin{picture}(30,36)(0,0)
\thicklines
\multiput(1.7,0)(0,30){2}{\line(1,0){26.5}}
\multiput(0,1.7)(30,0){2}{\line(0,1){26.5}}
\put(45,45){\makebox(0,0){$\mathbf{i}$}}
\put(45,-15){\makebox(0,0){$\mathbf{j}$}}
\put(-15,-15){\makebox(0,0){$\mathbf{k}$}}
\put(-15,45){\makebox(0,0){$\mathbf{l}$}}
\put(42,15){\makebox(0,0){$ilS$}}
\put(15,-10){\makebox(0,0){$ijS$}}
\put(-14,15){\makebox(0,0){$jkS$}}
\put(15,40){\makebox(0,0){$klS$}}
\put(15,15){\makebox(0,0){$jlS$}}
\put(0,0){\circle{4}}
\put(0,30){\circle*{4}}
\put(30,30){\circle{4}}
\put(30,0){\circle*{4}}
\put(-10,40){\line(1,-1){9}}
\put(-10,-10){\line(1,1){8.5}}
\put(40,-10){\line(-1,1){9}}
\put(40,40){\line(-1,-1){8.5}}
\end{picture}
\end{center}
\vspace{.2in}
\caption{The effect of the square move (M1) on the  face labeling.  Here
$\mathbf{i,j,k,l}$ are the (source or target) labels of the trips that traverse
the outer edges towards the central square;
$S$~is an arbitrary set of labels disjoint from $\{i,j,k,l\}$; and
$abS$ is a shorthand for the set $\{a,b\} \cup S$.}
\label{fig:M1Plucker}
\end{figure}
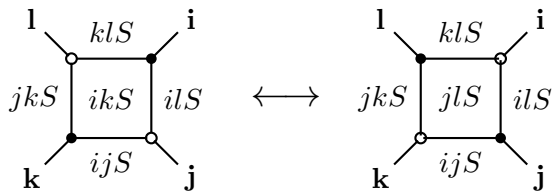

The face labelings of plabic graphs 
can be used to recover %other combinatorial labelings that we encountered before: 
the labelings of diagonals in a polygon by Pl\"ucker coordinates 
as well as the labelings of chambers in (ordinary or double) wiring diagrams by minors: 

\pagebreak[3]

\begin{exercise}
Let $T$ be a triangulation of a convex $m$-gon $\mathbf{P}_m$,
and let $G(T)$ be the plabic graph defined in \cref{TriangulationA}. 
Explain how to label the boundary vertices of $G(T)$ in such a way that the 
face labeling of $G(T)$ recovers the labeling of diagonals of $\mathbf{P}_m$ by pairs of integers.
Cf.\ \cref{fig:triang-plabic}. 
	%Show that the source face labeling of $G(T)$ recovers
	%the natural labeling of the triangulation $T$ induced by 
	%the labeling of the vertices of $\mathbf{P}_m$.
\end{exercise}

\begin{figure}[ht]
\begin{center}
\hspace{-.6in}
\setlength{\unitlength}{2.3pt}
\begin{picture}(60,65)(0,-3)
\thicklines
  \multiput(0,20)(60,0){2}{\line(0,1){20}}
  \multiput(20,0)(0,60){2}{\line(1,0){20}}
  \multiput(0,40)(40,-40){2}{\line(1,1){20}}
  \multiput(20,0)(40,40){2}{\line(-1,1){20}}

  \multiput(20,0)(20,0){2}{\circle*{1}}
  \multiput(20,60)(20,0){2}{\circle*{1}}
  \multiput(0,20)(0,20){2}{\circle*{1}}
  \multiput(60,20)(0,20){2}{\circle*{1}}

% \thinlines
\put(40,0){\line(1,2){20}}
\put(0,40){\line(1,0){60}}
\put(0,20){\line(2,-1){40}}
\put(0,40){\line(1,-1){40}}
\put(20,60){\line(2,-1){40}}

\put(15,16){\makebox(0,0){$P_{46}$}}
\put(23,22){\makebox(0,0){$P_{47}$}}
\put(46,22){\makebox(0,0){$P_{24}$}}
\put(34,42.5){\makebox(0,0){$P_{27}$}}
\put(30,52){\makebox(0,0){$P_{28}$}}
\put(40,-3){\makebox(0,0){$\mathbf{4}$}}
\put(63,20){\makebox(0,0){$\mathbf{3}$}}
\put(63,40){\makebox(0,0){$\mathbf{2}$}}
\put(40,63){\makebox(0,0){$\mathbf{1}$}}
\put(20,63){\makebox(0,0){$\mathbf{8}$}}
\put(-3,40){\makebox(0,0){$\mathbf{7}$}}
\put(-3,20){\makebox(0,0){$\mathbf{6}$}}
\put(20,-3){\makebox(0,0){$\mathbf{5}$}}
\end{picture}
	\hspace{1.7cm}\setlength{\unitlength}{2pt}
\begin{picture}(50,50)(0,-8)
\thicklines
% \thinlines
	\multiput(0,20)(60,0){2}{\red{\line(0,1){20}}}
	\multiput(20,0)(0,60){2}{\red{\line(1,0){20}}}
	\multiput(0,40)(40,-40){2}{\red{\line(1,1){20}}}
	\multiput(20,0)(40,40){2}{\red{\line(-1,1){20}}}
\put(40,0){\red{\line(1,2){20}}}
	\put(0,40){\red{\line(1,0){60}}}
	\put(0,20){\red{\line(2,-1){40}}}
	\put(0,40){\red{\line(1,-1){40}}}
	\put(20,60){\red{\line(2,-1){40}}}

\thicklines
  \multiput(20,0)(20,0){2}{\circle{3}}
  \multiput(20,60)(20,0){2}{\circle{3}}
  \multiput(0,20)(0,20){2}{\circle{3}}
  \multiput(60,20)(0,20){2}{\circle{3}}

\thicklines
\put(30,30){\circle{80}}

\put(32,28){\circle*{3}}
\put(12,21){\circle*{3}}
\put(22,5){\circle*{3}}
\put(55,21){\circle*{3}}
\put(20,48){\circle*{3}}
\put(38,55){\circle*{3}}

\put(32,28){\line(1,0.42){26.7}}
\put(32,28){\line(-1,0.38){30.5}}
\put(32,28){\line(0.3,-1){8}}

\put(12,21){\line(-0.65,1){11.5}}
\put(12,21){\line(1,-0.75){27}}
\put(12,21){\line(-1,-0.1){10.5}}

\put(22,5){\line(1,-0.27){16.4}}
\put(22,5){\line(-0.35,-1){1.3}}
\put(22,5){\line(-1,0.67){21}}

\put(55,21){\line(1, -0.15){3.8}}
\put(55,21){\line(0.27, 1){4.7}}
\put(55,21){\line(-0.7, -1){14}}

\put(20,48){\line(0, 1){10.5}}
\put(20,48){\line(-1, -0.4){18.5}}
\put(20,48){\line(5, -1){38.5}}

\put(38,55){\line(1, 2.5){1.5}}
\put(38,55){\line(-3.6, 1){16.5}}
\put(38,55){\line(4.4, -3){20.8}}

\multiput(20,0)(20,0){2}{\circle{3}}
\multiput(20,60)(20,0){2}{\circle{3}}
\multiput(0,20)(0,20){2}{\circle{3}}
\multiput(60,20)(0,20){2}{\circle{3}}
  
\put(19,-1){\line(-0.6,-1){3.8}}
\put(41,-1){\line(0.6,-1){3.8}}
  
\put(19,61){\line(-0.6,1){3.8}}
\put(41,61){\line(0.6,1){3.8}}
  
\put(-1, 19){\line(-1, -0.6){6.1}}
\put(-1, 41){\line(-1, 0.6){6.1}}

\put(61, 19){\line(1, -0.6){6.1}}
\put(61, 41){\line(1, 0.6){6.1}}

\put(46,-10){\makebox(0,0){$\mathbf{5}$}}
\put(69,14){\makebox(0,0){$\mathbf{4}$}}
\put(70,46){\makebox(0,0){$\mathbf{3}$}}
\put(46,70){\makebox(0,0){$\mathbf{2}$}}
\put(14,70){\makebox(0,0){$\mathbf{1}$}}
\put(-9.5,46){\makebox(0,0){$\mathbf{8}$}}
\put(-9.5,14){\makebox(0,0){$\mathbf{7}$}}
\put(14,-10){\makebox(0,0){$\mathbf{6}$}}

\end{picture}
\end{center}
	\caption{A triangulation $T$ of an octagon and the corresponding
	plabic graph $G(T)$, \emph{cf.}\ Figure~\ref{fig:quiver-triangulation}. 
}
\label{fig:triang-plabic}
\end{figure}

\begin{exercise}
Let $D$ be a wiring diagram with $m$ wires. 
Let $G(D)$ be the plabic graph defined in
\cref{def:wiringplabic}, see also \cref{fig:wiring-plabic}.
Label the boundary vertices of $G(D)$ by the numbers $1,\dots,2m$ 
in the clockwise order, starting with a $1$ at the lower left boundary vertex of $G(D)$.  
Label the faces of $G(D)$ using the source labeling $\mathcal{F}_{\source}(G)$. 
Show that intersecting each face label with the set $\{1,2,\dots,m\}$ recovers the labeling of $D$
by chamber minors.
See \cref{fig:wiring-plabic}. 
\end{exercise}

%\newsavebox{\ssone}
%\setlength{\unitlength}{2.4pt} 
%\savebox{\ssone}(10,20)[bl]{
%\thicklines 
%\qbezier(5,5)(7,10)(10,10)
%\qbezier(5,5)(3,0)(0,0)
%\qbezier(5,5)(3,10)(0,10)
%\qbezier(5,5)(7,0)(10,0)
%\put(0,20){\line(1,0){10}} 
%}

\begin{figure}
\begin{center}
\vspace{-.3in}
\begin{picture}(60,40)(0,-18) 
%\begin{picture}(30,20)(-6,3) 
\put(0,0){\makebox(0,0){\usebox{\ssone}}} 
\put(10,0){\makebox(0,0){\usebox{\sstwo}}} 
\put(20,0){\makebox(0,0){\usebox{\ssone}}} 
\put(-8,-10){\makebox(0,0){$1$}}
\put(-8,0){\makebox(0,0){$2$}}
\put(-8,10){\makebox(0,0){$3$}}
  \put(-5,-5){$_{\mathbf{\dark{1}}}$}
  \put(8,-5){$_{\mathbf{\dark{2}}}$}
  \put(23,-5){$_{\mathbf{\dark{3}}}$}

  \put(0,5){$_{\mathbf{\dark{12}}}$}
  \put(15,5){$_{\mathbf{\dark{23}}}$}
  
	\put(7.5,15){$_{\mathbf{\dark{123}}}$}
\end{picture}
\setlength{\unitlength}{1.2pt}
\begin{picture}(60,40)(0,-18)
\thicklines
\put(0,40){\line(1,0){28.5}}
\put(60,40){\line(-1,0){28.5}}
\put(0,20){\line(1,0){8.5}}
\put(11.5,20){\line(1,0){37}}
\put(60,20){\line(-1,0){8.5}}
\put(0,0){\line(1,0){60}}
\put(10,1.5){\line(0,1){17}}
\put(30,21.5){\line(0,1){17}}
\put(50,1.5){\line(0,1){17}}
\put(10,0){\circle{3}}
\put(30,20){\circle{3}}
\put(50,0){\circle{3}}
\put(10,20){\circle*{3}}
\put(30,40){\circle*{3}}
\put(50,20){\circle*{3}}
\put(-8,0){\makebox(0,0){$1$}}
\put(-8,20){\makebox(0,0){$2$}}
\put(-8,40){\makebox(0,0){$3$}}
\put(68,0){\makebox(0,0){$6$}}
\put(68,20){\makebox(0,0){$5$}}
\put(68,40){\makebox(0,0){$4$}}
  \put(-3,10){$_{\mathbf{\dark{145}}}$}
  \put(24,10){$_{\mathbf{\dark{245}}}$}
  \put(52,10){$_{\mathbf{\dark{345}}}$}

  \put(10,30){$_{\mathbf{\dark{124}}}$}
  \put(40,30){$_{\mathbf{\dark{234}}}$}
  
	\put(24,50){$_{\mathbf{\dark{123}}}$}
	\put(24,-10){$_{\mathbf{\dark{456}}}$}

\put(10,0){\white{\circle*{2.3}}}
\put(30,20){\white{\circle*{2.3}}}
\put(50,0){\white{\circle*{2.3}}}
\end{picture}
\vspace{-.25in}
\end{center}
	\caption{A wiring diagram $D$ and the plabic graph $G(D)$ with the source 
labeling of its faces.}
\label{fig:wiring-plabic}
\end{figure}
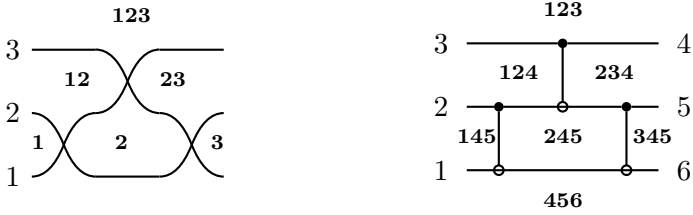

\pagebreak

\begin{exercise}
\label{exercise:DWD-source-labeling}
Let $D$ be a double wiring diagram with $m$ pairs of wires. 
Let $G(D)$ be the plabic graph defined in \cref{ex:DWD}.
Label the boundary vertices of $G(D)$ by the numbers
\begin{equation*}
1,2,\dots,m-1,m,m',\dots,2',1'
\end{equation*}
in clockwise order, starting with the label $1$ at the lower left boundary vertex of~$G(D)$.  
Label the faces of $G=G(D)$ using the source labeling $\mathcal{F}_{\source}(G)$, 
so that each face gets labeled by 
$I'\cup J$, where $I'\subset \{1',\dots,m'\}$ and $J\subset \{1,\dots,m\}$.  
Let $I$ denote the set obtained from $I'$ by replacing each $i'$ by $i$.
Show that mapping each face label $I'\cup J$ to the pair $([1,m]\setminus I, {J})$ 
recovers the labeling of $D$ by chamber minors.
See \cref{fig:better-0119}. 
\end{exercise}

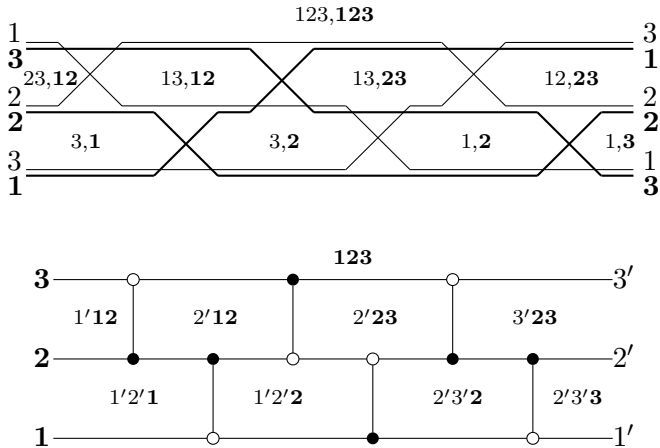
\begin{figure}[ht]
\setlength{\unitlength}{1.2pt}
\begin{center}
\begin{picture}(180,50)(6,0)
\thicklines
\dark{

  \put(0,0){\line(1,0){40}}
  \put(60,0){\line(1,0){100}}
  \put(180,0){\line(1,0){10}}
  \put(0,20){\line(1,0){40}}
  \put(60,20){\line(1,0){10}}
  \put(90,20){\line(1,0){70}}
  \put(180,20){\line(1,0){10}}
  \put(0,40){\line(1,0){70}}
  \put(90,40){\line(1,0){100}}

  \put(40,0){\line(1,1){20}}
  \put(70,20){\line(1,1){20}}
  \put(160,0){\line(1,1){20}}

  \put(40,20){\line(1,-1){20}}
  \put(70,40){\line(1,-1){20}}
  \put(160,20){\line(1,-1){20}}

  \put(193,-6){$\mathbf{3}$}
  \put(193,14){$\mathbf{2}$}
  \put(193,34){$\mathbf{1}$}

  \put(-6,-6){$\mathbf{1}$}
  \put(-6,14){$\mathbf{2}$}
  \put(-6,34){$\mathbf{3}$}
}

\light{
\thinlines

  \put(0,2){\line(1,0){100}}
  \put(120,2){\line(1,0){70}}
  \put(0,22){\line(1,0){10}}
  \put(30,22){\line(1,0){70}}
  \put(120,22){\line(1,0){10}}
  \put(150,22){\line(1,0){40}}
  \put(0,42){\line(1,0){10}}
  \put(30,42){\line(1,0){100}}
  \put(150,42){\line(1,0){40}}

  \put(10,22){\line(1,1){20}}
  \put(100,2){\line(1,1){20}}
  \put(130,22){\line(1,1){20}}

  \put(10,42){\line(1,-1){20}}
  \put(100,22){\line(1,-1){20}}
  \put(130,42){\line(1,-1){20}}

  \put(193,2){${1}$}
  \put(193,22){${2}$}
  \put(193,42){${3}$}

  \put(-6,2){${3}$}
  \put(-6,22){${2}$}
  \put(-6,42){${1}$}

  \put(14,10){$_{\light{3},\mathbf{\dark{1}}}$}
  \put(76,10){$_{\light{3},\mathbf{\dark{2}}}$}
  \put(136,10){$_{\light{1},\mathbf{\dark{2}}}$}
  \put(181,10){$_{\light{1},\mathbf{\dark{3}}}$}

  \put(-1,30){$_{\light{23},\mathbf{\dark{12}}}$}
  \put(42,30){$_{\light{13},\mathbf{\dark{12}}}$}
  \put(102,30){$_{\light{13},\mathbf{\dark{23}}}$}
  \put(162,30){$_{\light{12},\mathbf{\dark{23}}}$}
  \put(84,50){$_{\light{123},\mathbf{\dark{123}}}$}
}
\end{picture}
		\end{center}
			\vspace{.5cm}
		\begin{center}
	%%%%%%%%%%%%%%%%%%%%%%%%%%%%%%%%%%%%%%%%%%%%%%%%%%%%%%%%
    \setlength{\unitlength}{1.5pt}
\begin{picture}
% \thicklines
            (140,45)(0,-5)
  \put(-5,38){$\mathbf{3}$}
  \put(-5,18){$\mathbf{2}$}
  \put(-5,-2){$\mathbf{1}$}
			\put(140,38){${3'}$}
  \put(140,18){${2'}$}
  \put(140,-2){${1'}$}
				\put(70,45){$_{\mathbf{\dark{123}}}$}
  \put(14,10){$_{\light{1'2'}\mathbf{\dark{1}}}$}
  \put(50,10){$_{\light{1'2'}\mathbf{\dark{2}}}$}
  \put(95,10){$_{\light{2'3'}\mathbf{\dark{2}}}$}
  \put(125,10){$_{\light{2'3'}\mathbf{\dark{3}}}$}

  \put(5,30){$_{\light{1'}\mathbf{\dark{12}}}$}
  \put(35,30){$_{\light{2'}\mathbf{\dark{12}}}$}
  \put(75,30){$_{\light{2'}\mathbf{\dark{23}}}$}
  \put(115,30){$_{\light{3'}\mathbf{\dark{23}}}$}
				\thicklines
				\put(0,20){\line(1,0){58.5}}
				\put(61.5,20){\line(1,0){17}}
				\put(81.5,20){\line(1,0){58.5}}
				\put(0,0){\line(1,0){38.5}}
				\put(140,0){\line(-1,0){18.5}}
				\put(118.5,0){\line(-1,0){77}}
				\put(0,40){\line(1,0){18.5}}
				\put(21.5,40){\line(1,0){37}}
				\put(61.5,40){\line(1,0){37}}
				\put(101.5,40){\line(1,0){38.5}}
				
				\put(20,21.5){\line(0,1){17}}
				\put(80,1.5){\line(0,1){17}}
				\put(60,21.5){\line(0,1){17}}
				\put(40,1.5){\line(0,1){17}}
				\put(100,21.5){\line(0,1){17}}
				\put(120,1.5){\line(0,1){17}}
				
				\put(20,20){\circle*{3}}
				\put(40,20){\circle*{3}}
				\put(60,20){\circle{3}}
				\put(80,20){\circle{3}}
				\put(100,20){\circle*{3}}
				\put(120,20){\circle*{3}}
				
				\put(20,40){\circle{3}}
				\put(60,40){\circle*{3}}
				\put(100,40){\circle{3}}
				
				\put(40,0){\circle{3}}
				\put(80,0){\circle*{3}}
				\put(120,0){\circle{3}}
				% \thicklines
				% \linethickness{1.5pt}
\end{picture}
\end{center}
\caption{Double wiring diagram labeling from plabic graphs.
The~labeling of a double wiring diagram $D$ 
is obtained from the source labeling of the associated plabic graph $G(D)$ 
%as follows. 
%Each subset $I' \cup J$ (here $I'\subset \{1',2',\dots\}$ and $J\subset \{1,2,\dots\}$)
%is mapped to the pair $([1,m]\setminus I, \mathbf{J})$,
%where $I$ denotes the set obtained from $I'$ by mapping each $i'$ to~$i$.
using the recipe described in \cref{exercise:DWD-source-labeling}.}
\label{fig:better-0119}
\end{figure}

%\clearpage
\newpage

\section{Grassmann necklaces and weakly separated collections}\label{weaksep}

Fix two nonnegative integers $b$ and $a\leq b$.  
We denote by $\binom{[b]}{a}$ the set of all $a$-element subsets of $\{1,\dots,b\}$. 

In this section, we provide an intrinsic combinatorial characterization
of the subsets of $\binom{[b]}{a}$ that arise as sets of face labels of reduced plabic graphs.
The proofs are omitted. 

\begin{definition}[{\rm \cite{leclerc-zelevinsky}}]
We say that two $a$-element subsets $I,J\in\binom{[b]}{a}$
are \emph{weakly separated} if 
 and only if,
after drawing the numbers $1,2,\dots,b$ clockwise around a circle,
there exists a chord separating the sets $I\setminus J$ and $J\setminus I$ from each other.
More specifically, $I$ and $J$ are weakly separated if 
there do not exist $i, j, i', j'\in\{1,\dots,b\}$ such that
\begin{itemize}[leftmargin=.2in]
\item
%$i, j, i', j'$ are cyclically ordered, i.e., 
$i<j<i'<j'$ or $j<i<j'<i'$; % or $i'<j'<i<j$ or $j'<i<j<i'$; 
\item 
$i,i'\in I \setminus J$ and $j,j'\in J\setminus I$.
\end{itemize}
\end{definition}

\begin{theorem}[{\rm \cite{danilov-karzanov-koshevoy,oh-postnikov-speyer}}]
\label{th:face-labels-weakly-sep}
Let $I$ and $J$ be target face labels of two faces in a 
reduced plabic graph. 
Alternatively, let $I$ and $J$ be source face labels of two faces in a 
reduced plabic graph. 
Then $I$ and~$J$ are weakly separated.
\end{theorem}

\begin{definition}
\label{def:wekley-sep-collection}
A collection  $\mathcal{C} \subset \binom{[b]}{a}$
of $a$-element subsets of $[b]$ is \emph{weakly separated} if 
any $I,J\!\in\!\mathcal{C}$ are weakly separated.  
Thus, \cref{th:face-labels-weakly-sep} asserts that 
	the collection of target (or source) face labels of a 
reduced plabic graph is weakly separated.
%
%Moreover these collections are maximal in a sense
%	that will be made precise in \cref{thm:ops} and \cref{thm:ops2}.
%
A weakly separated collection $\mathcal{C}$ is called \emph{maximal} 
if it is not contained in any other weakly separated collection.
See \cref{fig:plabic4}.
\end{definition}

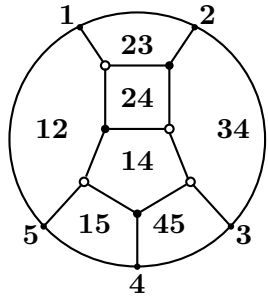
\begin{figure}[ht]%[htbp]
\begin{center}
\vspace{-.05in}
\setlength{\unitlength}{.8pt}
\begin{picture}(100,120)(-40,-70)
\thicklines
\multiput(1,0)(1,30){2}{\line(1,0){27.5}}
\multiput(0,1)(30,1){2}{\line(0,1){27.5}}
\put(0,30){\circle{4}}
\put(30,30){\circle*{4}}
\put(30,0){\circle{4}}
\put(0,0){\circle*{4}}
\put(-10,-25){\circle{4}}
\put(15,-40){\circle*{4}}
\put(40,-25){\circle{4}}
\put(0,0){\line(-10,-25){9}}
\put(30,-1.5){\line(10,-25){9}}
\put(15,-40){\line(-25,15){23}}
\put(15,-40){\line(25,15){23}}
\put(15,-5){\circle{120}}
\put(15,-65){\circle*{3}}
\put(15,-40){\line(0,-1){24}}
\put(-12,48){\circle*{3}}
\put(42,48){\circle*{3}}
\put(-12,48){\line(12,-18){11}}
\put(42,48){\line(-12,-18){11}}
\put(-29,-46){\circle*{3}}
\put(59,-46){\circle*{3}}
\put(-29,-46){\line(19,21){18}}
\put(59,-46){\line(-19,21){18}}
\put(-18,54){\makebox(0,0){$\mathbf{1}$}}
\put(48,54){\makebox(0,0){$\mathbf{2}$}}
\put(65,-50){\makebox(0,0){$\mathbf{3}$}}
\put(15,-73){\makebox(0,0){$\mathbf{4}$}}
\put(-35,-50){\makebox(0,0){$\mathbf{5}$}}
\put(15,40){\makebox(0,0){$\mathbf{23}$}}
\put(15,15){\makebox(0,0){$\mathbf{24}$}}
\put(15,-15){\makebox(0,0){$\mathbf{14}$}}
\put(60,0){\makebox(0,0){$\mathbf{34}$}}
\put(30,-45){\makebox(0,0){$\mathbf{45}$}}
\put(-5,-45){\makebox(0,0){$\mathbf{15}$}}
\put(-25,0){\makebox(0,0){$\mathbf{12}$}}
\end{picture}
\vspace{-.1in}
\end{center}
\caption{The target face labeling of the 
reduced plabic graph $G$ from \cref{fig:plabic}(a).
The set of labels $\{12, 23, 34, 45, 15, 24, 14\}$
is a maximal weakly separated collection in~$\binom{[5]}{2}$.
Here $\dpi_G = \dpi_{2,5}=(3,4,5,1,2)$.}
\label{fig:plabic4}
\end{figure}

\begin{theorem}[{\rm \cite{danilov-karzanov-koshevoy,oh-postnikov-speyer}}]
\label{thm:ops}
For $\mathcal{C} \subset \binom{[b]}{a}$, the following are equivalent:
\begin{itemize}[leftmargin=.2in]
\item 
$\mathcal{C}$ is a maximal weakly separated collection; 
\item
$\mathcal{C}$ is the set of target face labels of a reduced plabic graph~$G$
with $\dpi_G\!=\!\dpi_{a,b}$ (see \eqref{eq:dpi-ab}). 
\end{itemize}
In that case, the cardinality of $\mathcal{C}$ is equal to  $|\mathcal{C}| = a(b-a)+1$.
\end{theorem}

\begin{remark}
A general formula for the number of maximal 
weakly separated collections in $\binom{[b]}{a}$ is unknown.
For $a=2$, the maximal weakly separated collections in 
$\binom{[b]}{2}$ are in bijection with triangulations of a convex $b$-gon, 
so they are counted by the Catalan numbers~$C_{b-2}$, where 
$C_n = \frac{1}{n+1} \binom{2n}{n}$.
For $a=3$ and $b=6,\dots, 12$,  the number of maximal weakly separated collections 
in $\binom{[b]}{3}$ is equal to 
$34, 259, 2136, 18600, 168565, 1574298, 15051702$.
See~\cite{Early} for more data.  
\end{remark}

\cref{thm:ops} can be generalized to arbitrary reduced
plabic graphs.  To state this result, we will need the following notion.

\begin{definition}[{\cite[Definition 16.1]{postnikov}}] 
A \emph{Grassmann necklace} of type $(a,b)$ is a sequence
$\mathcal{I} = (I_1,\dots,I_b)$ of subsets $I_i\in \binom{[b]}{a}$ 
such that, for $i=1,\dots,b$, we have 
$I_{i+1}\supset I_i\setminus \{i\}$.
(Here the indices are taken modulo $b$, so that $I_1\supset I_b\setminus \{b\}$.)  
Thus, if $i\notin I_i$, then $I_{i+1}=I_i$.
\end{definition}

In other words, either $I_{i+1}=I_i$ 
or $I_{i+1}$ is obtained from $I_i$ by deleting
$i$ and adding another element.  
Note that if $I_{i+1}=I_i$, then either $i$ belongs to all elements
$I_j$ of the necklace, or $i$ belongs to none of them.

\begin{example}
\label{ex:grassmann-necklace}
The sequence $\mathcal{I}=(126, 236, 346, 456, 156, 126)$ is 
a Grassmann necklace of type $(3,6)$.
\end{example}

\begin{definition}
\label{def:ell-order}
For $\ell\in\{1,\dots,b\}$, 
we define the linear order $<_{\ell}$ on $\{1,\dots,b\}$ as follows:
\begin{equation*}
\ell <_{\ell} \ell+1 <_{\ell} \ell+2 <_{\ell} \dots <_{\ell}b <_{\ell} 1 <_{\ell} \dots 
<_{\ell} \ell-1.
\end{equation*}

For a decorated permutation $\dpi$ on $b$ letters, we say that 
$i\in \{1,\dots,b\}$ is an \emph{$\ell$-anti-excedance} of $\dpi$ if 
either $\dpi^{-1}(i) >_\ell i$ or if $\dpi(i)=\overline{i}$.
Thus, a $1$-anti-excedance is the same as an (ordinary) anti-excedance, 
as in \cref{def:anti}.  
\end{definition}

It is not hard to see that the
number of $\ell$-anti-excedances does not depend on the choice of $\ell\in \{1,\dots,b\}$,
so we simply refer to this quantity as the number of anti-excedances.

\begin{lemma}
\label{lem:decpermnecklace}
Decorated permutations 
on $b$ letters with $a$ anti-excedances are in bijection with 
 Grassmann necklaces
$\mathcal{I}$ 
of type $(a,b)$.  
\end{lemma}
\begin{proof}
	To go from $\mathcal{I}$ to the corresponding
decorated permutation $\dpi = \dpi(\mathcal{I})$, we set
$\dpi(i)=j$ whenever $I_{i+1}=(I_i\setminus \{i\}) \cup \{j\}$
for $i\neq j$.  
If $i\notin I_i=I_{i+1}$ then $\dpi(i) = 
\underline{i}$, and 
 if $i\in I_i=I_{i+1}$ then $\dpi(i) = 
\overline{i}$.  

Going in the other direction, let $\dpi$ be a decorated permutation. 
For $\ell\in\{1,\dots,b\}$, we denote by $I_\ell$ the set of $\ell$-anti-excedances of~$\dpi$.  
Then $\mathcal{I} =\mathcal{I}(\dpi)= (I_1,\dots, I_b)$ is the corresponding Grassmann necklace.
\end{proof}

%NOTE: $\mathcal{I}$ will correspond to the face labels of the boundary regions of $G$
%when we use the \emph{TARGET} labels of faces.  Need to discuss both of them here.

\begin{example}
Let $\mathcal{I}\!=\!(126, 236, 346, 456, 156, 126)$, cf.\ \cref{ex:grassmann-necklace}.
Then $\dpi(\mathcal{I}) = (3,4,5,1,2,\overline{6})$.
\end{example}

\begin{example}
\label{ex:pi-ab-necklace}
Let $\dpi_G=\dpi_{a,b}$, cf.~\eqref{eq:dpi-ab}. 
The corresponding Grassmann necklace (cf.\ \cref{lem:decpermnecklace})
is given by
\begin{equation}
\label{eq:I-dpi-ab}
\mathcal{I}(\dpi_{a,b})=(\{1,2,\dots,a\}, \{2,3,\dots,a,a+1\}, \dots, \{b,1,2,\dots, a-1\}). 
\end{equation}
\end{example}

\begin{definition}
\label{def:ell-partial-order}
We extend the linear order $<_{\ell}$ on $\{1,\dots,b\}$ 
to a partial order on~$\binom{[b]}{a}$, as follows.
Let
\begin{align*}
&I=\{i_1,\dots,i_a\}, \,\,\quad i_1<_{\ell} i_2 <_{\ell} \dots <_{\ell} i_a; \\
&J=\{j_1,\dots,j_a\}, \quad j_1<_{\ell} j_2 <_{\ell} \dots <_{\ell} j_a. 
\end{align*}
Then, by definition, $I \leq_{\ell} J$ 
if and only if 
$i_1 \leq_{\ell} j_1,\dots, i_a \leq_{\ell} j_a$.
\end{definition}

\begin{definition}
\label{def:necklace-to-positroid}
	For a Grassmann necklace $\mathcal{I}=(I_1,\dots,I_b)$ of type $(a,b)$, we define the associated 
\emph{positroid} $\mathcal{M}_{\mathcal{I}}$ by
\begin{equation*}
\mathcal{M}_{\mathcal{I}}=\{J\in \textstyle\binom{[b]}{a} \ \vert \ I_\ell \leq_{\ell}
J \text{ for all }\ell \in \{1,\dots,b\}\}.
\end{equation*}
\end{definition}

As we will see in a subsequent chapter, % ~\ref{ch:Grassmannians},  
positroids are the (realizable) \emph{matroids} that 
arise from full rank $a\times b$  matrices with all Pl\"ucker coordinates nonnegative.
Abstractly, one may also define a \emph{positively oriented matroid} to be an oriented
matroid on $\{1,2,\dots,b\}$ 
whose chirotope takes nonnegative values on any ordered subset $\{i_1< \dots < i_a\}$.
By \cite{Ardila-Rincon-Williams}, these two notions are the same, in other words,
every positively oriented matroid is realizable.  

\begin{example}
\label{ex:positroid-ab}
Let $\mathcal{I}=\mathcal{I}(\dpi_{a,b})$, see~\eqref{eq:I-dpi-ab}. 
Then $\mathcal{M}_{\mathcal{I}}=\binom{[b]}{a}$,
i.e., the positroid associated with~$\mathcal{I}$ contains all $a$-element subsets of $\{1,\dots,b\}$. 
\end{example}

\begin{definition}
\label{def:strongly equivalent}
Two reduced plabic graphs are called \emph{strongly equivalent}
if they have the same sets of face labels.  

We note that two plabic graphs which are connected
via moves (M2) and (M3) are strongly equivalent.
\end{definition}

Recall that $\mathcal{F}_{\target}(G)$ denotes the collection of target-labels of faces of 
a reduced plabic graph $G$.

\begin{theorem}[{\cite[Theorem~1.5]{oh-postnikov-speyer}}] 
\label{thm:ops2}
Fix a decorated permutation $\dpi$ on $b$ letters with $a$ anti-excedances.
Let~$\mathcal{I}$ be the corresponding Grassmann necklace of type $(a,b)$, 
cf.\ \cref{lem:decpermnecklace}. 
Let $\mathcal{M}_{\mathcal{I}}$ be the associated positroid, 
cf.\ \cref{def:necklace-to-positroid}.
Then the map $G \mapsto \mathcal{F}_{\target}(G)$ gives a bijection~between
\begin{itemize}[leftmargin=.2in]
\item 
the strong equivalence classes of reduced plabic graphs $G$ 
with decorated trip permutation $\dpi_G = \dpi$ and 
\item
the collections $\mathcal{C} \subset \binom{[b]}{a}$ that are maximal (with respect to inclusion) 
among the weakly separated collections satisfying 
$\mathcal{I} \subseteq \mathcal{C} \subseteq \mathcal{M}_{\mathcal{I}}$.
\end{itemize}
\end{theorem}

\begin{remark}
Let $\mathcal{I}=\mathcal{I}(\dpi_{a,b})$. 
Then $\mathcal{I}=\mathcal{I}(\dpi_{a,b})$ is given by~\eqref{eq:I-dpi-ab}. 
Each of the $b$ cyclically consecutive subsets in \eqref{eq:I-dpi-ab}
is weakly separated from every other $a$-element subset of $\{1,\dots,b\}$,
so every maximal weakly separated collection $\mathcal{C} \subset \binom{[b]}{a}$ 
must contain~$\mathcal{I}$. 
Furthermore, $\mathcal{M}_{\mathcal{I}}=\binom{[b]}{a}$ (see \cref{ex:positroid-ab}),
so any such~$\mathcal{C}$ automatically  
satisfies the inclusions $\mathcal{I} \subseteq \mathcal{C} \subseteq \mathcal{M}_{\mathcal{I}}$.
We thus recover \cref{thm:ops} as a special case of \cref{thm:ops2}.
\end{remark}

%%% Local Variables:
%%% mode: latex
%%% TeX-master: book.tex
%%% End:

\backmatter

%%%%%%\include{references}
%\include{bibliography}

%bibliographystyle{amsalpha}
\bibliographystyle{acm}
\bibliography{bibliography}
\label{sec:biblio}

%RESTORE IN THE FINAL VERSION:
%\printindex

\end{document}